\documentclass{scrartcl}
\usepackage{command_MU}
\usepackage{algorithm, algpseudocode} % use algorithm and algorithmicx for typesetting algorithms
\usepackage{mathrsfs} % for \mathscr command

\usepackage{amsmath}
\usepackage{amsthm}
\usepackage{amssymb}
\usepackage[shortlabels]{enumitem}
\usepackage{stfloats}
\usepackage{siunitx}
\usepackage[utf8]{inputenc}
\usepackage{booktabs}
\usepackage{hyperref}
\usepackage[sort&compress,numbers,square]{natbib}
\theoremstyle{plain}

\theoremstyle{definition}
\newtheorem{remark}{Remark}

\usepackage{graphicx, color}
\usepackage{subcaption}
\usepackage{tikz}
\usetikzlibrary{spy}
\DeclareMathOperator*{\argmax}{arg\,max}

\title{Stability of Data-Dependent Ridge-Regularization\\ for Inverse Problems}
\author{Sebastian Neumayer \and Fabian Altekrüger}
\date{\today}

\begin{document}

\maketitle
\begin{abstract}
    Theoretical guarantees for the robust solution of inverse problems have important implications for applications.
    To achieve both guarantees and high reconstruction quality, we propose learning a pixel-based ridge regularizer with a data-dependent and spatially varying regularization strength.
    For this architecture, we establish the existence of solutions to the associated variational problem and the stability of its solution operator.
    Further, we prove that the reconstruction forms a maximum-a-posteriori approach.
    Simulations for biomedical imaging and material sciences demonstrate that the approach yields high-quality reconstructions even if only a small instance-specific training set is available.
    
    \vspace{.1cm}\noindent \textbf{Keywords}: biomedical imaging, data-dependent regularization, data-driven priors, set-valued maps, spatial adaptivity, superresolution
\end{abstract}
\section{Introduction}
Inverse problems are ubiquitous in biomedical imaging \cite{MM2019}, with examples such as magnetic resonance imaging and computed tomography.
There, we aim to reconstruct an unknown image $\V x \in \R^n$ from noisy measurements $\V y = \mathrm{cor}(\M H \V x)\in \R^m$, where $\M H \in \R^{m, n}$ encodes the data-acquisition process, and the (possibly non-additive) noise model $\mathrm{cor}\colon \R^m \to \R^m $ accounts for imperfections in this description.
One of the most common reconstruction methods is to solve a variational problem of the form
\begin{equation}\label{eq:var_prob}
    \hat{\V x} \in \argmin_{\V x \in \mathcal X} \bigl( 
   %\frac12 \Vert\M H \V x - \V y \Vert^2 
    \mathcal{D} (\M H \V x, \V y)
    + \lambda \mathcal{R} (\V x) \bigr),
\end{equation}
with a data fidelity $\mathcal{D} \colon \R^m \times \R^m \to \R_{\geq 0}$, a regularizer $\mathcal{R}\colon \R^n \to \R_{\geq 0}$ and $\mathcal X \subset \R^n$.
In \eqref{eq:var_prob}, the data fidelity ensures the (approximate) consistency of the reconstruction $\hat{\V x}$ with the observed data $\V y$, while the regularizer $\mathcal R$, whose strength is controlled by $\lambda \in \R_{>0}$, promotes favorable properties of $\hat{\V x}$.

For many acquisition operators $\M H$ and noise models $\mathrm{cor}$, well-suited loss functions $\mathcal{D}$ exist \cite{ribes2008linear,Scherzer2009}.
While $\mathcal{D}$ is instance-specific, $\mathcal{R}$ is preferably agnostic to $\M H$ and $\mathrm{cor}$.
Then, $\mathcal{R}$ depends exclusively on the properties of the underlying images, and we can deploy it for any inverse problem.
In contrast, we can adapt $\mathcal R$ to $\M H$ by using adversarial regularization \cite{lunz2018adversarial,MukDit2021,SBMS2024} or the network Tikhonov \cite{LiSch2020}.
Operator agnostic attempts date back to the Tikhonov regularization \cite{tikhonov1963}, where images are modeled as smooth signals.
Over the years, better regularizers such as total variation \cite{rudin1992nonlinear}, compressed sensing \cite{donoho2006compressed}, and the field-of-experts \cite{RotBla2009} have been developed.
More recently, learned regularizers such as patch priors \cite{ADHHMS2023,prost2021learning}, plug-and-play priors \cite{HNS2021,hurault2022gradient,HurChaLec2023}, or total deep variation \cite{KobEff2020,KEKP2021}
led to huge improvements in reconstruction quality.
To maintain interpretability, the works \cite{GouNeuBoh2022,GouNeuUns2023} learn so-called ridge regularizers
\begin{equation}\label{eq:ridge_reg}
    \mathcal{R} \colon\V{x}\mapsto \sum_{c=1}^{ N_\mathrm{C}} \bigl \langle \boldsymbol 1_n, \boldsymbol \psi_c (\M W_{c} \V x)\bigr\rangle,
\end{equation}
with convolution matrices \smash{$\M W_{c} \in \R^{n,n}$} and potentials \smash{$\boldsymbol \psi_c (\V x) = (\psi_c(x_k))_{k=1}^n$ with $\psi_c \in \mathcal{C}_{\geq0}^{1,1}(\R)$}, where $\mathcal{C}_{\geq0}^{1,1}(\R)$ is the space of nonnegative differentiable functions with Lipschitz-continuous derivatives.
As shorthand, we introduce the notation \smash{$\M W  = (\M W_{c})_{c=1}^{N_C}$} and \smash{$\boldsymbol \psi  = (\boldsymbol \psi_{c})_{c=1}^{N_C}$}.
In \cite{GouNeuBoh2022}, they deploy convex $\psi_c$ ($\psi_c^{\prime\prime} \geq 0$) in \eqref{eq:ridge_reg}, which leads to a convex $\mathcal{R}$.
Later, it was found that the relaxation to $1$-weakly convex $\psi_c$ ($\psi_c^{\prime\prime} > -1$) significantly boosts the performance \cite{GouNeuUns2023}.
If $\lambda \leq 1$ and $\psi_c^{\prime\prime} > -1$, the objective \eqref{eq:var_prob} for denoising ($\M H = \M I$) remains convex.
In particular, there are no local minima.
Due to its clear interpretability as a filter-based model, we choose \eqref{eq:ridge_reg} as the starting point for our investigation.

So far, the regularization strength $\boldsymbol 1_n$ in \eqref{eq:ridge_reg} is the same for each entry in $\M W_{c} \V x$.
However, starting with total-variation regularization \cite{DonSch2020,HiPaRa2017,ChReSc2017}, it has been observed in several works (see, for example, \cite{KofAltBa2023,lefkimmiatis2023learning,ZhLiDu2022} in the data-driven setting) that the use of a spatially-adaptive regularization strength (as opposed to $\boldsymbol 1_n$) can significantly boost the performance of the reconstruction model \eqref{eq:var_prob}.
Following this idea, a data-dependent generalization of \eqref{eq:ridge_reg} was recently proposed in \cite{NeuPouGou2023} as
\begin{equation}\label{eq:rridge_reg_mask}
    \mathcal{R}_{\V y} \colon\V{x}\mapsto \sum_{c=1}^{ N_\mathrm{C}} \bigl \langle \boldsymbol \Lambda_c(\V y), \boldsymbol \psi_c (\M W_{c} \V x)\bigr\rangle,
\end{equation}
where the mask $\boldsymbol \Lambda = (\M \Lambda_c)_{c=1}^{N_C}$ with $\boldsymbol \Lambda_c\colon \R^n \to [\epsilon,1]^n$ and $\epsilon>0$ makes \eqref{eq:rridge_reg_mask} data-dependent without increasing the weak-convexity modulus compared to $\mathcal R$.
So far, only weakly convex $\psi_c$ have been studied, likely because they lead to a better baseline architecture \eqref{eq:ridge_reg}.
For \eqref{eq:rridge_reg_mask}, the spatial adaptivity can be motivated as follows.
First, the noise might be non-uniform across the data, and locally adapting the regularization strength is beneficial.
Further, the pixel-wise cost $\mathcal R_{\V y}$ often over-penalizes structure, which leads to smoothing.
Since structure might only respond to certain $\M W_c$, using different masks across the channels $c$ is reasonable.
For an overview, we refer to \cite{PraCalLan2023}.

Today, classic signal processing approaches such as \cite{tikhonov1963,rudin1992nonlinear,donoho2006compressed} often serve as a basic benchmark.
With this choice, the reconstruction model \eqref{eq:var_prob} is stable and converges for vanishing noise \cite{Scherzer2009,PlaNeuUns2023}.
Guarantees of this type play a crucial role in applications such as medical imaging, where diagnostic decisions are based on the provided reconstructions $\hat{\V x}$.
While the classic approaches are well-analyzed, many questions remain open for learned ones.
In particular, learned approaches might hallucinate or remove meaningful structures \cite{zbontar2018fastMRI}.
If the learned $\mathcal R$ is independent of the data $\V y$, some (weaker) theoretical guarantees exist \cite{LiSch2020,MukHauOek2023,ObmHal2023,SBMS2024}.
For \emph{data-dependent} regularizers $\mathcal R_{\V y}$, such as the one in \eqref{eq:rridge_reg_mask}, an analysis is to the best of our knowledge unavailable.
While some ideas from previous work are transferable, other aspects need a completely new treatment.

\paragraph{Contribution and Outline}
In Section~\ref{sec:TheoProp}, we prove nonemptiness, stability and convergence for vanishing noise for the data-to-reconstruction map $S \colon \R^m  \rightrightarrows \R^n$ associated with \eqref{eq:var_prob}, where we deploy the $\mathcal R_{\V y}$ from \eqref{eq:rridge_reg_mask}. 
To deal with the non-uniqueness of critical points for \eqref{eq:var_prob}, we use the theory of implicit multifunctions  \cite{DonRoc2009,Ioffe2017}, which enables a generic stability analysis \cite{GfrOut2016,CuoKru2021}.
Due to the specific structure of $\mathcal R_{\V y}$, we can derive even stronger results with tools from variational analysis.
These are also stronger compared to previous work on learned regularization.
The relevant stability concepts for set-valued maps and the necessary tools from variational analysis are summarized in Section~\ref{sec:Prelim}.
All stability concepts generalize the Lipschitz continuity of single-valued functions $S \colon \R^m  \to \R^n$.

Let us briefly discuss the stability results for $S$ ordered according to their strength.
The strongest one (see Section~\ref{sec:StabInv}) is obtained if we can characterize $S$ by the implicit function theorem.
Then, its values are isolated points, and each one is Lipschitz continuous in the data $\V y$.
Unfortunately, the theoretical requirements for this are hard to guarantee in practice.
Therefore, we also investigate two other settings.
In Section~\ref{sec:StabPoly}, we establish a weaker notion of continuity for the case $\boldsymbol \Lambda_c(\V y) = \boldsymbol 1_n$ and piecewise quadratic profiles $\psi_c$.
Then, $S$ is \emph{piecewise linear}, which ensures many favorable properties that general set-valued maps do not have.
As a second setting, we analyze (data-dependent) regularizers $\mathcal R_{\V y}$ with convex $\psi_c$ in Section~\ref{sec:StabConv}.
First, we sharpen the results from Section~\ref{sec:StabPoly} to Lipschitz continuity using the convexity.
Further, for a data-dependent mask  $\boldsymbol \Lambda(\V y)$, we establish a localized notion of Lipschitz continuity.
To complement the stability analysis, we consider the vanishing noise setting in Section~\ref{sec:VanishNoise}.
After this well-posedness discussion, a Bayesian interpretation of the model \eqref{eq:var_prob} is established in Section~\ref{sec:Bayesian}.
Here, we show that the regularizer \eqref{eq:rridge_reg_mask} induces a probability density and that \eqref{eq:var_prob} is a maximum-a-posteriori approach.

Our theoretical results for convex $\psi_c$ in \eqref{eq:rridge_reg_mask} are stronger than for weakly convex ones ($\psi_c^{\prime\prime} > -1$).
Hence, one might wonder how this restriction affects practical performance.
For the baseline model \eqref{eq:ridge_reg} with $\boldsymbol \Lambda_c(\V y) = \boldsymbol 1_n$, a possible explanation for the performance improvement is that weakly convex $\psi_c$ reduce the over-penalization of structure.
The data-dependent mask $\boldsymbol \Lambda$ is an alternative approach to resolve this issue.
To understand the interactions of both strategies, we learn a $\boldsymbol \Lambda$ for both convex and weakly convex $\mathcal{R}_{\V y}$ of the form \eqref{eq:rridge_reg_mask}, and compare their denoising performance in Section~\ref{sec:Denoising}.
Interestingly, the results are very similar and far superior to the weakly convex baseline \eqref{eq:ridge_reg}.

Encouraged by these findings, we extend the experimental validation for magnetic resonance imaging from \cite{NeuPouGou2023}, which is revisited in Section \ref{sec:mri}, to two new tasks, namely computed tomography in Section~\ref{sec:ct} and superresolution of material microstructures in Section~\ref{sec:SiC}.
For these tasks, we usually lack a large training database because of privacy preservation in medical imaging or destructive/costly processes in material sciences.
Hence, we train the data-dependent $\mathcal R_{\V y}$ on a generic image database such as BSDS500 to get a universal baseline.
Then, if specific data is available, we fine-tune the regularization strength $\boldsymbol \Lambda$ in the learned $\mathcal R_{\V y}$ for the task at hand.
In our simulations, this enables high-quality reconstructions $\hat{\V x}$ with only a small instance-specific training set.
% Additional Refs:
% Variational inequalities:     https://arxiv.org/pdf/2109.13569.pdf
% Conservative Jacobian: https://arxiv.org/pdf/2106.04350.pdf, https://epubs.siam.org/doi/epdf/10.1137/22M1541630
% Lasso and Prox: https://arxiv.org/pdf/2205.06872.pdf, https://arxiv.org/pdf/2102.06809.pdf0'o'0
% https://link.springer.com/article/10.1007/s10107-020-01515-z$
\section{A Brief on Set-Valued Maps}\label{sec:Prelim}
We start with some general notation.
The closed unit ball in $\R^n$ is denoted by $\mathbb{B}$.
The collection of all subsets $N$ of $\N$ such that $\N \backslash N$ is finite is denoted by $\mathscr{N}$, whereas the collection of all infinite $N \subset \mathbb{N}$ is denoted by $\mathscr{N}^{\sharp}$.
With this, we can conveniently handle subsequences of a given sequence.
For instance, if we have $\{\V x_{k}\}_{k\in \N} \subset \mathbb{R}^{n}$, then $\{\V x_{k}\}_{k \in N}$ for either $N \in \mathscr{N}$ or $N \in \mathscr{N}^{\sharp}$ designates a subsequence.
Limits as $k \rightarrow \infty$ with $k \in N$ are either indicated by $\lim_{k \in N}$ or by \smash{$\xrightarrow{N}$}.
The function $\mathbf{diag} \colon \R^n \to \R^{n,n}$ returns a diagonal matrix whose diagonal entries are the input vector.

\paragraph{Set-Valued Maps}
Set-valued maps $S \colon \R^m \rightrightarrows \R^n$ appear as solution operators to \eqref{eq:var_prob}.
For an in-depth introduction, we refer to \cite[Cha.~3]{DonRoc2009}.
The set $\mathrm{dom}(S) = \{\V y \in \R^m : S(\V y) \neq \emptyset\}$ is the \emph{domain} of $S$, and the \emph{graph} of $S$ is the set $\operatorname{Graph}(S)=\{(\V y, \V x) \in \R^m \times \R^n: \V x \in S(\V y)\}$.
The inverse $S^{-1}\colon \R^n \rightrightarrows \R^m$ of $S$ is defined via $\operatorname{Graph}(S^{-1})=\{(\V x, \V y) \in \R^n \times \R^m: (\V y, \V x) \in \operatorname{Graph}(S)\}$.
Recall that
\begin{equation}
\limsup_{\V y \rightarrow \bar{\V y}} S(\V y) =\left\{\V x : \exists \V y_{k} \rightarrow \bar{\V y}, \exists \V x_{k} \rightarrow \V x \text{ with } \V x_{k} \in S\left(\V y_{k}\right)\right\}
\end{equation}
and
\begin{equation}
\liminf_{\V y \rightarrow \bar{\V y}} S(\V y)  =\bigl\{\V x : \forall \V y_{k} \rightarrow \bar{\V y}, \exists N \in \mathscr{N}, \V x_{k} \xrightarrow{N} \V x \text{ with } \V x_{k} \in S\left(\V y_{k}\right)\bigr\}.
\end{equation}
Now, we discuss several notions of continuity.
A map $S$ is \emph{outer semicontinuous (osc)} relative to $D \subset \mathrm{dom}(S)$ at $\bar{\V y} \in \R^{m}$ if $\limsup_{\V y \rightarrow \bar{\V y}} S(\V y) \subset S(\bar{\V y})$, where the sequences are taken in $D$.
This is equivalent to $\operatorname{Graph}(S)$ being closed in $D \times \R^n$.
Further, $S$ is \emph{inner semicontinuous (isc)} relative to $D$ at $\bar{\V y}$ if $\liminf_{\V y \rightarrow \bar{\V y}} S(\V y) \supset S(\bar{\V y})$.
This is equivalent to $S^{-1}(O)$ being open in D for every open $O \subset \R^n$.
Finally, $S$ is continuous relative to $D$ at $\bar{\V y}$ iff it is both osc and isc relative to $D$.
For single-valued $S \colon \R^m  \to \R^n$, this coincides with the usual notion of continuity.
Conversely, any closed-convex-valued and isc $S$ admits a continuous single-valued selection $\tilde S \colon \mathrm{dom}(S) \to \R^n$, see \cite[Thm.~5J.5]{DonRoc2009}.

The map $S$ is \emph{Lipschitz continuous} relative to $D \subset \mathrm{dom}(S)$  if $S$ is closed-valued, and there exists a Lipschitz constant  $\kappa > 0$ such that
\begin{equation}\label{eq:LipSetVal}
S(\V y^{\prime}) \subset S(\V y)+\kappa\Vert \V y^{\prime} - \V y \Vert \mathbb{B} \quad \text { for all } \V y^{\prime}, \V y \in D.
\end{equation}
Lipschitz continuity is closely related to \emph{strict continuity} of $S$ at $\bar{\V y}$, which means that for each $\rho>0$, there exist $\kappa>0$ and a neighbourhood $V$ of $\bar{\V y}$ such that
\begin{equation}\label{eq:StrictCon}
    S(\V y) \cap \rho \mathbb{B} \subset S(\V y^{\prime})+\kappa\Vert \V y-\V y^{\prime}\Vert  \mathbb{B} \quad \text{ for all } \V y, \V y^{\prime} \in V.
\end{equation}
The difference to \eqref{eq:LipSetVal} is the intersection with $\rho \mathbb{B}$, which resolves issues that can arise if $S$ takes unbounded sets as value.
As a local counterpart to \eqref{eq:StrictCon}, we say that $S$ has the \emph{Aubin (or Lipschitz-like) property} at $\bar{\V y}$ for $\bar{\V x} \in S(\bar{\V y})$ if $\operatorname{Graph}(S)$ is locally closed at $(\bar{\V y}, \bar{\V x})$, and there is $\kappa>0$ along with neighborhoods $V$ of $\bar{\V y}$ and $U$ of $\bar{\V x}$  such that 
\begin{equation}\label{eq:Aubin}
    S(\V y^{\prime}) \cap U \subset S(\V y)+\kappa\Vert \V y^{\prime} - \V y\Vert  \mathbb{B} \quad \text { for all } \V y^{\prime}, \V y \in V.
\end{equation}
For single-valued $S \colon \R^m  \to \R^n$, the concepts \eqref{eq:LipSetVal}-\eqref{eq:Aubin} coincide with the usual Lipschitz continuity.

Next, we introduce weaker continuity notations, where the reference point $\V y$ in the right-hand sides of \eqref{eq:LipSetVal} and \eqref{eq:Aubin} is fixed.
The map $S$ is \emph{outer Lipschitz continuous} at $\bar{\V y} \in D$ relative to $D \subset \mathrm{dom}(S)$ if $S(\bar{\V y})$ is closed, and there is $\kappa > 0$ along with a neighborhood $V$ of $\bar{\V y}$ such that
\begin{equation}\label{eq:OutLip}
S(\V y) \subset S(\bar{\V y})+\kappa\Vert \V y-\bar{\V y}\Vert  \mathbb{B} \quad \text { for all } \V y \in V \cap D.
\end{equation}
If $S$ is outer Lipschitz continuous at every $\V y \in D$ with the same $\kappa$, then $S$ is outer Lipschitz continuous relative to $D$.
%A map $S$ with closed values is Lipschitz continuous relative to $D$ with constant $\kappa$ iff it is both isc and outer Lipschitz continuous relative to $D$ with constant $\kappa$ \cite[Thm.~3D.3]{DonRoc2009}.
A yet weaker notion originates from \eqref{eq:Aubin}.
The map $S$ is \emph{calm} at $(\bar{\V y}, \bar{\V x}) \in \operatorname{Graph}(S)$ if there is $\kappa>0$ along with neighborhoods $V$ of $\bar{\V y}$ and $U$ of $\bar{\V x}$  such that
\begin{equation}\label{eq:Calm}
    S(\V y) \cap U \subset S(\bar{\V y}) + \kappa \Vert \V y - \bar{\V y}\Vert \mathbb B \quad \text{ for all } \V y \in V.
\end{equation}
Compared with \eqref{eq:OutLip}, the inclusion only holds locally, similar to the relation between the Lipschitz continuity \eqref{eq:LipSetVal} and the Aubin property \eqref{eq:Aubin}.

An important class of set-valued maps $S$ are \emph{polyhedral} ones, for which $\operatorname{Graph}(S)$ is a finite union of polyhedra.
Recall that a \emph{polyhedron} in $\R^{m+n}$ is of the form $\{\V x \in \R^{m+n}: \M A \V x \leq \V b\}$ with $\M A \in \R^{p, m+n}$ and $\V b \in \R^p$. 
Polyhedral maps are strictly continuous \eqref{eq:StrictCon} outside of a lower-dimensional set $D$ ($\dim(D) < m$) and the Lipschitz modulus $\kappa$ can be chosen independent of $\bar{\V y}$ \cite[Prop.~35]{DanPan2011}.
Moreover, they are outer Lipschitz continuous \eqref{eq:OutLip} with a uniform modulus \cite[Thm.~3D.1]{DonRoc2009}.

\paragraph{Hoffman bound}
The Hoffman bound \cite{Hof1952} as in \cite[Props.\ 5 and 6]{PenVerZul2021} states that for $\M E \in \R^{n,n}$, $\M F \in \R^{l,n}$, $\V b\in \R^n$ and $\V q\in \R^l$ with $\{\V x \in \R^n : \M E \V x = \V b, \M F \V x \leq \V q\} \neq \emptyset$ and any $\V x^\prime \in \R^n$ it holds that 
\begin{equation}\label{eq:HoffmannEst}
    \min_{\M E\V x = \V b, \M F\V x \leq \V q}\Vert \V x - \V x^\prime \Vert  \leq K(\M E, \M F) \biggl \Vert
    \begin{pmatrix}
        \M E \V x^\prime - \V b\\
        (\M F \V x^\prime - \V q)_+
    \end{pmatrix}\biggr\Vert,
\end{equation}
where
\begin{align}\label{eq:hoffmanBound}
    K(\M E, \M F) &= \max_{J \in U(\M E, \M F)} \biggl(\min \Bigl \{
    \bigl\Vert \M E^T \V u + (\M F_J)^T \V v \bigr\Vert:  \V u \in \ran(\M E),\, \V v \geq 0,\, \Vert (\V u, \V v) \Vert = 1
    \Bigr\}\biggr)^{-1} < \infty.
\end{align}
In \eqref{eq:hoffmanBound}, $\M F_J$ denotes the matrix that consists of the rows indexed by $J$, and
\begin{equation}\label{eq:SetS}
    U(\M E, \M F) \coloneqq  \bigl\{J \subset [l]: \{(\M F \V x + \V s, \M E \V x) : \V x \in \R^{n}, \V s \in \R^l, s_j \geq 0\, \forall j \in J\}\text{ is a subspace}\bigr\},
\end{equation}
where $[l]$ is the powerset of $\{1,\ldots,l\}$.
For our setting, the set \eqref{eq:SetS} simplifies as
\begin{align}
    U(\M E, \M F)=& \bigl\{J \subset [l]: \{(\M F_J \V x + \V s, \M E \V x) : \V x \in \R^{n}, \V s \in \R_{\geq 0}^{\vert J \vert}\}= \R^{\vert J \vert} \times \ran(\M E)\bigr\}\notag\\
    =& \bigl\{J \subset [l]: \{(\M F_J \V x, \M E \V x) : \V x \in \R^{n}\}= \R^{\vert J \vert} \times \ran(\M E)\bigr\}\notag\\
    =& \bigl\{J \subset [l]: \{\M F_J \V x : \V x \in \ker(\M E)\}= \R^{\vert J \vert}\bigr\}\notag\\
    =& \bigl\{J \subset [l]: \textrm{rank}(\M F_J) = \vert J \vert, \ker(\M E) + \ker(\M F_J) = \R^n\bigr\}\notag\\
    = &\bigl\{J \subset [l] : \textrm{rank}(\M F_J) = \vert J \vert, \ran(\M E^T) \cap \ran(\M F_J^T) = \{\V 0\}\bigr\}.
\end{align}
The estimate \eqref{eq:HoffmannEst} bounds the distance to a polyhedron $\{\V x \in \R^n : \M E \V x = \V b, \M F \V x \leq \V q\}$.
This will be useful when estimating the Lipschitz constant in \eqref{eq:LipSetVal} for polyhedral maps throughout Section~\ref{sec:TheoProp}. 

\paragraph{Implicit Set-Valued Maps}
Given a Lipschitz continuous $f \colon \R^{m+n} \to \R^n$, which in our case will be the derivative of the objective in \eqref{eq:var_prob}, we can implicitly define a set-valued map $S \colon \R^m \rightrightarrows \R^n$ via
\begin{equation}\label{eq:ImplicitMap}
    S(\V y) = \bigl\{\V x \in \R^n : f(\V y,\V x) = 0\bigr\}.
\end{equation}
Strong regularity statements for $S$ exist if the Jacobian $\mathbf J_f (\V y, \V x) \coloneqq \begin{pmatrix} \M A & \M B \end{pmatrix} \in \R^{n,m+n}$ with $\M A = (\nabla_{y_j} f_i(\V y, \V x))_{i,j=1}^{n,m}$ and $\M B = (\nabla_{x_j} f_i(\V y, \V x))_{i,j=1}^{n,n}$ exists and if $\V B$ is (locally) invertible.
Then, the classic implicit function theorem guarantees that $S$ is nonempty and admits single-valued localizations.
Moreover, each localization is differentiable and has a formula for its derivative.
If $f$ is merely Lipschitz continuous, Clarke's implicit function theorem \cite{Clarke1990} guarantees the existence and Lipschitz continuity of single-valued localizations.
This requires to replace $\mathbf J_f (\V y, \V x)$ by the \emph{Clarke Jacobian}
\begin{equation}\label{eq:JacClarke}
    \M J_f^c(\V y, \V x) = \mathrm{conv}\bigl\{\M C \in \R^{n,m+n}: \exists (\V y_k, \V x_k) \in R \text{ with } (\V y_k, \V x_k) \to (\V y, \V x) \text{ s.t. } \M J_f(\V y_k, \V x_k) \to \M C \bigr\},
\end{equation}
where $R \subset \R^{m+n}$ is the set of full measure on which $f$ is differentiable (Rademacher theorem).
Unfortunately, no formula for the Clarke Jacobian $\M J_S^c$ of $S$ exists.
Recently, an improved implicit function theorem based on conservative Jacobians was established \cite{BolPauSil2024}.
We say that $J_f \colon \R^{m+n} \rightrightarrows \R^{n,m+n}$ is a \emph{conservative Jacobian} for $f$ if $\operatorname{Graph}(J_f)$ is closed, locally bounded, and  nonempty with
\begin{equation}\label{eq:ConsJac}
    \frac{\dx}{\dx t}f \circ \gamma(t) = \begin{pmatrix} \M A & \M B \end{pmatrix} \gamma^\prime (t) \quad \text{ for all } \begin{pmatrix} \M A & \M B \end{pmatrix} \in J_f (\gamma(t)) \text{ and a.e.\ } t,
\end{equation}
where $\gamma$ is an absolutely continuous curve in $\R^{m+n}$.
The convex hull of any $J_f(\V y, \V x)$ is a superset of $\M J_f^c(\V y, \V x)$ and coincides with $\M J_f(\V y, \V x)$ on $R$.
Compared to $\M J_f^c(\V y, \V x)$ as in \eqref{eq:JacClarke}, $J_f(\V y, \V x)$ admits a proper calculus \cite{BolPau2021}.
This makes its computation much easier as we can decompose complicated functions into simpler components.
Note that a conservative $J_f$ is usually not unique.
\begin{theorem}[Implicit function theorem \cite{BolPauSil2024}]\label{thm:ImplicitFunction}   
Let $f \colon \R^{m+n} \to \R^n$ with a conservative Jacobian $J_f$ and $\bar{\V x} \in S(\bar{\V y})$ with $\bar{\V y} \in \R^m$.
Assume that $J_f(\bar{\V y}, \bar{\V x})\subset \R^{n,m+n}$ is convex and that for each $\begin{pmatrix} \M A & \M B \end{pmatrix} \in J_f(\bar{\V y}, \bar{\V x})$ the matrix $\M B$ is invertible.
Then, there exists an open neighborhood $U \times V \subset \R^{m+n}$ of $(\bar{\V y}, \bar{\V x})$ and $s\colon U \to V$ such that $s(\V y) \in  S(\V y)$ for all $\V y \in U$.
If for each $\V y \in U$ and $\begin{pmatrix} \M A & \M B \end{pmatrix} \in J_f(\V y, s(\V y))$ the matrix $\M B$ is invertible, then a possible conservative Jacobian $J_{s}$ of $s$ is given by
\begin{equation}\label{eq:JacS}
    J_s \colon \V y \rightrightarrows \bigl\{-\M B^{-1}\M A :
    \begin{pmatrix}
        \M A & \M B
    \end{pmatrix}
    \in J_f(\V y, s(\V y))\bigr\}.
\end{equation}
\end{theorem}
\begin{remark}\label{eq:RemLip}
    Using \eqref{eq:ConsJac}, we get for the $s\colon U \to V$ from Theorem \ref{thm:ImplicitFunction} and any $\V y_1, \V y_2 \in U$ that
    \begin{equation}
        \Vert s(\V y_1) - s(\V y_2) \Vert \leq \sup_{\V y \in U}\sup_{\V A \in J_S(\V y)} \Vert \V A \Vert \Vert \V y_1- \V y_2\Vert,
    \end{equation}
    where $J_s$ is the conservative Jacobian from \eqref{eq:JacS}.
    Hence, if $\sup_{\V y \in U}\sup_{\V A \in J_S(\V y_)} \Vert \V A \Vert<\infty$, then $s\colon U \to V$ is Lipschitz continuous in the usual sense.
\end{remark}
In principle, the minimal $J_f$ for Theorem~\ref{thm:ImplicitFunction} is $\M J^c_f$ \cite[Cor.~1]{BolPau2021}, which might be impossible to compute.
With the (possibly larger) conservative Jacobians, we gain flexibility at the cost of having to check the invertibility of more matrices $\M B$.
If a $\M B$ is non-invertible, we can instead apply the more general implicit function theorems in \cite{DonRoc2009,Ioffe2017}, which leads to the Aubin property.
\begin{theorem}[Implicit functions {\cite[Thm.\ 2.84]{Ioffe2017}}]\label{thm:GeneralImplicitFunction}
Let $f \colon \R^{m+n} \to \R^n$ and $\bar{\V x} \in S(\bar{\V y})$.
Suppose
that
\begin{itemize}
    \item for all $\V y$ in a neighborhood of $\bar{\V y}$, the partial inverse $f^{-1}(\V y,\cdot)\colon \R^{n} \rightrightarrows \R^{n}$ is calm \eqref{eq:Calm} at $\V0$ for $\bar{\V x}$ with a constant $K$ independent of $\V y$;
    \item $f(\cdot,\V x)\colon \R^m \to \R^n$ is $\alpha$-Lipschitz for all $\V x$ in some neighborhood of $\bar{\V x}$.
\end{itemize}
Then, $S$ as in \eqref{eq:ImplicitMap} has the Aubin property \eqref{eq:Aubin} near $(\bar{\V y}, \bar{\V x})$ and $\kappa  \leq K \alpha$.
\end{theorem}
Compared to the version in \cite[Thm.~2.84]{Ioffe2017}, we replaced the metric subregularity of $f$ by the equivalent calmness of $f^{-1}$ for convenience of notation, see \cite[Thm.~3H.3]{DonRoc2009}.
In theory, a necessary and sufficient condition for the Aubin property \eqref{eq:Aubin} is given by the Mordukhovich criterion \cite[Thm.~9.40]{RocWet1998}.
However, the required coderivatives of $S$ are difficult to compute, see \cite{GfrOut2016} for a recent discussion, and we prefer to rely on Theorem~\ref{thm:GeneralImplicitFunction} instead.

\section{Theoretical Properties of Ridge Regularizers}\label{sec:TheoProp}
For the theoretical analysis of the problem \eqref{eq:var_prob}, we restrict ourselves to $\mathcal{D}_P \colon \R^m \times \R^m \to \R_{\geq 0}$ of the form 
\begin{equation}\label{eq:data_term}
    \mathcal{D}_P(\M H \V x, \V y) = \sum_{j=1}^m \phi_j\bigl((\M H \V x)_j, y_j\bigr) = \bigl \langle \boldsymbol 1_m, \boldsymbol \phi (\M H \V x, \V y)\bigr\rangle,
\end{equation}
where $\boldsymbol \phi (\V x, \V y) = (\phi_j(x_j, y_j))_{j=1}^m$ and the \smash{$\phi_j \in \mathcal{C}_{\geq0}^{1,1}(\R^2)$} are piecewise-polynomial with finitely many pieces.
A very popular instance of this form is $\mathcal{D}_P(\M H \V x, \V y) = \frac{1}{2} \Vert \M H \V x - \V y \Vert^2$. 
If we further choose $\mathcal R_{\V y}$ as in \eqref{eq:rridge_reg_mask} with piecewise-polynomial \smash{$\psi_c \in \mathcal{C}_{\geq0}^{1,1}(\R)$}, then the problem \eqref{eq:var_prob} turns into
\begin{equation}
    \label{eq:inv_spc}
    \argmin_{\V x \in \X} \Big(g_{\V y}(\V x) \coloneqq  \mathcal D_P(\M H \V x,\V y) + \lambda \sum_{c=1}^{ N_\mathrm{C}} \bigl \langle \boldsymbol \Lambda_c(\V y), \boldsymbol \psi_c(\M W_{c} \V x )\bigr\rangle \Big).
\end{equation}
From now on, we assume that the \smash{$\boldsymbol \Lambda_c\colon \R^m \to [\epsilon,1]^n$} are Lipschitz continuous and bounded from below by $\epsilon>0$ .
Due to the upper-bound $1$, the $\boldsymbol \Lambda = (\M \Lambda_c)_{c=1}^{N_C}$ does not increase the weak-convexity modulus compared to the baseline $\mathcal R$ in \eqref{eq:ridge_reg}.
First, we provide a sufficient condition for the coercivity of $g_{\V y}$.
\begin{lemma}
    The function $g_{\V y}$ in \eqref{eq:inv_spc} is coercive if $\ker(\M W) \cap \ker(\M H)= \{0\}$ and all $\phi_j$ and $\psi_c$, $j=1,...,m$, $c=1,...,N_c$ are coercive.
\end{lemma}
\begin{proof}
    Let $\{\V x_k\}_{k \in \N}$ with $\Vert \V x_k \Vert \to \infty$.
    Then, either $\Vert\M H \V x_k\Vert \to \infty$ or there exists $c$ with $\Vert\M W_c \V x_k \Vert \to \infty$.
    Due to the coercivity of $\phi_j$ and $\psi_c$, the claim readily follows.
\end{proof}
Further, the set \eqref{eq:inv_spc} is nonempty if the profiles $\psi_c$ are piecewise-polynomial.
\begin{theorem}[Existence]\label{thm:existence}
    Let $\psi_c \in \mathcal{C}_{\geq0}^{1,1}(\R)$, $c=1,\ldots,N_C$, be piecewise-polynomial with finitely many pieces, and $\X \subset \R^n$ be a polyhedron.
    Then, \eqref{eq:inv_spc} is nonempty. 
\end{theorem}
\begin{proof}
    Since the $\psi_c$ and $\phi_j(\cdot,y_j)$ are piecewise polynomial, they partition $\R$ into finitely many closed intervals $(I_l)_{l=1}^L$ on which they are polynomials.
    Hence, if we denote the $i$-th row of $\M W_c$ and $\M H$ with $\M W_{c,i}$ and $\M H_i$, respectively, then $\psi_c( \M W_{c,i} \cdot )$ and $ \phi_j( \M H_{j} \cdot , y_j)$ partition $\X$ into $L$ polyhedra $A_{c,i}^{l}=\{\V x \in \X :  \M W_{c,i} \V x \in I_l\}$ and $B_{j}^{l}=\{\V x \in \X : \M H_{j} \V x \in I_l\}$, respectively.
    By intersecting over all possible combinations of $c$, $i$, and $j$, we get a finite decomposition of $\X$ into polyhedra on which $g_{\V y}$ is a polynomial.
    The infimum of $g_{\V y}$ is attained on (at least) one polyhedron, which we denote with $P$.
    
    Now, we pick a minimizing sequence $\{\V x_k\}_{k\in \N}\subset P$.
    Let $\M M$ be the vertical concatenation of all rows $\M H_{j}$ and $\M W_{c,i}$ such that $\{\M W_{c,i} \V x_k\}_{k \in \N}$ and $\{\M H_{j} \V x_k\}_{k \in \N}$ remain bounded.
    Since $\{\M M \V x_k\}_{k\inN}$ is bounded by construction, we can extract a convergent subsequence indexed by $N \in \mathscr{N}^{\sharp}$ with limit $\V u \in \ran(\M M)$.
    The associated set
    \begin{equation}
    Q=\{\V x\in\R^n \colon \M M\V x = \V u\} = \{\M M^\dagger \V u\} + \ker(\M M)
    \end{equation}
    is a polyhedron.
    It holds that
    \begin{equation}
    \mathrm{dist}(\V x_k, Q) = \mathrm{dist}\bigl(\M M^\dagger \M M \V x_k + \mathrm P_{\ker(\M M)}(\V x_k), Q\bigr) \leq \Vert \M M^\dagger \M M \V x_k - \M M^\dagger \V u \Vert \xrightarrow{N} 0
    \end{equation}
    and, thus, that $\mathrm{dist}(P, Q)=0$, which is only possible if $P\cap Q \neq \emptyset$ \cite[Thm.~1]{W1968}.
    
    The rows $\M W_{c,i}$ and $\M H_j$ that were not added to $\M M$ satisfy $\vert \M W_{c,i} \V x_k \vert \to \infty$ and $\vert \M H_ j \V x_k \vert \to \infty$ as $k \to \infty$.
    Hence, they are contained in unbounded 
    intervals \smash{$I_{l_{c,i}}$} and \smash{$I_{l_j}$}, respectively.
    Since $\psi_c$ and $\phi_j(\cdot, y_j)$ are nonnegative polynomials on them, $\psi_c(\M W_{c,i} \cdot)$ and $\phi_j (\M H_j \cdot, y_j)$ are constant\footnote{A non-constant polynomial cannot have a finite limit for $x \to \pm\infty$.} on $P$.
    Hence, $\lim_{k \to \infty} \psi_c(\M W_{c,i} \V x_k) = \psi_c(\M W_{c,i} \V x)$ and $\lim_{k \to \infty} \phi_j(\M H_j \V x_k, y_j) = \phi_j(\M H_j \V x, y_j)$ for every $\V x\in P\cap Q$, $c$, $i$ and $j$.
    Consequently, every $\V x\in P\cap Q$ is a minimizer of \eqref{eq:inv_spc}.
\end{proof}
\begin{remark}
    Our proof does not rely on the direct method of calculus.
    Hence, Theorem~\ref{thm:existence} does not involve the common assumption that $g_{\V y}$ is coercive.
\end{remark}
Without further assumptions (such as strict convexity), several critical points may exist for the objective $g_{\V y}$ in \eqref{eq:inv_spc}.
Hence, the critical point map $S_\lambda \colon \R^m \rightrightarrows \R^n$ with
\begin{equation}\label{eq:CritPointMap}
    S_\lambda(\V y) = \{\V x \in \R^n: \boldsymbol \nabla g_{\V y}(\V x) = \V 0 \}
\end{equation}
is set-valued.
Thus, the usual Lipschitz continuity has to be replaced with the stability concepts from Section~\ref{sec:Prelim}.
In the following, we restrict $g_{\V y}$ in three different ways, ordered according to the strength of the associated stability results.
For convex objectives $g_{\V y}$, it holds that $S_\lambda(\V y) = \argmin_{\V x \in \X} g_{\V y}(\V x)$, namely that it coincides with \eqref{eq:inv_spc}.
This is not necessarily true for nonconvex $g_{\V y}$.
\subsection{Stability: Objective with Invertible Conservative Hessian}\label{sec:StabInv}
Here, we want to analyze  $S_\lambda$ based on Theorem \ref{thm:ImplicitFunction}.
To simplify the derivations, we consider $\boldsymbol \Lambda_c \in (0,\infty)^n$ as additional input of $S_1$ ($\lambda=1$) instead of tracking the dependence on $\V y$.
In particular, we also absorb the regularization strength $\lambda$.
Despite these simplifications, we can capture the Lipschitz dependence of the original $S_\lambda$ on $\V y$ by choosing $\boldsymbol \Lambda_1 = \lambda \boldsymbol \Lambda(\V y_1)$ and $\boldsymbol \Lambda_2 = \lambda \boldsymbol \Lambda(\V y_2)$ for $\V y_1, \V y_2 \in \R^m$.

The function $f(\boldsymbol \Lambda, \V y, \V x) \coloneqq \boldsymbol \nabla_{\V x} g_{\V y}(\V x)$ induced by \eqref{eq:inv_spc} is piecewise polynomial.
Moreover, $\mathrm{graph}(f)$ is a semi-algebraic set, namely a finite union of sets of the form $\{\V z \in \R^{n}: P_1(\V z) = 0, P_2(\V z) > 0\}$ with $P_1$, $P_2$ being collections of polynomials.
Hence, $f$ is semi-algebraic  and a conservative Jacobian $J_f$ of $f$ exists \cite[Prop.\ 2(iv)]{BolPau2021}.
Due to the sum rule \cite[Cor.\ 4]{BolPau2021}, a possible one is given by
\begin{equation}\label{eq:JacFSum}
    J_{f}(\boldsymbol \Lambda, \V y, \V x) = J_{\nabla_{\V x} \mathcal D_P(\M H \V x, \V y)}(\boldsymbol \Lambda, \V y, \V x) + J_{\nabla_{\V x} \mathcal R_{\V y}}(\boldsymbol \Lambda, \V y, \V x).
\end{equation}
For the first summand in \eqref{eq:JacFSum}, we use the sum rule, the chain rule \cite[Lems.\ 5 and 6]{BolPau2021}, and the aggregation property \cite[Lem.\ 3]{BolPau2021} to obtain the conservative Jacobian
\begin{equation}\label{eq:ConJacX}
    J_{\nabla_{\V x} \mathcal D_P(\M H \V x, \V y)}(\boldsymbol \Lambda, \V y, \V x) = \begin{pmatrix}\boldsymbol{0} & \M H^T \M J^c_{\boldsymbol \phi_{\V x}(\M H \V x, \cdot)}(\V y) & \M H^T \M J^c_{\boldsymbol \phi_{\V x}(\cdot, \V y)}(\M H \V x) \M H \end{pmatrix},
\end{equation}
where $\boldsymbol \phi_{\V x}(\V x, \V y) = (\partial_{x_j} \phi_j(x_j, y_j))_{j=1}^m$.
Note that both $\M J^c_{\boldsymbol \phi_{\V x}(\M H \V x, \cdot)}(\V y)$ and $\M J^c_{\boldsymbol \phi_{\V x}(\cdot, \V y)}(\M H \V x)$ are diagonal matrices.
Using the same arguments for the second summand in \eqref{eq:JacFSum}, we get the possible choice
\begin{equation}\label{eq:ConsJacReg}
    J_{\nabla_{\V x} \mathcal R_{\V y}}(\boldsymbol \Lambda, \V y, \V x) =
    \begin{pmatrix}
        \M A(\V x) & \boldsymbol{0} &\sum_{c=1}^{N_C} \M W_c^T \mathbf{diag}\bigl(\boldsymbol \Lambda_c \odot 
        \M J^c_{\boldsymbol \psi_c^\prime}(\M W_c \V x)\bigr)  \M W_c  
    \end{pmatrix},
\end{equation}
where $\odot$ denotes the element-wise product and
\begin{equation}
    \M A(\V x) = \begin{pmatrix}
        \M W_1^T \boldsymbol \psi_1^{\prime}(\M W_1 \V x) & \dots & \M W_{N_c}^T \boldsymbol \psi_{N_c}^{\prime}(\M W_{N_c} \V x)
    \end{pmatrix} \in \R^{n,N_c \cdot n}.
\end{equation}
Now we apply Theorem \ref{thm:ImplicitFunction} to get the following result.
\begin{theorem}\label{thm:InvertJacobLip}
    Assume that $\bar{\V x} \in S_1(\bar{\boldsymbol \Lambda}, \bar{\V y})$ and that for any $\begin{pmatrix} \M A & \M B & \M C \end{pmatrix} \in
    J_{f}(\bar{\boldsymbol \Lambda},\bar{\V y}, \bar{\V x})$ as in \eqref{eq:JacFSum} the matrix $\M C$ is invertible.
    Then, there exists a neighborhood $U$ of $(\bar{\boldsymbol \Lambda}, \bar{\V y})$ such that $S_1\colon U \rightrightarrows \R^n$ has a single-valued localization $s_1\colon U \to \R^n$ which is locally Lipschitz continuous.
    If for each $(\boldsymbol \Lambda, \V y) \in U$ and $\begin{pmatrix} \M A & \M B & \M C \end{pmatrix} \in J_f(\boldsymbol \Lambda, \V y, s_1(\boldsymbol \Lambda, \V y))$ the matrix $\M C$ is invertible, a possible $J_{s_1}$ is given by
\begin{equation}\label{eq:ConJac2}
    J_{s_1} \colon  (\boldsymbol \Lambda, \V y) \rightrightarrows \bigl\{-
    \begin{pmatrix}
        \M C^{-1} \M A & \M C^{-1} \M B
    \end{pmatrix} :
    \begin{pmatrix}
        \M A & \M B & \M C
    \end{pmatrix} \in J_f(\boldsymbol \Lambda, \V y, s_1(\boldsymbol \Lambda, \V y))\bigr\}.
\end{equation}
\end{theorem}
The Lipschitz constant of $s_1$ in Theorem \ref{thm:InvertJacobLip} can be readily estimated if $f(\boldsymbol \Lambda, \V y, \V x) \coloneqq \boldsymbol \nabla_{\V x} g_{\V y}(\V x)$ is induced by $\mathcal{D}_P(\M H \V x, \V y) = \frac{1}{2} \Vert \M H \V x - \V y \Vert^2$ with an invertible operator $\M H$ and if the $\psi_c$ are convex.
Then, \eqref{eq:ConJacX} simplifies to
\begin{equation}
J_{\nabla_{\V x} \mathcal D_P(\M H \V x, \V y)}(\boldsymbol \Lambda, \V y, \V x) = \begin{pmatrix}\boldsymbol{0} & - \M H^T  & \M H^T \M H \end{pmatrix}.
\end{equation}
Moreover, the convexity of $\psi_c$ implies $\sum_{c=1}^{N_C} \M W_c^T \mathbf{diag}(\boldsymbol \Lambda_c \odot \smash{\M J^c_{\boldsymbol \psi_c^\prime}}(\M W_c \V x))  \M W_c \succeq 0$.
Thus, we get that any \smash{$\begin{pmatrix}\M A & \M B & \M C \end{pmatrix} \in J_f(\boldsymbol \Lambda, \V y, s_1(\boldsymbol \Lambda, \V y))$} satisfies $\M C \succeq \M H^T \M H$.
Since $\M H$ is invertible, this further implies $\M C^{-1} \preceq (\M H^T \M H)^{-1}$.
Using Remark \ref{eq:RemLip} and the formula \eqref{eq:ConJac2}, we get the following estimates.
\begin{remark}[Convex $\mathcal R$]\label{rem:fixedMask}
    Let $\bar{\boldsymbol \Lambda}$ be fixed and $U$ be the neighborhood from Theorem \ref{thm:InvertJacobLip}.
    Then, a possible Lipschitz constant of $s_1(\bar{\boldsymbol \Lambda}, \cdot)$ on $\{\V y: (\bar{\boldsymbol \Lambda}, \V y ) \in U\}$ reads
    \begin{align}\label{eq:EstConv}
        L_{\V y}(\bar{\boldsymbol \Lambda}) & = \sup_{\V y: (\bar{\boldsymbol \Lambda}, \V y ) \in U} \sup_{(\V V_{\boldsymbol \Lambda},\V V_{\V y}) \in J_{s_1}(\bar{\boldsymbol \Lambda}, \V y)} \Vert \V V_{\V y} \Vert = \sup_{\V y: (\bar{\boldsymbol \Lambda}, \V y ) \in U} \sup_{\M C \in J_{f}(\bar{\boldsymbol \Lambda}, \V y, s_1(\bar{\boldsymbol \Lambda}, \V y))} \Vert \M C^{-1} \M H^T \Vert \notag\\
        & \leq \Vert (\M H^T \M H)^{-1} \M H^T \Vert = \Vert \M H^{-1}\Vert,
    \end{align}
    which coincides with the result in \cite[Prop.\ III.2.]{GouNeuBoh2022}.
    To obtain the estimate \eqref{eq:EstConv}, we dropped the summand $\sum_{c=1}^{N_C} \M W_c^T \mathbf{diag}(\boldsymbol \Lambda_c \odot \smash{\M J^c_{\boldsymbol \psi_c^\prime}}(\M W_c  s_1(\bar{\boldsymbol \Lambda}, \V y)))  \M W_c$ of $\M C$, which can locally improve the estimate.
\end{remark}
\begin{remark}[Mask sensitivity for convex $\mathcal R_{\V y}$]\label{rem:fixedData}
     For fixed $\bar{\V y}$, we bound the Lipschitz constant of $s_1(\cdot, \bar{\V y})$ on $\{\boldsymbol \Lambda: (\boldsymbol \Lambda, \bar{\V y} ) \in U\}$ as
    \begin{align}
        L_{\boldsymbol \Lambda}(\bar{\V y}) &= \sup_{\boldsymbol \Lambda: (\boldsymbol \Lambda, \bar{\V y} ) \in U} \sup_{(\V V_{\boldsymbol \Lambda},\V V_{\V y}) \in J_{s_1}(\boldsymbol \Lambda, \bar{\V y})} \Vert \V V_{\boldsymbol \Lambda} \Vert = \sup_{\boldsymbol \Lambda: (\boldsymbol \Lambda, \bar{\V y} ) \in U} \sup_{\M C \in J_{f}(\boldsymbol \Lambda, \bar{\V y}s_1(\boldsymbol \Lambda, \bar{\V y}))} \Vert \M C^{-1} \M A(s_1(\boldsymbol \Lambda, \bar{\V y})) \Vert \notag\\
        & \leq  \Vert \M H^{-1}\Vert^2  \sup_{\boldsymbol \Lambda: (\boldsymbol \Lambda, \bar{\V y} ) \in U} \Vert \M A(s_1(\boldsymbol \Lambda, \bar{\V y}))\Vert\leq  \Vert \M H^{-1}\Vert^2 \biggl(\sum_{c=1}^{N_c} \Vert \psi_c^\prime \Vert_{C^0}^2 \Vert \M W_c \Vert^2\biggr)^{1/2}.
    \end{align}
    This estimate depends inherently on $\Vert \M H^{-1}\Vert$ and is similar to the one in \cite[Prop.\ 4.3(i)]{KofAltBa2023}.
\end{remark}
\begin{remark}
    Combining the estimates in Remark \ref{rem:fixedMask} and \ref{rem:fixedData}, we get a (loose) Lipschitz bound with respect to both $\boldsymbol \Lambda$ and $\V y$.
    To get a better local estimate, we can directly compute a bound based on $J_f(\boldsymbol \Lambda, \V y, s_1(\boldsymbol \Lambda, \V y))$ around $(\bar{\boldsymbol \Lambda}, \bar{\V y}, s_1(\bar{\boldsymbol \Lambda}, \bar{\V y}))$.
    In either case, $s_1$ is Lipschitz continuous in $\V y$ and $\boldsymbol \Lambda$.
    Hence, if we decrease the noise level and $\lambda$, the reconstructions will converge to the noise-free solution (the so-called convergence for vanishing noise, see Section \ref{sec:VanishNoise}).
\end{remark}
\subsection{Stability: Polyhedral Reconstruction Map}\label{sec:StabPoly}
Now, we assume that $\partial_x \phi_j$ and $\psi_c^\prime$ are piecewise affine (instead of only piecewise polynomial) and set $\boldsymbol{\Lambda_c} = \boldsymbol{1}_n$ in \eqref{eq:inv_spc}.
Recall that any $\V x \in S_\lambda(\V y)$ (see \eqref{eq:CritPointMap}) with $\V y \in \R^d$ satisfies
\begin{equation}\label{eq:optimalityeq}
    f(\V y,\V x) \coloneqq  \nabla_{\V x} g_{\V y} (\V x) = \M H^T \boldsymbol \phi_{\V x}(\M H \V x, \V y) + \lambda \sum_{c=1}^{N_C} \M W_c^T \boldsymbol \psi_c^\prime(\M W_c \V x) = \V 0.
\end{equation}
Since $\partial_x \phi_j$ and $\psi_c^\prime$ are piecewise affine, we can decompose $\R^n \times \R^m$ into polyhedra on which \eqref{eq:optimalityeq} turns into a linear equation.
Hence, $\operatorname{Graph}(S_\lambda)$ is the union of finitely many polyhedra and therefore $S_\lambda$ a polyhedral map.
A dependence of $\boldsymbol{\Lambda_c}$ on $\V y$ would destroy this structure due to the multiplicative relation with $\V x$ in \eqref{eq:inv_spc}.
The observations in Section \ref{sec:Prelim} lead to the following result.
\begin{theorem}[Stability of Minimizers]\label{thm:stab}
Let $\psi_c^\prime$ and $\partial_{\V x}\phi_j$ be piecewise affine, $\lambda>0$ and $\V y_k \to \V y \in \R^m$.
%For each $k \in \mathbb{N}$, fix some $\V x_k \in S_\lambda(\V y_k)$.
Then, $S_\lambda$ is outer Lipschitz continuous, namely there exists $\kappa>0$ such that every $\V y$ has a neighborhood $V$ such that $ S_\lambda(\V y_k) \subset S_\lambda(\V y) + \kappa \Vert \V y_k -\V y \Vert \mathbb B$ for all $\V y_k \in V$.
Moreover, there exists a set $D \subset \R^m$ with $\dim(D) \leq m-1$ such that for any $\V x \in S_\lambda(\V y)$ with $\V y \in \R^m \setminus D$, there exists a neighborhood $V$ of $\V y$ such that $\V x \in S_\lambda(\V y_k)+\kappa \Vert \V y_k - \V y \Vert \mathbb{B}$ for all $\V y_k \in V$.
\end{theorem}
\begin{proof}
    The first claim follows directly from the fact that $S_\lambda$ is a polyhedral map \cite[Thm.~3D.1]{DonRoc2009}.
    For the second part, we apply \cite[Prop.\ 35]{DanPan2011}, which yields strict continuity \eqref{eq:StrictCon} of $S_\lambda$ outside a set $D$ with $\mathrm{dim}(d) \le m-1$.
    Given $\V x \in S_\lambda(\V y)$ with $\V y \in \R^m \setminus D$, we choose $\rho = \Vert x \Vert + 1 >0$.
    Then, there exists a neighborhood $V$ of $\V y$ such that $\V x \in S_\lambda(\V y_k)+\kappa \Vert \V y_k - \V y \Vert \mathbb{B}$ for all $\V y_k \in V$.
\end{proof}
\begin{remark}
    Let $\mathcal{D}_P(\M H \V x, \V y) = \frac{1}{2} \Vert \M H \V x - \V y \Vert^2$ and $\M D_{\boldsymbol \psi_c,j} \in \R^{n,n}$ be the diagonal matrices with the slopes of $\psi_c^\prime$ on the polyhedra where $f$ is affine (these are independent of $\V y$).
    Then, the $\kappa$ in Theorem \ref{thm:stab} is proportional to the inverse of the smallest singular value of the matrices $\M H^T \M H + \lambda \sum_{c=1}^{N_c} \M W_c^T \M D_{\boldsymbol \psi_c,j} \M W_c$.
    If $S_\lambda$ is single-valued, Theorem \ref{thm:stab} tells us that $S_\lambda$ is locally Lipschitz continuous.
\end{remark}

\subsection{Stability: Convex Objective}\label{sec:StabConv}
Let $\psi_c^\prime$ and $\partial_{x}\phi_j$ be piecewise affine with $\psi_c^{\prime \prime}\geq 0$ and $\partial^2_{x}\phi_j> 0$ a.e.
Then, $\mathcal R_{\V y}$ is convex, and we derive two different results.
First, we show the Lipschitz continuity of the $S_\lambda$ from \eqref{eq:CritPointMap} for $\boldsymbol{\Lambda}_c = \V 1_n$.
For a Lipschitz continuous mask $\boldsymbol{\Lambda}(\V y)$, we show the weaker Aubin property \eqref{eq:Aubin}.

\paragraph{Constant Mask}
If the problem underlying the critical-point condition \eqref{eq:optimalityeq} is convex $\mathcal R_{\V y}$, then $S_\lambda$ is Lipschitz continuous.
\begin{theorem}[Stability for Constant Mask]\label{thm:LipConv}
    Let $\lambda>0$,  $\psi_c^\prime$ and $\partial_{x}\phi_j$ be piecewise affine with $\psi_c^{\prime \prime}\geq 0$ and $\partial^2_{x}\phi_j> 0$ a.e., and $\boldsymbol{\Lambda}_c = \V 1_n$, $c=1,...,N_c$.
    Then, $S_\lambda$ is Lipschitz continuous, i.e., there exists $\kappa>0$ such that
    \begin{equation}
    S_\lambda(\V y^{\prime}) \subset S_\lambda(\V y)+\kappa\Vert \V y^{\prime} - \V y\Vert  \mathbb{B} \quad \text { for all } \V y^{\prime}, \V y \in \R^m.
    \end{equation}
\end{theorem}
\begin{proof}
    The map $S_\lambda$ is Lipschitz continuous with constant $\kappa$ if it is both isc and outer Lipschitz continuous with constant $\kappa$ \cite[Thm.~3D.3]{DonRoc2009}.
    By Theorem~\ref{thm:stab}, $S_\lambda$ is outer Lipschitz continuous.
    Thus, it suffices to show that $S_\lambda$ is isc, i.e., that for every $\V y_0 \in \mathrm{dom}(S_{\lambda})$, $\V x_0 \in S_{\lambda}(\V y_0)$ and every sequence $\V y_k \to \V y_0$ there exist  $\V x_k \to \V x_0$ such that $\V x_k \in S_{\lambda}(\V y_k)$ for all $k \in \N$.
    Equivalently, we can show that for any ball $B_\epsilon(\V x_0) = \{ \V x \in \R^n: \Vert \V x - \V x_0 \Vert < \epsilon \}$ there exists $k_\epsilon$ such that for $k \geq k_\epsilon$ it holds
    \begin{equation}\label{eq:LocalProblem}
        S_{\lambda,\epsilon}(\V y_k) \coloneqq \bigl \{\V h_{\V x} \in B_\epsilon(\V 0): \ f(\V y_k,\V x_0 + \V h_{\V x}) = \V 0 \bigr\} \neq \emptyset.
    \end{equation}
    Since $\psi_c^\prime$ is piecewise affine, we can take $B_\epsilon(\V x_0)$ small enough such that all $\boldsymbol \psi^\prime_c(\M W_c \cdot) \colon B_\epsilon(\V x_0) \to \R^n$ coincide with their directional linearizations at $\V x_0$, namely
    \begin{equation} \label{eq:psi_prime_linearization}
        \boldsymbol \psi_c^\prime \bigl(\M W_c (\V x_0 + \V h_{\V x})\bigr) = \boldsymbol \psi_c^\prime(\M W_c \V x_0) + \boldsymbol \psi^{\prime\prime}_c (\M W_c \V x_0; \M W_c \V h_{\V x}) \M W_c \V h_{\V x} \quad \text{ for all } \V h_{\V x} \in B_\epsilon(\V 0),
    \end{equation}
    where $\boldsymbol \psi^{\prime\prime}_c (\V x; \V h) = \mathbf{diag}((\psi_c^{\prime\prime}(x_k; h_k))_{k=1}^n)$ with $\psi_c^{\prime\prime}(x_k; h_k) = \lim_{t \searrow 0} \frac{\psi_c^{\prime}(x_k + t h_k) - \psi_c^{\prime}(x_k)}{t h_k}$.
    Similarly, for $B_\epsilon(\V x_0)$ small enough, we can find $B_\delta(\V y_0)$ independent of $\epsilon$ such that $\boldsymbol \phi_{\V x} (\M H \cdot, \cdot) \colon B_\epsilon(\V x_0) \times B_\delta(\V y_0) \to \R^n$
    coincides with its directional linearization at $(\V x_0, \V y_0)$, namely
    \begin{equation} \label{eq:phi_prime_linearization}
    \boldsymbol \phi_{\V x}\bigl(\M H (\V x_0 + \V h_{\V x}),\V y_0 + \V h_{\V y}\bigr) = \boldsymbol \phi_{\V x}(\M H \V x_0,\V y_0) + \boldsymbol \phi_{\V x \V x} (\M H \V x_0, \V y_0; \M H \V h_{\V x}) \M H \V h_{\V x} + \boldsymbol \phi_{\V x \V y}(\M H \V x_0, \V y_0; \V h_{\V y}) \V h_{\V y}
    \end{equation}
    for all $\V h_{\V x} \in B_\epsilon(\V 0)$ and $\V h_{\V y} \in B_\delta(\V 0)$, where the directional derivatives $\boldsymbol \phi_{\V x \V x}$ and $\boldsymbol \phi_{\V x \V y}$ of $\boldsymbol \phi_{\V x}$ are defined like $\boldsymbol \psi^{\prime\prime}_c$.
    Using that \eqref{eq:optimalityeq} holds for $\V x_0 \in S_{\lambda}(\V y_0)$, we get with the shorthands $\M D_{\boldsymbol \phi,\V x}(\V h_{\V x}) = \boldsymbol \phi_{\V x \V x}(\M H \V x_0, \V y_0; \M H \V h_{\V x} )$, $\M D_{\boldsymbol \phi,\V y}(\V h_{\V y}) =  \boldsymbol \phi_{\V x \V y}(\M H \V x_0, \V y_0; \V h_{\V y} )$ and $\M D_{\boldsymbol \psi_c}(\V h_{\V x}) = \boldsymbol \psi_c^{\prime\prime} (\M W_c \V x_0; \M W_c \V h_{\V x})$ that for any $\V h_{\V y} \in B_\delta(\V 0)$ it holds 
    \begin{align}
        S_{\lambda,\epsilon}(\V y_0 + \V h_{\V y}) = & \biggl \{\V h_{\V x} \in B_\epsilon(\V 0): \M H^T \bigl(\M D_{\boldsymbol \phi,\V x}(\V h_{\V x}) \M H \V h_{\V x} + \M D_{\boldsymbol \phi,\V y}(\V h_{\V y}) \V h_{\V y} \bigr) + \lambda \sum_{c=1}^{N_C} \M W_c^T \M D_{\boldsymbol \psi_c}(\V h_{\V x}) \M W_c \V h_{\V x} = \V 0 \biggr\}\notag\\
        \subset & \biggl \{\V x \in \R^n : \M H^T \bigl(\M D_{\boldsymbol \phi,\V x}(\V x) \M H \V x + \M D_{\boldsymbol \phi,\V y}(\V h_{\V y}) \V h_{\V y}\bigr) + \lambda \sum_{c=1}^{N_C} \M W_c^T \M D_{\boldsymbol \psi_c}(\V x) \M W_c \V x = \V 0 \biggr\}.\label{eq:AuxProb}
    \end{align}
    Since $\psi_c^{\prime\prime} \geq 0$ and $\partial^2_{x}\phi_j>0$ a.e., \eqref{eq:AuxProb} coincides with the minimizers of the convex problem
    \begin{equation}
        \argmin_{\V x \in \R^n} \bigl\langle \M D_{\boldsymbol \phi,\V y}(\V h_{\V y})\V h_{\V y}, \M H \V x \bigr\rangle + \frac12 \Vert \M H \V x \Vert^2_{\M D_{\boldsymbol \phi,\V x}(\V x)} + \frac{\lambda}{2} \sum_{c=1}^{N_C} \Vert \M W_c \V x \Vert^2_{\M D_{\boldsymbol \psi_c}(\V x)},\label{eq:AuxProb2}
    \end{equation}
    where $\Vert \V x \Vert_{\M D}^2 = \langle \V x, \M D \V x \rangle$ (any symmetric positive semi-definite matrix $\M D$ induces a semi-norm). 
    Since the objective in \eqref{eq:AuxProb2} is lower bounded, Theorem~\ref{thm:existence} implies that \eqref{eq:AuxProb2} is nonempty.
    
    Next, we show that if $\V h_{\V y,k} \to \V 0$, then there exist $\hat{\V x}_k \in S_{\lambda,\epsilon}(\V y_0 + \V h_{\V y_k})$ with $\hat{\V x}_k \to \V 0$. 
    For this, we partition $\R^n$ into polyhedra $\Omega_i = \{\V x \in \R^n : \M F_i \V x \leq \V 0\}$ on which the equation in \eqref{eq:AuxProb} becomes affine in $\V x$ with matrices $\M E_i = \M H^T \M D_{\boldsymbol \phi,\V x, i} \M H + \lambda \sum_{c=1}^{N_C} \M W_c^T \M D_{\boldsymbol \psi_c,i} \M W_c$.
    %By refining this partition, we can assume without loss of generality that no $\Omega_i$ contains a linear subspace of $\R^n$.
    Since \eqref{eq:AuxProb} is nonempty, at least one of the polyhedra
    \begin{equation}
        P_i(\V h_{\V y_k}) = \bigl \{\V x \in \R^n:\M E_i \V x = - \M H^T \M D_{\boldsymbol \phi,\V y_k,i} \V h_{\V y_k},\, \M F_{i} \V x \leq \V 0 \bigr\}
    \end{equation}
    is nonempty.
    Now, we make use of the Hoffman bound \eqref{eq:HoffmannEst}, which gives us for $\V x^\prime = \V 0$ that
    \begin{align}
    \min_{\V x \in P_i(\V h_{\V y_k})} \Vert \V x \Vert  \leq K(\M E_i, \M F_i) \bigl\Vert \M H^T \M D_{\boldsymbol \phi,\V y_k,i} \V h_{\V y_k}\bigr\Vert,
    \end{align}
    from which we infer the existence of $\hat{\V x}_k \in S_{\lambda,\epsilon}(\V y_0 + \V h_{\V y_k})$ with $\hat{\V x}_k \to \V 0$.
    Thus, also \eqref{eq:LocalProblem} is satisfied.
\end{proof}

\begin{remark}
    For the weakly convex setting from \cite{GouNeuUns2023}, this proof does not apply as \eqref{eq:AuxProb2} might be empty.
    Instead, we could check for the openness of $S_\lambda$ directly to obtain Lipschitz continuity.
\end{remark}

\paragraph{Data-Dependent Mask}

Now, we provide a stability result for a Lipschitz continuous mask $\boldsymbol{\Lambda}(\V y)$. 
Recall that the critical point condition for \eqref{eq:inv_spc} given some $\V y \in \R^m$ reads
\begin{equation}\label{eq:optimalityeq_mask}
    f(\V y,\V x) = \M H^T \boldsymbol \phi_{\V x}(\M H \V x, \V y) + \lambda \sum_{c=1}^{N_C} \M W_c^T \boldsymbol \Lambda_c(\V{y}) \odot \boldsymbol \psi_c^\prime(\M W_c \V x) = \V 0,
\end{equation}
where the $\boldsymbol \Lambda_c(\V{y})$ makes $\operatorname{Graph}(S_\lambda)$ non-polyhedral.
Hence, the proof of Theorem \ref{thm:LipConv} needs to be modified, and we can only establish the weaker Aubin property \eqref{eq:Aubin}.
First, we decompose $\R^n \times \R^m$ into polyhedra $\Omega_{i} = \{(\V x, \V y) \in \R^n \times \R^m: \M F_i \V x + \M G_i \V y \leq \V q_i \}$ with $\M F_i \in \R^{l,n}$, $\M G_i \in \R^{l,m}$ and $\V q_i \in \R^{l}$ such that that on each $\Omega_i$ we have $\boldsymbol \psi_c^\prime(\M W_c \V x) = \M D_{\boldsymbol \psi_c,i} \M W_c \V x + \V d_{\boldsymbol \psi_c,i}$ and $\boldsymbol \phi_{\V x}(\M H \V x, \V y) = \M D_{\boldsymbol \phi, \V x,i} \M H \V x + \M D_{\boldsymbol \phi, \V y,i} \V y + \V d_{\boldsymbol \phi,i}$ with nonnegative diagonal matrices $\M D_{\boldsymbol \psi_c,i} \in \R^{n,n}$ and $\M D_{\boldsymbol \phi, \V x, i} , \M D_{\boldsymbol \phi, \V y, i} \in \R^{m,m}$, see also \eqref{eq:psi_prime_linearization} and \eqref{eq:phi_prime_linearization}.
Then, $f_{\V y} \coloneq f(\V y,\cdot)$ has the affine components $f_{\V y, i} (\V x) = \M E_{\V y,i} \V x + \V b_{\V y,i}$ with
\begin{align}
    \M E_{\V y,i} &= \M H^T \M D_{\boldsymbol \phi, \V x,i} \M H + \lambda \sum_{c=1}^{N_C} \M W_c^T \mathbf{diag}\bigl(\boldsymbol \Lambda_c(\V{y})\bigr) \M D_{\boldsymbol \psi_c,i} \M W_c \in \R^{n,n}, \\
    \V b_{\V y,i} &=  \M H^T ( \M D_{\boldsymbol \phi, \V y,i} \V y  + \V d_{\phi,i} ) + \lambda \sum_{c=1}^{N_C} \M W_c^T \boldsymbol \Lambda_c(\V{y}) \odot \V d_{\boldsymbol{\psi}_c,i} \in \R^{n},
\end{align}
and $\mathrm{graph} (f_{\V y}) = \bigcup_i \mathrm{graph} (f_{\V y,i})$.
With this, we can formulate the main result of this section.
\begin{theorem}[Stability for Lipschitz Mask]
    Let $\boldsymbol{\Lambda}$ be Lipschitz continuous and  $\psi_c^\prime$, $\partial_{ x}\phi_j$ be piecewise affine with $\psi_c^{\prime \prime}\geq 0$ and $\partial^2_{x}\phi_j> 0$ a.e. 
    Then, $S_\lambda$ has the Aubin property at every $(\bar{\V y}, \bar{\V x}) \in \mathrm{Graph}(S_\lambda)$, namely there exists $\kappa>0$ and neighborhoods $V$ and $U$ of $\bar{\V y}$ and $\bar{\V x}$, respectively, such that    \begin{equation}\label{eq:AubinconvMask}
    S_\lambda(\V y^{\prime})\cap U \subset S_\lambda(\V y)+\kappa\Vert \V y^{\prime} - \V y\Vert  \mathbb{B} \quad \text { for all } \V y^{\prime}, \V y \in V.
    \end{equation}
    Moreover, if all $\boldsymbol \Lambda_c$ are semi-algebraic, then there exists a set $D \subset \R^m$ with $\dim(D) \leq m-1$ such that $S_\lambda$ is even strictly continuous \eqref{eq:StrictCon} at any $\bar{\V y} \in D$.
\end{theorem}
\begin{proof}
We establish \eqref{eq:AubinconvMask} based on Theorem \ref{thm:GeneralImplicitFunction}.
To this end, we note that $f(\cdot,\V x)$ is Lipschitz continuous since $\boldsymbol \Lambda$ and $\boldsymbol \phi_{\V x}(\M H \cdot, \cdot)$ are Lipschitz continuous.
Moreover, the dependence of the constant on $\V x$ is continuous.
Hence, the first requirement of Theorem \ref{thm:GeneralImplicitFunction} is satisfied.
Second, we need to show that $f_{\V y}^{-1}$ is calm \eqref{eq:Calm} at $\V 0$ with a uniform constant on some neighborhood of $\bar{\V y}$.
To this end, we show a stronger property, namely that $f_{\V y}^{-1}$ is outer Lipschitz continuous \eqref{eq:OutLip} at $\V 0$ with a uniform constant.
Our proof is inspired by \cite[Thm.\ 3C.3]{DonRoc2009}, with an explicit treatment of the dependence on $\V y$.

First, note that $\mathrm{graph} (f_{\V y}^{-1}) = \bigcup_i \mathrm{graph} (f_{\V y,i}^{-1})$, where any $\V x \in f^{-1}_{\V y,i}(\V z)$ satisfies $\M E_{\V y,i} \V x = \V z - \V b_{\V y,i}$ and $\M F_i \V x \leq \V q_i - \M G_i \V y$.
For every $(\V z, \V x) \in \mathrm{graph} (f_{\V y,i}^{-1})$ and $\V z^\prime \in \mathrm{dom}(f_{\V y,i}^{-1})$, we apply the Hoffman bound \eqref{eq:HoffmannEst} to get
\begin{align}\label{eq:HoffEstConv}
    \min_{\V x^{\prime} \in f_{\V y,i}^{-1}(\V z^\prime)}\Vert \V x - \V x^{\prime}\Vert  \leq& K (\M E_{\V y,i},\M F_i)  \biggl \Vert
    \begin{pmatrix}
        \M E_{\V y,i} \V x - \V b_{\V y,i} - \V z^\prime\\
        (\M F \V x + \M G_i \V y - \V q_{i})_+
    \end{pmatrix}\biggr\Vert
    = K (\M E_{\V y,i},\M F_i)  \Vert \V z - \V z^{\prime} \Vert.
\end{align}
Since \eqref{eq:HoffEstConv} holds for any $\V x \in f_{\V y,i}^{-1}(\V z)$, we obtain that $f_{\V y,i}^{-1}$ is Lipschitz continuous on $\operatorname{dom}(f_{\V y,i}^{-1})$, namely
\begin{equation}\label{eq:LipContMask}
    f_{\V y,i}^{-1}(\V z) \subset  f_{\V y,i}^{-1}(\V z^{\prime}) + K (\M E_{\V y,i},\M F_i)  \Vert \V z - \V z^{\prime} \Vert \mathbb B.
\end{equation}

Next, we show by contradiction that $\sup_{i,\V y \in B_1(\bar {\V y})} K(\M E_{\V y,i},\M F_i) < \infty$.
Since $\M E_{\V y,i}$ is symmetric and $\boldsymbol \Lambda_c$ is positive, it holds that $\ker(\M E_{\V y,i}) = \bigcap_c\ker(\M D_{\boldsymbol \psi_c,i} \M W_c) \cap \ker (\M D_{\boldsymbol \phi, \V x,i} \M H)$.
Hence, the spaces $\ran(\M E_{\V y,i}) = \ran(\M E_{\V y,i}^T) = \mathrm{ker}(\M E_{\V y,i})^\perp$ are independent of $\V y$.
Thus, the same holds true for $U(\M E_{\V y,i}, \M F_i)$.
Now, assume that there exists $\V y_k \to \V y^* \in B_1(\bar {\V y})$ such that $K(\M E_{\V y_k,i},\M F_i) \to \infty$ for some $i$.
Then, there exists $J \in U(\M E_{\V y_k,i}, \M F_i) = U(\M E_{\V y,i}, \M F_i)$ such that for some subsequence indexed by $N \in \mathscr{N}^{\sharp}$ it holds
\begin{equation}
    \min \Bigl \{
    \bigl\Vert \M E_{\V y_k,i} \V u + (\M F_i)_J^T \V v \bigr\Vert:  \V u \in \ker(\M E_{\V y^*,i})^\perp,\, \V v \geq \V 0,\, \Vert (\V u, \V v) \Vert = 1
    \Bigr\} \xrightarrow{N} 0.
\end{equation}
In particular, there exists a convergent sequence $\{(\V u_k, \V v_k)\}_{k \in \N} \subset \ker(\M E_{\V y^*,i})^\perp \times \R_{\geq 0}^{\vert J \vert}$ with $\Vert (\V u_k, \V v_k) \Vert = 1$ and $\Vert \M E_{\V y_k,i} \V u_k + (\M F_i)_J^T \V v_k \Vert \to 0$.
Due to the Lipschitz continuity of $\boldsymbol \Lambda_c$, the limit $(\V u, \V v)$ satisfies
\begin{equation}\label{eq:EstLipData}
    \bigl \Vert \M E_{\V y^*,i} \V u + (\M F_i)_J^T \V v \bigr\Vert = 0 \quad \text{and} \quad \Vert (\V u, \V v) \Vert = 1.
\end{equation}
Since $\ran((\M F_i)^T_J) \cap \ran(\M E_{\V y^*,i}) = \{\V 0\}$, \eqref{eq:EstLipData} implies $\M E_{\V y^*,i} \V u = \V 0$ and $(\M F_i)_J^T \V v = \V 0$.
As $(\M F_i)_J$ has full rank and $\V u \in \mathrm{ker}(\M E_{\V y^*,i})^\perp$, we get the contradiction $(\V u, \V v) = \V 0$.

Let \smash{$I_{\V y}=\{i : f_{\V y,i}^{-1}(\V 0) \neq \emptyset\}$} such that
\smash{$f_{\V y}^{-1}(\V 0)=\bigcup_{i \in I_{\V y}} f_{\V y,i}^{-1}(\V 0)$}.
For $i \notin I_{\V y}$,  the polyhedra $\{\V 0\} \times \mathbb{R}^{n}$ and \smash{$\mathrm{graph} (f_{\V y,i}^{-1})$} are disjoint.
Hence, there are neighborhoods $V_{\V y, i}$ of $\V 0$ with $(V_{\V y, i} \times \mathbb{R}^{n}) \cap \mathrm{graph} (f_{\V y,i}^{-1})=\emptyset$.
For the neighborhood $V_{\V y}=\bigcap_{i \notin I_{\V y}} V_{\V y, i}$ of $\V 0$, we have
\begin{equation}\label{eq:GraphSelection}
\left(V_{\V y} \times \mathbb{R}^{n}\right) \cap \operatorname{graph} (f_{\V y}^{-1}) \subset \bigcup_{i}\mathrm{graph} (f_{\V y,i}^{-1}) \backslash \bigcup_{i \notin I_{\V y}} \mathrm{graph} (f_{\V y,i}^{-1}) \subset \bigcup_{i \in I_{\V y}} \mathrm{graph} (f_{\V y,i}^{-1}).
\end{equation}
For $\V z \in V_{\V y}$ and $\V x \in f_{\V y}^{-1}(\V z)$, we get from \eqref{eq:GraphSelection} that
\smash{$(\V z, \V x) \in \bigcup_{i \in I_{\V y}} \mathrm{graph} (f_{\V y,i}^{-1})$}.
Hence, we have $\V x \in f_{\V y,i}^{-1}(\V z)$ for some $i \in I_{\V y}$.
Since $\V 0 \in \mathrm{dom}(f_{\V y,i}^{-1})$, we can use \eqref{eq:LipContMask} to obtain
\begin{equation}
    \V x \in  \Bigl (f_{\V y,i}^{-1}(\V 0) + \sup_{i,\V y \in B_1(\bar {\V y})} K (\M E_{\V y,i},\M F_i) \Vert \V z - \V 0 \Vert \mathbb B\Bigr) \subset f_{\V y}^{-1}(\V 0) + \sup_{i,\V y \in B_1(\bar {\V y})} K (\M E_{\V y,i},\M F_i) \Vert \V z \Vert \mathbb B.
\end{equation}
Since $\V x \in f_{\V y}^{-1}(\V z)$ was arbitrary, we conclude that $f_{\V y}^{-1}$ is outer Lipschitz continuous at $\V 0$ with constant independent of $\V y \in B_1(\bar {\V y})$.
Now, the Aubin property follows from Theorem \ref{thm:GeneralImplicitFunction}.

For the second part, we note that $\mathrm{graph}({S_\lambda})$ is semi-algebraic if all the $\boldsymbol \Lambda_c$ are semi-algebraic.
Hence, we can apply \cite[Thm.\ 28]{DanPan2011}, which yields strict continuity of $S_\lambda$ outside a set $D$ with $\mathrm{dim}(d) \le m-1$.
\end{proof}
\begin{remark}
    By slightly modifying the proof, we can also treat the mask $\boldsymbol{\Lambda}$ as an input parameter of  $S_\lambda$.
    Then, the Aubin property and the strict continuity hold for both inputs.
\end{remark}

\subsection{Convergence for Vanishing Noise}\label{sec:VanishNoise}
Now, we look into the setting of vanishing noise.
To emphasize the dependence on the regularization parameter $\lambda$, we use the notation $f_{\V y, \lambda}$ instead of $f_{\V y}$.
Our first result is derived along the lines of \cite[Thm.\ 3.26]{Scherzer2009} with several modifications due to the data dependence of $\mathcal R_{\V y}$.
Throughout, $\mathcal R$ denotes the baseline regularizer with $\boldsymbol \Lambda_c = \boldsymbol 1_n$.
\begin{theorem}[Convergence of Reconstructions for Vanishing Noise]\label{thm:ConvNoise}
Let $\V y \in \R^m$, $\lambda\colon\mathbb{R}_{>0}\to \mathbb{R}_{>0}$ satisfy $\lambda(\delta)\to 0$ and $\delta/\lambda(\delta) \to 0$ as $\delta \to 0$, and $\boldsymbol{\Lambda}$ be Lipschitz continuous.
Moreover, assume that $\mathcal D_P$ is a divergence, that $\mathcal R$ is coercive and that $S_0(\V y) \coloneqq \argmin_{\{\V x: \M H \V x = \V y\}} \mathcal R_{\V y}(\V x) \neq \emptyset$.
Then, for any $\{\delta_k\}_{k \in \N} \subset \mathbb{R}_{>0}$ converging to zero and any $\{\V y_k\}_{k \in \N} \subset \R^m$ with $\mathcal D_P(\V y,\V y_k) \leq \delta_k$, it holds with $\lambda_k = \lambda(\delta_k)$ that any sequence $\V x_k \in S_{\lambda_k}(\V y_k)$ has a convergent subsequence with limit $\V x^* \in S_0(\V y)$.
\end{theorem}
\begin{proof}
Fix some $\hat{\V x} \in S_0(\V y)$.
Then, it holds that
\begin{equation}\label{eq:EstEnergy}
    \mathcal D_P(\M H \V x_k, \V y_k ) + \lambda_k \mathcal R_{\V y_k}(\V x_k) = f_{\V y_k, \lambda_k}(\V x_k) \leq f_{\V y_k,\lambda_k}(\hat{\V x}) 
    = \mathcal D_P(\V y,\V y_k)  +  \lambda_{k} \mathcal R_{\V y_k} (\hat{\V x}) .
\end{equation}
Since $\lim_{k \to \infty} \mathcal R_{\V y_k} (\hat{\V x}) = \mathcal R_{\V y} (\hat{\V x})$, the right-hand side in \eqref{eq:EstEnergy} converges to zero.
As $\mathcal D_P$ is a divergence and continuous due to \eqref{eq:data_term}, this implies $\lim_{k \to \infty} \M H \V x_k = \V y$.
Further, \eqref{eq:EstEnergy} implies that $\mathcal R_{\V y_k}(\V x_k) \leq  \delta_k/\lambda_{k} +  \mathcal R_{\V y_k}(\hat{\V x})$, from which we conclude that $\limsup_{k \to \infty} \mathcal R_{\V y_k}(\V x_k) \leq \mathcal R_{\V y}(\hat{\V x})$.
Let us define $\bar \lambda = \max_k \lambda_k$ and 
\begin{equation}\label{eq:Boundcoerv}
    M \coloneqq \bar \lambda \mathcal R_{\V y}(\hat{\V x}) \geq \limsup_{k \to \infty} \mathcal D_P(\M H \V x_k, \V y_k ) + \bar \lambda \mathcal R_{\V y_k}(\V x_k) \geq \limsup_{k \to \infty} \mathcal D_P(\M H \V x_k, \V y_k ) + \bar \lambda \epsilon \mathcal R(\V x_k),
\end{equation}
which holds true since $\boldsymbol \Lambda_c \ge \epsilon \boldsymbol 1_n$ for all $c=1,\ldots,N_c$.
As $\mathcal R$ is coercive, the sequence $\{\V x_k\}_{k \in \N}$ is bounded and admits a convergent subsequence with limit point $\V x^*$.
By continuity of $\M H$, it holds $\M H \V x^* = \V y$.
Moreover, the continuity of $\mathcal R_{\V y_k}$ implies 
\begin{equation}
    \mathcal R_{\V y} (\V x^*) \leq \limsup_{k \to \infty} \mathcal R_{\V y} (\V x_k) \leq \limsup_{k \to \infty} \mathcal R_{\V y_k} (\V x_k) + \vert \mathcal R_{\V y} (\V x_k) - \mathcal R_{\V y_k} (\V x_k) \vert \leq \mathcal R_{\V y} (\hat{\V x}), 
\end{equation}
where we use the uniform Lipschitz continuity of $\V y \mapsto \mathcal R_{\boldsymbol{\V y}}(\V x)$ on compact sets for the last inequality.
Thus, it holds that $\V x^* \in S_0(\V y)$ as desired.
\end{proof}
Next, we establish convergence rates based on a source condition.
Here, we restrict ourselves to $\boldsymbol \Lambda_c = \boldsymbol 1_n$ and  $\mathcal{D}_P(\M H \V x, \V y) = \frac{1}{2} \Vert \M H \V x - \V y \Vert^2$.
\begin{theorem}[Convergence Rates for Vanishing Noise]
Let $\V y \in \R^m$, $\mathcal D_P (\M H \V x, \V y) = \frac12 \Vert \M H \V x - \V y \Vert^2$, $\psi_c^\prime$ be piecewise affine with $\psi_c^{\prime \prime}\geq 0$ a.e.\ and $\boldsymbol \Lambda_c = \boldsymbol 1_n$.
Assume that there are $\hat{\V x}, \V u  \in \R^n$ with $\M H \hat{\V x} = \V y$ and $\nabla \mathcal R (\hat{\V x}) = \M H^T \M H \V u$ such that the $\boldsymbol \psi^\prime_c(\M W_c \V x)$ are differentiable in direction $\V u$ at $\hat{\V x}$.
Further, assume that $\lambda\colon\mathbb{R}_{>0}\to \mathbb{R}_{>0}$ satisfies $\lambda(\delta)\to 0$ and $\delta/\lambda(\delta) \to 0$ for $\delta \to 0$.
If $\{\delta_k\}_{k \in \N} \subset \mathbb{R}_{>0}$ converges to zero, and $\{\V y_k\}_{k \in \N} \subset \R^m$ satisfies $\Vert \V y - \V y_k\Vert \leq \delta_k$, then it holds for $k\in \mathbb{N}$ and $\lambda_k = \lambda(\delta_k)$ that $\Vert S_{\lambda_k}(\V y_k) - \hat{\V x}\Vert \leq C \min (\lambda_k, \delta_k/\lambda_k)$.
\end{theorem}
\begin{proof}
Using the same notation as in the proof of Theorem \ref{thm:LipConv}, we linearize $f_{\V y_k,\lambda_k}(\hat {\V x} + \V x) = \V 0$  around $\hat {\V x}$.
By incorporating $\M H \hat{\V x} = \V y$ and setting $\V h_{\V y_k} = \V y_k - \V y$, we get the linearized equation
\begin{equation}\label{eq:LinarizedConv}
    \M H^T \bigl(\M H \V x - \V h_{\V y_k} + \lambda_k \M H \V u) + \lambda_k \sum_{c=1}^{N_C} \M W_c^T \M D_{\boldsymbol \psi_c}(\V x) \M W_c \V x = \V 0.
\end{equation}
From Theorem~\ref{thm:existence}, we infer that the associated variational formulation 
 \begin{equation}
    \argmin_{\V x \in \R^n} \frac12 \Vert \M H \V x - \V h_{\V y,k} + \lambda_k \M H \V u  \Vert^2 + \frac{\lambda_k}{2} \sum_{c=1}^{N_C} \Vert \M W_c \V x \Vert^2_{\M D_{\boldsymbol \psi_c}(\V x)} \label{eq:LinSource}
\end{equation}
admits solutions.
Now, we show that there exist solutions to \eqref{eq:LinSource} that converge to zero.

For this, we partition $\R^n$ into polyhedra $\Omega_i = \{\V x \in \R^n : \M F_i \V x \leq \V 0\}$ on which \eqref{eq:LinarizedConv} is affine with matrices $\M E_{i,k} = \M H^T \M H + \lambda_k \sum_{c=1}^{N_C} \M W_c^T \M D_{\boldsymbol \psi_c,i} \M W_c$.
Since \eqref{eq:LinSource} is nonempty, at least one of the polyhedra
\begin{equation}
        P_{i,k}(\V h_{\V y_k}) = \bigl \{\V x \in \R^n:\M E_{i,k} \V x = \M H^T \bigl(\V h_{\V y_k} - \lambda_k \M H \V u\bigr),\, \M F_{i} \V x \leq \V 0 \bigr\}
\end{equation}
is nonempty.
We estimate $\min_{\V x\in P_{i,k}(\V h_{\V y_k})} \Vert \V x + \lambda_k \V u \Vert$ using the Hoffman bound \eqref{eq:HoffmannEst}, namely
\begin{align}\label{eq:DistPoly}
    \min_{\V x \in P_{i,k}(\V h_{\V y_k})} \Vert \V x + \lambda_k \V u \Vert  \leq K(\M E_{i,k}, \M F_i) \biggl \Vert
    \begin{pmatrix}
        \M H^T \V h_{\V y_k} + \lambda_k^2 \sum_{c=1}^{N_C} \M W_c^T \M D_{\boldsymbol \psi_c,i} \M W_c \V u\\
        (- \lambda_k \M F_i \V u)^+
    \end{pmatrix}\biggr\Vert.
\end{align}
Here, the Hoffman constant, see \eqref{eq:hoffmanBound}, is given by
\begin{align}
    K(\M E_{i,k}, \M F_i) &= \max_{J \in U(\M E_{i,k}, \M F_i)} \Bigl(\min \bigl \{
    \bigl\Vert \M E_{i,k} \V u + (\M F_i)_J^T \V v \bigr\Vert:  \V u \in \ker(\M E_{i,k})^\perp,\, \V v \geq 0,\, \Vert (\V u, \V v) \Vert = 1
    \bigr\}\Bigr)^{-1}.
\end{align}
Since $(A \cap B)^\perp = A^\perp + B^\perp$ for subspaces $A,B \subset \R^n$, $\ker (\M E_{i,1}) = \bigcap_c \ker(\M D_{\boldsymbol \psi_{c,i}} \M W_c) \cap \ker(\M H)$ implies
\begin{equation}\label{eq:SumRange}
    \ran (\M E_{i,k}) = \ran (\M E_{i,1}) = \Bigl(\bigcap_c \ker(\M D_{\boldsymbol \psi_{c,i}} \M W_c) \cap \ker(\M H)\Bigr)^\perp = \ran(\M H^T \M H) + \ran\Bigl(\sum_{c=1}^{N_C} \M W_c^T \M D_{\boldsymbol \psi_c,i} \M W_c\Bigr).
\end{equation}
For a contradiction, assume that $\limsup_{k \to \infty} \lambda_k K(\M E_{i,k}, \M F_i) = \infty$ for some $i \in \N$.
Then, there exists $J \in U(\M E_{i,k}, \M F_i) = U(\M E_{i,1}, \M F_i)$ such that for some subsequence indexed by $N_1$ it holds
\begin{equation}\label{eq:LimConst}
    \lim_{k \in N_1} \lambda_k^{-1} \min \bigl \{
    \bigl\Vert \M E_{i,k} \V u + (\M F_i)_J^T \V v \bigr\Vert:  \V u \in \ker(\M E_{i,1})^\perp,\, \V v \geq \V 0,\, \Vert (\V u, \V v) \Vert = 1
    \bigr\} = 0.
\end{equation}
Thus, there exists $\{(\V u_k, \V v_k)\}_{k \in N_1} \subset \ker(\M E_{i,1})^\perp \times \R_{\geq 0}^{\vert J \vert}$ with
\begin{equation}\label{eq:LimConst2}
    \lim_{k \in N_1} \lambda_k^{-1} \bigl(\M H^T \M H \V u_k + (\M F_i)_J^T \V v_k \bigr) + \sum_{c=1}^{N_C} \M W_c^T \M D_{\boldsymbol \psi_c,i} \M W_c \V u_k = \V 0.
\end{equation}
Now, we decompose $\V u_k = \V u_{k,\perp} + \V u_{k,\ker}$ with $\V u_{k,\perp} \in \ker(\M H)^\perp \cap \ran(\M E_{i,1})$ and $\V u_{k,\ker} \in \ker(\M H) \cap \ran(\M E_{i,1})$, which satisfy $\Vert \V u_{k,\perp} \Vert^2 + \Vert \V v_k \Vert^2 = 1 - \Vert \V u_{k,\ker} \Vert^2$.
Plugging this into \eqref{eq:LimConst2} leads to
\begin{equation}
    \lim_{k \to \infty} \frac{\Vert (\V u_{k,\perp}, \V v_{k}) \Vert}{\lambda_k}\frac{ \M H^T \M H \V u_{k,\perp} + (\M F_i)_J^T \V v_{k}}{\Vert (\V u_{k,\perp}, \V v_{k}) \Vert} + \sum_{c=1}^{N_C} \M W_c^T \M D_{\boldsymbol \psi_c,i} \M W_c \V u_k = \V 0.
\end{equation}
Due to compactness, there are indices $N_2$ with $\lim_{k \in N_2} (\V u_{k,\perp}, \V v_{k}) / \Vert (\V u_{k,\perp}, \V v_{k}) \Vert = (\V u_{\perp}, \V v) \neq \V 0 \in \smash{\ker(\M H)^\perp \times \R_{\geq 0}^{\vert J \vert}}$ and $\lim_{k \in N_2} \V u_k = \hat{\V u} \in \ker(\M E_{i,1})^\perp$.
For this subsequence, it holds that
\begin{equation}
    \lim_{k \in N_2} \frac{\Vert (\V u_{k,\perp}, \V v_{k}) \Vert}{\lambda_k} \bigl(\M H^T \M H \V u_{\perp} + (\M F_i)_J^T \V v \bigr) + \sum_{c=1}^{N_C} \M W_c^T \M D_{\boldsymbol \psi_c,i} \M W_c \hat{\V u} = \V 0.
\end{equation}
As $\ran(\M H^T \M H) \cap \ran((\M F_i)_J^T) = \{\V 0\}$, we get from $(\V u_{\perp}, \V v) \neq \V 0$ that $\M H^T \M H \V u_{\perp} + (\M F_i)_J^T \V v \neq \V 0$.
This implies that $\bar{\lambda} = \lim_{k \in N_2} \lambda_k /\Vert (\V u_{k,\perp}, \V v_{k}) \Vert > 0$ and $\lim_{k \in N_2} \Vert (\V u_{k,\perp}, \V v_{k}) \Vert = 0$.
Thus, we have that $\hat{\V u} = \lim_{k \in N_2} \V u_{k,\perp} + \V u_{k,\ker} \in \ker(\M H)$ with $\Vert \hat{\V u} \Vert = 1$ and by \eqref{eq:SumRange} that
\begin{equation}\label{eq:EstRange}
    - (\M F_i)_J^T \V v = \M H^T \M H \V u_{\perp} + \bar{\lambda} \sum_{c=1}^{N_C} \M W_c^T \M D_{\boldsymbol \psi_c,i} \M W_c \hat{\V u} \in \ran(\M E_{i,1}).
\end{equation}
Since $\ran(\M E_{i,1}) \cap \ran((\M F_i)_J^T) = \{\V 0\}$, \eqref{eq:EstRange} implies $\M H^T \M H \V u_{\perp} + \bar{\lambda} \sum_{c=1}^{N_C} \M W_c^T \M D_{\boldsymbol \psi_c,i} \M W_c \hat{\V u} = \V 0$.
By multiplying this from the left with $\hat{\V u}^T \in \ker(\M H)$, we get $\hat{\V u} \in \bigcap_c \ker(\M D_{\boldsymbol \psi_{c,i}} \M W_c)$.
Thus, $\hat{\V u} \in \ker(\M E_{i,1})$ with  $\Vert \hat{\V u} \Vert = 1$, which contradicts $\hat{\V u} \in \ker(\M E_{i,1})^\perp$.
Hence, it holds $\limsup_{k \to \infty} \lambda_k K(\M E_{i,k}, \M F_i) < \infty$ and there exists $C>0$ with $K(\M E_{i,k}, \M F_i) \leq C \lambda_k^{-1}$ for all $i$ and $k$.

Now, we are able to finally estimate
\begin{align}
    \min_{i, \V x \in P_{i,k}(\V h_{\V y_k})} \Vert \V x \Vert \leq & \lambda_k \Vert \V u \Vert + \min_{i, \V x^\prime \in P_{i,k}(\V h_{\V y_k})} \Vert \V x^\prime + \lambda_k\V u \Vert \notag\\
    \leq & \lambda_k \Vert \V u \Vert + C \lambda_k^{-1} \biggl \Vert
    \begin{pmatrix}
        \M H^T \V h_{\V y_k} + \lambda_k^2 \sum_{c=1}^{N_C} \M W_c^T \M D_{\boldsymbol \psi_c,i} \M W_c \V u\\
        (- \lambda_k \M F_i \V u)^+
    \end{pmatrix}\biggr\Vert\notag\\
    \leq & \lambda_k \Vert \V u \Vert + C \lambda_k^{-1} \biggl( \Vert \M H \Vert \delta_k + \lambda_k^2 \Bigl\Vert\sum_{c=1}^{N_C} \M W_c^T \M D_{\boldsymbol \psi_c,i} \M W_c \V u \Bigr\Vert + \lambda_k \Vert (\M F_i \V u)^+\Vert \biggr).
\end{align}
Since the $\psi^\prime_c(\M W_c \V x)$ are differentiable in direction $\V u$ at $\hat{\V x}$, the line $t \mapsto t \V u$ lies in a single $\Omega_i$. 
This implies $\M F_i \V u = \V 0$ and the convergence rate is $\min(\lambda_k, \delta_k/\lambda_k)$ as desired.
\end{proof}
\begin{remark}
    If we choose $\lambda_k \sim \sqrt{\delta_k}$, we get the optimal convergence order $\sqrt{\delta_k}$, which coincides with the result in \cite{Scherzer2009}.
    However, in contrast to that work, our estimates are expressed using the (often stronger) set-wise distance instead of the Bregman distance \begin{equation}
        D_{\mathcal R}(\V x, \hat{\V x}) = \mathcal R(\V x) - \mathcal R (\hat{\V x}) - \langle \nabla \mathcal R (\hat{\V x}), \V x - \hat{\V x} \rangle.
    \end{equation}
\end{remark}
\subsection{The Bayesian Perspective} \label{sec:Bayesian}

Let $\V X \in \R^n$ be an absolutely continuous random variable with law $P_{\V X}$ and density $p_{\V X} \colon \R^n \to [0,\infty)$.
We can interpret the inverse problem $\V y = \mathrm{cor}(\M H \V x)$ as a Bayesian one
\begin{equation}
\V Y = \mathrm{cor} (\M H \V X),
\end{equation}
where we want to sample from the conditional posterior $P_{\V X|\V Y=\V y}$.
The most likely sample is the maximum a-posteriori (MAP) estimator, which is given by
\begin{equation}
\V x_{\mathrm{MAP}}(\V y) \in \argmax_{\V x \in \R^n} p_{\V X|\V Y= \V y} (\V x).
\end{equation}
By Bayes' theorem, finding the MAP estimator is equivalent to
\begin{equation}\label{eq:MAP}
\V x_{\mathrm{MAP}}(\V y) \in   \argmax_{\V x \in \R^n} \log p_{\V X|\V Y= \V y} (\V x) 
= \argmin_{\V x \in \R^n} \bigl\{- \log p_{\V Y|\V X = \V x} (\V y) - \log p_{\V X} (\V x) \bigr\}.
\end{equation}
For the data likelihood $p_{\V Y|\V X=\V x} (\V y) = C \exp ( -\mathcal{D}_P(\M H \V x, \V y))$ and a Gibbs prior $p_{\V X}(\V x) = C_\lambda \exp (- \lambda \mathcal{R}(\V x) )$ with $\mathcal R \colon \R^n \to \R$, computing \eqref{eq:MAP} is equivalent to minimizing the variational problem \eqref{eq:var_prob}, namely
\begin{align*}
\V x_{\mathrm{MAP}}(\V y) \in 
\argmin_{\V x \in \R^n} \{\mathcal{D}_P(\M H \V x, \V y) + \lambda \mathcal{R}(\V x) \}.
\end{align*}
To make this rigorous, $C_\lambda \exp (- \lambda \mathcal{R}(\V x) )$ has to be integrable so that it defines a density.
In Theorem~\ref{thm:reg_integrable} we provide mild assumptions under which $\mathcal R$ of the form \eqref{eq:ridge_reg} defines a density.
\begin{theorem} \label{thm:reg_integrable}
If $\bigcap_c\mathrm{ker}(\M W_c) = \{ \V 0 \}$ and there are $a>0$, $b \in \R$ with $\psi_c(x) \geq a \vert x \vert + b$ for all $c$, then $\mathcal{R}(\V x) = \sum_{c=1}^{N_c} \langle \boldsymbol{1}_n, \boldsymbol \psi_c (\M W_c \V x) \rangle$ induces a Gibbs prior, i.e., $\exp (-\lambda \mathcal{R}(\V x))$ is integrable for any $\lambda > 0$.
\end{theorem}
\begin{proof}
%We will show that it holds $R(x) = \sum_{j=1}^{N_c} \Vert \psi_j (W_j x) \Vert_1 \ge C \Vert x \Vert_1$ for some constant 
First, we show that there exists some $\gamma>0$ with $\max_c \Vert \M W_c \V x \Vert_1 \ge \gamma \Vert \V x \Vert_1 $.
To this end, let $m(\V x) = \max_c \Vert \M W_c \V x \Vert_1$.
Note that $m$ is continuous and positively homogenous, namely 
$m(a \V x) = \vert a \vert m( \V x)$ for all $a>0$. 
Let $\gamma = \min_{\V x \in \R^n: \Vert \V x \Vert_1=1} m(\V x)$.
Since, $m$ is continuous, it attains its minimum for some $\hat{\V x}$ in the compact set $\{\V x \in \R^n: \Vert \V x \Vert_1=1\}$.
Since $\hat{\V x} \notin \bigcap_c \mathrm{ker}(\M W_c)$, it holds $\gamma = m(\hat{\V x})>0$.
Now, $\max_c \Vert \M W_c \V x \Vert_1 \ge \gamma \Vert \V x \Vert_1 $ follows from the positive homogeneity.
Thus, it holds
\begin{equation}
\mathcal{R}(\V x) = \sum_{c=1}^{N_c} \langle \boldsymbol{1}_n, \boldsymbol \psi_c (\M W_c \V x) \rangle \geq \max_c \langle \boldsymbol{1}_n, \boldsymbol \psi_c (\M W_c \V x) \ge a \max_c \Vert \M W_c \V x \Vert_1 + b \geq a \gamma \Vert \V x \Vert_1 + b.
\end{equation}
For $\lambda>0$, this yields
\begin{align} \label{eq:int_max}
\int_{\R^n} \exp( - \lambda \mathcal{R}(\V x)) \dx \V x & \le \int_{\R^n} \exp( - \lambda a \gamma \Vert \V x \Vert_1 - \lambda b) \dx \V x  = \exp( - \lambda b) \int_{\R^n} \exp( - \lambda a \gamma \Vert \V x \Vert_1) \dx \V x.
\end{align}
Since \eqref{eq:int_max} is finite, $\exp( - \lambda \mathcal{R}(\V x))$ is integrable.
\end{proof}
\begin{remark}
    The assumptions of Theorem~\ref{thm:reg_integrable} are satisfied by our learned convex baseline $\mathcal R$.
    Moreover, Theorem~\ref{thm:reg_integrable} also holds for the adaptive regularizer $\mathcal{R}_{\V y}$ \eqref{eq:rridge_reg_mask} with fixed $\V y$ since $\mathcal{R}_{\V y}(\V x) \ge \epsilon \mathcal{R}(\V x)$.
\end{remark}

\section{Parameterization and Training}
\begin{figure}
\centering
\begin{subfigure}[t]{\textwidth}
\includegraphics[width=\linewidth]{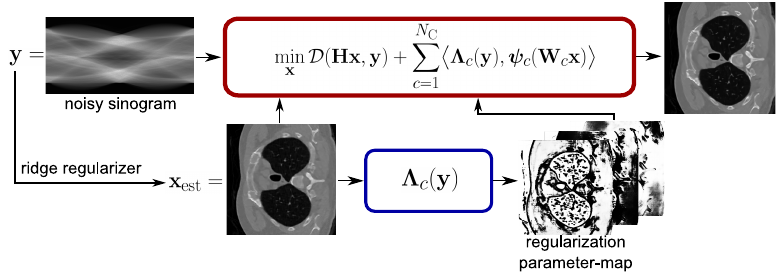}
\end{subfigure}
\caption{Visualization of our scheme.
First, we use a ridge regularizer $\mathcal R$ to compute an initial reconstruction $\V x_{\mathrm{est}}$ from the data $\V y$ based on \eqref{eq:var_prob}.
Then, we generate the masks $\bm \Lambda_c (\V y)$ based on $\V x_{\mathrm{est}}$, and use $\V x_{\mathrm{est}}$ as initialization for the second minimization.} \label{fig:visualization_scheme}
\end{figure}

In \cite{GouNeuUns2023}\footnote{Code: \url{https://github.com/axgoujon/weakly_convex_ridge_regularizer}} the unadaptive regularizer \eqref{eq:ridge_reg} was learned on 400 images of the BSDS500 \cite{bsd_data} training set such that
\begin{equation}\label{eq:Denoiser}
D_{\sigma}(\V y) = \mathrm{prox}_{\mathcal{R}_{\sigma}} (\V y) = \argmin_{\V x \in \R^n} \frac{1}{2} \Vert \V x - \V y \Vert_2^2 + \mathcal{R}_{\sigma} (\V x)
\end{equation}
is a good Gaussian denoiser for any noise level $\sigma \in [0,30/255]$.
For this, we compute \eqref{eq:Denoiser} based on accelerated gradient descent (AGD) \cite[Alg.\ 1]{GouNeuUns2023}.
After the training, the same algorithm is used to solve arbitrary inverse problems  with the variational reconstruction model
\begin{equation}\label{eq:InvProbLearn}
    \argmin_{\V x \in \R^n} \frac{1}{2}\Vert \M H \V x - \V y \Vert_2^2 + \lambda R_{\sigma}(\V x),
\end{equation}
where $\sigma$  and $\lambda$ are  hyperparameters.
For image denoising, namely $\M H = \text{Id}$, we use the actual standard deviation $\sigma$ of the noise model and set $\lambda=1$.
Otherwise, we tune them on a small validation set using a grid search.
The learned $N_C = 60$ convolution matrices $\M W_c$ from \cite{GouNeuUns2023} satisfy $\Vert \M W_c \Vert_2 = 1$.
To cope with the varying noise, the potentials $\psi_c (x) = \alpha_c(\sigma)^{-2} \psi ( \alpha_c(\sigma) x)$ depend on $\sigma$ through $\alpha_c (\sigma) = \exp(s_{\alpha_c} (\sigma)) / (\sigma + 10^{-5})$, where the $s_{\alpha_c}$ are linear splines with 11 equally distant knots in $[0,30/255]$.
The $\psi$ is parameterized as $\psi = \mu \psi_{+} - c_{\mathrm{cvx}} \psi_{-}$ with $\mu>0$ and two quadratic splines $\psi_{+}$, $\psi_{-}$ with 101 equally spaced knots, spacing $\Delta = 0.002$ and $\psi^{\prime \prime}_+ (x), \psi^{\prime \prime}_- (x) \in [0,1]$ for a.e.\ $x \in \R$.
Setting $c_{\mathrm{cvx}}=0$ leads to a convex $\mathcal R$ (CRR) \cite{GouNeuBoh2022}, while $c_{\mathrm{cvx}}=1$ leads to a $1$-weakly convex one (WCRR) \cite{GouNeuUns2023}.

Recall that adaptivity should improve the baseline $\mathcal R_{\sigma}$.
Therefore, we reuse the parametrization from \cite{GouNeuUns2023} as well as their learned parameters for the matrices $\M W$ and the common profile $\psi$ in the adaptive regularizer \eqref{eq:rridge_reg_mask}.
Thus, we only adapt the profiles $\psi_c$ in terms of $\alpha_c (\sigma)$ and $\mu$.
To build the additional $\boldsymbol \Lambda_c\colon \R^m \to [0,1]^n$, we first compute an initial reconstruction $\V x_{\mathrm{est}} \in \R^n$ from $\V y \in \R^m$ based on the unadaptive $\mathcal R_{\sigma}$ with the optimal hyperparameters $\lambda$, $\sigma>0$.
The obtained $\V x_{\mathrm{est}}$ is plugged into a network $G \colon \R^n \to ([0,1]^n)^{N_c}$, which we choose as a variant of SwinIR \cite{LCSZGT2021}.
This architecture first deploys a $3\times3$ convolution to extract 60 features.
Then, it applies 4 residual Swin transformer blocks with window size 8 followed again by a $3\times3$ convolution.
Each Swin block has 6 attention heads and contains 6 Swin transformer layers (where the MLPs have hidden dimension 120), which are followed by a $3\times3$ convolution layer.
The final one-step upsamling layer increases the numbers of channels to the output dimension 80, namely the number of filters in \eqref{eq:rridge_reg_mask}, and applies the algebraic function $x \mapsto x / \sqrt{x^2 + 1}$ to enforce the output range $[0,1]$.
All other parameters are chosen as the default ones.
Since $\V x_{\mathrm{est}}$ is always an image, we can learn the parameters of $G$ for denoising as before.
We perform AGD with starting point $\V x_{\mathrm{est}}$ to minimize \eqref{eq:InvProbLearn} for the adaptive $\mathcal R_{\V y,\sigma}$. 
The whole pipeline for a CT setup is visualized in Figure~\ref{fig:visualization_scheme}. 

In summary, we have to learn the $G$, the scalings $\alpha_c (\sigma)$ and the $\mu$.
Let $\theta$ be the aggregation of these parameters and let us denote $\mathcal{R}_{\V y, \sigma, \theta}$ for an explicit reference.
Similar as for $\mathcal R_{\sigma}$, we aim to learn $\theta$ such that $\mathcal{R}_{\V y, \sigma, \theta}$ induces a good multi-noise-level Gaussian denoiser \smash{$D_{\sigma, \theta}(\V y) = \mathrm{prox}_{\mathcal{R}_{\V y, \sigma, \theta}} (\V y)$}.
For this, we extract $M=110400$ patches $\{ \V x^k \}_{k=1}^M$ of size $80 \times 80$ from the 400 training images.
Compared to \cite{GouNeuUns2023}, we doubled the patch size to account for the larger field of view in $G$.
The corresponding data is simulated as $\V y^k = \V x^k + \sigma^k \V n^k$ with Gaussian noise $\V n^k \sim \mathcal{N}(\V 0 , \V I)$ and standard deviations $\sigma^k$ uniformly drawn from $[0,30/255]$.
This leads to the supervised training problem 
\begin{align} \label{eq:training_mask}
\argmin_{\theta} \sum_{k=1}^M \mathbb{E}_{\V n^k \sim  \mathcal{N}(\V 0 , \V I)} \mathbb{E}_{\sigma^k \sim \mathcal{U}([0,30/255])} \bigl(\bigl\Vert D_{\sigma^k, \theta} (\V y^k ) - \V x^k \bigr\Vert_1\bigr).
\end{align}
Here, $\nabla_{\theta} D_{\sigma^k, \theta} (\V y^k)$ is computed using the implicit differentiation techniques detailed in \cite{GouNeuUns2023}.
We initialize $\alpha_c(\sigma)$ and $\mu$ from the baseline $\mathcal R_{\sigma}$, and the parameters of $G$ randomly.
Before training with all parameters, the $G$ is pre-trained for 2 epochs with $\sigma^k =25/255$ using the same procedure as for the full training described below.
Then, we minimize \eqref{eq:training_mask} using Adam \cite{KinJim2015} with batch size 32 and learning rate $\mathrm{lr}=\num{1e-4}$ for $G$, $\mathrm{lr}=\num{5e-2}$ for $\mu$, and $\mathrm{lr}=\num{5e-3}$ for $s_{\alpha_c}$.
These correspond to the initial learning rates from \cite{LCSZGT2021} and \cite{GouNeuUns2023}, respectively.
Adapting these did not lead to significant improvements in our experiments.
Moreover, we use a scheduler that reduces the learning rates linearly for the first 5 epochs such that the final ones are half the initial ones.
In total, we train for 35 epochs, after which the model performance stagnates.
In terms of time, this can be achieved in about 2 days on a consumer GPU.
Then, we extract the model parameters $\theta$ for which the MSE on the validation Set12 \cite{ZZCMZ2017} is lowest.
When deploying $\mathcal R_{\V y,\sigma}$ to inverse problems, we need to tune $\lambda$ and $\sigma$ for \eqref{eq:InvProbLearn}.

\section{Experiments} \label{sec:experiments}

Now, we evaluate the two proposed regularizers $\mathcal R_{\V y, \sigma}$, which refer to as Mask-CRR and Mask-WCRR, respectively, on various inverse problems.
First, we investigate image denoising.
Then, we move to magnetic resonance imaging (MRI) in Section~\ref{sec:mri}.
In Section~\ref{sec:ct}, we deploy $\mathcal R_{\V y,\sigma}$ for computed tomography (CT).
Finally, we show results for the superresolution of material microstructures in Section~\ref{sec:SiC}. 
We quantitatively compare with established methods that can deal with limited training data (a common setting for medical imaging or material sciences), namely with
\begin{itemize}
    \item the total variation (TV) regularizer \cite{rudin1992nonlinear} for the variational model \eqref{eq:var_prob} as a well-known convex baseline.
    This method has a single hyperparameter $\lambda$;
    \item the patchNR \cite{ADHHMS2023,PAHHWS2023} as regularizer for the variational model \eqref{eq:var_prob}.
    It performs particularly well in the regime of limited training data.
    This regularizer is non-convex and, apart from the existence of minimizers and an empirical convergence analysis, no other theoretical guarantees are known.
    This regularizer needs to be learned on the validation set;
    \item the Prox-DRUNet \cite{HurLec2022,hurault2022gradient} as one of the best-performing methods for convergent plug-and-play image-reconstruction \cite{HNS2021,venkatakrishnan2013plug,Drunet2022}.
    It is inspired by the approach \eqref{eq:var_prob}, but models the regularizer $\mathcal R$ implicitly.
    For MRI and image superresolution, we deploy the DRS-PnP version described in \cite{HurLec2022}, and the Prox-PnP-$\alpha$PGD described in \cite{HurChaLec2023} for CT. 
    Here, one should keep in mind that the theoretical requirements are only imposed via regularization.
\end{itemize}
To optimize the hyperparameters in \eqref{eq:var_prob}, we deploy the coarse-to-fine grid search from \cite{GouNeuBoh2022} for a given validation set.
The reported quality metrics are the peak-signal-to-noise ratio (PSNR) and the structural similarity index measure (SSIM).
Additionally, we provide the runtime of the methods for a single reconstruction. 
All experiments are implemented in PyTorch and the code is available online\footnote{\url{https://github.com/FabianAltekrueger/DataAdaptiveRR}}.
The evaluations are performed on a NVIDIA GeForce RTX 2060 with 6GB GPU memory.

\begin{table}[t]
\begin{center}
\begin{tabular}[t]{c|ccc|cc|ccc} 
Method    &    BM3D   & CRR & WCRR & Mask-CRR & Mask-WCRR & Prox-DRUNet \\
\hline
$\sigma=5/255$   &   37.54     & 36.96  & 37.63 & 37.70 & \underline{37.80} &  \textbf{37.98} \\
$\sigma=15/255$  &   31.11     & 30.53  & 31.18 & 31.56 & \underline{31.61} &  \textbf{31.70} \\
$\sigma=25/255$   &   28.60     &  28.08 & 28.68 & \underline{29.13} & \underline{29.13} &  \textbf{29.18} \\
\hline 
\end{tabular}
\caption{Denoising. Average PSNR on the BSDS68 set.
Best is bold and second best is underlined.} 
\label{tab:psnr_denoising}
\end{center}
\end{table}
\begin{figure}[t]
\centering
\begin{subfigure}[t]{.185\textwidth}
\centering
\includegraphics[width=0.8\linewidth]{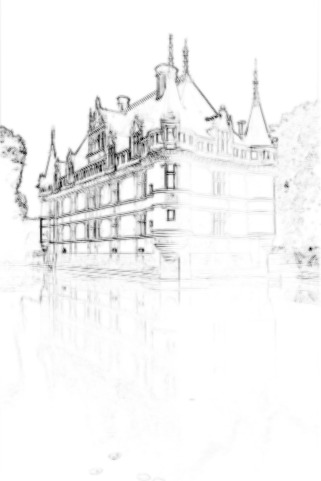}
\caption*{$\mathcal{R}(\V x_{\mathrm{gt}})$ (clean data)}
\end{subfigure}
\hfill
\begin{subfigure}[t]{.185\textwidth}
\centering
\includegraphics[width=0.8\linewidth]{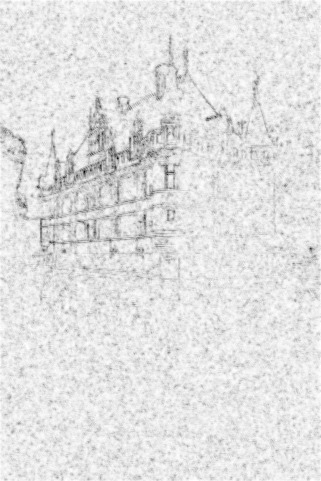}
\caption*{$\mathcal{R}(\V y)$ (noisy data)}
\end{subfigure}
\hfill
\begin{subfigure}[t]{.185\textwidth}
\centering
\includegraphics[width=0.8\linewidth]{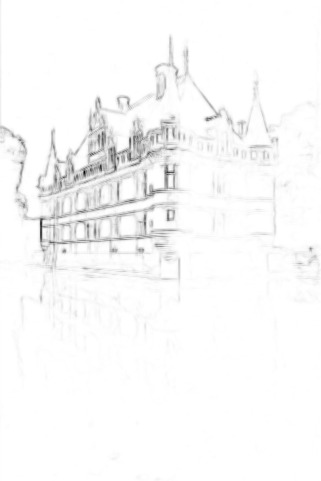}
\caption*{$\mathcal{R}(\V x_{\mathrm{est}})$ (result)}
\end{subfigure}
\hfill
\begin{subfigure}[t]{.185\textwidth}
\centering
\includegraphics[width=0.8\linewidth]{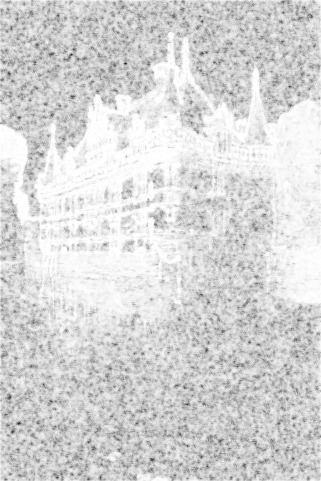}
\caption*{$\mathcal{R}_{\V y}(\V y)$ (noisy data)}
\end{subfigure}
\hfill
\begin{subfigure}[t]{.185\textwidth}
\centering
\includegraphics[width=0.8\linewidth]{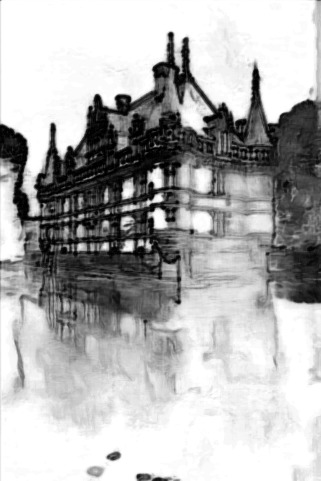}
\caption*{$\boldsymbol \Lambda(\V y)$ (average mask)}
\end{subfigure}

%\begin{subfigure}[t]{.17\textwidth}
%\includegraphics[width=\linewidth]{imgs/responses/wcrr/reg_gt}
%\end{subfigure}
%\hfill
%\begin{subfigure}[t]{.17\textwidth}
%\includegraphics[width=\linewidth]{imgs/responses/wcrr/reg_noisy}
%\end{subfigure}
%\hfill
%\begin{subfigure}[t]{.17\textwidth}
%\includegraphics[width=\linewidth]{imgs/responses/wcrr/reg_init}
%\end{subfigure}
%\hfill
%\begin{subfigure}[t]{.17\textwidth}
%\includegraphics[width=\linewidth]{imgs/responses/wcrr/reg_mask}
%\end{subfigure}
%\hfill
%\begin{subfigure}[t]{.17\textwidth}
%\includegraphics[width=\linewidth]{imgs/responses/wcrr/mask_mean}
%\end{subfigure}
\caption{Pixel-wise cost $\mathcal{R}(\V x) = \sum_c \boldsymbol \psi_c (\M W_c \V x)$ for $\V x \in \{ \V x_{\mathrm{gt}}, \V y, \V x_{\mathrm{est}} \}$ and data-dependent cost $\mathcal{R}_{\V y}(\V y) = \sum_c \langle \boldsymbol \Lambda_c (\V y ) , \boldsymbol \psi_c (\M W_c \V y) \rangle$.
Here, black corresponds to higher values.
For the last image $\boldsymbol \Lambda (\V y ) = \sum_c \boldsymbol \Lambda_c (\V y )$ black corresponds to smaller values.}
\label{fig:respones}
\vspace{.5cm}

\centering
\begin{subfigure}[t]{.123\textwidth}  
\begin{tikzpicture}[spy using outlines=
{rectangle,white,magnification=3.5,size=2.115cm, connect spies}]
\node[anchor=south west,inner sep=0]  at (0,0) {\includegraphics[width=\linewidth]{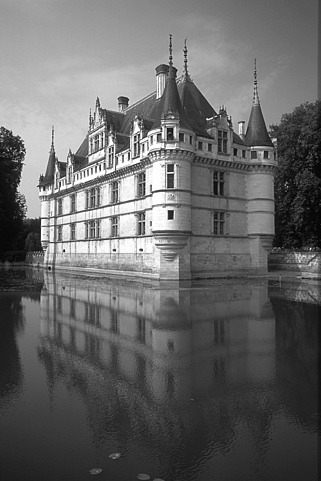}};
\spy on (1.65,2.5) in node [right] at (-0.005,-.87);
\end{tikzpicture}
\caption*{\footnotesize Original}
\end{subfigure}%
\hfill
\begin{subfigure}[t]{.123\textwidth}
\begin{tikzpicture}[spy using outlines=
{rectangle,white,magnification=3.5,size=2.115cm, connect spies}]
\node[anchor=south west,inner sep=0]  at (0,0) {\includegraphics[width=\linewidth]{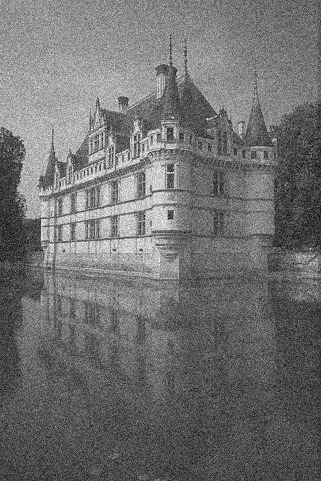}};
\spy on (1.65,2.5) in node [right] at (-0.005,-.87);
\end{tikzpicture}
  \caption*{\footnotesize Noisy}
\end{subfigure}%
\hfill
\begin{subfigure}[t]{.123\textwidth}
\begin{tikzpicture}[spy using outlines=
{rectangle,white,magnification=3.5,size=2.115cm, connect spies}]
\node[anchor=south west,inner sep=0]  at (0,0) {\includegraphics[width=\linewidth]{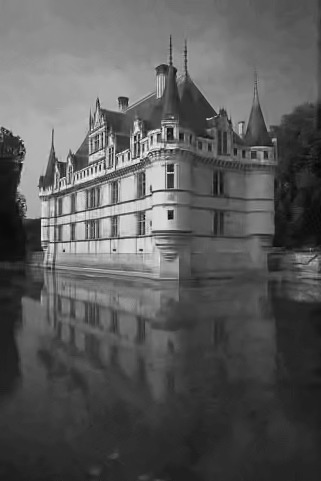}};
\spy on (1.65,2.5) in node [right] at (-0.005,-.87);
\end{tikzpicture}
  \caption*{\footnotesize BM3D}
\end{subfigure}%
\hfill
\begin{subfigure}[t]{.123\textwidth}
\begin{tikzpicture}[spy using outlines=
{rectangle,white,magnification=3.5,size=2.115cm, connect spies}]
\node[anchor=south west,inner sep=0]  at (0,0) {\includegraphics[width=\linewidth]{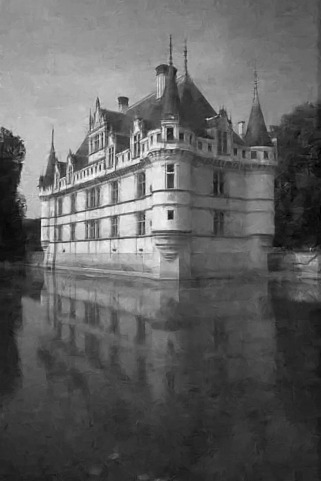}};
\spy on (1.65,2.5) in node [right] at (-0.005,-.87);
\end{tikzpicture}
  \caption*{\footnotesize CRR}
\end{subfigure}%
\hfill
\begin{subfigure}[t]{.123\textwidth}
\begin{tikzpicture}[spy using outlines=
{rectangle,white,magnification=3.5,size=2.115cm, connect spies}]
\node[anchor=south west,inner sep=0]  at (0,0) {\includegraphics[width=\linewidth]{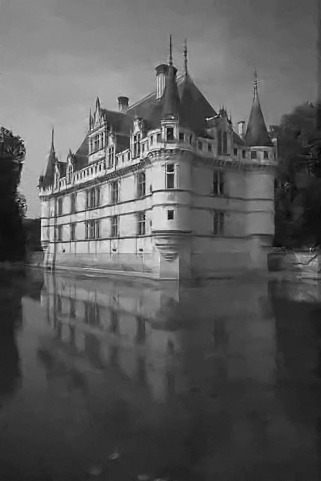}};
\spy on (1.65,2.5) in node [right] at (-0.005,-.87);
\end{tikzpicture}
  \caption*{\footnotesize WCRR}
\end{subfigure}%
\hfill
\begin{subfigure}[t]{.123\textwidth}
\begin{tikzpicture}[spy using outlines=
{rectangle,white,magnification=3.5,size=2.115cm, connect spies}]
\node[anchor=south west,inner sep=0]  at (0,0) {\includegraphics[width=\linewidth]{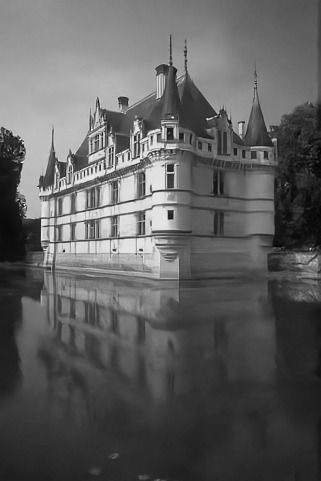}};
\spy on (1.65,2.5) in node [right] at (-0.005,-.87);
\end{tikzpicture}
  \caption*{\footnotesize Mask-CRR}
\end{subfigure}%
\hfill
\begin{subfigure}[t]{.123\textwidth}
\begin{tikzpicture}[spy using outlines=
{rectangle,white,magnification=3.5,size=2.115cm, connect spies}]
\node[anchor=south west,inner sep=0]  at (0,0) {\includegraphics[width=\linewidth]{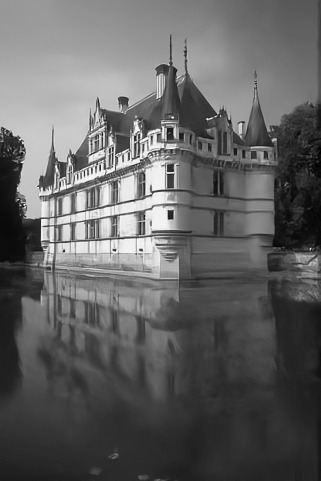}};
\spy on (1.65,2.5) in node [right] at (-0.005,-.87);
\end{tikzpicture}
  \caption*{\footnotesize Mask-WCRR}
\end{subfigure}%
\hfill
\begin{subfigure}[t]{.123\textwidth}
\begin{tikzpicture}[spy using outlines=
{rectangle,white,magnification=3.5,size=2.115cm, connect spies}]
\node[anchor=south west,inner sep=0]  at (0,0) {\includegraphics[width=\linewidth]{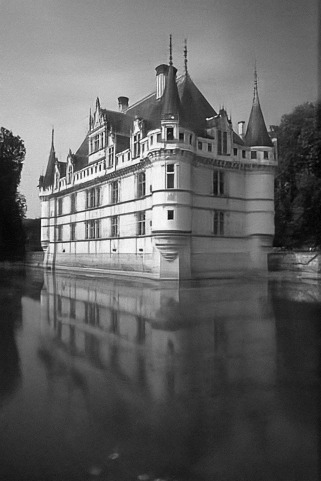}};
\spy on (1.65,2.5) in node [right] at (-0.005,-.87);
\end{tikzpicture}
  \caption*{\footnotesize Prox-DRUNet}
\end{subfigure}%
\caption{Denoising of the \emph{castle} image for $\sigma=25/255$. 
The white box marks the zoomed area.} \label{fig:denoising_comparison}
\end{figure}

\subsection{Denoising}\label{sec:Denoising}
First, we evaluate the methods for non-blind denoising on the BSDS68 dataset \cite{bsd_data}.
This means that the standard deviation $\sigma$ of the noise is used as input.
The quantitative results for different $\sigma$ are given in Table~\ref{tab:psnr_denoising}, where we also include the popular BM3D denoiser \cite{DabFoiKat2007}.
Most importantly, the mask $\boldsymbol \Lambda$ improves both the CRR and the WCRR.
To analyze this behavior, we visualize the pixel-wise regularization cost for CRR and Mask-CRR in Figure~\ref{fig:respones}. The images would look similar for WCRR.
First, note that each $\boldsymbol \Lambda_c \colon \R^n \to [0,1]^{n}$ dampens the penalized filter response $\boldsymbol \psi_c (\M W_c \V x)$ at the structure.
In particular, the values of $\boldsymbol \Lambda$ are smaller (black) in structured areas and higher (white) in constant ones, which prevents oversmoothing.
Secondly, after adding the $\boldsymbol \Lambda$, the two regularizers CRR and WCRR perform similarly.
This contrasts with the baseline setting, where the WCRR outperforms the CRR.
A qualitative result for the \emph{castle} image and $\sigma=25/255$ is given in Figure~\ref{fig:denoising_comparison}.
While the results with CRR and WCRR look very different, the addition of $\boldsymbol \Lambda$ leads to visually similar results with only minor differences:
Since the top of the spire is still contained in the CRR result, the structure-aware mask $\boldsymbol \Lambda$ enables us to recover most of it.
In contrast, it is already lost for the WCRR result, and $\boldsymbol \Lambda$ does not help anymore.
Hence, we conclude that the mask $\boldsymbol \Lambda$ enhances the structure contained in $\V x_\text{est}$.

\subsection{Magnetic Resonance Imaging} \label{sec:mri}
\begin{table}[t]
\begin{center}
\begin{tabular}[t]{c|ccccc} 
    & \multicolumn{5}{c}{4-fold single coil}\\    
     & PD  & PDFS & PD & PDFS & time \\    
\midrule
Zero-Filled &  27.40 $\pm$ 2.41  &  29.68 $\pm$ 1.58     &  0.714 $\pm$ 0.053  &  0.719 $\pm$ 0.049 &   0.02s  \\
TV &  32.30 $\pm$ 2.62 &  32.56 $\pm$ 1.78  & 0.827 $\pm$ 0.052 &  0.775 $\pm$ 0.061 &   4s  \\
CRR   &   33.39 $\pm$ 2.65 &  33.44 $\pm$ 1.82 & 0.870 $\pm$ 0.042 & 0.824 $\pm$ 0.050 &  24s \\
WCRR   &   35.74 $\pm$ 2.51 &  34.60 $\pm$ 1.96 & 0.897 $\pm$ 0.039 &  0.837 $\pm$ 0.053 & 20s \\
\hline
Mask-CRR   &   35.86 $\pm$ 2.49 &  34.60  $\pm$ 1.97 & \underline{0.903} $\pm$ 0.036 &  0.838 $\pm$ 0.051 &  38s \\
Mask-WCRR   &   \textbf{36.19} $\pm$ 2.53  &  \underline{34.68} $\pm$ 1.98 & \textbf{0.904} $\pm$ 0.038 &  \underline{0.840} $\pm$ 0.051 & 37s\\
Prox-DRUNet DRS    & \underline{36.14} $\pm$ 2.52 & \textbf{35.05} $\pm$ 1.97 & 0.900 $\pm$ 0.039 & \textbf{0.847} $\pm$ 0.048 &  625s         \\
\hline
patchNR     &   35.80  $\pm$ 2.84 &  34.26 $\pm$ 2.06 & 0.895  $\pm$ 0.043 &  0.830 $\pm$ 0.057  &  116s      \\

\toprule
\toprule
    &  \multicolumn{5}{c}{8-fold multi coil}\\    
     &  PD  & PDFS & PD & PDFS & time\\    
\midrule
Zero-Filled &     23.80  $\pm$ 2.34     & 27.19 $\pm$ 1.62    &  0.628 $\pm$ 0.061   &  0.649 $\pm$ 0.065 &  0.02s   \\
TV &    31.14 $\pm$ 3.14 & 32.12 $\pm$ 2.15 & 0.807 $\pm$ 0.060 & 0.769 $\pm$ 0.070 &  16s  \\
CRR   &   33.90 $\pm$ 2.81 &  34.28 $\pm$ 2.08 & 0.875 $\pm$ 0.038 &  0.847 $\pm$ 0.049 &  27s \\
WCRR   &   35.50 $\pm$ 2.40 & 35.12 $\pm$ 2.13 & 0.894 $\pm$ 0.032 & 0.855 $\pm$ 0.052 &  27s\\
\hline
Mask-CRR   &  35.73 $\pm$ 2.49 &  \underline{35.20} $\pm$ 2.11 & \underline{0.900} $\pm$ 0.031 &  \textbf{0.858} $\pm$ 0.049 & 38s \\
Mask-WCRR   &   \textbf{36.00} $\pm$ 2.36 & \textbf{35.23} $\pm$ 2.12 & \textbf{0.902} $\pm$ 0.030 &  \textbf{0.858} $\pm$ 0.050 & 38s \\
Prox-DRUNet DRS    &  \underline{35.81} $\pm$ 2.36 & 35.13 $\pm$ 2.14 & 0.894 $\pm$ 0.031 & \underline{0.857} $\pm$ 0.052 &   625s      \\
\hline
patchNR     &    35.35  $\pm$ 3.09 & 34.84 $\pm$ 2.15 & 0.888 $\pm$ 0.042 & 0.849 $\pm$ 0.054 & 119s   \\
\end{tabular}
\caption{MRI. Average metrics and standard deviations for 50 images of the fastMRI dataset.
Best is bold and second best is underlined.} 
\label{tab:error_MRI}
\end{center}
\end{table}
Next, we investigate the single- and 15-coil MRI setups detailed in \cite{GouNeuBoh2022}, where $\M H$ is (essentially) a subsampled Fourier transform.
We want to reconstruct proton-density-weighted knee images from the fastMRI dataset \cite{zbontar2018fastMRI}, which consists of the two subsets PDFS and PD (with and without fat suppression).
To generate ground truth images, we use the fully sampled k-space measurements. 
For the single-coil setup, we generate the measurements by a $4$-fold ($M_{\text{acc}} = 4$) subsampling.
The acceleration factor $M_{\text{acc}}$ determines the ratio, i.e., we keep $1/M_{\text{acc}}$ of the columns in the k-space.
In the 15-coil setup, we apply $8$-fold ($M_{\text{acc}} = 8$) subsampling to the Fourier transforms of the ground-truth images multiplied by the respective sensitivity maps.
For this, we use the BART \cite{uecker2013software} implementation of the ESPIRiT algorithm \cite{Uecker2014-uv}.
The data is corrupted with additive Gaussian noise of standard deviation $\sigma = \num{2e-3}$ to the real and imaginary part.
From the Bayesian viewpoint, the negative log-likelihood reads
\begin{equation}\label{eq:DataMRI}
    - \log(p_{\M Y|\M X= \V x}(\V y)) \propto \frac{1}{2\sigma^2} \Vert \M H \V x - \V y \Vert^2.
\end{equation}

For each task and data set, we use ten validation images to determine the hyperparameters.
We also use these images to finetune the $\boldsymbol \Lambda$ in \eqref{eq:rridge_reg_mask} for 30 epochs and to train the patchNR.
The quantitative results on 50 test images are reported in Table~\ref{tab:error_MRI}.
Moreover, qualitative results are given in Figures~\ref{fig:MRI_comparison_4} and \ref{fig:MRI_comparison_8}.
Here, the error refers to the absolute difference with the ground truth.
Again, the performance of CRR and WCRR improves with the addition of $\boldsymbol \Lambda \colon \R^m \to [0,1]^{nN_c}$.
In particular, the Mask-CRR improves significantly for all cases.
Overall, the Mask-WCRR performs best.
Similar reconstruction quality is achieved with the Prox-DRUNet.
Its drawback is that it takes 15 times longer than the Mask-WCRR.
The patchNR struggles to reconstruct the finer structures and performs worse than the other methods.
More reconstructions are given in Appendix~\ref{app:more_examples}.
\begin{figure}[hp]
\centering
\begin{subfigure}[t]{.123\textwidth}  
\begin{tikzpicture}[spy using outlines=
{rectangle,white,magnification=4.5,size=2.115cm, connect spies}]
\node[anchor=south west,inner sep=0]  at (0,0) {\includegraphics[width=\linewidth]{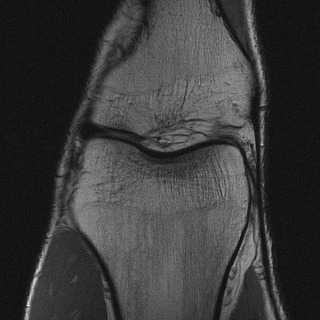}};
\spy on (1.55,1.08) in node [right] at (-0.005,-1.1);
\end{tikzpicture}
%\caption*{\footnotesize GT}
\end{subfigure}%
\hfill
\begin{subfigure}[t]{.123\textwidth}
\begin{tikzpicture}[spy using outlines=
{rectangle,white,magnification=4.5,size=2.115cm, connect spies}]
\node[anchor=south west,inner sep=0]  at (0,0) {\includegraphics[width=\linewidth]{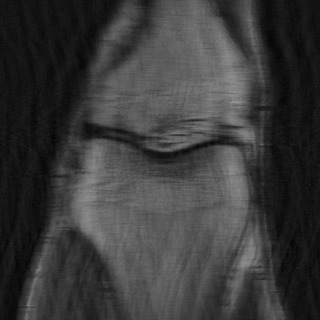}};
\spy on (1.55,1.08) in node [right] at (-0.005,-1.1);
\end{tikzpicture}
%  \caption*{\footnotesize FBP}
\end{subfigure}%
\hfill
\begin{subfigure}[t]{.123\textwidth}
\begin{tikzpicture}[spy using outlines=
{rectangle,white,magnification=4.5,size=2.115cm, connect spies}]
\node[anchor=south west,inner sep=0]  at (0,0) {\includegraphics[width=\linewidth]{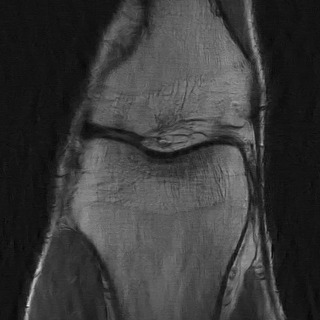}};
\spy on (1.55,1.08) in node [right] at (-0.005,-1.1);
\end{tikzpicture}
%  \caption*{\footnotesize CRR}
\end{subfigure}%
\hfill
\begin{subfigure}[t]{.123\textwidth}
\begin{tikzpicture}[spy using outlines=
{rectangle,white,magnification=4.5,size=2.115cm, connect spies}]
\node[anchor=south west,inner sep=0]  at (0,0) {\includegraphics[width=\linewidth]{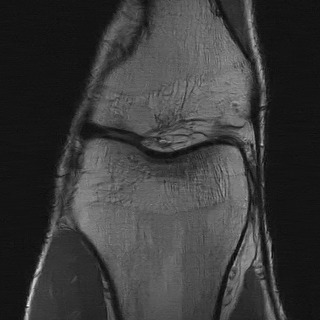}};
\spy on (1.55,1.08) in node [right] at (-0.005,-1.1);
\end{tikzpicture}
%  \caption*{\footnotesize WCRR}
\end{subfigure}%
\hfill
\begin{subfigure}[t]{.123\textwidth}
\begin{tikzpicture}[spy using outlines=
{rectangle,white,magnification=4.5,size=2.115cm, connect spies}]
\node[anchor=south west,inner sep=0]  at (0,0) {\includegraphics[width=\linewidth]{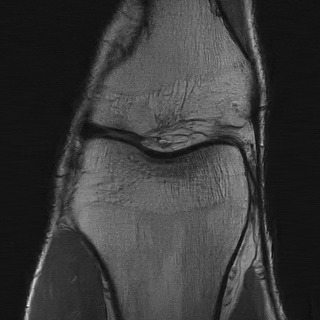}};
\spy on (1.55,1.08) in node [right] at (-0.005,-1.1);
\end{tikzpicture}
%  \caption*{\footnotesize Mask-CRR}
\end{subfigure}%
\hfill
\begin{subfigure}[t]{.123\textwidth}
\begin{tikzpicture}[spy using outlines=
{rectangle,white,magnification=4.5,size=2.115cm, connect spies}]
\node[anchor=south west,inner sep=0]  at (0,0) {\includegraphics[width=\linewidth]{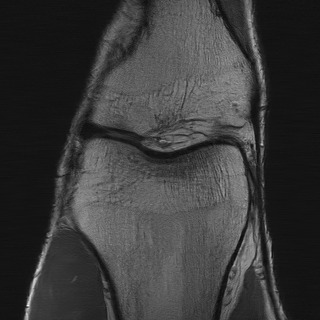}};
\spy on (1.55,1.08) in node [right] at (-0.005,-1.1);
\end{tikzpicture}
%  \caption*{\footnotesize Mask-WCRR}
\end{subfigure}%
\hfill
\begin{subfigure}[t]{.123\textwidth}
\begin{tikzpicture}[spy using outlines=
{rectangle,white,magnification=4.5,size=2.115cm, connect spies}]
\node[anchor=south west,inner sep=0]  at (0,0) {\includegraphics[width=\linewidth]{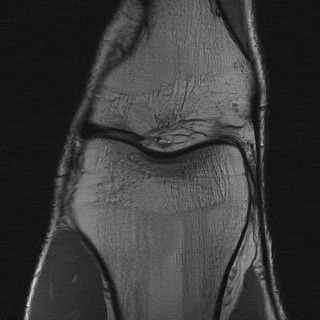}};
\spy on (1.55,1.08) in node [right] at (-0.005,-1.1);
\end{tikzpicture}
%  \caption*{\footnotesize Prox-DRUNet}
\end{subfigure}%
\hfill
\begin{subfigure}[t]{.123\textwidth}
\begin{tikzpicture}[spy using outlines=
{rectangle,white,magnification=4.5,size=2.115cm, connect spies}]
\node[anchor=south west,inner sep=0]  at (0,0) {\includegraphics[width=\linewidth]{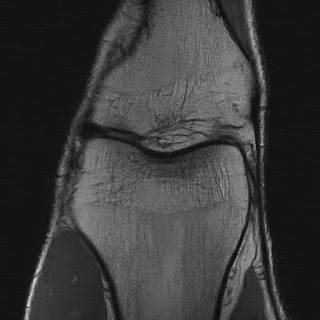}};
\spy on (1.55,1.08) in node [right] at (-0.005,-1.1);
\end{tikzpicture}
%  \caption*{\footnotesize patchNR}
\end{subfigure}%

\begin{subfigure}[t]{.123\textwidth}
\includegraphics[width=\linewidth]{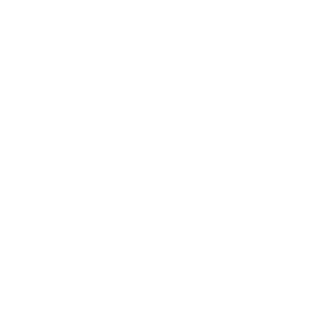}
  \caption*{\footnotesize GT}
\end{subfigure}%
\hfill
\begin{subfigure}[t]{.123\textwidth}
\includegraphics[width=\linewidth]{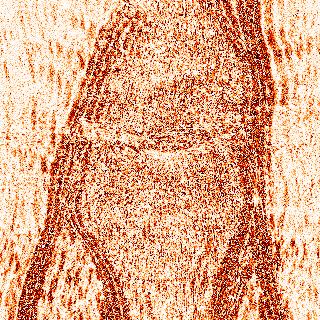}
  \caption*{\footnotesize Zero-filled}
\end{subfigure}%
\hfill
\begin{subfigure}[t]{.123\textwidth}
\includegraphics[width=\linewidth]{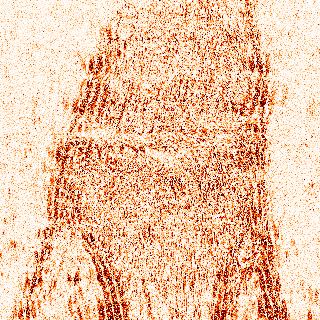}
  \caption*{\footnotesize CRR}
\end{subfigure}%
\hfill
\begin{subfigure}[t]{.123\textwidth}
\includegraphics[width=\linewidth]{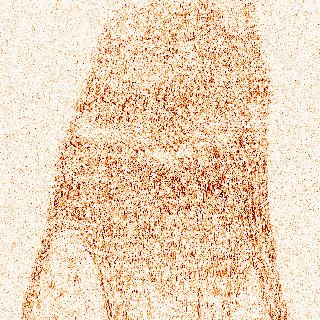}
  \caption*{\footnotesize WCRR}
\end{subfigure}%
\hfill
\begin{subfigure}[t]{.123\textwidth}
\includegraphics[width=\linewidth]{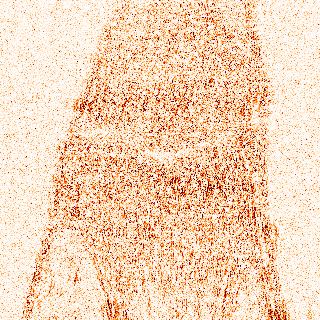}
  \caption*{\footnotesize CRR-Mask}
\end{subfigure}%
\hfill
\begin{subfigure}[t]{.123\textwidth}
\includegraphics[width=\linewidth]{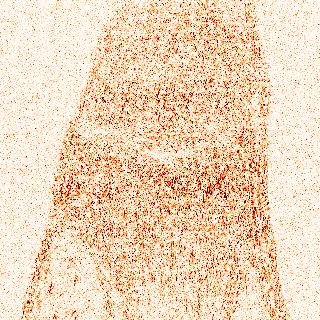}
  \caption*{\footnotesize WCRR-Mask}
\end{subfigure}%
\hfill
\begin{subfigure}[t]{.123\textwidth}
\includegraphics[width=\linewidth]{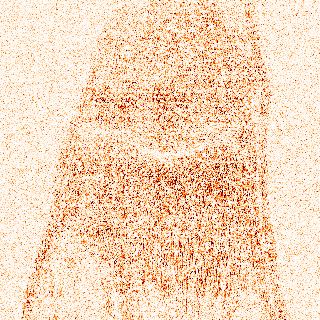}
  \caption*{\footnotesize Prox-DRUNet}
\end{subfigure}%
\hfill
\begin{subfigure}[t]{.123\textwidth}
\includegraphics[width=\linewidth]{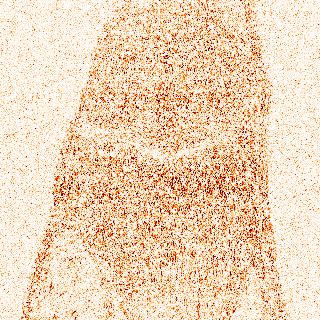}
  \caption*{\footnotesize patchNR}
\end{subfigure}%

\caption{4-fold single-coil MRI on PD data set.
The white box marks the zoomed area.
\textit{Top}: full image; \textit{middle}: zoomed-in part; \textit{bottom}: error.
} \label{fig:MRI_comparison_4}
\vspace{.5cm}

\centering
\begin{subfigure}[t]{.123\textwidth}  
\begin{tikzpicture}[spy using outlines=
{rectangle,white,magnification=4.5,size=2.115cm, connect spies}]
\node[anchor=south west,inner sep=0]  at (0,0) {\includegraphics[width=\linewidth]{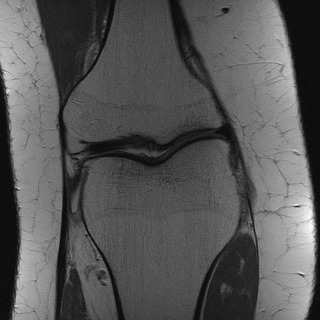}};
\spy on (.5,.3) in node [right] at (-0.005,-1.1);
\end{tikzpicture}
%\caption*{\footnotesize GT}
\end{subfigure}%
\hfill
\begin{subfigure}[t]{.123\textwidth}
\begin{tikzpicture}[spy using outlines=
{rectangle,white,magnification=4.5,size=2.115cm, connect spies}]
\node[anchor=south west,inner sep=0]  at (0,0) {\includegraphics[width=\linewidth]{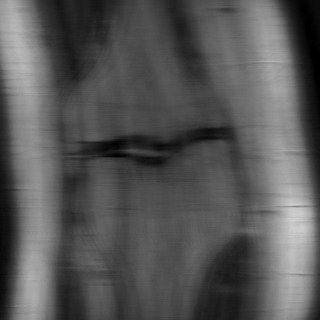}};
\spy on (.5,.3) in node [right] at (-0.005,-1.1);
\end{tikzpicture}
  %\caption*{\footnotesize FBP}
\end{subfigure}%
\hfill
\begin{subfigure}[t]{.123\textwidth}
\begin{tikzpicture}[spy using outlines=
{rectangle,white,magnification=4.5,size=2.115cm, connect spies}]
\node[anchor=south west,inner sep=0]  at (0,0) {\includegraphics[width=\linewidth]{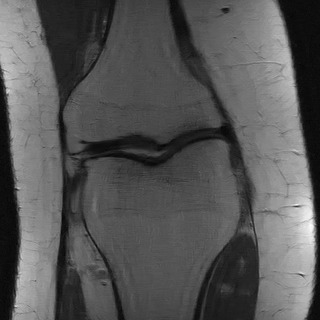}};
\spy on (.5,.3) in node [right] at (-0.005,-1.1);
\end{tikzpicture}
%  \caption*{\footnotesize CRR}
\end{subfigure}%
\hfill
\begin{subfigure}[t]{.123\textwidth}
\begin{tikzpicture}[spy using outlines=
{rectangle,white,magnification=4.5,size=2.115cm, connect spies}]
\node[anchor=south west,inner sep=0]  at (0,0) {\includegraphics[width=\linewidth]{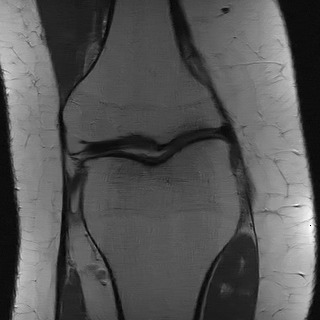}};
\spy on (.5,.3) in node [right] at (-0.005,-1.1);
\end{tikzpicture}
%  \caption*{\footnotesize WCRR}
\end{subfigure}%
\hfill
\begin{subfigure}[t]{.123\textwidth}
\begin{tikzpicture}[spy using outlines=
{rectangle,white,magnification=4.5,size=2.115cm, connect spies}]
\node[anchor=south west,inner sep=0]  at (0,0) {\includegraphics[width=\linewidth]{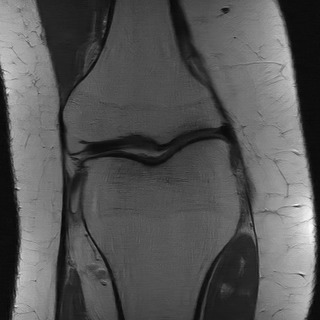}};
\spy on (.5,.3) in node [right] at (-0.005,-1.1);
\end{tikzpicture}
%  \caption*{\footnotesize Mask-CRR}
\end{subfigure}%
\hfill
\begin{subfigure}[t]{.123\textwidth}
\begin{tikzpicture}[spy using outlines=
{rectangle,white,magnification=4.5,size=2.115cm, connect spies}]
\node[anchor=south west,inner sep=0]  at (0,0) {\includegraphics[width=\linewidth]{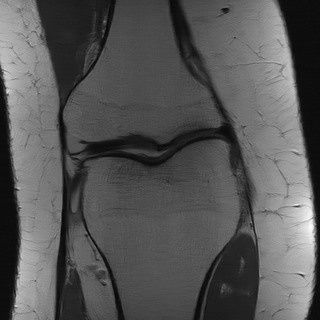}};
\spy on (.5,.3) in node [right] at (-0.005,-1.1);
\end{tikzpicture}
%  \caption*{\footnotesize Mask-WCRR}
\end{subfigure}%
\hfill
\begin{subfigure}[t]{.123\textwidth}
\begin{tikzpicture}[spy using outlines=
{rectangle,white,magnification=4.5,size=2.115cm, connect spies}]
\node[anchor=south west,inner sep=0]  at (0,0) {\includegraphics[width=\linewidth]{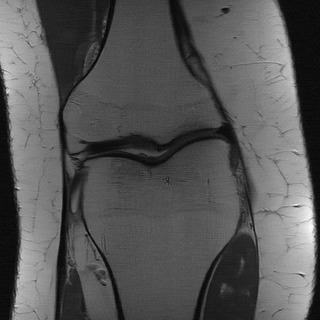}};
\spy on (.5,.3) in node [right] at (-0.005,-1.1);
\end{tikzpicture}
%  \caption*{\footnotesize Prox-DRUNet}
\end{subfigure}%
\hfill
\begin{subfigure}[t]{.123\textwidth}
\begin{tikzpicture}[spy using outlines=
{rectangle,white,magnification=4.5,size=2.115cm, connect spies}]
\node[anchor=south west,inner sep=0]  at (0,0) {\includegraphics[width=\linewidth]{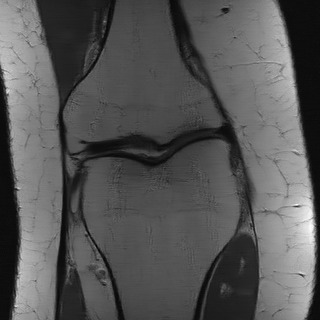}};
\spy on (.5,.3) in node [right] at (-0.005,-1.1);
\end{tikzpicture}
%  \caption*{\footnotesize patchNR}
\end{subfigure}%

\begin{subfigure}[t]{.123\textwidth}
\includegraphics[width=\linewidth]{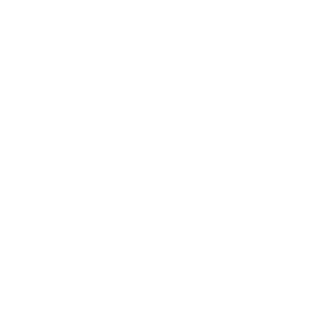}
  \caption*{\footnotesize GT}
\end{subfigure}%
\hfill
\begin{subfigure}[t]{.123\textwidth}
\includegraphics[width=\linewidth]{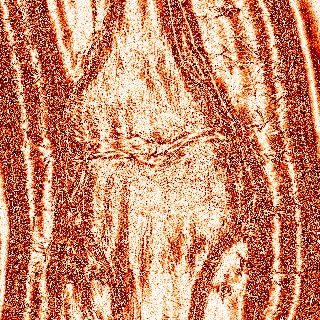}
  \caption*{\footnotesize Zero-filled}
\end{subfigure}%
\hfill
\begin{subfigure}[t]{.123\textwidth}
\includegraphics[width=\linewidth]{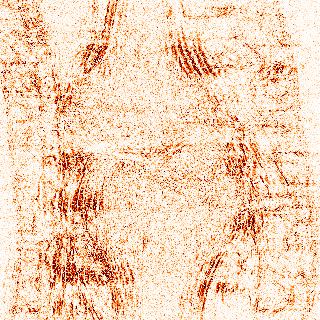}
  \caption*{\footnotesize CRR}
\end{subfigure}%
\hfill
\begin{subfigure}[t]{.123\textwidth}
\includegraphics[width=\linewidth]{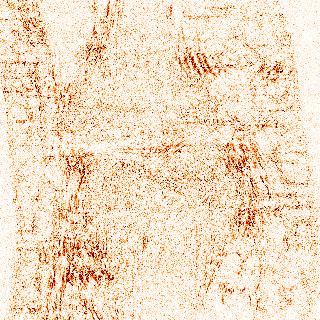}
  \caption*{\footnotesize WCRR}
\end{subfigure}%
\hfill
\begin{subfigure}[t]{.123\textwidth}
\includegraphics[width=\linewidth]{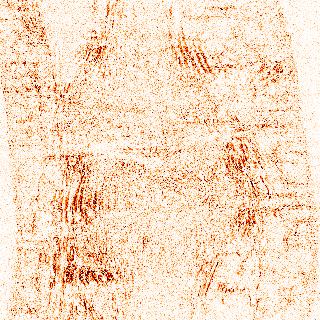}
  \caption*{\footnotesize CRR-Mask}
\end{subfigure}%
\hfill
\begin{subfigure}[t]{.123\textwidth}
\includegraphics[width=\linewidth]{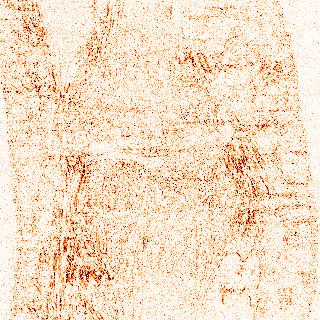}
  \caption*{\footnotesize WCRR-Mask}
\end{subfigure}%
\hfill
\begin{subfigure}[t]{.123\textwidth}
\includegraphics[width=\linewidth]{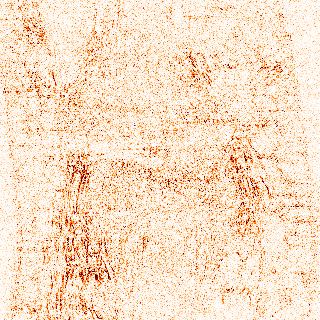}
  \caption*{\footnotesize Prox-DRUNet}
\end{subfigure}%
\hfill
\begin{subfigure}[t]{.123\textwidth}
\includegraphics[width=\linewidth]{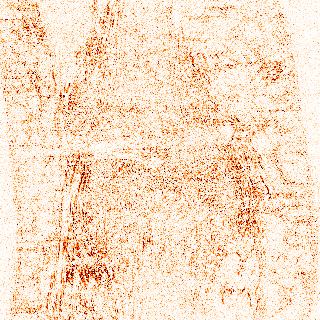}
  \caption*{\footnotesize patchNR}
\end{subfigure}%

\begin{subfigure}[t]{.123\textwidth}  
\begin{tikzpicture}[spy using outlines=
{rectangle,white,magnification=4.5,size=2.115cm, connect spies}]
\node[anchor=south west,inner sep=0]  at (0,0) {\includegraphics[width=\linewidth]{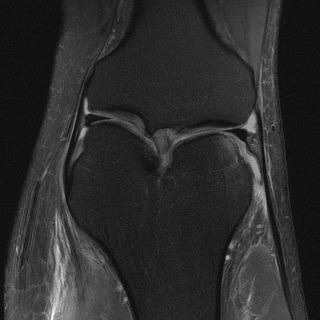}};
\spy on (.5,1.32) in node [right] at (-0.005,-1.1);
\end{tikzpicture}
%\caption*{\footnotesize GT}
\end{subfigure}%
\hfill
\begin{subfigure}[t]{.123\textwidth}
\begin{tikzpicture}[spy using outlines=
{rectangle,white,magnification=4.5,size=2.115cm, connect spies}]
\node[anchor=south west,inner sep=0]  at (0,0) {\includegraphics[width=\linewidth]{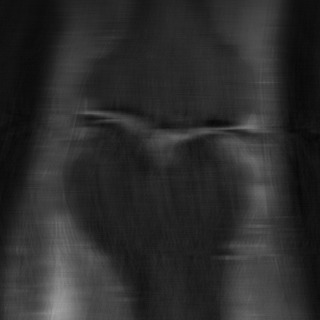}};
\spy on (.5,1.32) in node [right] at (-0.005,-1.1);
\end{tikzpicture}
%  \caption*{\footnotesize FBP}
\end{subfigure}%
\hfill
\begin{subfigure}[t]{.123\textwidth}
\begin{tikzpicture}[spy using outlines=
{rectangle,white,magnification=4.5,size=2.115cm, connect spies}]
\node[anchor=south west,inner sep=0]  at (0,0) {\includegraphics[width=\linewidth]{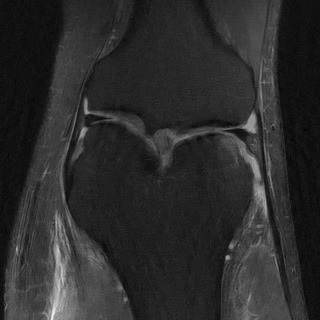}};
\spy on (.5,1.32) in node [right] at (-0.005,-1.1);
\end{tikzpicture}
%  \caption*{\footnotesize CRR}
\end{subfigure}%
\hfill
\begin{subfigure}[t]{.123\textwidth}
\begin{tikzpicture}[spy using outlines=
{rectangle,white,magnification=4.5,size=2.115cm, connect spies}]
\node[anchor=south west,inner sep=0]  at (0,0) {\includegraphics[width=\linewidth]{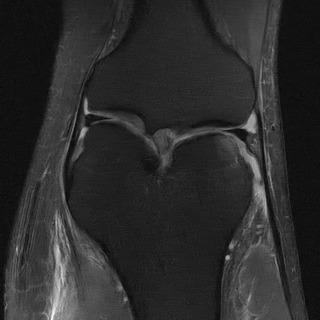}};
\spy on (.5,1.32) in node [right] at (-0.005,-1.1);
\end{tikzpicture}
%  \caption*{\footnotesize WCRR}
\end{subfigure}%
\hfill
\begin{subfigure}[t]{.123\textwidth}
\begin{tikzpicture}[spy using outlines=
{rectangle,white,magnification=4.5,size=2.115cm, connect spies}]
\node[anchor=south west,inner sep=0]  at (0,0) {\includegraphics[width=\linewidth]{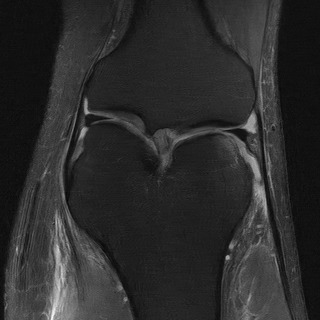}};
\spy on (.5,1.32) in node [right] at (-0.005,-1.1);
\end{tikzpicture}
%  \caption*{\footnotesize Mask-CRR}
\end{subfigure}%
\hfill
\begin{subfigure}[t]{.123\textwidth}
\begin{tikzpicture}[spy using outlines=
{rectangle,white,magnification=4.5,size=2.115cm, connect spies}]
\node[anchor=south west,inner sep=0]  at (0,0) {\includegraphics[width=\linewidth]{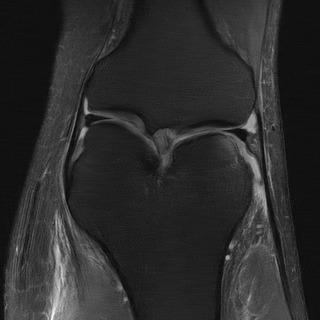}};
\spy on (.5,1.32) in node [right] at (-0.005,-1.1);
\end{tikzpicture}
%  \caption*{\footnotesize Mask-WCRR}
\end{subfigure}%
\hfill
\begin{subfigure}[t]{.123\textwidth}
\begin{tikzpicture}[spy using outlines=
{rectangle,white,magnification=4.5,size=2.115cm, connect spies}]
\node[anchor=south west,inner sep=0]  at (0,0) {\includegraphics[width=\linewidth]{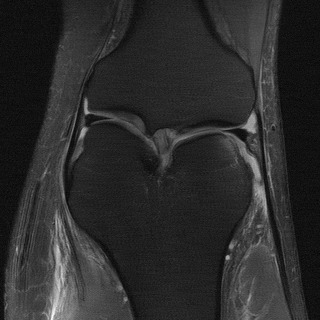}};
\spy on (.5,1.32) in node [right] at (-0.005,-1.1);
\end{tikzpicture}
%  \caption*{\footnotesize Prox-DRUNet}
\end{subfigure}%
\hfill
\begin{subfigure}[t]{.123\textwidth}
\begin{tikzpicture}[spy using outlines=
{rectangle,white,magnification=4.5,size=2.115cm, connect spies}]
\node[anchor=south west,inner sep=0]  at (0,0) {\includegraphics[width=\linewidth]{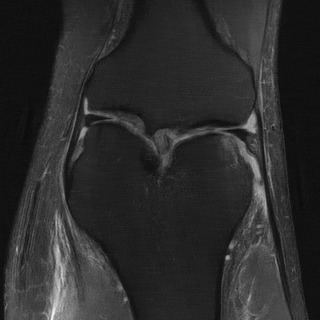}};
\spy on (.5,1.32) in node [right] at (-0.005,-1.1);
\end{tikzpicture}
%  \caption*{\footnotesize patchNR}
\end{subfigure}%

\begin{subfigure}[t]{.123\textwidth}
\includegraphics[width=\linewidth]{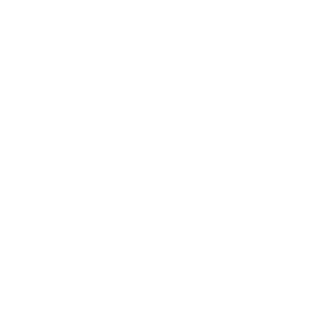}
  \caption*{\footnotesize GT}
\end{subfigure}%
\hfill
\begin{subfigure}[t]{.123\textwidth}
\includegraphics[width=\linewidth]{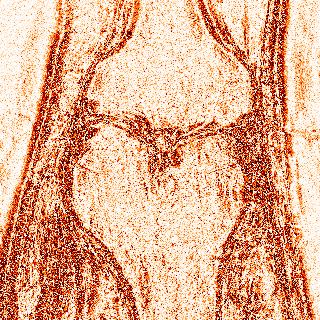}
  \caption*{\footnotesize Zero-filled}
\end{subfigure}%
\hfill
\begin{subfigure}[t]{.123\textwidth}
\includegraphics[width=\linewidth]{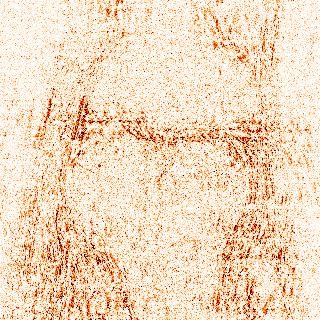}
  \caption*{\footnotesize CRR}
\end{subfigure}%
\hfill
\begin{subfigure}[t]{.123\textwidth}
\includegraphics[width=\linewidth]{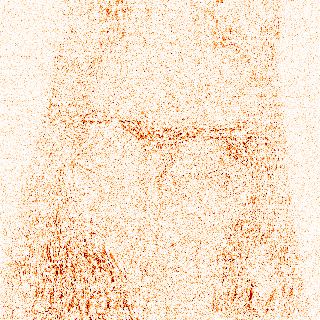}
  \caption*{\footnotesize WCRR}
\end{subfigure}%
\hfill
\begin{subfigure}[t]{.123\textwidth}
\includegraphics[width=\linewidth]{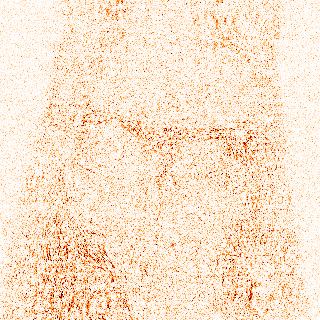}
  \caption*{\footnotesize CRR-Mask}
\end{subfigure}%
\hfill
\begin{subfigure}[t]{.123\textwidth}
\includegraphics[width=\linewidth]{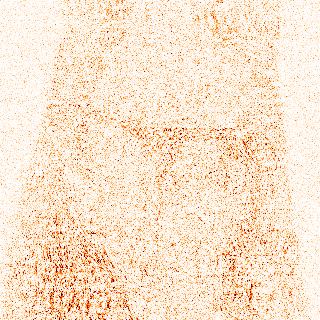}
  \caption*{\footnotesize WCRR-Mask}
\end{subfigure}%
\hfill
\begin{subfigure}[t]{.123\textwidth}
\includegraphics[width=\linewidth]{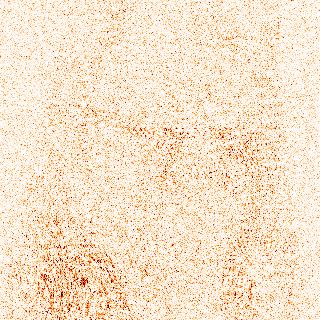}
  \caption*{\footnotesize Prox-DRUNet}
\end{subfigure}%
\hfill
\begin{subfigure}[t]{.123\textwidth}
\includegraphics[width=\linewidth]{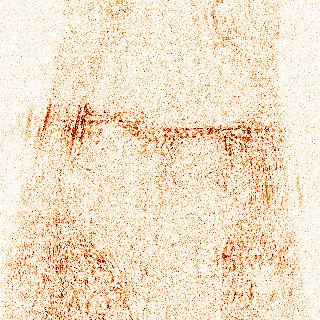}
  \caption*{\footnotesize patchNR}
\end{subfigure}%
\caption{8-fold multi-coil MRI on PDFS (bottom) and PD (top) data set.
The white box marks the zoomed area.
\textit{Top}: full image; \textit{middle}: zoomed-in part; \textit{bottom}: error.
} \label{fig:MRI_comparison_8}
\end{figure}

\subsection{Computed Tomography} \label{sec:ct}

Here, $\M H$ is the discrete Radon transform with a parallel beam geometry using 513 equidistant detector bins and 1000 equidistant angles between $0$ and $\pi$.
The corrupted  measurement $\V y$ is modeled as
\begin{equation}
\V y = - \frac{1}{\mu} \log \Bigl( \frac{\tilde {\V y}}{N_0} \Bigr), \quad \tilde {\V y} \sim \mathrm{Pois}\bigl(N_0 \exp (-\M H \V x \mu)\bigr),
\end{equation}
where $\mathrm{Pois}(\rho)$ is the Poisson distribution with mean $\rho$, $N_0=4096$ is the mean photon count per detector bin without attenuation, and $\mu = 81.35858$ is a normalization constant.
By computing $-\log p_{\V Y|\V X = \V x} (\V y)$ and omitting summands depending only on $\V y$, we arrive at the data fidelity
\begin{align} \label{eq:data_fidelity_CT}
\mathcal D(\M H \V x, \V y) =
\sum_{i=1}^{m}  \exp\left(-(\M H \V x)_i \mu\right) N_0 
+ \exp \left(- \mu \V y_i \right) N_0 \bigl(\mu(\M H \V x)_i - \log(N_0) \bigr).
\end{align}
The ground truth images from the LoDoPaB dataset \cite{LoDoPaB21}\footnote{Data available at \url{https://zenodo.org/record/3384092\#.Ylglz3VBwgM}.} have a resolution of $362 \times 362$.
To get a fair comparison with the patchNR, which is trained on 6 images of the training set, we also finetune the $\boldsymbol \Lambda \colon \R^m \to [0,1]^{nN_c}$ in \eqref{eq:rridge_reg_mask} for 30 epochs on them.
Once finetuned, we use $\mathcal R_{\V y, \sigma}$ for both full-view and limited-angle CT.
All hyperparameters are tuned on a single validation image.

\paragraph{Full view CT}
In Table~\ref{tab:error_CT}, we report quantitative results for the first 100 images of the LoDoPaB test split.
Again, the addition of $\boldsymbol \Lambda$ leads to a significant improvement of the metrics.
In particular, see Figure~\ref{fig:CT_masks}, the mask generation process behaves robust to the task shift.
In Figure~\ref{fig:CT_comparison}, we provide several reconstructions, including the filtered back-projection (FBP) \cite{Radon86} as a baseline.
The squared difference to the ground truth is taken into account as the error. For better visibility, the differences are rescaled with a square root.
While the WCRR  outperforms CRR both visually and in terms of the metrics, the data-dependent counterparts behave very similar.
Despite the similar metric, the patchNR procudes visually more appealing results.
For the data fidelity term \eqref{eq:data_fidelity_CT}, no closed form of its proximal operator exists.
Therefore, the DRS-PnP algorithm is impractical and we use a relaxed version of the PGD-PnP \cite{HurChaLec2023} instead.
Within this setting, we did not achieve competitive results and only report the obtained ones for completeness.
More reconstructions are given in Appendix~\ref{app:more_examples}.

\begin{table}[t]
\begin{center}
\scalebox{.9}{
\begin{tabular}[t]{c|ccc|cccc} 
    & \multicolumn{3}{c}{full view CT}   & \multicolumn{3}{c}{limited-angle CT} &   \\
    &    PSNR    & SSIM & time  & PSNR   & SSIM & time\\
\hline
FBP   &   30.37 $\pm$ 2.95 &  0.739 $\pm$ 0.141 & 0.03s &  21.96 $\pm$ 2.25  &  0.531 $\pm$ 0.097 & 0.03s \\
TV    &   32.87 $\pm$ 3.79     &  0.796 $\pm$ 0.151   &  12s& 29.40  $\pm$  2.73  &  0.756 $\pm$ 0.151 &  16s  \\
CRR   &   34.37 $\pm$ 4.20 & 0.818 $\pm$ 0.154 & 16s & 31.75 $\pm$ 3.24  & 0.792$\pm$ 0.155 & 18s  \\
WCRR   &  34.87 $\pm$ 4.40 & 0.822$\pm$ 0.156 & 19s & 32.54  $\pm$ 3.46  &  0.797 $\pm$ 0.157 & 19s \\
\hline 
%#Mask-TV  &  33.90 & 0.809  & 30.31  & 0.765 \\
Mask-CRR   &  \underline{35.14} $\pm$ 4.53 & \underline{0.826} $\pm$ 0.154 &  34s & 33.02 $\pm$ 3.61 & \underline{0.805} $\pm$ 0.156 & 38s\\
Mask-WCRR   &  35.07 $\pm$ 4.54 & 0.824 $\pm$ 0.157 &  32s & \underline{33.09} $\pm$ 3.62 & \underline{0.805} $\pm$ 0.156& 30s \\
Prox-DRUNet-$\alpha$PGD   & 34.23 $\pm$ 4.32  & 0.811 $\pm$ 0.154 &   593s   &  31.04 $\pm$ 3.23 & 0.771 $\pm$ 0.157 & 593s \\
\hline 
patchNR & \textbf{35.19}  $\pm$ 4.52 & \textbf{0.829} $\pm$ 0.152 &  48s& \textbf{33.20}  $\pm$ 3.55  & \textbf{0.811} $\pm$ 0.151 & 486s
\end{tabular}
}
\caption{CT. Average metrics and standard deviations for the first 100 images of the LoDoPaB test set.
Best is bold and second best is underlined.} 
\label{tab:error_CT}
\end{center}
\end{table}

\begin{figure}[p]
\centering
\begin{subfigure}[t]{.21\textwidth}  
\includegraphics[width=\linewidth]{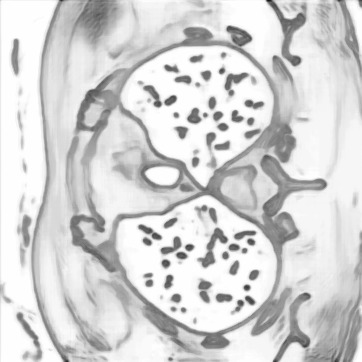}
\caption*{\footnotesize Mask-CRR}
\end{subfigure}%\begin{subfigure}[t]{.123\textwidth}  
\hspace{.04cm}
\begin{subfigure}[t]{.21\textwidth}  
\includegraphics[width=\linewidth]{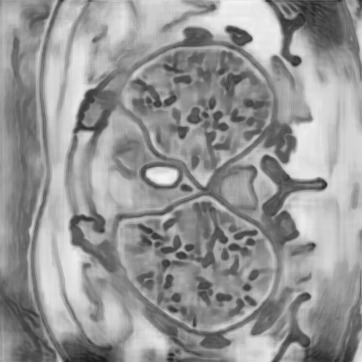}
\caption*{\footnotesize Mask-WCRR}
\end{subfigure}%
\hspace{1.5cm}
\begin{subfigure}[t]{.21\textwidth}  
\includegraphics[width=\linewidth]{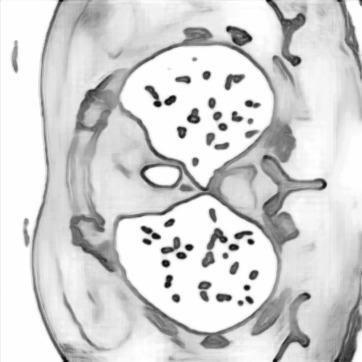}
\caption*{\footnotesize Mask-CRR}
\end{subfigure}%\begin{subfigure}[t]{.123\textwidth}  
\hspace{.04cm}
\begin{subfigure}[t]{.21\textwidth}  
\includegraphics[width=\linewidth]{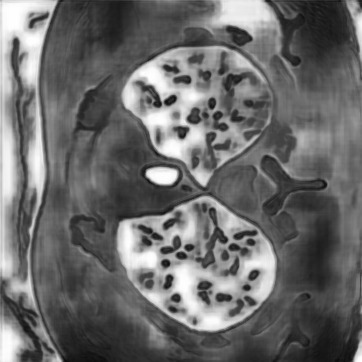}
\caption*{\footnotesize Mask-WCRR}
\end{subfigure}%
\caption{Average masks $\boldsymbol \Lambda (\V y ) = \sum_c \boldsymbol \Lambda_c (\V y )$ for CRR and WCRR in the full view (left) and limited-angle CT (right) setting.
Black corresponds to smaller values.} \label{fig:CT_masks}
\vspace{.5cm}

\centering
\begin{subfigure}[t]{.123\textwidth}  
\begin{tikzpicture}[spy using outlines=
{rectangle,white,magnification=5.5,size=2.115cm, connect spies}]
\node[anchor=south west,inner sep=0]  at (0,0) {\includegraphics[width=\linewidth]{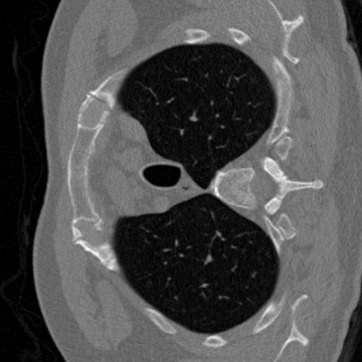}};
\spy on (1.7,1.02) in node [right] at (-0.005,-1.1);
\end{tikzpicture}
%\caption*{\footnotesize GT}
\end{subfigure}%
\hfill
\begin{subfigure}[t]{.123\textwidth}
\begin{tikzpicture}[spy using outlines=
{rectangle,white,magnification=5.5,size=2.115cm, connect spies}]
\node[anchor=south west,inner sep=0]  at (0,0) {\includegraphics[width=\linewidth]{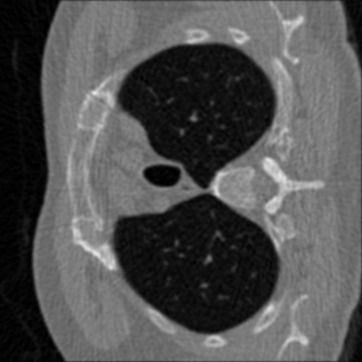}};
\spy on (1.7,1.02) in node [right] at (-0.005,-1.1);
\end{tikzpicture}
% \caption*{\footnotesize FBP}
\end{subfigure}%
\hfill
\begin{subfigure}[t]{.123\textwidth}
\begin{tikzpicture}[spy using outlines=
{rectangle,white,magnification=5.5,size=2.115cm, connect spies}]
\node[anchor=south west,inner sep=0]  at (0,0) {\includegraphics[width=\linewidth]{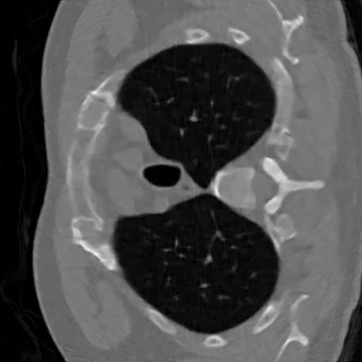}};
\spy on (1.7,1.02) in node [right] at (-0.005,-1.1);
\end{tikzpicture}
%\caption*{\footnotesize CRR}
\end{subfigure}%
\hfill
\begin{subfigure}[t]{.123\textwidth}
\begin{tikzpicture}[spy using outlines=
{rectangle,white,magnification=5.5,size=2.115cm, connect spies}]
\node[anchor=south west,inner sep=0]  at (0,0) {\includegraphics[width=\linewidth]{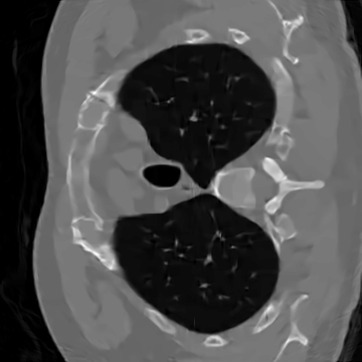}};
\spy on (1.7,1.02) in node [right] at (-0.005,-1.1);
\end{tikzpicture}
%  \caption*{\footnotesize WCRR}
\end{subfigure}%
\hfill
\begin{subfigure}[t]{.123\textwidth}
\begin{tikzpicture}[spy using outlines=
{rectangle,white,magnification=5.5,size=2.115cm, connect spies}]
\node[anchor=south west,inner sep=0]  at (0,0) {\includegraphics[width=\linewidth]{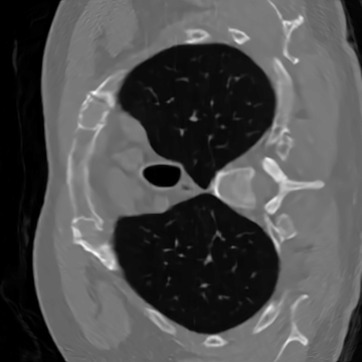}};
\spy on (1.7,1.02) in node [right] at (-0.005,-1.1);
\end{tikzpicture}
%  \caption*{\footnotesize Mask-CRR}
\end{subfigure}%
\hfill
\begin{subfigure}[t]{.123\textwidth}
\begin{tikzpicture}[spy using outlines=
{rectangle,white,magnification=5.5,size=2.115cm, connect spies}]
\node[anchor=south west,inner sep=0]  at (0,0) {\includegraphics[width=\linewidth]{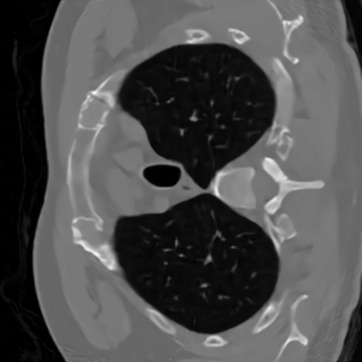}};
\spy on (1.7,1.02) in node [right] at (-0.005,-1.1);
\end{tikzpicture}
%  \caption*{\footnotesize Mask-WCRR}
\end{subfigure}%
\hfill
\begin{subfigure}[t]{.123\textwidth}
\begin{tikzpicture}[spy using outlines=
{rectangle,white,magnification=5.5,size=2.115cm, connect spies}]
\node[anchor=south west,inner sep=0]  at (0,0) {\includegraphics[width=\linewidth]{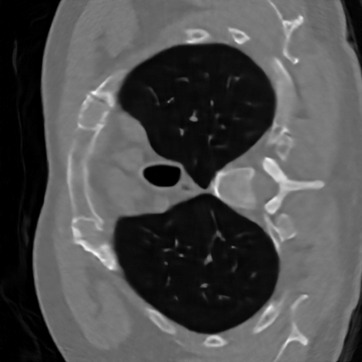}};
\spy on (1.7,1.02) in node [right] at (-0.005,-1.1);
\end{tikzpicture}
%  \caption*{\footnotesize Prox-DRUNet}
\end{subfigure}%
\hfill
\begin{subfigure}[t]{.123\textwidth}
\begin{tikzpicture}[spy using outlines=
{rectangle,white,magnification=5.5,size=2.115cm, connect spies}]
\node[anchor=south west,inner sep=0]  at (0,0) {\includegraphics[width=\linewidth]{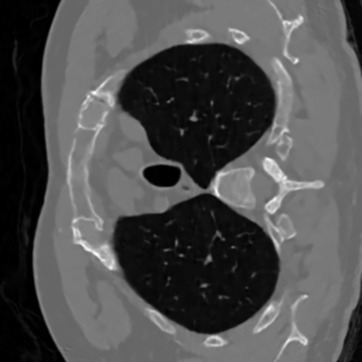}};
\spy on (1.7,1.02) in node [right] at (-0.005,-1.1);
\end{tikzpicture}
%  \caption*{\footnotesize patchNR}
\end{subfigure}%

%\begin{comment}
\begin{subfigure}[t]{.123\textwidth}
\includegraphics[width=\linewidth]{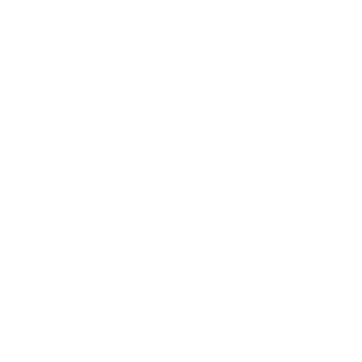}
  \caption*{\footnotesize GT}
\end{subfigure}%
\hfill
\begin{subfigure}[t]{.123\textwidth}
\includegraphics[width=\linewidth]{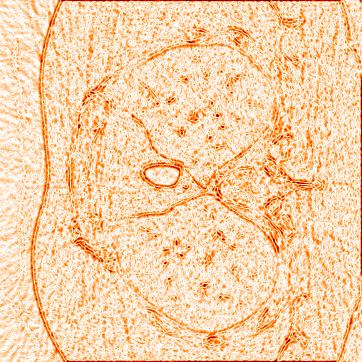}
  \caption*{\footnotesize FBP}
\end{subfigure}%
\hfill
\begin{subfigure}[t]{.123\textwidth}
\includegraphics[width=\linewidth]{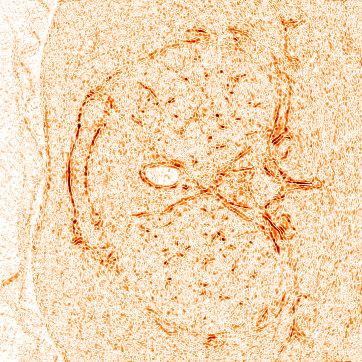}
  \caption*{\footnotesize CRR}
\end{subfigure}%
\hfill
\begin{subfigure}[t]{.123\textwidth}
\includegraphics[width=\linewidth]{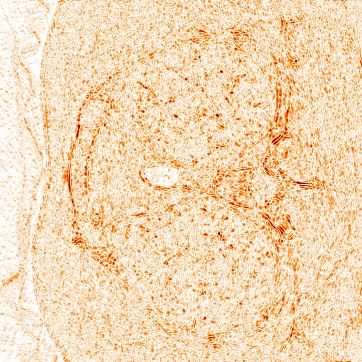}
  \caption*{\footnotesize WCRR}
\end{subfigure}%
\hfill
\begin{subfigure}[t]{.123\textwidth}
\includegraphics[width=\linewidth]{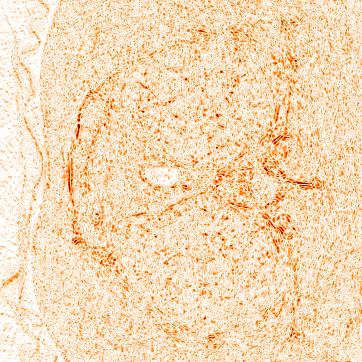}
  \caption*{\footnotesize Mask-CRR}
\end{subfigure}%
\hfill
\begin{subfigure}[t]{.123\textwidth}
\includegraphics[width=\linewidth]{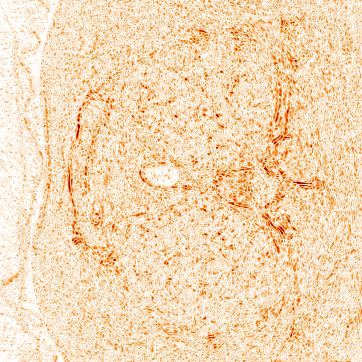}
  \caption*{\footnotesize Mask-WCRR}
\end{subfigure}%
\hfill
\begin{subfigure}[t]{.123\textwidth}
\includegraphics[width=\linewidth]{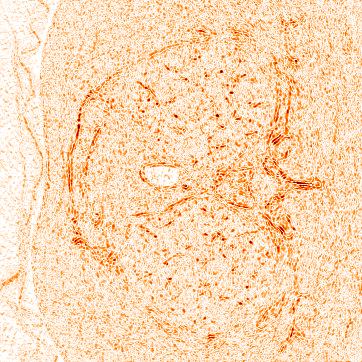}
  \caption*{\footnotesize Prox-DRUNet}
\end{subfigure}%
\hfill
\begin{subfigure}[t]{.123\textwidth}
\includegraphics[width=\linewidth]{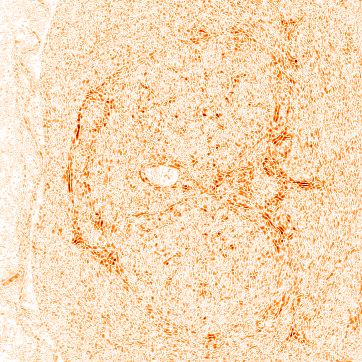}
  \caption*{\footnotesize patchNR}
\end{subfigure}%
%\end{comment}
\caption{Full view CT.
The white box marks the zoomed area.
\textit{Top}: full image; \textit{middle}: zoomed-in part; \textit{bottom}: error.
}\label{fig:CT_comparison}
\vspace{.5cm}

\centering
\begin{subfigure}[t]{.123\textwidth}  
\begin{tikzpicture}[spy using outlines=
{rectangle,white,magnification=6.5,size=2.115cm, connect spies}]
\node[anchor=south west,inner sep=0]  at (0,0) {\includegraphics[width=\linewidth]{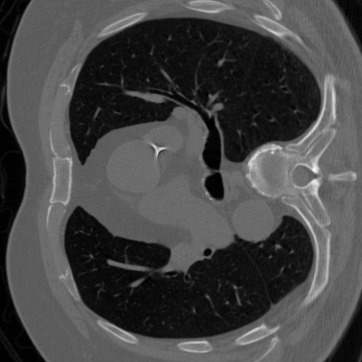}};
\spy on (1.92,1.08) in node [right] at (-0.005,-1.1);
\end{tikzpicture}
%\caption*{\footnotesize GT}
\end{subfigure}%
\hfill
\begin{subfigure}[t]{.123\textwidth}
\begin{tikzpicture}[spy using outlines=
{rectangle,white,magnification=6.5,size=2.115cm, connect spies}]
\node[anchor=south west,inner sep=0]  at (0,0) {\includegraphics[width=\linewidth]{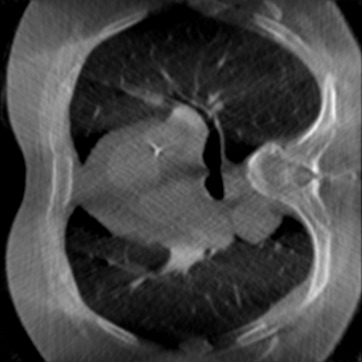}};
\spy on (1.92,1.08) in node [right] at (-0.005,-1.1);
\end{tikzpicture}
% \caption*{\footnotesize FBP}
\end{subfigure}%
\hfill
\begin{subfigure}[t]{.123\textwidth}
\begin{tikzpicture}[spy using outlines=
{rectangle,white,magnification=6.5,size=2.115cm, connect spies}]
\node[anchor=south west,inner sep=0]  at (0,0) {\includegraphics[width=\linewidth]{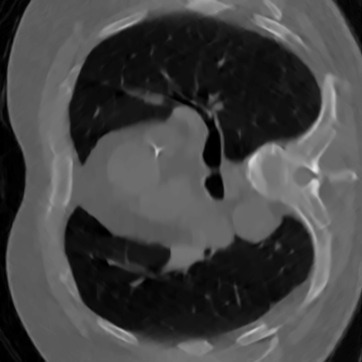}};
\spy on (1.92,1.08) in node [right] at (-0.005,-1.1);
\end{tikzpicture}
%  \caption*{\footnotesize CRR}
\end{subfigure}%
\hfill
\begin{subfigure}[t]{.123\textwidth}
\begin{tikzpicture}[spy using outlines=
{rectangle,white,magnification=6.5,size=2.115cm, connect spies}]
\node[anchor=south west,inner sep=0]  at (0,0) {\includegraphics[width=\linewidth]{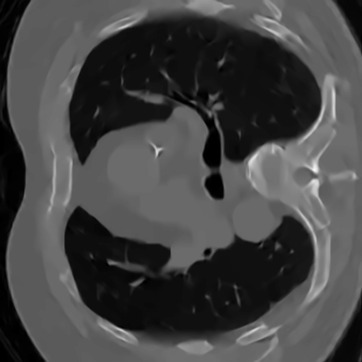}};
\spy on (1.92,1.08) in node [right] at (-0.005,-1.1);
\end{tikzpicture}
% \caption*{\footnotesize WCRR}
\end{subfigure}%
\hfill
\begin{subfigure}[t]{.123\textwidth}
\begin{tikzpicture}[spy using outlines=
{rectangle,white,magnification=6.5,size=2.115cm, connect spies}]
\node[anchor=south west,inner sep=0]  at (0,0) {\includegraphics[width=\linewidth]{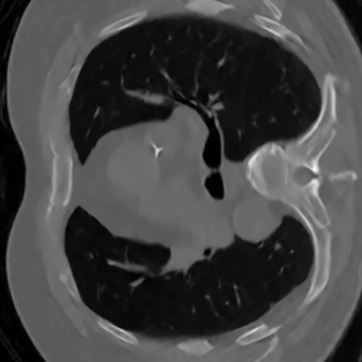}};
\spy on (1.92,1.08) in node [right] at (-0.005,-1.1);
\end{tikzpicture}
%  \caption*{\footnotesize Mask-CRR}
\end{subfigure}%
\hfill
\begin{subfigure}[t]{.123\textwidth}
\begin{tikzpicture}[spy using outlines=
{rectangle,white,magnification=6.5,size=2.115cm, connect spies}]
\node[anchor=south west,inner sep=0]  at (0,0) {\includegraphics[width=\linewidth]{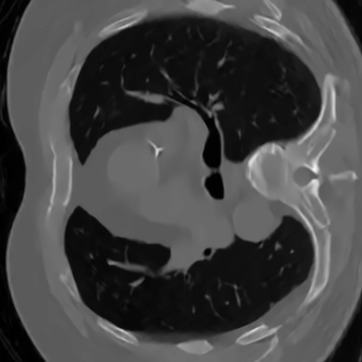}};
\spy on (1.92,1.08) in node [right] at (-0.005,-1.1);
\end{tikzpicture}
% \caption*{\footnotesize Mask-WCRR}
\end{subfigure}%
\hfill
\begin{subfigure}[t]{.123\textwidth}
\begin{tikzpicture}[spy using outlines=
{rectangle,white,magnification=6.5,size=2.115cm, connect spies}]
\node[anchor=south west,inner sep=0]  at (0,0) {\includegraphics[width=\linewidth]{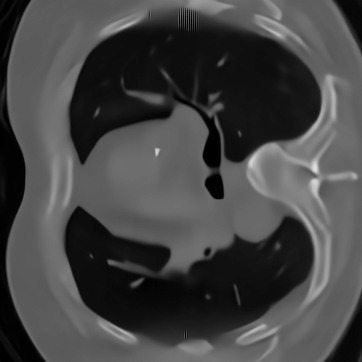}};
\spy on (1.92,1.08) in node [right] at (-0.005,-1.1);
\end{tikzpicture}
% \caption*{\footnotesize Prox-DRUNet}
\end{subfigure}%
\hfill
\begin{subfigure}[t]{.123\textwidth}
\begin{tikzpicture}[spy using outlines=
{rectangle,white,magnification=6.5,size=2.115cm, connect spies}]
\node[anchor=south west,inner sep=0]  at (0,0) {\includegraphics[width=\linewidth]{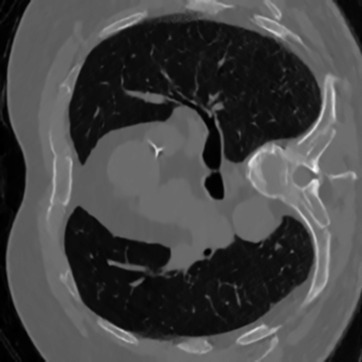}};
\spy on (1.92,1.08) in node [right] at (-0.005,-1.1);
\end{tikzpicture}
% \caption*{\footnotesize patchNR}
\end{subfigure}%

\begin{subfigure}[t]{.123\textwidth}
\includegraphics[width=\linewidth]{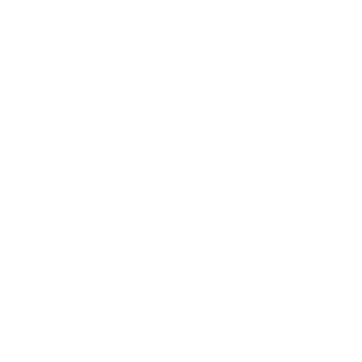}
  \caption*{\footnotesize GT}
\end{subfigure}%
\hfill
\begin{subfigure}[t]{.123\textwidth}
\includegraphics[width=\linewidth]{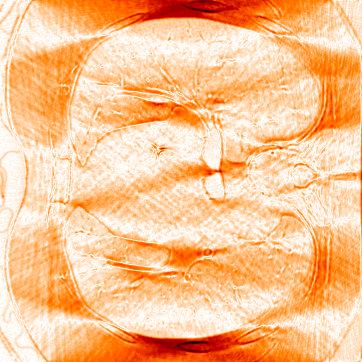}
  \caption*{\footnotesize FBP}
\end{subfigure}%
\hfill
\begin{subfigure}[t]{.123\textwidth}
\includegraphics[width=\linewidth]{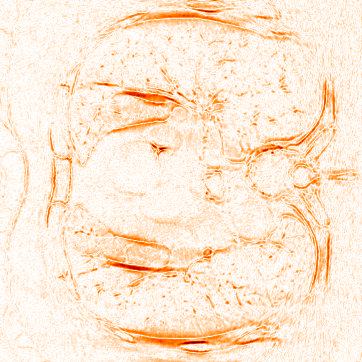}
  \caption*{\footnotesize CRR}
\end{subfigure}%
\hfill
\begin{subfigure}[t]{.123\textwidth}
\includegraphics[width=\linewidth]{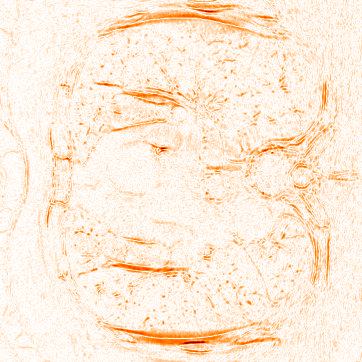}
  \caption*{\footnotesize WCRR}
\end{subfigure}%
\hfill
\begin{subfigure}[t]{.123\textwidth}
\includegraphics[width=\linewidth]{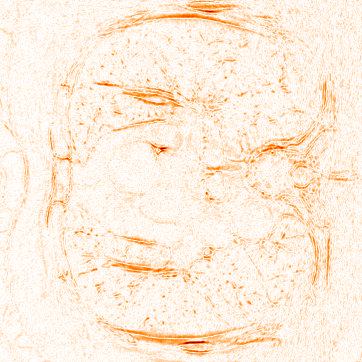}
  \caption*{\footnotesize Mask-CRR}
\end{subfigure}%
\hfill
\begin{subfigure}[t]{.123\textwidth}
\includegraphics[width=\linewidth]{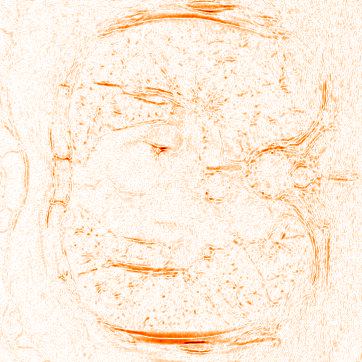}
  \caption*{\footnotesize Mask-WCRR}
\end{subfigure}%
\hfill
\begin{subfigure}[t]{.123\textwidth}
\includegraphics[width=\linewidth]{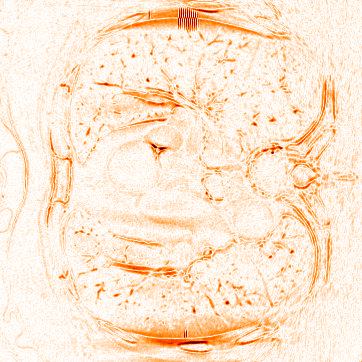}
  \caption*{\footnotesize Prox-DRUNet}
\end{subfigure}%
\hfill
\begin{subfigure}[t]{.123\textwidth}
\includegraphics[width=\linewidth]{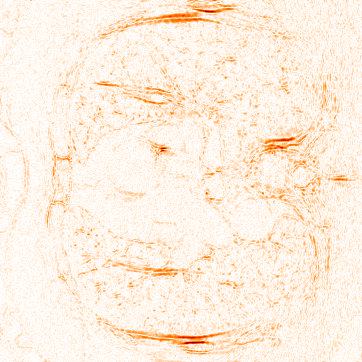}
  \caption*{\footnotesize patchNR}
\end{subfigure}%

\caption{Limited-angle CT.
The white box marks the zoomed area.
\textit{Top}: full image; \textit{middle}: zoomed-in part; \textit{bottom}: error.
} \label{fig:CT_comparison_limited}
\end{figure}

\paragraph{Subsampling: Limited-Angle CT}

Now, we remove the first and last 100 angles from the data, which corresponds to $36^\circ$ out of $180^\circ$.
Several reconstructions and quantitative metrics are given in Figure~\ref{fig:CT_comparison_limited} and  Table~\ref{tab:error_CT}, respectively.
This task is much harder than full view CT, which becomes most evident when inspecting the FBP reconstruction.
Overall, the performance comparison turns out similar as before.
Interestingly, the performance boost attributed to $\boldsymbol \Lambda$ becomes even higher for this setup.
In Figure~\ref{fig:CT_masks}, we can see that $\boldsymbol{\Lambda}$ remains reasonable despite the large amount of missing data.
While the CRR masks look similar for both CT setups, the WCRR ones are considerably different.
This is a bit surprising and calls for further investigations.
More reconstructions are provided in Appendix~\ref{app:more_examples}.

\subsection{Image Superresolution} \label{sec:SiC}
Here, we consider the superresolution of microstructures in a composite of silicon and diamonds (``SiC Diamonds''), see also \cite{Hertrich21,ADHHMS2023}.
The dataset consists of 2D slices of size $600 \times 600$.
These were acquired by microcomputed tomography at the SLS beamline TOMCAT.
The operator $\M H$ is a convolution with a $16 \times 16$ Gaussian kernel with standard deviation 2 and stride 4, namely with 4-fold subsampling.
The data is simulated by applying $\M H$ to the groundtruth and adding Gaussian noise with $\sigma=0.01$, which leads to $\mathcal D(\M H \V x , \V y) = \frac{1}{2\sigma^2} \Vert \M H \V x - \V y \Vert^2$ as in Section~\ref{sec:mri}. 
The patchNR is trained on one image, which is also used to determine the hyperparameters.

In Figure~\ref{fig:SiC_comparison}, we provide visual results and in Figure~\ref{fig:SiC_masks} we show the $\boldsymbol \Lambda \colon \R^m \to [0,1]^{nN_c}$ for Mask-CRR and Mask-WCRR.
Note that CRR and WCRR suffer from blur in their reconstructions, especially in the small regions between structures.
Here, the inclusion of $\boldsymbol \Lambda$ results in noticeable improvement.
The contrast in $\boldsymbol \Lambda$ is much higher for the CRR compared to the WCRR.
For the latter, the penalization of structure is already notably dampened by the $1$-weakly convex potentials and a less pronounced mask seems to suffice.
The results of Mask-CRR, Mask-WCRR and Prox-DRUNet are comparable, but do not match the patchNR.
In particular, this method is the only one that reconstructs the tiny bright structure in the bottom of the zoomed-in part.
These observations are confirmed by the quantitative results in Table~\ref{tab:error_SiC}.
Since the other priors are not tailored to this very particular data, the superiority of the patchNR is not surprising.
More results are included in Appendix~\ref{app:more_examples}.

\begin{table}[t]
\begin{center}
\begin{tabular}[t]{c|cccc} 
    & \multicolumn{3}{c}{SiC}   &   \\
    &    PSNR    & SSIM & time  \\
\hline
Bicubic   &   25.63 $\pm$ 0.56 &  0.699$\pm$ 0.012 & 0.0002s \\
TV    &   26.14  $\pm$  0.61   &  0.710  $\pm$ 0.018   &  11s  \\
CRR   &   27.83$\pm$ 0.52  & 0.772 $\pm$ 0.008 & 38s \\
WCRR   &  28.17$\pm$ 0.54 & 0.788 $\pm$ 0.010 & 38s  \\
\hline 
Mask-CRR   &  28.21 $\pm$ 0.55 & 0.788 $\pm$ 0.010 &  45s \\
Mask-WCRR   &  \underline{28.27} $\pm$ 0.55 & \textbf{0.790} $\pm$ 0.010 &  45s \\
Prox-DRUNet DRS   &  28.21   $\pm$ 0.52 & \underline{0.789} $\pm$ 0.010  &   652s  \\
\hline 
patchNR & \textbf{28.53} $\pm$ 0.49 & 0.780 $\pm$ 0.008 &  151s 
\end{tabular}
\caption{Superresolution of SiC. Averaged metrics and standard deviations evaluated on 100 test images.
Best is bold and second best is underlined.} 
\label{tab:error_SiC}
\end{center}
\end{table}

\begin{figure}[t!]
\centering
\begin{subfigure}[t]{.123\textwidth}  
\begin{tikzpicture}[spy using outlines=
{rectangle,white,magnification=10,size=2.115cm, connect spies}]
\node[anchor=south west,inner sep=0]  at (0,0) {\includegraphics[width=\linewidth]{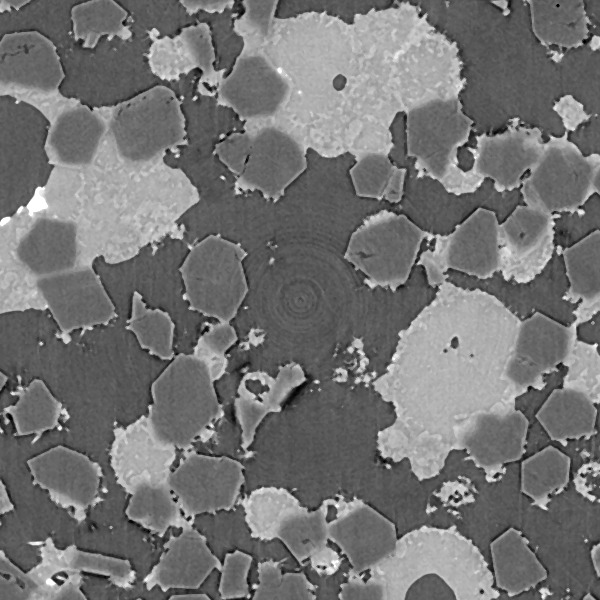}};
\spy on (1.25,.83) in node [right] at (-0.005,-1.1);
\end{tikzpicture}
\caption*{\footnotesize HR}
\end{subfigure}%
\hfill
\begin{subfigure}[t]{.123\textwidth}
\begin{tikzpicture}[spy using outlines=
{rectangle,white,magnification=10,size=2.115cm, connect spies}]
\node[anchor=south west,inner sep=0]  at (0,0) {\includegraphics[width=\linewidth]{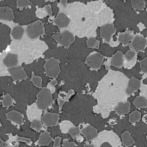}};
\spy on (1.25,.83) in node [right] at (-0.005,-1.1);
\end{tikzpicture}
  \caption*{\footnotesize LR}
\end{subfigure}%
\hfill
\begin{subfigure}[t]{.123\textwidth}
\begin{tikzpicture}[spy using outlines=
{rectangle,white,magnification=10,size=2.115cm, connect spies}]
\node[anchor=south west,inner sep=0]  at (0,0) {\includegraphics[width=\linewidth]{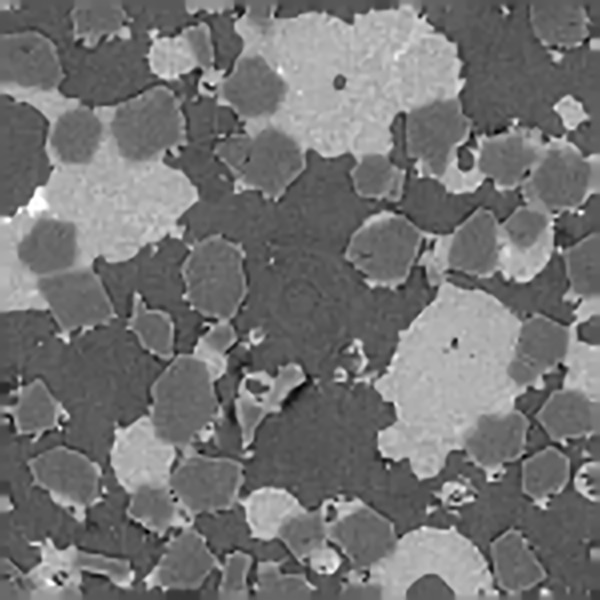}};
\spy on (1.25,.83) in node [right] at (-0.005,-1.1);
\end{tikzpicture}
  \caption*{\footnotesize CRR}
\end{subfigure}%
\hfill
\begin{subfigure}[t]{.123\textwidth}
\begin{tikzpicture}[spy using outlines=
{rectangle,white,magnification=10,size=2.115cm, connect spies}]
\node[anchor=south west,inner sep=0]  at (0,0) {\includegraphics[width=\linewidth]{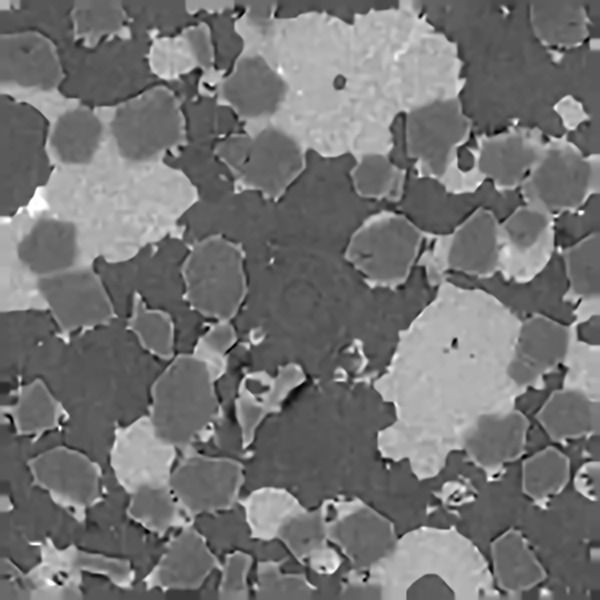}};
\spy on (1.25,.83) in node [right] at (-0.005,-1.1);
\end{tikzpicture}
  \caption*{\footnotesize WCRR}
\end{subfigure}%
\hfill
\begin{subfigure}[t]{.123\textwidth}
\begin{tikzpicture}[spy using outlines=
{rectangle,white,magnification=10,size=2.115cm, connect spies}]
\node[anchor=south west,inner sep=0]  at (0,0) {\includegraphics[width=\linewidth]{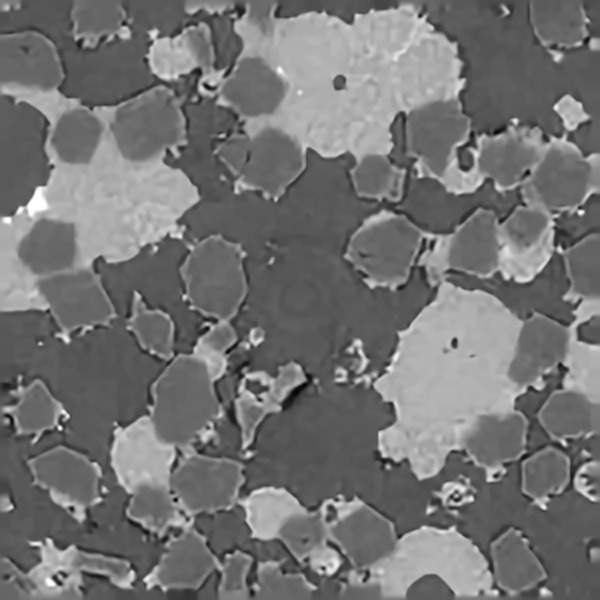}};
\spy on (1.25,.83) in node [right] at (-0.005,-1.1);
\end{tikzpicture}
  \caption*{\footnotesize Mask-CRR}
\end{subfigure}%
\hfill
\begin{subfigure}[t]{.123\textwidth}
\begin{tikzpicture}[spy using outlines=
{rectangle,white,magnification=10,size=2.115cm, connect spies}]
\node[anchor=south west,inner sep=0]  at (0,0) {\includegraphics[width=\linewidth]{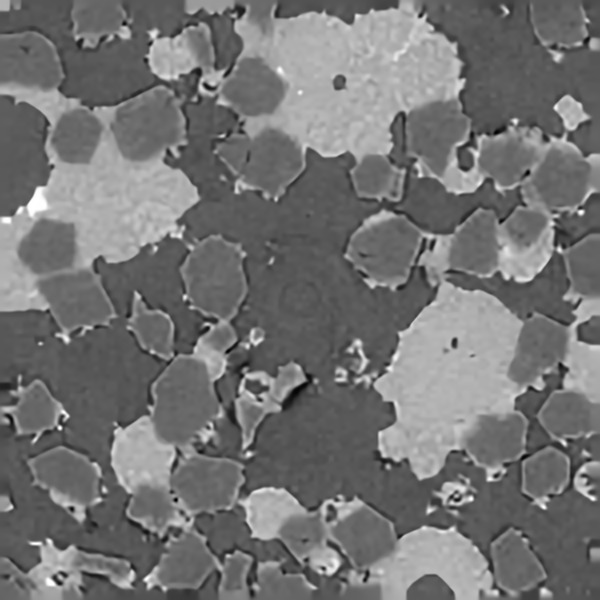}};
\spy on (1.25,.83) in node [right] at (-0.005,-1.1);
\end{tikzpicture}
  \caption*{\footnotesize Mask-WCRR}
\end{subfigure}%
\hfill
\begin{subfigure}[t]{.123\textwidth}
\begin{tikzpicture}[spy using outlines=
{rectangle,white,magnification=10,size=2.115cm, connect spies}]
\node[anchor=south west,inner sep=0]  at (0,0) {\includegraphics[width=\linewidth]{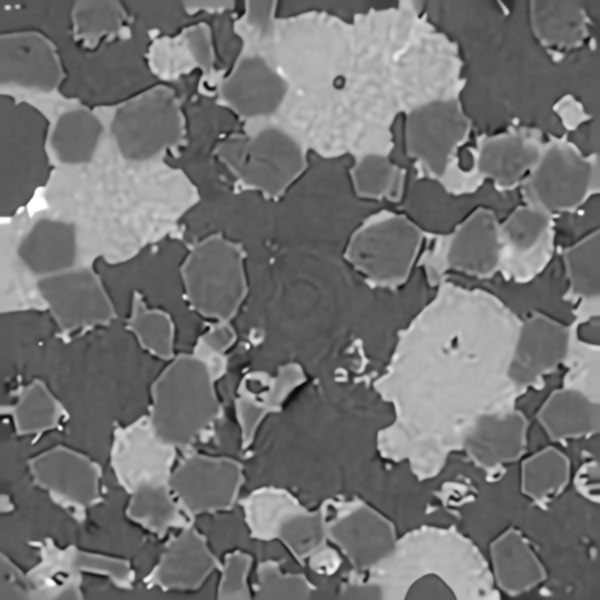}};
\spy on (1.25,.83) in node [right] at (-0.005,-1.1);
\end{tikzpicture}
  \caption*{\footnotesize Prox-DRUNet}
\end{subfigure}%
\hfill
\begin{subfigure}[t]{.123\textwidth}
\begin{tikzpicture}[spy using outlines=
{rectangle,white,magnification=10,size=2.115cm, connect spies}]
\node[anchor=south west,inner sep=0]  at (0,0) {\includegraphics[width=\linewidth]{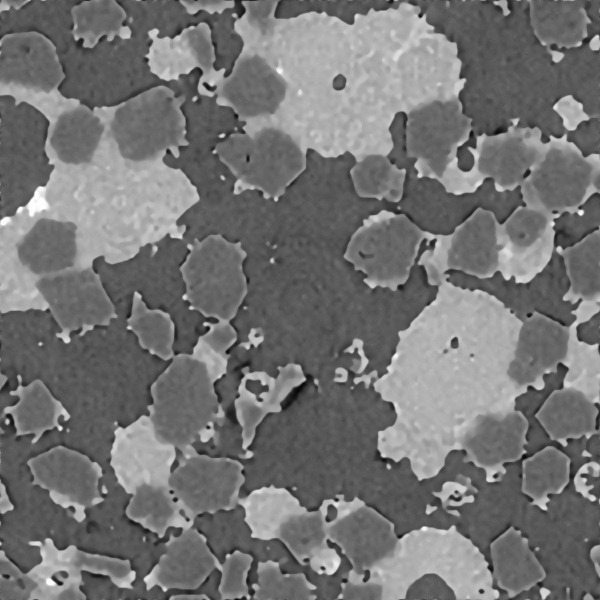}};
\spy on (1.25,.83) in node [right] at (-0.005,-1.1);
\end{tikzpicture}
  \caption*{\footnotesize patchNR}
\end{subfigure}%
\caption{Superresolution of material microstructures.
The white box marks the zoomed area.
\textit{Top}: full image; \textit{bottom}: zoomed-in part.
} \label{fig:SiC_comparison}
\vspace{.5cm}

\centering
\begin{subfigure}[t]{.28\textwidth}  
\includegraphics[width=\linewidth]{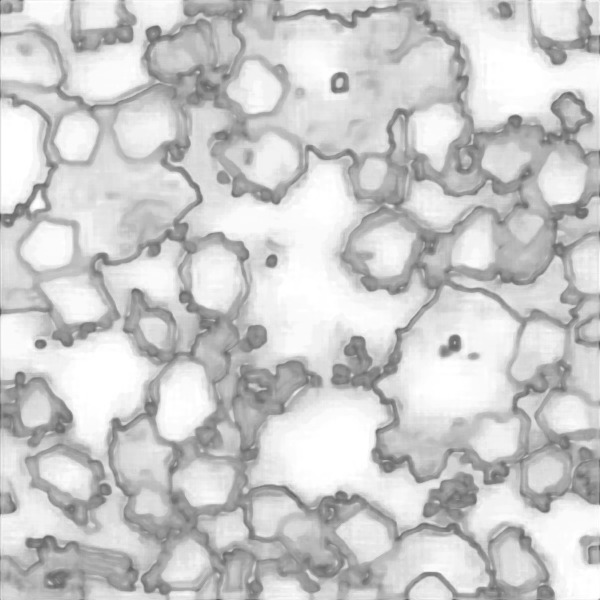}
\caption*{\footnotesize Mask-CRR}
\end{subfigure}%\begin{subfigure}[t]{.123\textwidth}  
\hspace{.1cm}
\begin{subfigure}[t]{.28\textwidth}  
\includegraphics[width=\linewidth]{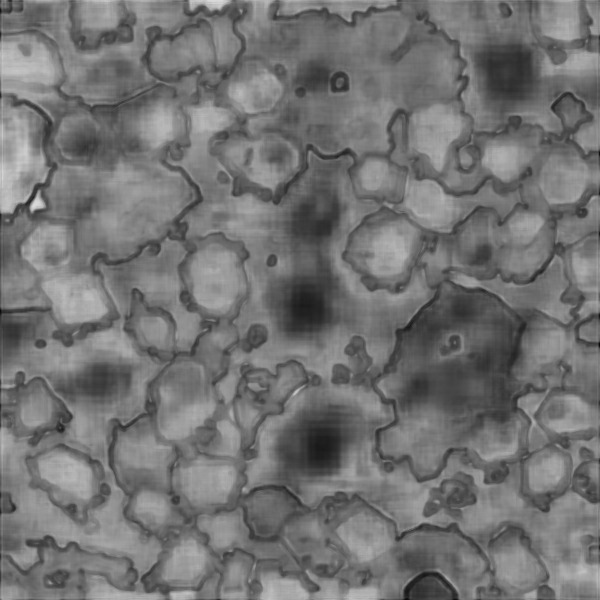}
\caption*{\footnotesize Mask-WCRR}
\end{subfigure}%
\caption{Average mask $\boldsymbol \Lambda (\V y ) = \sum_c \boldsymbol \Lambda_c (\V y )$ for image superresolution of material`s microstructures.
Black corresponds to smaller values.} \label{fig:SiC_masks}
\end{figure}

\section{Conclusion}

We established theoretical guarantees for ridge-based regularization.
Here, the (rather uncommon) data-dependence of the regularizer $\mathcal R_{\V y}$ made a new analysis necessary.
Our results for Mask-CRR are stronger than for Mask-WCRR, which presents an interesting direction for future work.
Given their similar experimental performance, we recommend to use the Mask-CRR for now.   
An important direction for future work is to investigate conditions for the existence of a unique critical point, akin to results in compressed sensing.
Such an analysis would extend our findings to the stronger notion of classical Lipschitz continuity, further strengthening the theoretical foundation.

From a practical perspective, a big advantage of our learned variational approaches (and those in the comparison) is that they do not require large amounts of training data.
This comes at a cost: computing the minimizers is more time-consuming compared to evaluating an end-to-end trained reconstruction network.
However, these methods are highly specialized for a specific task and cannot be easily adapted to different datasets or operators.
Moreover, they only work if sufficient training data is available.
In contrast, we trained both the mask $\boldsymbol \Lambda$ and the baseline regularizer  $\mathcal R_\sigma$ for denoising of natural images.
Hence, our approach is extremely flexible.
We can apply it to any inverse problem by adjusting the regularization strength $\lambda$ and the noise level $\sigma$ in \eqref{eq:var_prob} without additional training.

Our approach can be extended into several directions.
Since $\mathcal R_{\V y,\sigma}$ induces a probability distribution, we can apply it for uncertainty quantification based on invertible architectures \cite{AKRK2019,dinhrnvp,HHS2021} or Langevin methods as done for the patchNR in \cite{CTMLSZ2024}.
Moreover, if we have knowledge about the operator $\M H$, the noise model, or sufficiently many task specific training images, we can adapt $\mathcal{R}_{\V y}$ for the specific setting.
In particular, we can parameterize $\boldsymbol \Lambda$ directly based on $\V y$ instead of using an initial reconstruction $\V x_\mathrm{est}$.
This likely improves the performance and makes our approach, besides its theoretical guarantees, even more relevant to practical applications.

\section*{Acknowledgments}
The authors have no conflict of interest to declare.
S.N.\ acknowledges support from the DFG within the SPP2298 under the project number 543939932.
F.A.\ acknowledges support from the DFG under Germany‘s Excellence Strategy – The Berlin Mathematics Research Center MATH+ (project AA5-6). 
The data from Section \ref{sec:SiC} has been acquired in the frame of the EU Horizon 2020 Marie Sklodowska-Curie Actions Innovative Training Network MUMMERING (MUltiscale, Multimodal and Multidimensional imaging for EngineeRING, Grant Number 765604) at the beamline TOMCAT of the SLS by A.\ Saadaldin, D.\ Bernard, and F.\ Marone Welford. We acknowledge the Paul Scherrer Institut, Villigen, Switzerland for provision of synchrotron radiation beamtime at the TOMCAT beamline X02DA of the SLS.

\bibliography{references,references_cvx,refs}

\appendix

\section{Reconstruction Examples} \label{app:more_examples}

Here, we provide additional qualitative results for the numerical examples from Section~\ref{sec:experiments}, see Figures~\ref{fig:CT_comparison_appendix} and \ref{fig:CT_comparison_limited_appendix} for full view and limited-angle CT, respectively, and Figures~\ref{fig:MRI_comparison_pd4_appendix} to \ref{fig:MRI_comparison_pdfs8_appendix} for the MRI experiments, and Figure~\ref{fig:SiC_comparison_appendix} for the superresolution example.
The error is the absolute difference to the ground truth.
For the CT experiments, it is rescaled with a square root for better visibility.

\begin{figure}[t]
\centering
\begin{subfigure}[t]{.123\textwidth}  
\begin{tikzpicture}[spy using outlines=
{rectangle,white,magnification=3.8,size=2.115cm, connect spies}]
\node[anchor=south west,inner sep=0]  at (0,0) {\includegraphics[width=\linewidth]{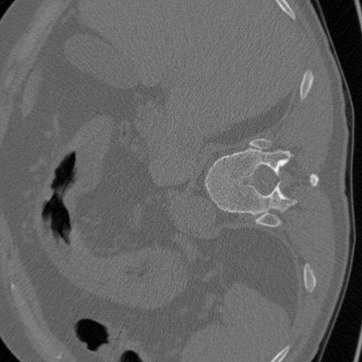}};
\spy on (1.45,1.04) in node [right] at (-0.005,-1.1);
\end{tikzpicture}
%\caption*{\footnotesize GT}
\end{subfigure}%
\hfill
\begin{subfigure}[t]{.123\textwidth}
\begin{tikzpicture}[spy using outlines=
{rectangle,white,magnification=3.8,size=2.115cm, connect spies}]
\node[anchor=south west,inner sep=0]  at (0,0) {\includegraphics[width=\linewidth]{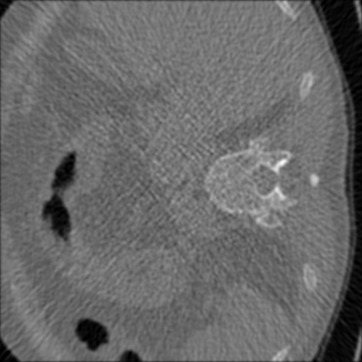}};
\spy on (1.45,1.04) in node [right] at (-0.005,-1.1);
\end{tikzpicture}
% \caption*{\footnotesize FBP}
\end{subfigure}%
\hfill
\begin{subfigure}[t]{.123\textwidth}
\begin{tikzpicture}[spy using outlines=
{rectangle,white,magnification=3.8,size=2.115cm, connect spies}]
\node[anchor=south west,inner sep=0]  at (0,0) {\includegraphics[width=\linewidth]{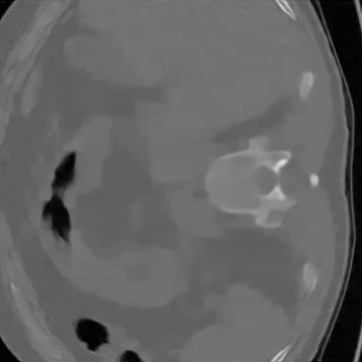}};
\spy on (1.45,1.04) in node [right] at (-0.005,-1.1);
\end{tikzpicture}
%  \caption*{\footnotesize CRR}
\end{subfigure}%
\hfill
\begin{subfigure}[t]{.123\textwidth}
\begin{tikzpicture}[spy using outlines=
{rectangle,white,magnification=3.8,size=2.115cm, connect spies}]
\node[anchor=south west,inner sep=0]  at (0,0) {\includegraphics[width=\linewidth]{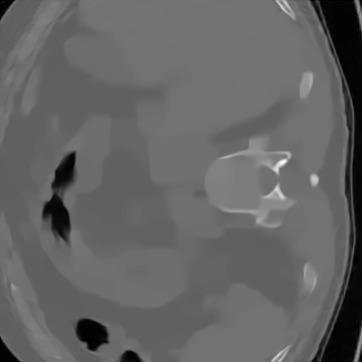}};
\spy on (1.45,1.04) in node [right] at (-0.005,-1.1);
\end{tikzpicture}
% \caption*{\footnotesize WCRR}
\end{subfigure}%
\hfill
\begin{subfigure}[t]{.123\textwidth}
\begin{tikzpicture}[spy using outlines=
{rectangle,white,magnification=3.8,size=2.115cm, connect spies}]
\node[anchor=south west,inner sep=0]  at (0,0) {\includegraphics[width=\linewidth]{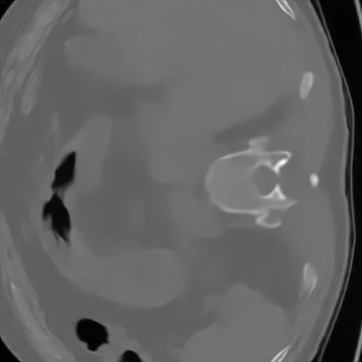}};
\spy on (1.45,1.04) in node [right] at (-0.005,-1.1);
\end{tikzpicture}
%  \caption*{\footnotesize Mask-CRR}
\end{subfigure}%
\hfill
\begin{subfigure}[t]{.123\textwidth}
\begin{tikzpicture}[spy using outlines=
{rectangle,white,magnification=3.8,size=2.115cm, connect spies}]
\node[anchor=south west,inner sep=0]  at (0,0) {\includegraphics[width=\linewidth]{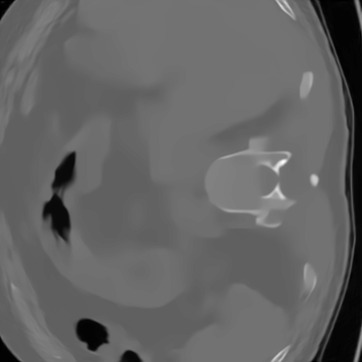}};
\spy on (1.45,1.04) in node [right] at (-0.005,-1.1);
\end{tikzpicture}
% \caption*{\footnotesize Mask-WCRR}
\end{subfigure}%
\hfill
\begin{subfigure}[t]{.123\textwidth}
\begin{tikzpicture}[spy using outlines=
{rectangle,white,magnification=3.8,size=2.115cm, connect spies}]
\node[anchor=south west,inner sep=0]  at (0,0) {\includegraphics[width=\linewidth]{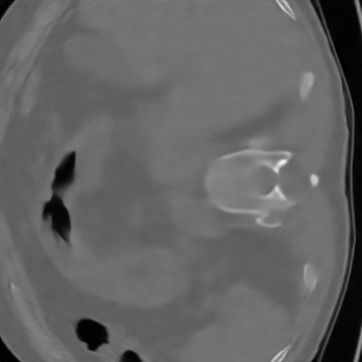}};
\spy on (1.45,1.04) in node [right] at (-0.005,-1.1);
\end{tikzpicture}
% \caption*{\footnotesize Prox-DRUNet}
\end{subfigure}%
\hfill
\begin{subfigure}[t]{.123\textwidth}
\begin{tikzpicture}[spy using outlines=
{rectangle,white,magnification=3.8,size=2.115cm, connect spies}]
\node[anchor=south west,inner sep=0]  at (0,0) {\includegraphics[width=\linewidth]{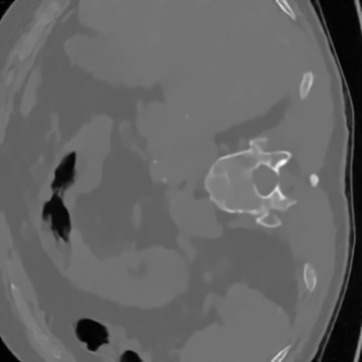}};
\spy on (1.45,1.04) in node [right] at (-0.005,-1.1);
\end{tikzpicture}
% \caption*{\footnotesize patchNR}
\end{subfigure}%

\begin{subfigure}[t]{.123\textwidth}
\includegraphics[width=\linewidth]{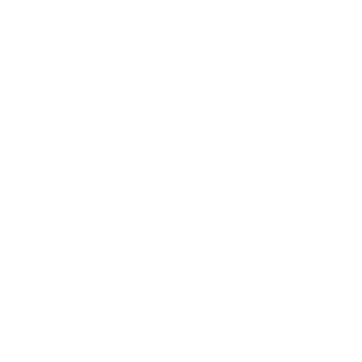}
  \caption*{\footnotesize GT}
\end{subfigure}%
\hfill
\begin{subfigure}[t]{.123\textwidth}
\includegraphics[width=\linewidth]{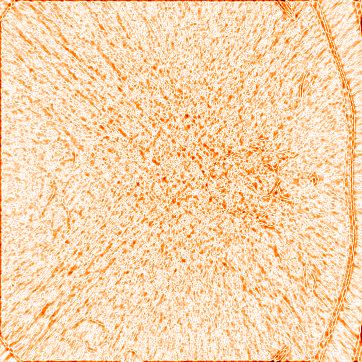}
  \caption*{\footnotesize FBP}
\end{subfigure}%
\hfill
\begin{subfigure}[t]{.123\textwidth}
\includegraphics[width=\linewidth]{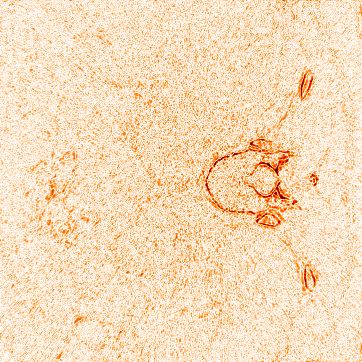}
  \caption*{\footnotesize CRR}
\end{subfigure}%
\hfill
\begin{subfigure}[t]{.123\textwidth}
\includegraphics[width=\linewidth]{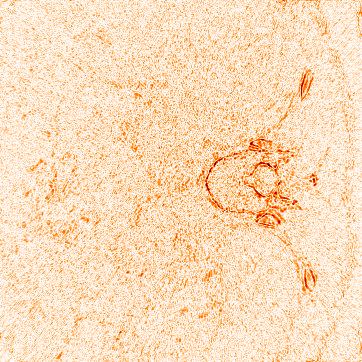}
  \caption*{\footnotesize WCRR}
\end{subfigure}%
\hfill
\begin{subfigure}[t]{.123\textwidth}
\includegraphics[width=\linewidth]{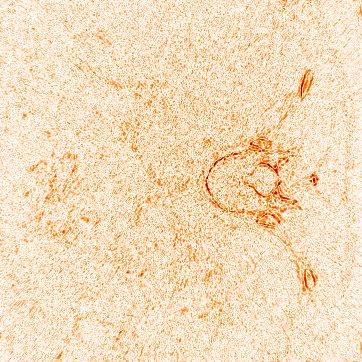}
  \caption*{\footnotesize Mask-CRR}
\end{subfigure}%
\hfill
\begin{subfigure}[t]{.123\textwidth}
\includegraphics[width=\linewidth]{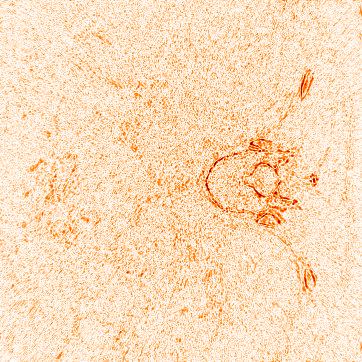}
  \caption*{\footnotesize Mask-WCRR}
\end{subfigure}%
\hfill
\begin{subfigure}[t]{.123\textwidth}
\includegraphics[width=\linewidth]{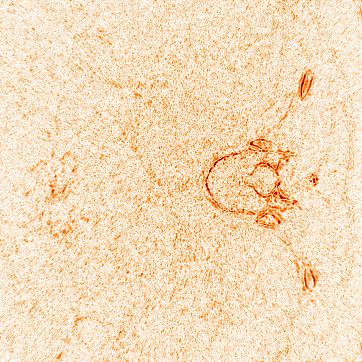}
  \caption*{\footnotesize Prox-DRUNet}
\end{subfigure}%
\hfill
\begin{subfigure}[t]{.123\textwidth}
\includegraphics[width=\linewidth]{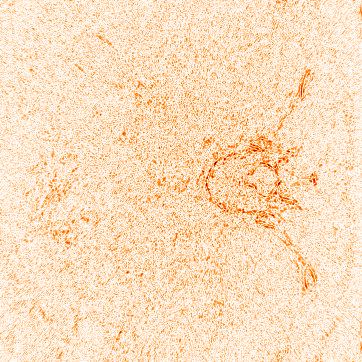}
  \caption*{\footnotesize patchNR}
\end{subfigure}%

\begin{subfigure}[t]{.123\textwidth}  
\begin{tikzpicture}[spy using outlines=
{rectangle,white,magnification=5.,size=2.115cm, connect spies}]
\node[anchor=south west,inner sep=0]  at (0,0) {\includegraphics[width=\linewidth]{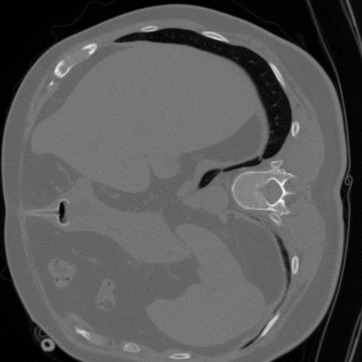}};
\spy on (1.53,.99) in node [right] at (-0.005,-1.1);
\end{tikzpicture}
%\caption*{\footnotesize GT}
\end{subfigure}%
\hfill
\begin{subfigure}[t]{.123\textwidth}
\begin{tikzpicture}[spy using outlines=
{rectangle,white,magnification=5.,size=2.115cm, connect spies}]
\node[anchor=south west,inner sep=0]  at (0,0) {\includegraphics[width=\linewidth]{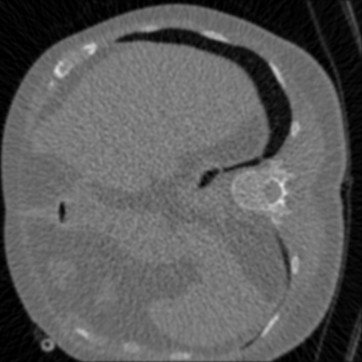}};
\spy on (1.53,.99) in node [right] at (-0.005,-1.1);
\end{tikzpicture}
% \caption*{\footnotesize FBP}
\end{subfigure}%
\hfill
\begin{subfigure}[t]{.123\textwidth}
\begin{tikzpicture}[spy using outlines=
{rectangle,white,magnification=5.,size=2.115cm, connect spies}]
\node[anchor=south west,inner sep=0]  at (0,0) {\includegraphics[width=\linewidth]{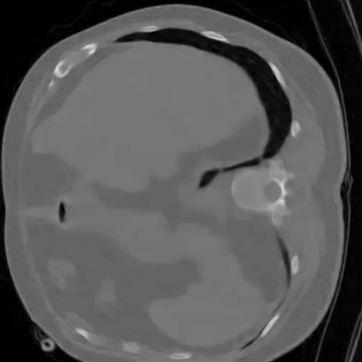}};
\spy on (1.53,.99) in node [right] at (-0.005,-1.1);
\end{tikzpicture}
%  \caption*{\footnotesize CRR}
\end{subfigure}%
\hfill
\begin{subfigure}[t]{.123\textwidth}
\begin{tikzpicture}[spy using outlines=
{rectangle,white,magnification=5.,size=2.115cm, connect spies}]
\node[anchor=south west,inner sep=0]  at (0,0) {\includegraphics[width=\linewidth]{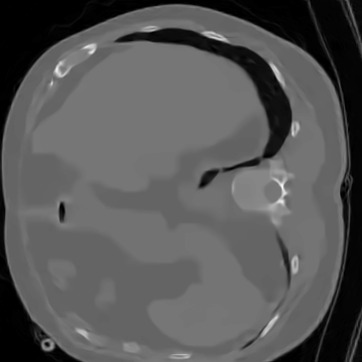}};
\spy on (1.53,.99) in node [right] at (-0.005,-1.1);
\end{tikzpicture}
% \caption*{\footnotesize WCRR}
\end{subfigure}%
\hfill
\begin{subfigure}[t]{.123\textwidth}
\begin{tikzpicture}[spy using outlines=
{rectangle,white,magnification=5.,size=2.115cm, connect spies}]
\node[anchor=south west,inner sep=0]  at (0,0) {\includegraphics[width=\linewidth]{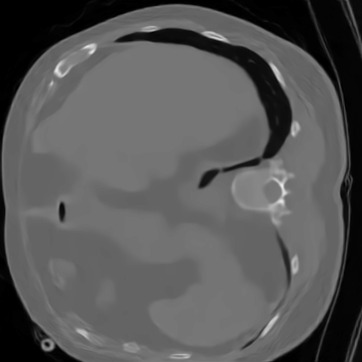}};
\spy on (1.53,.99) in node [right] at (-0.005,-1.1);
\end{tikzpicture}
%  \caption*{\footnotesize Mask-CRR}
\end{subfigure}%
\hfill
\begin{subfigure}[t]{.123\textwidth}
\begin{tikzpicture}[spy using outlines=
{rectangle,white,magnification=5.,size=2.115cm, connect spies}]
\node[anchor=south west,inner sep=0]  at (0,0) {\includegraphics[width=\linewidth]{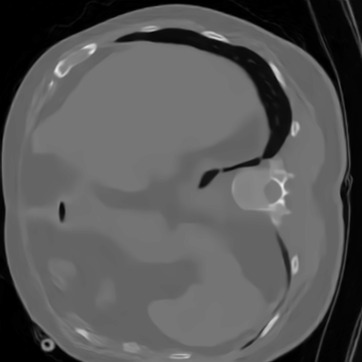}};
\spy on (1.53,.99) in node [right] at (-0.005,-1.1);
\end{tikzpicture}
% \caption*{\footnotesize Mask-WCRR}
\end{subfigure}%
\hfill
\begin{subfigure}[t]{.123\textwidth}
\begin{tikzpicture}[spy using outlines=
{rectangle,white,magnification=5.,size=2.115cm, connect spies}]
\node[anchor=south west,inner sep=0]  at (0,0) {\includegraphics[width=\linewidth]{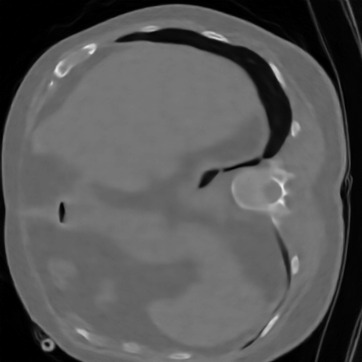}};
\spy on (1.53,.99) in node [right] at (-0.005,-1.1);
\end{tikzpicture}
% \caption*{\footnotesize Prox-DRUNet}
\end{subfigure}%
\hfill
\begin{subfigure}[t]{.123\textwidth}
\begin{tikzpicture}[spy using outlines=
{rectangle,white,magnification=5.,size=2.115cm, connect spies}]
\node[anchor=south west,inner sep=0]  at (0,0) {\includegraphics[width=\linewidth]{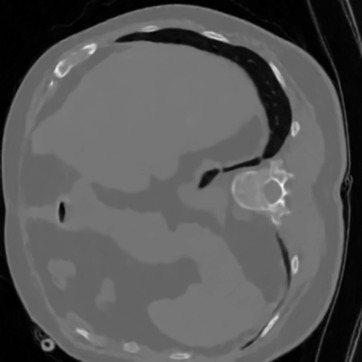}};
\spy on (1.53,.99) in node [right] at (-0.005,-1.1);
\end{tikzpicture}
% \caption*{\footnotesize patchNR}
\end{subfigure}%

\begin{subfigure}[t]{.123\textwidth}
\includegraphics[width=\linewidth]{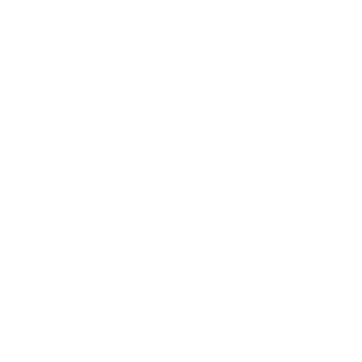}
  \caption*{\footnotesize GT}
\end{subfigure}%
\hfill
\begin{subfigure}[t]{.123\textwidth}
\includegraphics[width=\linewidth]{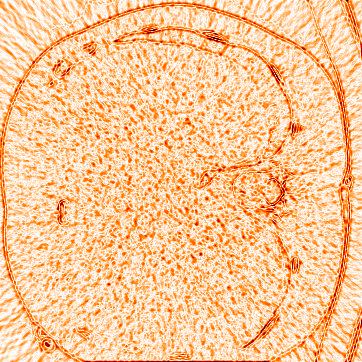}
  \caption*{\footnotesize FBP}
\end{subfigure}%
\hfill
\begin{subfigure}[t]{.123\textwidth}
\includegraphics[width=\linewidth]{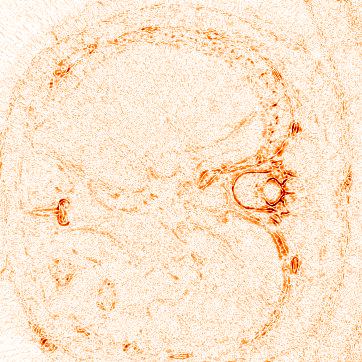}
  \caption*{\footnotesize CRR}
\end{subfigure}%
\hfill
\begin{subfigure}[t]{.123\textwidth}
\includegraphics[width=\linewidth]{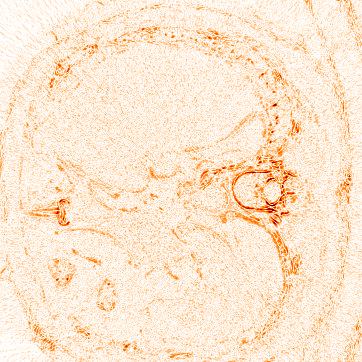}
  \caption*{\footnotesize WCRR}
\end{subfigure}%
\hfill
\begin{subfigure}[t]{.123\textwidth}
\includegraphics[width=\linewidth]{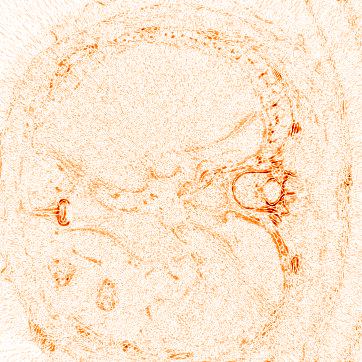}
  \caption*{\footnotesize Mask-CRR}
\end{subfigure}%
\hfill
\begin{subfigure}[t]{.123\textwidth}
\includegraphics[width=\linewidth]{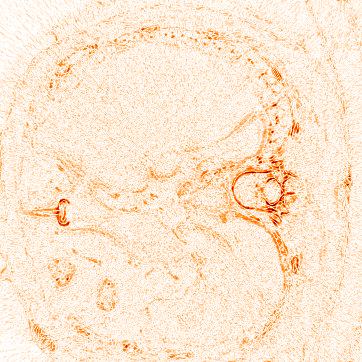}
  \caption*{\footnotesize Mask-WCRR}
\end{subfigure}%
\hfill
\begin{subfigure}[t]{.123\textwidth}
\includegraphics[width=\linewidth]{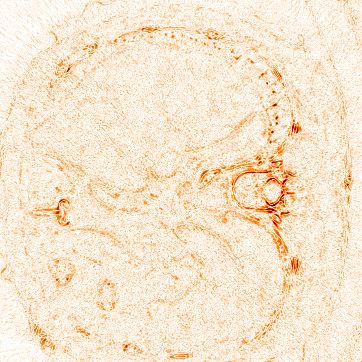}
  \caption*{\footnotesize Prox-DRUNet}
\end{subfigure}%
\hfill
\begin{subfigure}[t]{.123\textwidth}
\includegraphics[width=\linewidth]{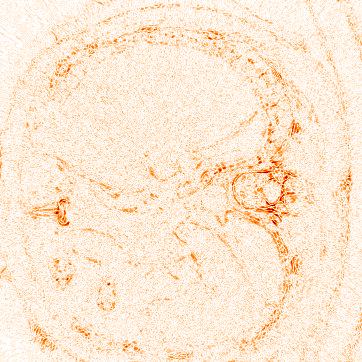}
  \caption*{\footnotesize patchNR}
\end{subfigure}%

\begin{subfigure}[t]{.123\textwidth}  
\begin{tikzpicture}[spy using outlines=
{rectangle,white,magnification=6.,size=2.115cm, connect spies}]
\node[anchor=south west,inner sep=0]  at (0,0) {\includegraphics[width=\linewidth]{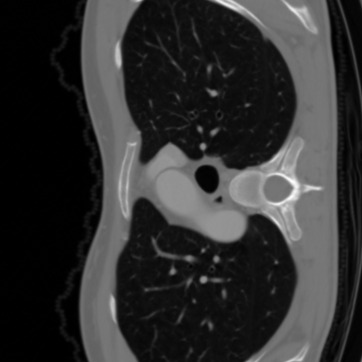}};
\spy on (1.68,1.15) in node [right] at (-0.005,-1.1);
\end{tikzpicture}
%\caption*{\footnotesize GT}
\end{subfigure}%
\hfill
\begin{subfigure}[t]{.123\textwidth}
\begin{tikzpicture}[spy using outlines=
{rectangle,white,magnification=6.,size=2.115cm, connect spies}]
\node[anchor=south west,inner sep=0]  at (0,0) {\includegraphics[width=\linewidth]{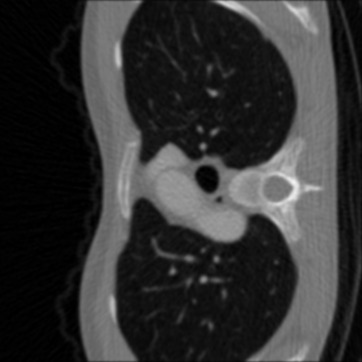}};
\spy on (1.68,1.15) in node [right] at (-0.005,-1.1);
\end{tikzpicture}
% \caption*{\footnotesize FBP}
\end{subfigure}%
\hfill
\begin{subfigure}[t]{.123\textwidth}
\begin{tikzpicture}[spy using outlines=
{rectangle,white,magnification=6.,size=2.115cm, connect spies}]
\node[anchor=south west,inner sep=0]  at (0,0) {\includegraphics[width=\linewidth]{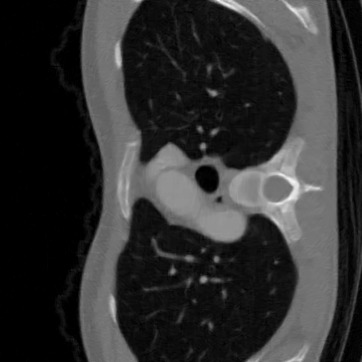}};
\spy on (1.68,1.15) in node [right] at (-0.005,-1.1);
\end{tikzpicture}
%  \caption*{\footnotesize CRR}
\end{subfigure}%
\hfill
\begin{subfigure}[t]{.123\textwidth}
\begin{tikzpicture}[spy using outlines=
{rectangle,white,magnification=6.,size=2.115cm, connect spies}]
\node[anchor=south west,inner sep=0]  at (0,0) {\includegraphics[width=\linewidth]{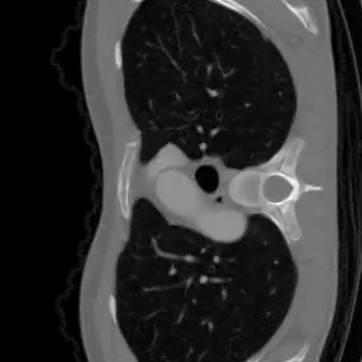}};
\spy on (1.68,1.15) in node [right] at (-0.005,-1.1);
\end{tikzpicture}
 %\caption*{\footnotesize WCRR}
\end{subfigure}%
\hfill
\begin{subfigure}[t]{.123\textwidth}
\begin{tikzpicture}[spy using outlines=
{rectangle,white,magnification=6.,size=2.115cm, connect spies}]
\node[anchor=south west,inner sep=0]  at (0,0) {\includegraphics[width=\linewidth]{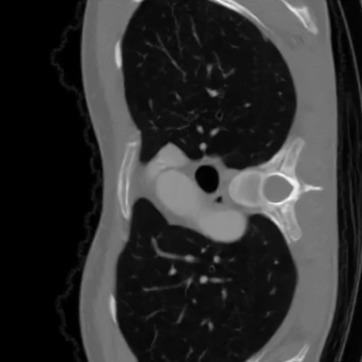}};
\spy on (1.68,1.15) in node [right] at (-0.005,-1.1);
\end{tikzpicture}
%  \caption*{\footnotesize Mask-CRR}
\end{subfigure}%
\hfill
\begin{subfigure}[t]{.123\textwidth}
\begin{tikzpicture}[spy using outlines=
{rectangle,white,magnification=6.,size=2.115cm, connect spies}]
\node[anchor=south west,inner sep=0]  at (0,0) {\includegraphics[width=\linewidth]{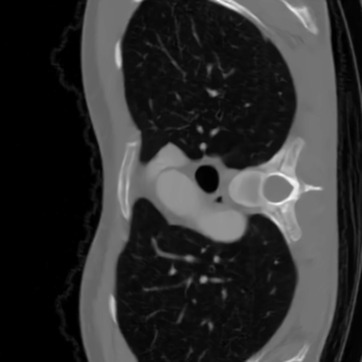}};
\spy on (1.68,1.15) in node [right] at (-0.005,-1.1);
\end{tikzpicture}
% \caption*{\footnotesize Mask-WCRR}
\end{subfigure}%
\hfill
\begin{subfigure}[t]{.123\textwidth}
\begin{tikzpicture}[spy using outlines=
{rectangle,white,magnification=6.,size=2.115cm, connect spies}]
\node[anchor=south west,inner sep=0]  at (0,0) {\includegraphics[width=\linewidth]{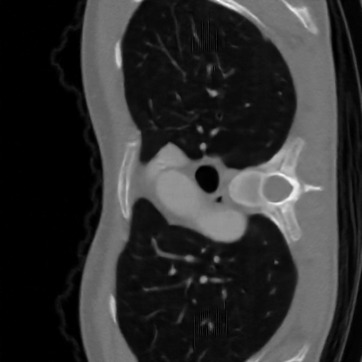}};
\spy on (1.68,1.15) in node [right] at (-0.005,-1.1);
\end{tikzpicture}
% \caption*{\footnotesize Prox-DRUNet}
\end{subfigure}%
\hfill
\begin{subfigure}[t]{.123\textwidth}
\begin{tikzpicture}[spy using outlines=
{rectangle,white,magnification=6.,size=2.115cm, connect spies}]
\node[anchor=south west,inner sep=0]  at (0,0) {\includegraphics[width=\linewidth]{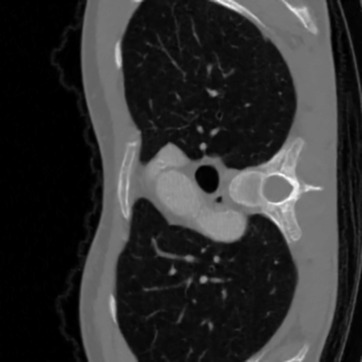}};
\spy on (1.68,1.15) in node [right] at (-0.005,-1.1);
\end{tikzpicture}
% \caption*{\footnotesize patchNR}
\end{subfigure}%

\begin{subfigure}[t]{.123\textwidth}
\includegraphics[width=\linewidth]{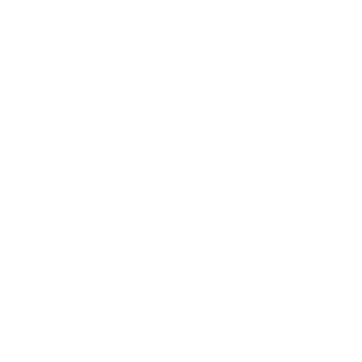}
  \caption*{\footnotesize GT}
\end{subfigure}%
\hfill
\begin{subfigure}[t]{.123\textwidth}
\includegraphics[width=\linewidth]{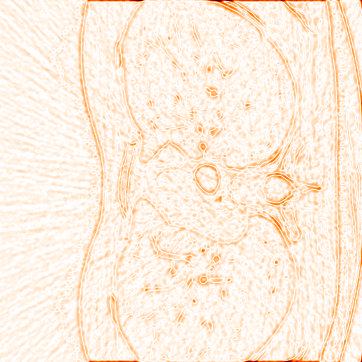}
  \caption*{\footnotesize FBP}
\end{subfigure}%
\hfill
\begin{subfigure}[t]{.123\textwidth}
\includegraphics[width=\linewidth]{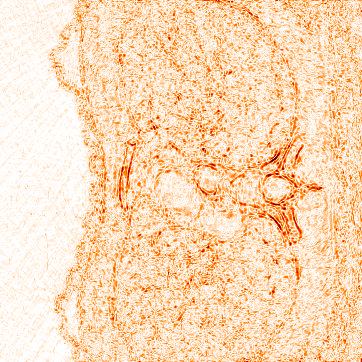}
  \caption*{\footnotesize CRR}
\end{subfigure}%
\hfill
\begin{subfigure}[t]{.123\textwidth}
\includegraphics[width=\linewidth]{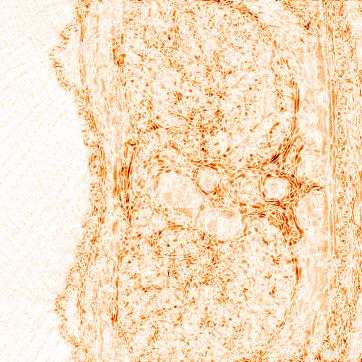}
  \caption*{\footnotesize WCRR}
\end{subfigure}%
\hfill
\begin{subfigure}[t]{.123\textwidth}
\includegraphics[width=\linewidth]{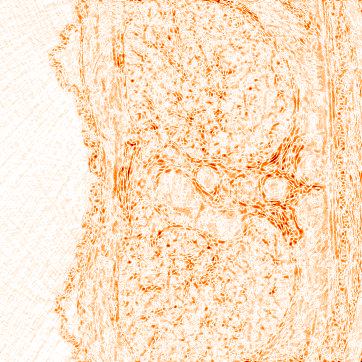}
  \caption*{\footnotesize Mask-CRR}
\end{subfigure}%
\hfill
\begin{subfigure}[t]{.123\textwidth}
\includegraphics[width=\linewidth]{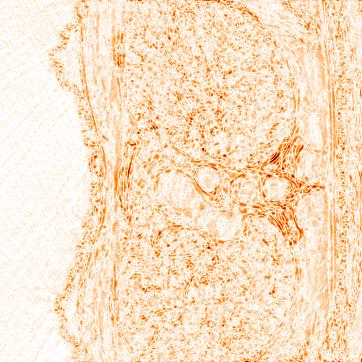}
  \caption*{\footnotesize Mask-WCRR}
\end{subfigure}%
\hfill
\begin{subfigure}[t]{.123\textwidth}
\includegraphics[width=\linewidth]{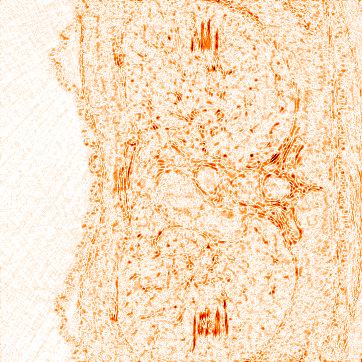}
  \caption*{\footnotesize Prox-DRUNet}
\end{subfigure}%
\hfill
\begin{subfigure}[t]{.123\textwidth}
\includegraphics[width=\linewidth]{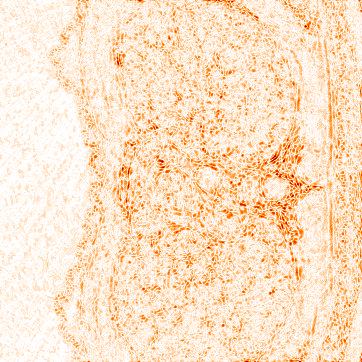}
  \caption*{\footnotesize patchNR}
\end{subfigure}%

\caption{Full view CT.
The white box marks the zoomed area.
\textit{Top}: full image; \textit{middle}: zoomed-in part; \textit{bottom}: error.
} \label{fig:CT_comparison_appendix}
\end{figure}

%---------------------------------------------------------------------------------------

\begin{figure}[t]
\centering

\centering
\begin{subfigure}[t]{.123\textwidth}  
\begin{tikzpicture}[spy using outlines=
{rectangle,white,magnification=5.5,size=2.115cm, connect spies}]
\node[anchor=south west,inner sep=0]  at (0,0) {\includegraphics[width=\linewidth]{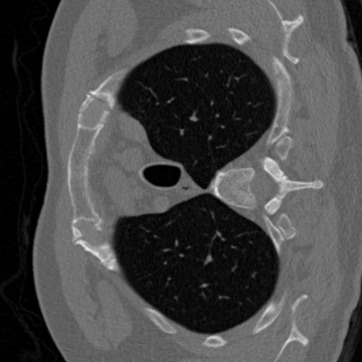}};
\spy on (1.,1.02) in node [right] at (-0.005,-1.1);
\end{tikzpicture}
%\caption*{\footnotesize GT}
\end{subfigure}%
\hfill
\begin{subfigure}[t]{.123\textwidth}
\begin{tikzpicture}[spy using outlines=
{rectangle,white,magnification=5.5,size=2.115cm, connect spies}]
\node[anchor=south west,inner sep=0]  at (0,0) {\includegraphics[width=\linewidth]{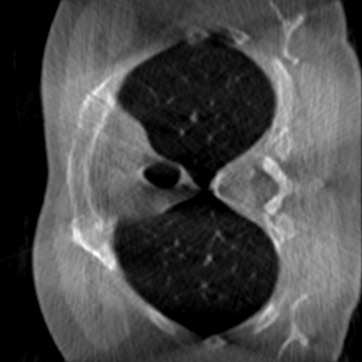}};
\spy on (1.,1.02) in node [right] at (-0.005,-1.1);
\end{tikzpicture}
%  \caption*{\footnotesize FBP}
\end{subfigure}%
\hfill
\begin{subfigure}[t]{.123\textwidth}
\begin{tikzpicture}[spy using outlines=
{rectangle,white,magnification=5.5,size=2.115cm, connect spies}]
\node[anchor=south west,inner sep=0]  at (0,0) {\includegraphics[width=\linewidth]{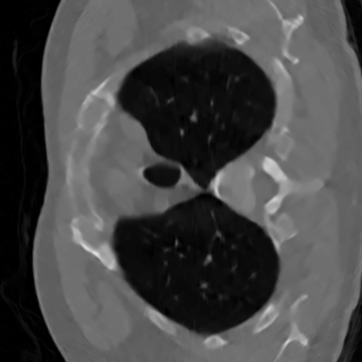}};
\spy on (1.,1.02) in node [right] at (-0.005,-1.1);
\end{tikzpicture}
%  \caption*{\footnotesize CRR}
\end{subfigure}%
\hfill
\begin{subfigure}[t]{.123\textwidth}
\begin{tikzpicture}[spy using outlines=
{rectangle,white,magnification=5.5,size=2.115cm, connect spies}]
\node[anchor=south west,inner sep=0]  at (0,0) {\includegraphics[width=\linewidth]{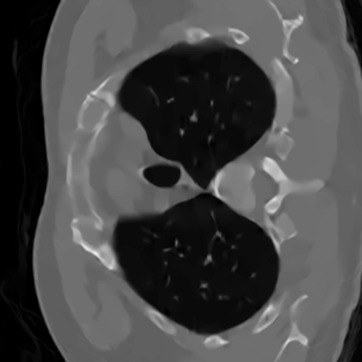}};
\spy on (1.,1.02) in node [right] at (-0.005,-1.1);
\end{tikzpicture}
%  \caption*{\footnotesize WCRR}
\end{subfigure}%
\hfill
\begin{subfigure}[t]{.123\textwidth}
\begin{tikzpicture}[spy using outlines=
{rectangle,white,magnification=5.5,size=2.115cm, connect spies}]
\node[anchor=south west,inner sep=0]  at (0,0) {\includegraphics[width=\linewidth]{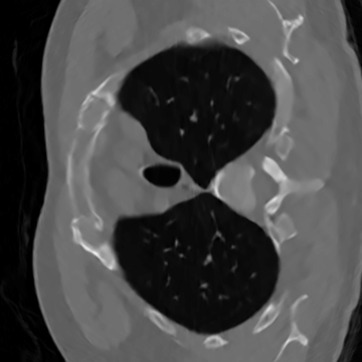}};
\spy on (1.,1.02) in node [right] at (-0.005,-1.1);
\end{tikzpicture}
%  \caption*{\footnotesize Mask-CRR}
\end{subfigure}%
\hfill
\begin{subfigure}[t]{.123\textwidth}
\begin{tikzpicture}[spy using outlines=
{rectangle,white,magnification=5.5,size=2.115cm, connect spies}]
\node[anchor=south west,inner sep=0]  at (0,0) {\includegraphics[width=\linewidth]{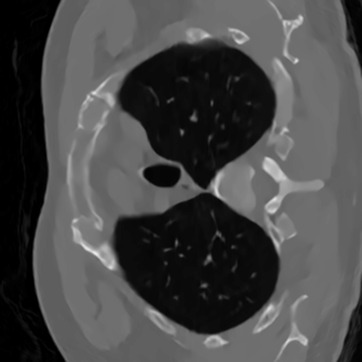}};
\spy on (1.,1.02) in node [right] at (-0.005,-1.1);
\end{tikzpicture}
%  \caption*{\footnotesize Mask-WCRR}
\end{subfigure}%
\hfill
\begin{subfigure}[t]{.123\textwidth}
\begin{tikzpicture}[spy using outlines=
{rectangle,white,magnification=5.5,size=2.115cm, connect spies}]
\node[anchor=south west,inner sep=0]  at (0,0) {\includegraphics[width=\linewidth]{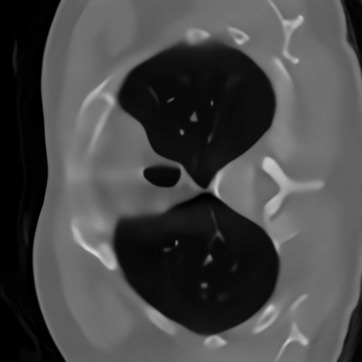}};
\spy on (1.,1.02) in node [right] at (-0.005,-1.1);
\end{tikzpicture}
%  \caption*{\footnotesize Prox-DRUNet}
\end{subfigure}%
\hfill
\begin{subfigure}[t]{.123\textwidth}
\begin{tikzpicture}[spy using outlines=
{rectangle,white,magnification=5.5,size=2.115cm, connect spies}]
\node[anchor=south west,inner sep=0]  at (0,0) {\includegraphics[width=\linewidth]{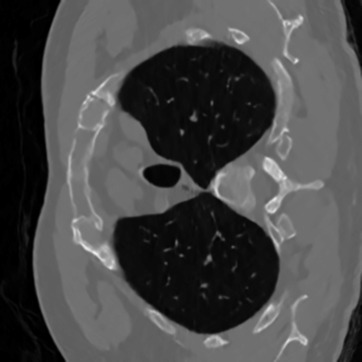}};
\spy on (1.,1.02) in node [right] at (-0.005,-1.1);
\end{tikzpicture}
%\caption*{\footnotesize patchNR}
\end{subfigure}%

%\begin{comment}
\begin{subfigure}[t]{.123\textwidth}
\includegraphics[width=\linewidth]{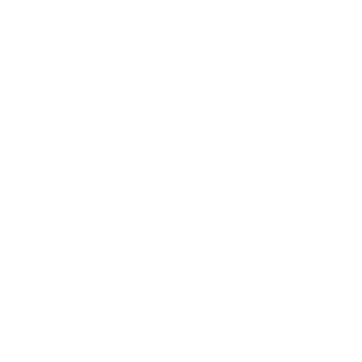}
  \caption*{\footnotesize GT}
\end{subfigure}%
\hfill
\begin{subfigure}[t]{.123\textwidth}
\includegraphics[width=\linewidth]{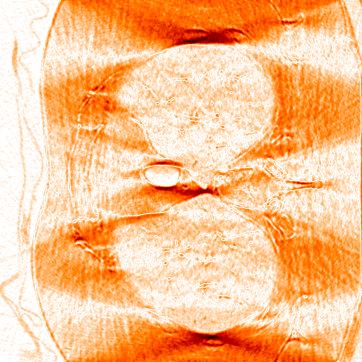}
  \caption*{\footnotesize FBP}
\end{subfigure}%
\hfill
\begin{subfigure}[t]{.123\textwidth}
\includegraphics[width=\linewidth]{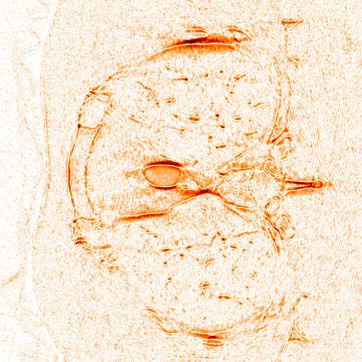}
  \caption*{\footnotesize CRR}
\end{subfigure}%
\hfill
\begin{subfigure}[t]{.123\textwidth}
\includegraphics[width=\linewidth]{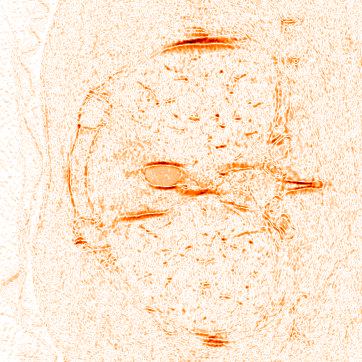}
  \caption*{\footnotesize WCRR}
\end{subfigure}%
\hfill
\begin{subfigure}[t]{.123\textwidth}
\includegraphics[width=\linewidth]{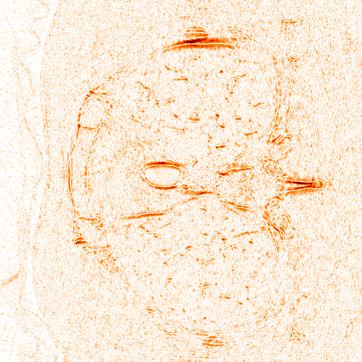}
  \caption*{\footnotesize Mask-CRR}
\end{subfigure}%
\hfill
\begin{subfigure}[t]{.123\textwidth}
\includegraphics[width=\linewidth]{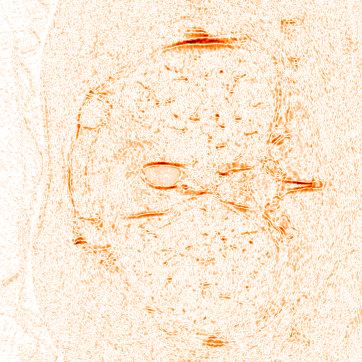}
  \caption*{\footnotesize Mask-WCRR}
\end{subfigure}%
\hfill
\begin{subfigure}[t]{.123\textwidth}
\includegraphics[width=\linewidth]{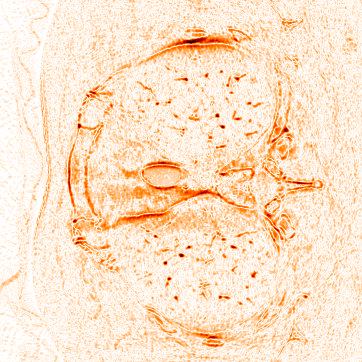}
  \caption*{\footnotesize Prox-DRUNet}
\end{subfigure}%
\hfill
\begin{subfigure}[t]{.123\textwidth}
\includegraphics[width=\linewidth]{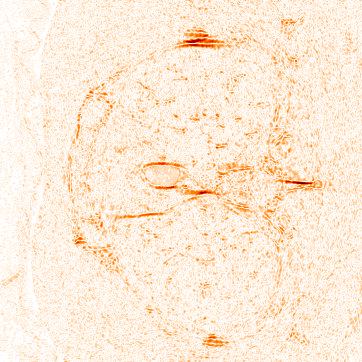}
  \caption*{\footnotesize patchNR}
\end{subfigure}%
%\end{comment}

%----------------------------------------------------------------------------------

\begin{subfigure}[t]{.123\textwidth}  
\begin{tikzpicture}[spy using outlines=
{rectangle,white,magnification=3.8,size=2.115cm, connect spies}]
\node[anchor=south west,inner sep=0]  at (0,0) {\includegraphics[width=\linewidth]{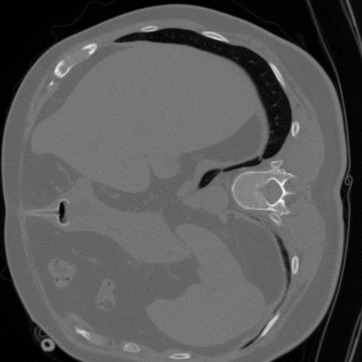}};
\spy on (1.45,1.04) in node [right] at (-0.005,-1.1);
\end{tikzpicture}
%\caption*{\footnotesize GT}
\end{subfigure}%
\hfill
\begin{subfigure}[t]{.123\textwidth}
\begin{tikzpicture}[spy using outlines=
{rectangle,white,magnification=3.8,size=2.115cm, connect spies}]
\node[anchor=south west,inner sep=0]  at (0,0) {\includegraphics[width=\linewidth]{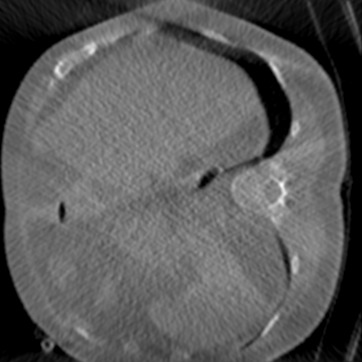}};
\spy on (1.45,1.04) in node [right] at (-0.005,-1.1);
\end{tikzpicture}
% \caption*{\footnotesize FBP}
\end{subfigure}%
\hfill
\begin{subfigure}[t]{.123\textwidth}
\begin{tikzpicture}[spy using outlines=
{rectangle,white,magnification=3.8,size=2.115cm, connect spies}]
\node[anchor=south west,inner sep=0]  at (0,0) {\includegraphics[width=\linewidth]{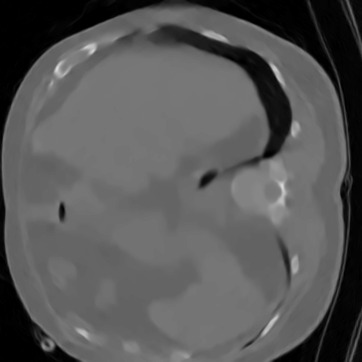}};
\spy on (1.45,1.04) in node [right] at (-0.005,-1.1);
\end{tikzpicture}
%  \caption*{\footnotesize CRR}
\end{subfigure}%
\hfill
\begin{subfigure}[t]{.123\textwidth}
\begin{tikzpicture}[spy using outlines=
{rectangle,white,magnification=3.8,size=2.115cm, connect spies}]
\node[anchor=south west,inner sep=0]  at (0,0) {\includegraphics[width=\linewidth]{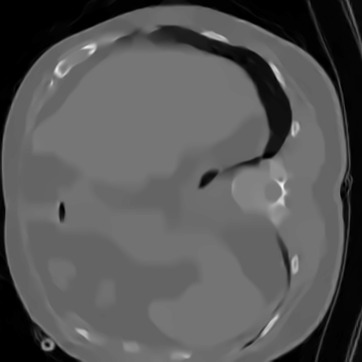}};
\spy on (1.45,1.04) in node [right] at (-0.005,-1.1);
\end{tikzpicture}
% \caption*{\footnotesize WCRR}
\end{subfigure}%
\hfill
\begin{subfigure}[t]{.123\textwidth}
\begin{tikzpicture}[spy using outlines=
{rectangle,white,magnification=3.8,size=2.115cm, connect spies}]
\node[anchor=south west,inner sep=0]  at (0,0) {\includegraphics[width=\linewidth]{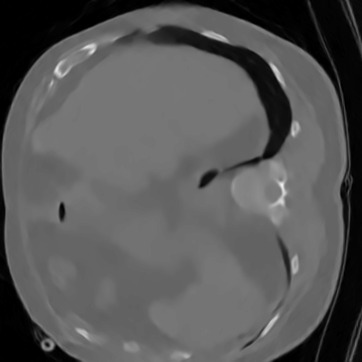}};
\spy on (1.45,1.04) in node [right] at (-0.005,-1.1);
\end{tikzpicture}
%  \caption*{\footnotesize Mask-CRR}
\end{subfigure}%
\hfill
\begin{subfigure}[t]{.123\textwidth}
\begin{tikzpicture}[spy using outlines=
{rectangle,white,magnification=3.8,size=2.115cm, connect spies}]
\node[anchor=south west,inner sep=0]  at (0,0) {\includegraphics[width=\linewidth]{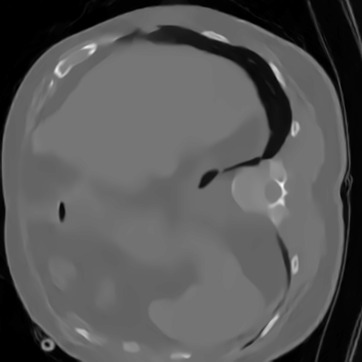}};
\spy on (1.45,1.04) in node [right] at (-0.005,-1.1);
\end{tikzpicture}
% \caption*{\footnotesize Mask-WCRR}
\end{subfigure}%
\hfill
\begin{subfigure}[t]{.123\textwidth}
\begin{tikzpicture}[spy using outlines=
{rectangle,white,magnification=3.8,size=2.115cm, connect spies}]
\node[anchor=south west,inner sep=0]  at (0,0) {\includegraphics[width=\linewidth]{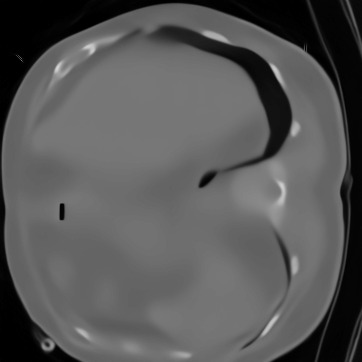}};
\spy on (1.45,1.04) in node [right] at (-0.005,-1.1);
\end{tikzpicture}
% \caption*{\footnotesize Prox-DRUNet}
\end{subfigure}%
\hfill
\begin{subfigure}[t]{.123\textwidth}
\begin{tikzpicture}[spy using outlines=
{rectangle,white,magnification=3.8,size=2.115cm, connect spies}]
\node[anchor=south west,inner sep=0]  at (0,0) {\includegraphics[width=\linewidth]{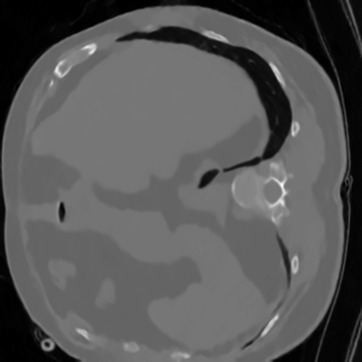}};
\spy on (1.45,1.04) in node [right] at (-0.005,-1.1);
\end{tikzpicture}
% \caption*{\footnotesize patchNR}
\end{subfigure}%

\begin{subfigure}[t]{.123\textwidth}
\includegraphics[width=\linewidth]{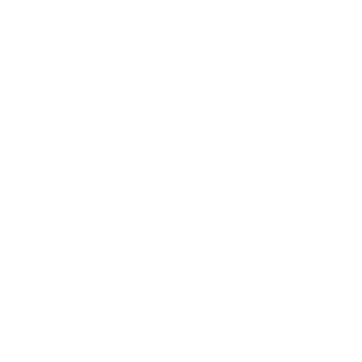}
  \caption*{\footnotesize GT}
\end{subfigure}%
\hfill
\begin{subfigure}[t]{.123\textwidth}
\includegraphics[width=\linewidth]{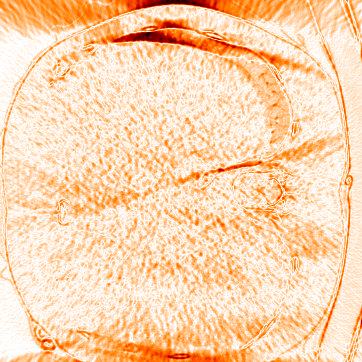}
  \caption*{\footnotesize FBP}
\end{subfigure}%
\hfill
\begin{subfigure}[t]{.123\textwidth}
\includegraphics[width=\linewidth]{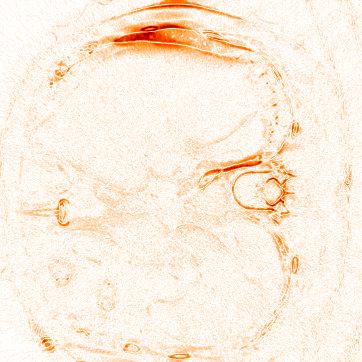}
  \caption*{\footnotesize CRR}
\end{subfigure}%
\hfill
\begin{subfigure}[t]{.123\textwidth}
\includegraphics[width=\linewidth]{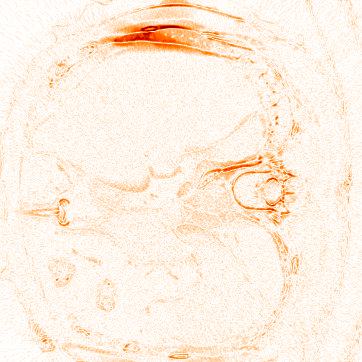}
  \caption*{\footnotesize WCRR}
\end{subfigure}%
\hfill
\begin{subfigure}[t]{.123\textwidth}
\includegraphics[width=\linewidth]{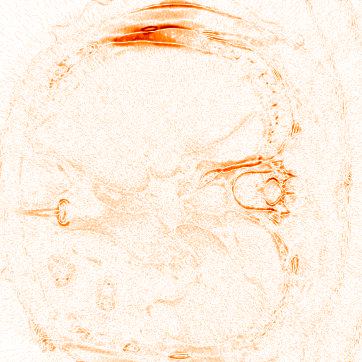}
  \caption*{\footnotesize Mask-CRR}
\end{subfigure}%
\hfill
\begin{subfigure}[t]{.123\textwidth}
\includegraphics[width=\linewidth]{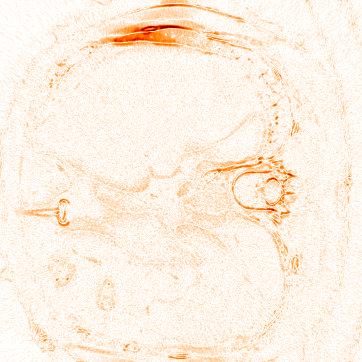}
  \caption*{\footnotesize Mask-WCRR}
\end{subfigure}%
\hfill
\begin{subfigure}[t]{.123\textwidth}
\includegraphics[width=\linewidth]{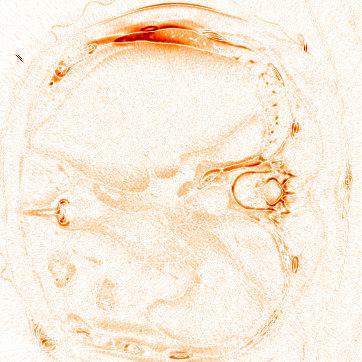}
  \caption*{\footnotesize Prox-DRUNet}
\end{subfigure}%
\hfill
\begin{subfigure}[t]{.123\textwidth}
\includegraphics[width=\linewidth]{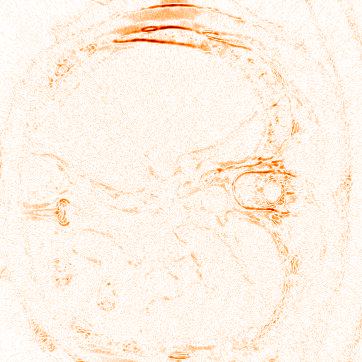}
  \caption*{\footnotesize patchNR}
\end{subfigure}%

%----------------------------------------------------------------------------------

\begin{subfigure}[t]{.123\textwidth}  
\begin{tikzpicture}[spy using outlines=
{rectangle,white,magnification=5.5,size=2.115cm, connect spies}]
\node[anchor=south west,inner sep=0]  at (0,0) {\includegraphics[width=\linewidth]{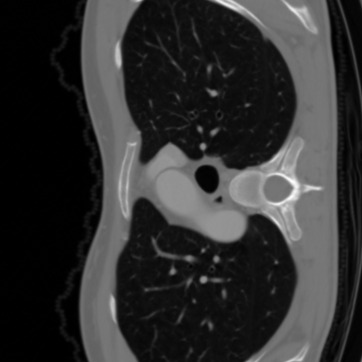}};
\spy on (1.7,1.0) in node [right] at (-0.005,-1.1);
\end{tikzpicture}
%\caption*{\footnotesize GT}
\end{subfigure}%
\hfill
\begin{subfigure}[t]{.123\textwidth}
\begin{tikzpicture}[spy using outlines=
{rectangle,white,magnification=5.5,size=2.115cm, connect spies}]
\node[anchor=south west,inner sep=0]  at (0,0) {\includegraphics[width=\linewidth]{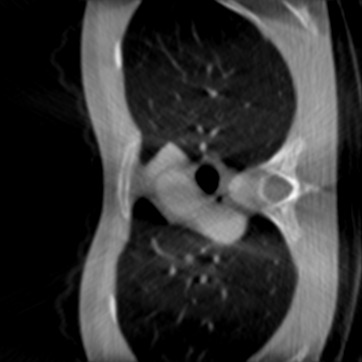}};
\spy on (1.7,1.0) in node [right] at (-0.005,-1.1);
\end{tikzpicture}
% \caption*{\footnotesize FBP}
\end{subfigure}%
\hfill
\begin{subfigure}[t]{.123\textwidth}
\begin{tikzpicture}[spy using outlines=
{rectangle,white,magnification=5.5,size=2.115cm, connect spies}]
\node[anchor=south west,inner sep=0]  at (0,0) {\includegraphics[width=\linewidth]{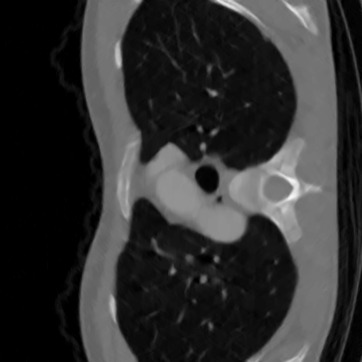}};
\spy on (1.7,1.0) in node [right] at (-0.005,-1.1);
\end{tikzpicture}
%  \caption*{\footnotesize CRR}
\end{subfigure}%
\hfill
\begin{subfigure}[t]{.123\textwidth}
\begin{tikzpicture}[spy using outlines=
{rectangle,white,magnification=5.5,size=2.115cm, connect spies}]
\node[anchor=south west,inner sep=0]  at (0,0) {\includegraphics[width=\linewidth]{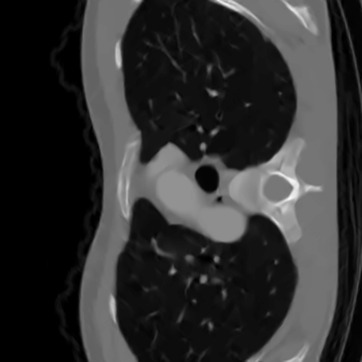}};
\spy on (1.7,1.0) in node [right] at (-0.005,-1.1);
\end{tikzpicture}
% \caption*{\footnotesize WCRR}
\end{subfigure}%
\hfill
\begin{subfigure}[t]{.123\textwidth}
\begin{tikzpicture}[spy using outlines=
{rectangle,white,magnification=5.5,size=2.115cm, connect spies}]
\node[anchor=south west,inner sep=0]  at (0,0) {\includegraphics[width=\linewidth]{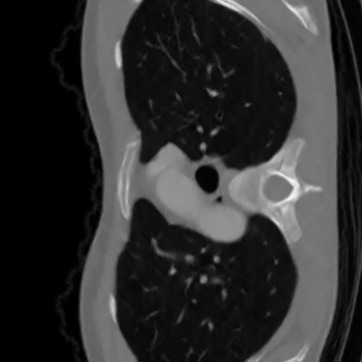}};
\spy on (1.7,1.0) in node [right] at (-0.005,-1.1);
\end{tikzpicture}
%  \caption*{\footnotesize Mask-CRR}
\end{subfigure}%
\hfill
\begin{subfigure}[t]{.123\textwidth}
\begin{tikzpicture}[spy using outlines=
{rectangle,white,magnification=5.5,size=2.115cm, connect spies}]
\node[anchor=south west,inner sep=0]  at (0,0) {\includegraphics[width=\linewidth]{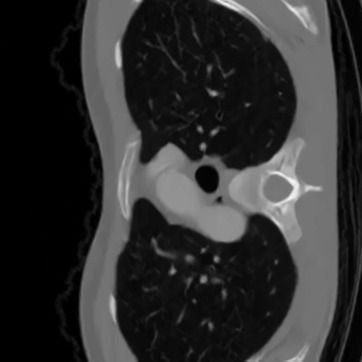}};
\spy on (1.7,1.0) in node [right] at (-0.005,-1.1);
\end{tikzpicture}
% \caption*{\footnotesize Mask-WCRR}
\end{subfigure}%
\hfill
\begin{subfigure}[t]{.123\textwidth}
\begin{tikzpicture}[spy using outlines=
{rectangle,white,magnification=5.5,size=2.115cm, connect spies}]
\node[anchor=south west,inner sep=0]  at (0,0) {\includegraphics[width=\linewidth]{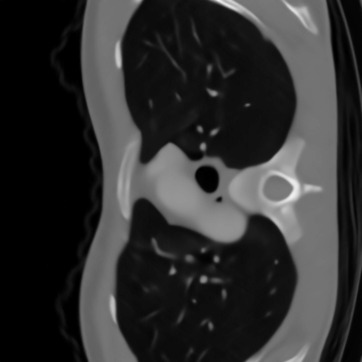}};
\spy on (1.7,1.0) in node [right] at (-0.005,-1.1);
\end{tikzpicture}
% \caption*{\footnotesize Prox-DRUNet}
\end{subfigure}%
\hfill
\begin{subfigure}[t]{.123\textwidth}
\begin{tikzpicture}[spy using outlines=
{rectangle,white,magnification=5.5,size=2.115cm, connect spies}]
\node[anchor=south west,inner sep=0]  at (0,0) {\includegraphics[width=\linewidth]{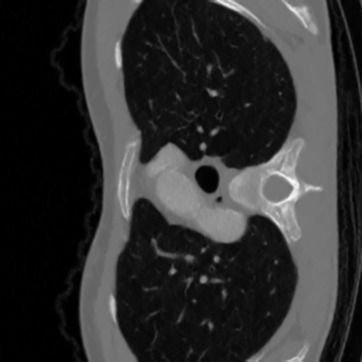}};
\spy on (1.7,1.0) in node [right] at (-0.005,-1.1);
\end{tikzpicture}
% \caption*{\footnotesize patchNR}
\end{subfigure}%

\begin{subfigure}[t]{.123\textwidth}
\includegraphics[width=\linewidth]{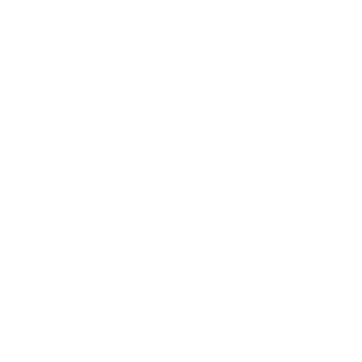}
  \caption*{\footnotesize GT}
\end{subfigure}%
\hfill
\begin{subfigure}[t]{.123\textwidth}
\includegraphics[width=\linewidth]{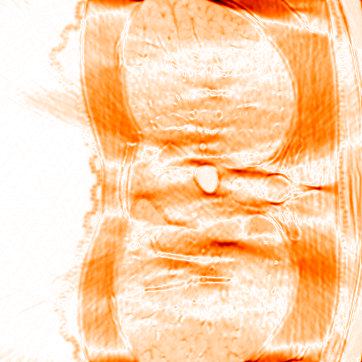}
  \caption*{\footnotesize FBP}
\end{subfigure}%
\hfill
\begin{subfigure}[t]{.123\textwidth}
\includegraphics[width=\linewidth]{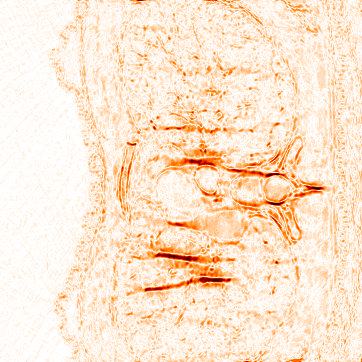}
  \caption*{\footnotesize CRR}
\end{subfigure}%
\hfill
\begin{subfigure}[t]{.123\textwidth}
\includegraphics[width=\linewidth]{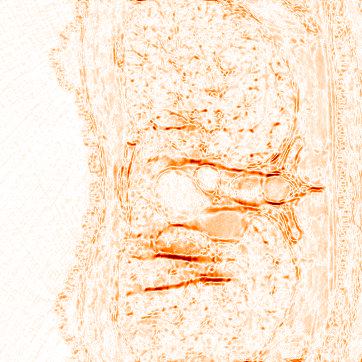}
  \caption*{\footnotesize WCRR}
\end{subfigure}%
\hfill
\begin{subfigure}[t]{.123\textwidth}
\includegraphics[width=\linewidth]{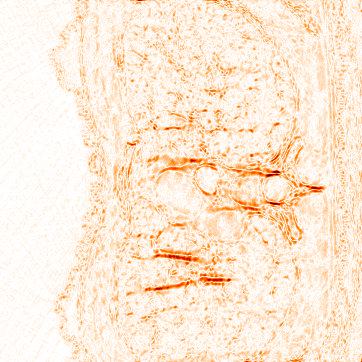}
  \caption*{\footnotesize Mask-CRR}
\end{subfigure}%
\hfill
\begin{subfigure}[t]{.123\textwidth}
\includegraphics[width=\linewidth]{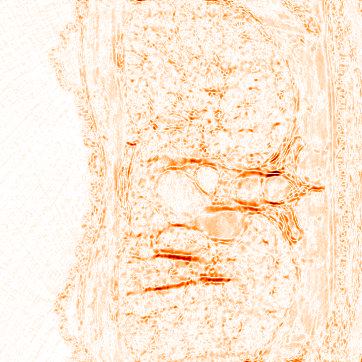}
  \caption*{\footnotesize Mask-WCRR}
\end{subfigure}%
\hfill
\begin{subfigure}[t]{.123\textwidth}
\includegraphics[width=\linewidth]{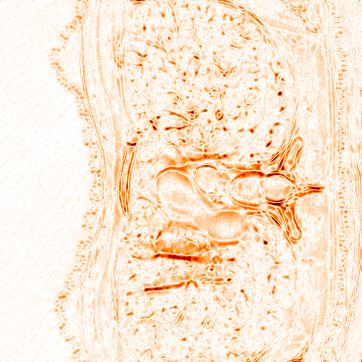}
  \caption*{\footnotesize Prox-DRUNet}
\end{subfigure}%
\hfill
\begin{subfigure}[t]{.123\textwidth}
\includegraphics[width=\linewidth]{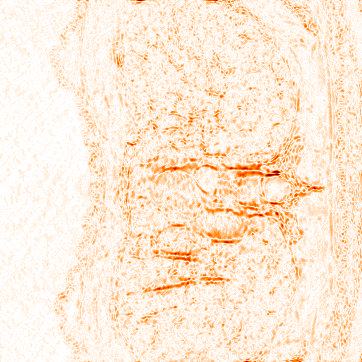}
  \caption*{\footnotesize patchNR}
\end{subfigure}%

\caption{Limited-angle CT.
The white box marks the zoomed area.
\textit{Top}: full image; \textit{middle}: zoomed-in part; \textit{bottom}: error.
} \label{fig:CT_comparison_limited_appendix}
\end{figure}

%----------------------------------------------------------------------------------------

\begin{figure}[t]
\centering
\begin{subfigure}[t]{.123\textwidth}  
\begin{tikzpicture}[spy using outlines=
{rectangle,white,magnification=4.5,size=2.115cm, connect spies}]
\node[anchor=south west,inner sep=0]  at (0,0) {\includegraphics[width=\linewidth]{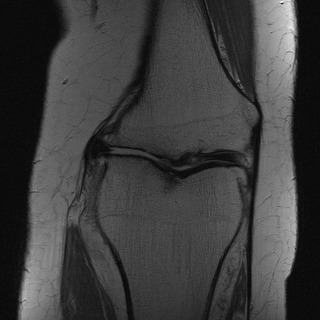}};
\spy on (1.6,1.08) in node [right] at (-0.005,-1.1);
\end{tikzpicture}
%\caption*{\footnotesize GT}
\end{subfigure}%
\hfill
\begin{subfigure}[t]{.123\textwidth}
\begin{tikzpicture}[spy using outlines=
{rectangle,white,magnification=4.5,size=2.115cm, connect spies}]
\node[anchor=south west,inner sep=0]  at (0,0) {\includegraphics[width=\linewidth]{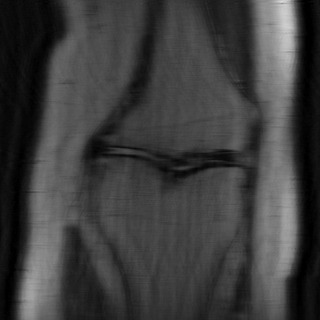}};
\spy on (1.6,1.08) in node [right] at (-0.005,-1.1);
\end{tikzpicture}
%  \caption*{\footnotesize FBP}
\end{subfigure}%
\hfill
\begin{subfigure}[t]{.123\textwidth}
\begin{tikzpicture}[spy using outlines=
{rectangle,white,magnification=4.5,size=2.115cm, connect spies}]
\node[anchor=south west,inner sep=0]  at (0,0) {\includegraphics[width=\linewidth]{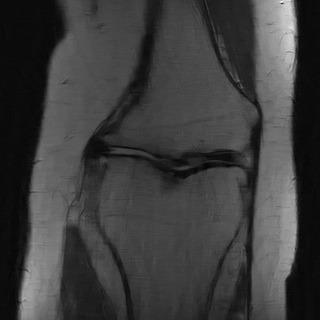}};
\spy on (1.6,1.08) in node [right] at (-0.005,-1.1);
\end{tikzpicture}
%  \caption*{\footnotesize CRR}
\end{subfigure}%
\hfill
\begin{subfigure}[t]{.123\textwidth}
\begin{tikzpicture}[spy using outlines=
{rectangle,white,magnification=4.5,size=2.115cm, connect spies}]
\node[anchor=south west,inner sep=0]  at (0,0) {\includegraphics[width=\linewidth]{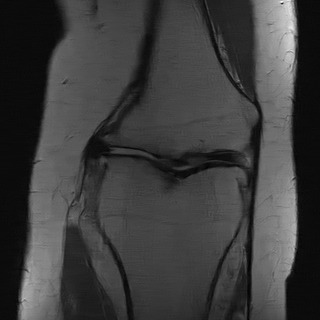}};
\spy on (1.6,1.08) in node [right] at (-0.005,-1.1);
\end{tikzpicture}
%  \caption*{\footnotesize WCRR}
\end{subfigure}%
\hfill
\begin{subfigure}[t]{.123\textwidth}
\begin{tikzpicture}[spy using outlines=
{rectangle,white,magnification=4.5,size=2.115cm, connect spies}]
\node[anchor=south west,inner sep=0]  at (0,0) {\includegraphics[width=\linewidth]{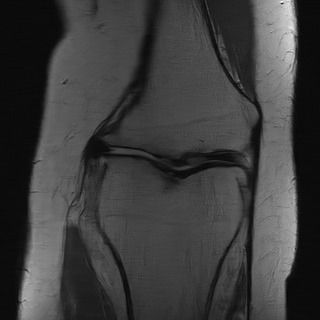}};
\spy on (1.6,1.08) in node [right] at (-0.005,-1.1);
\end{tikzpicture}
%  \caption*{\footnotesize Mask-CRR}
\end{subfigure}%
\hfill
\begin{subfigure}[t]{.123\textwidth}
\begin{tikzpicture}[spy using outlines=
{rectangle,white,magnification=4.5,size=2.115cm, connect spies}]
\node[anchor=south west,inner sep=0]  at (0,0) {\includegraphics[width=\linewidth]{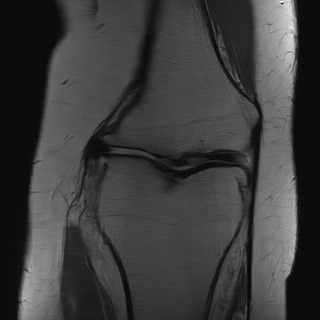}};
\spy on (1.6,1.08) in node [right] at (-0.005,-1.1);
\end{tikzpicture}
%  \caption*{\footnotesize Mask-WCRR}
\end{subfigure}%
\hfill
\begin{subfigure}[t]{.123\textwidth}
\begin{tikzpicture}[spy using outlines=
{rectangle,white,magnification=4.5,size=2.115cm, connect spies}]
\node[anchor=south west,inner sep=0]  at (0,0) {\includegraphics[width=\linewidth]{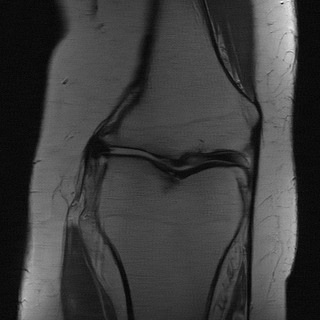}};
\spy on (1.6,1.08) in node [right] at (-0.005,-1.1);
\end{tikzpicture}
%  \caption*{\footnotesize Prox-DRUNet}
\end{subfigure}%
\hfill
\begin{subfigure}[t]{.123\textwidth}
\begin{tikzpicture}[spy using outlines=
{rectangle,white,magnification=4.5,size=2.115cm, connect spies}]
\node[anchor=south west,inner sep=0]  at (0,0) {\includegraphics[width=\linewidth]{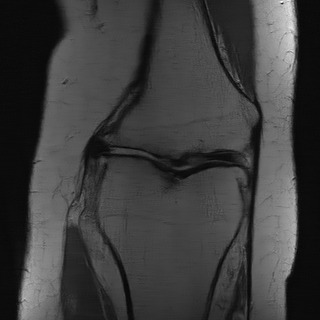}};
\spy on (1.6,1.08) in node [right] at (-0.005,-1.1);
\end{tikzpicture}
%  \caption*{\footnotesize patchNR}
\end{subfigure}%

\begin{subfigure}[t]{.123\textwidth}
\includegraphics[width=\linewidth]{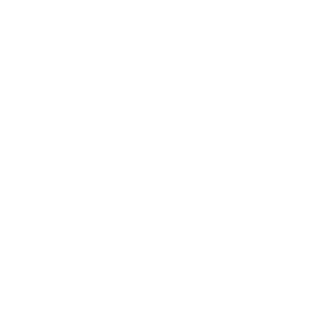}
  \caption*{\footnotesize GT}
\end{subfigure}%
\hfill
\begin{subfigure}[t]{.123\textwidth}
\includegraphics[width=\linewidth]{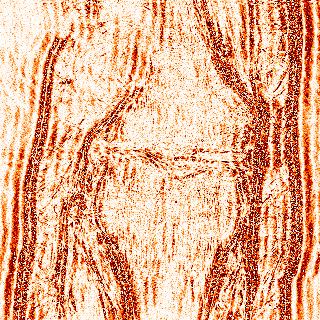}
  \caption*{\footnotesize Zero-filled}
\end{subfigure}%
\hfill
\begin{subfigure}[t]{.123\textwidth}
\includegraphics[width=\linewidth]{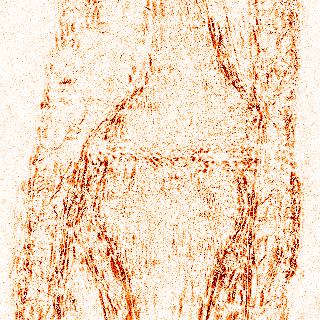}
  \caption*{\footnotesize CRR}
\end{subfigure}%
\hfill
\begin{subfigure}[t]{.123\textwidth}
\includegraphics[width=\linewidth]{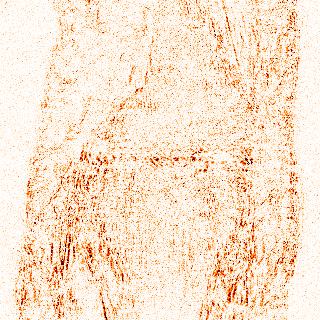}
  \caption*{\footnotesize WCRR}
\end{subfigure}%
\hfill
\begin{subfigure}[t]{.123\textwidth}
\includegraphics[width=\linewidth]{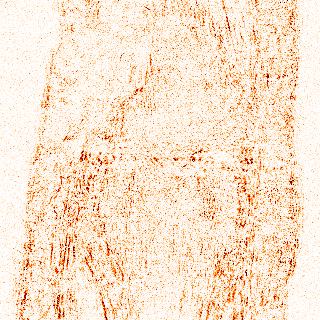}
  \caption*{\footnotesize CRR-Mask}
\end{subfigure}%
\hfill
\begin{subfigure}[t]{.123\textwidth}
\includegraphics[width=\linewidth]{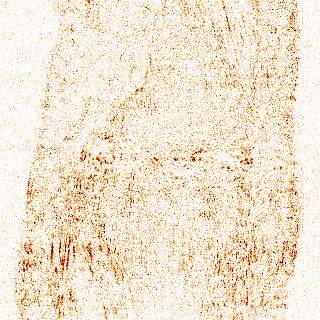}
  \caption*{\footnotesize WCRR-Mask}
\end{subfigure}%
\hfill
\begin{subfigure}[t]{.123\textwidth}
\includegraphics[width=\linewidth]{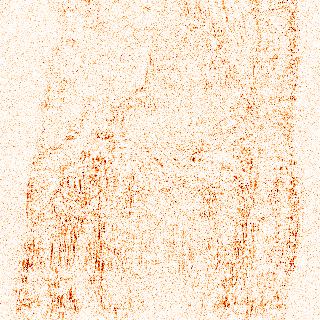}
  \caption*{\footnotesize Prox-DRUNet}
\end{subfigure}%
\hfill
\begin{subfigure}[t]{.123\textwidth}
\includegraphics[width=\linewidth]{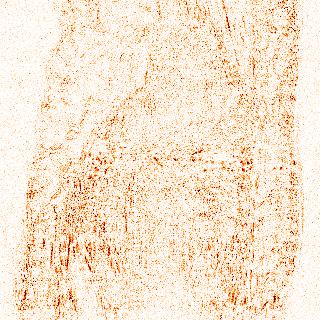}
  \caption*{\footnotesize patchNR}
\end{subfigure}%

%-----------------------------------------------------------------------------

\begin{subfigure}[t]{.123\textwidth}  
\begin{tikzpicture}[spy using outlines=
{rectangle,white,magnification=4.5,size=2.115cm, connect spies}]
\node[anchor=south west,inner sep=0]  at (0,0) {\includegraphics[width=\linewidth]{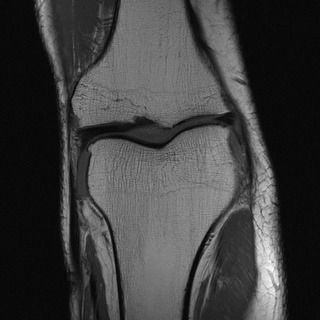}};
\spy on (.65,.7) in node [right] at (-0.005,-1.1);
\end{tikzpicture}
%\caption*{\footnotesize GT}
\end{subfigure}%
\hfill
\begin{subfigure}[t]{.123\textwidth}
\begin{tikzpicture}[spy using outlines=
{rectangle,white,magnification=4.5,size=2.115cm, connect spies}]
\node[anchor=south west,inner sep=0]  at (0,0) {\includegraphics[width=\linewidth]{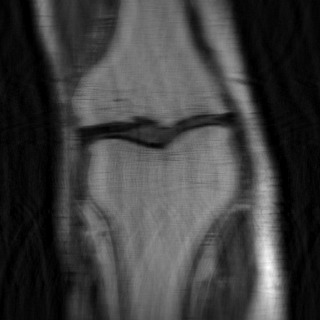}};
\spy on (.65,.7) in node [right] at (-0.005,-1.1);
\end{tikzpicture}
%  \caption*{\footnotesize FBP}
\end{subfigure}%
\hfill
\begin{subfigure}[t]{.123\textwidth}
\begin{tikzpicture}[spy using outlines=
{rectangle,white,magnification=4.5,size=2.115cm, connect spies}]
\node[anchor=south west,inner sep=0]  at (0,0) {\includegraphics[width=\linewidth]{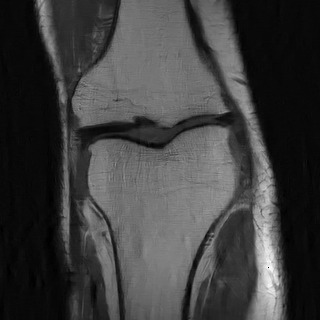}};
\spy on (.65,.7) in node [right] at (-0.005,-1.1);
\end{tikzpicture}
%  \caption*{\footnotesize CRR}
\end{subfigure}%
\hfill
\begin{subfigure}[t]{.123\textwidth}
\begin{tikzpicture}[spy using outlines=
{rectangle,white,magnification=4.5,size=2.115cm, connect spies}]
\node[anchor=south west,inner sep=0]  at (0,0) {\includegraphics[width=\linewidth]{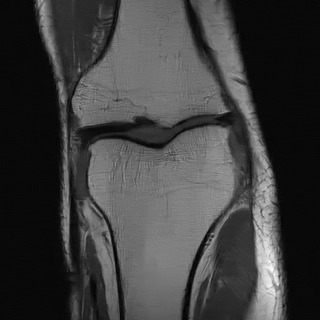}};
\spy on (.65,.7) in node [right] at (-0.005,-1.1);
\end{tikzpicture}
%  \caption*{\footnotesize WCRR}
\end{subfigure}%
\hfill
\begin{subfigure}[t]{.123\textwidth}
\begin{tikzpicture}[spy using outlines=
{rectangle,white,magnification=4.5,size=2.115cm, connect spies}]
\node[anchor=south west,inner sep=0]  at (0,0) {\includegraphics[width=\linewidth]{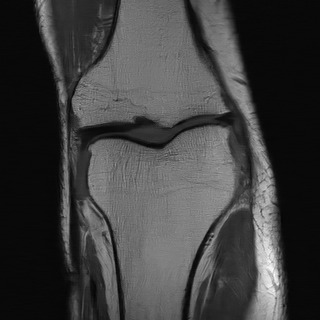}};
\spy on (.65,.7) in node [right] at (-0.005,-1.1);
\end{tikzpicture}
%  \caption*{\footnotesize Mask-CRR}
\end{subfigure}%
\hfill
\begin{subfigure}[t]{.123\textwidth}
\begin{tikzpicture}[spy using outlines=
{rectangle,white,magnification=4.5,size=2.115cm, connect spies}]
\node[anchor=south west,inner sep=0]  at (0,0) {\includegraphics[width=\linewidth]{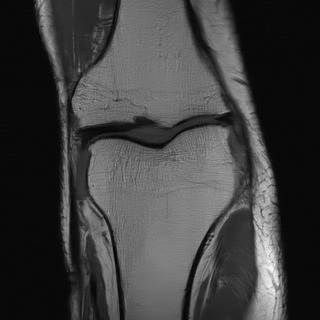}};
\spy on (.65,.7) in node [right] at (-0.005,-1.1);
\end{tikzpicture}
%  \caption*{\footnotesize Mask-WCRR}
\end{subfigure}%
\hfill
\begin{subfigure}[t]{.123\textwidth}
\begin{tikzpicture}[spy using outlines=
{rectangle,white,magnification=4.5,size=2.115cm, connect spies}]
\node[anchor=south west,inner sep=0]  at (0,0) {\includegraphics[width=\linewidth]{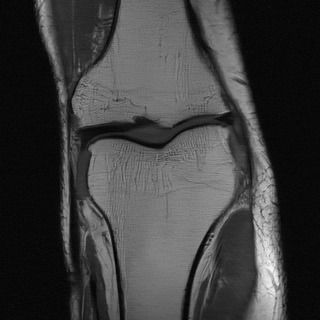}};
\spy on (.65,.7) in node [right] at (-0.005,-1.1);
\end{tikzpicture}
%  \caption*{\footnotesize Prox-DRUNet}
\end{subfigure}%
\hfill
\begin{subfigure}[t]{.123\textwidth}
\begin{tikzpicture}[spy using outlines=
{rectangle,white,magnification=4.5,size=2.115cm, connect spies}]
\node[anchor=south west,inner sep=0]  at (0,0) {\includegraphics[width=\linewidth]{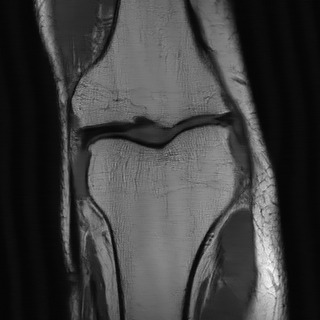}};
\spy on (.65,.7) in node [right] at (-0.005,-1.1);
\end{tikzpicture}
%  \caption*{\footnotesize patchNR}
\end{subfigure}%

\begin{subfigure}[t]{.123\textwidth}
\includegraphics[width=\linewidth]{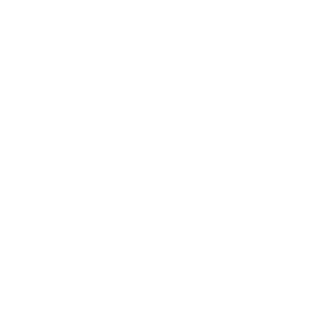}
  \caption*{\footnotesize GT}
\end{subfigure}%
\hfill
\begin{subfigure}[t]{.123\textwidth}
\includegraphics[width=\linewidth]{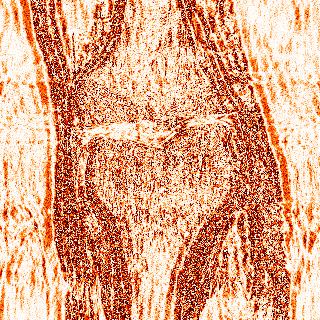}
  \caption*{\footnotesize Zero-filled}
\end{subfigure}%
\hfill
\begin{subfigure}[t]{.123\textwidth}
\includegraphics[width=\linewidth]{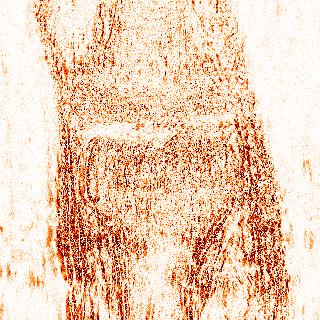}
  \caption*{\footnotesize CRR}
\end{subfigure}%
\hfill
\begin{subfigure}[t]{.123\textwidth}
\includegraphics[width=\linewidth]{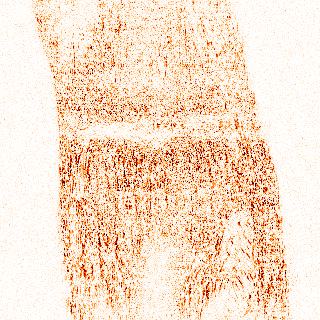}
  \caption*{\footnotesize WCRR}
\end{subfigure}%
\hfill
\begin{subfigure}[t]{.123\textwidth}
\includegraphics[width=\linewidth]{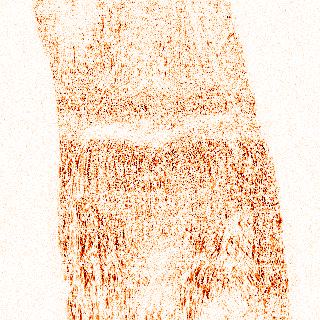}
  \caption*{\footnotesize CRR-Mask}
\end{subfigure}%
\hfill
\begin{subfigure}[t]{.123\textwidth}
\includegraphics[width=\linewidth]{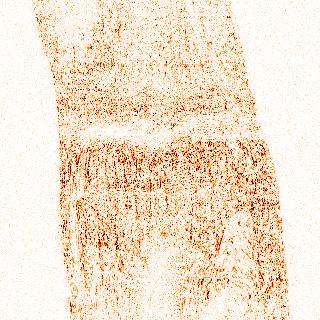}
  \caption*{\footnotesize WCRR-Mask}
\end{subfigure}%
\hfill
\begin{subfigure}[t]{.123\textwidth}
\includegraphics[width=\linewidth]{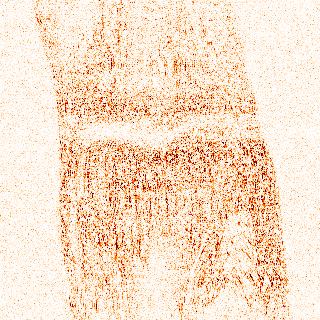}
  \caption*{\footnotesize Prox-DRUNet}
\end{subfigure}%
\hfill
\begin{subfigure}[t]{.123\textwidth}
\includegraphics[width=\linewidth]{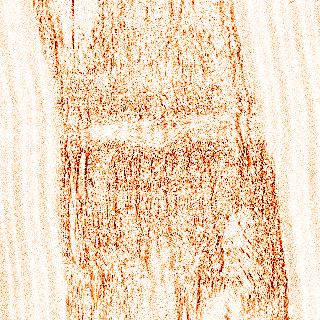}
  \caption*{\footnotesize patchNR}
\end{subfigure}%

%-----------------------------------------------------------------------------

\begin{subfigure}[t]{.123\textwidth}  
\begin{tikzpicture}[spy using outlines=
{rectangle,white,magnification=4.5,size=2.115cm, connect spies}]
\node[anchor=south west,inner sep=0]  at (0,0) {\includegraphics[width=\linewidth]{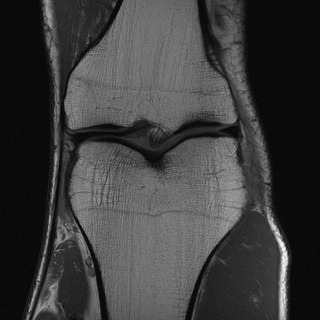}};
\spy on (1.55,1.08) in node [right] at (-0.005,-1.1);
\end{tikzpicture}
%\caption*{\footnotesize GT}
\end{subfigure}%
\hfill
\begin{subfigure}[t]{.123\textwidth}
\begin{tikzpicture}[spy using outlines=
{rectangle,white,magnification=4.5,size=2.115cm, connect spies}]
\node[anchor=south west,inner sep=0]  at (0,0) {\includegraphics[width=\linewidth]{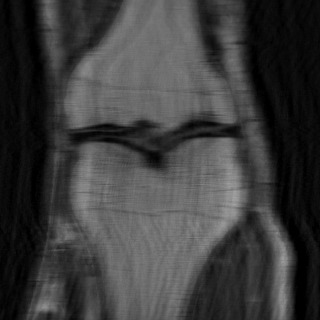}};
\spy on (1.55,1.08) in node [right] at (-0.005,-1.1);
\end{tikzpicture}
%  \caption*{\footnotesize FBP}
\end{subfigure}%
\hfill
\begin{subfigure}[t]{.123\textwidth}
\begin{tikzpicture}[spy using outlines=
{rectangle,white,magnification=4.5,size=2.115cm, connect spies}]
\node[anchor=south west,inner sep=0]  at (0,0) {\includegraphics[width=\linewidth]{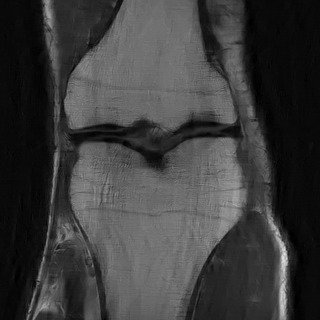}};
\spy on (1.55,1.08) in node [right] at (-0.005,-1.1);
\end{tikzpicture}
%  \caption*{\footnotesize CRR}
\end{subfigure}%
\hfill
\begin{subfigure}[t]{.123\textwidth}
\begin{tikzpicture}[spy using outlines=
{rectangle,white,magnification=4.5,size=2.115cm, connect spies}]
\node[anchor=south west,inner sep=0]  at (0,0) {\includegraphics[width=\linewidth]{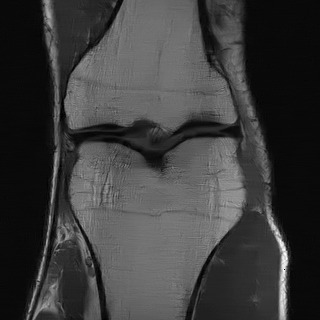}};
\spy on (1.55,1.08) in node [right] at (-0.005,-1.1);
\end{tikzpicture}
%  \caption*{\footnotesize WCRR}
\end{subfigure}%
\hfill
\begin{subfigure}[t]{.123\textwidth}
\begin{tikzpicture}[spy using outlines=
{rectangle,white,magnification=4.5,size=2.115cm, connect spies}]
\node[anchor=south west,inner sep=0]  at (0,0) {\includegraphics[width=\linewidth]{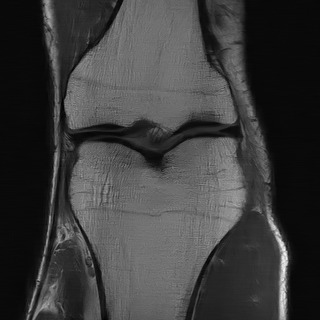}};
\spy on (1.55,1.08) in node [right] at (-0.005,-1.1);
\end{tikzpicture}
%  \caption*{\footnotesize Mask-CRR}
\end{subfigure}%
\hfill
\begin{subfigure}[t]{.123\textwidth}
\begin{tikzpicture}[spy using outlines=
{rectangle,white,magnification=4.5,size=2.115cm, connect spies}]
\node[anchor=south west,inner sep=0]  at (0,0) {\includegraphics[width=\linewidth]{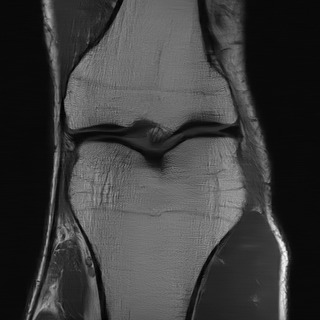}};
\spy on (1.55,1.08) in node [right] at (-0.005,-1.1);
\end{tikzpicture}
%  \caption*{\footnotesize Mask-WCRR}
\end{subfigure}%
\hfill
\begin{subfigure}[t]{.123\textwidth}
\begin{tikzpicture}[spy using outlines=
{rectangle,white,magnification=4.5,size=2.115cm, connect spies}]
\node[anchor=south west,inner sep=0]  at (0,0) {\includegraphics[width=\linewidth]{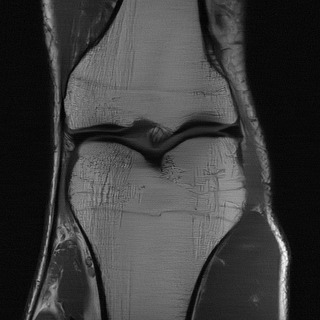}};
\spy on (1.55,1.08) in node [right] at (-0.005,-1.1);
\end{tikzpicture}
%  \caption*{\footnotesize Prox-DRUNet}
\end{subfigure}%
\hfill
\begin{subfigure}[t]{.123\textwidth}
\begin{tikzpicture}[spy using outlines=
{rectangle,white,magnification=4.5,size=2.115cm, connect spies}]
\node[anchor=south west,inner sep=0]  at (0,0) {\includegraphics[width=\linewidth]{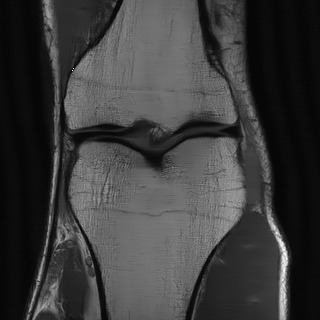}};
\spy on (1.55,1.08) in node [right] at (-0.005,-1.1);
\end{tikzpicture}
%  \caption*{\footnotesize patchNR}
\end{subfigure}%

\begin{subfigure}[t]{.123\textwidth}
\includegraphics[width=\linewidth]{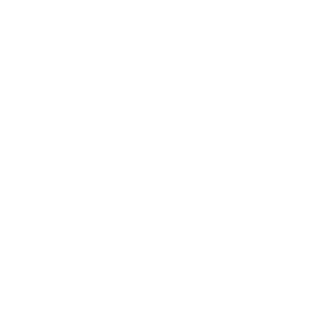}
  \caption*{\footnotesize GT}
\end{subfigure}%
\hfill
\begin{subfigure}[t]{.123\textwidth}
\includegraphics[width=\linewidth]{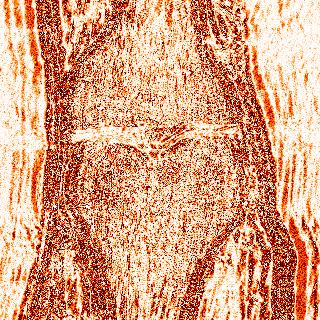}
  \caption*{\footnotesize Zero-filled}
\end{subfigure}%
\hfill
\begin{subfigure}[t]{.123\textwidth}
\includegraphics[width=\linewidth]{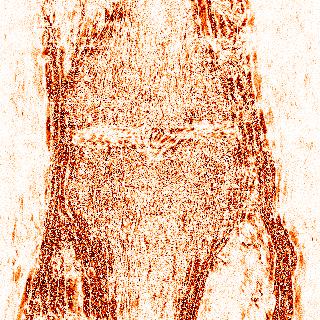}
  \caption*{\footnotesize CRR}
\end{subfigure}%
\hfill
\begin{subfigure}[t]{.123\textwidth}
\includegraphics[width=\linewidth]{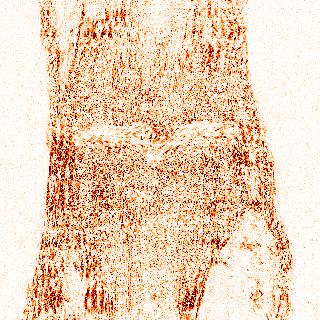}
  \caption*{\footnotesize WCRR}
\end{subfigure}%
\hfill
\begin{subfigure}[t]{.123\textwidth}
\includegraphics[width=\linewidth]{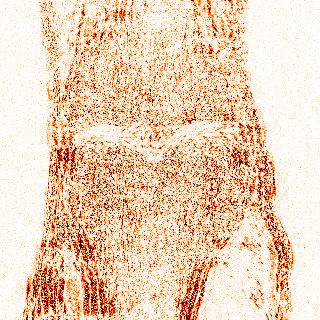}
  \caption*{\footnotesize CRR-Mask}
\end{subfigure}%
\hfill
\begin{subfigure}[t]{.123\textwidth}
\includegraphics[width=\linewidth]{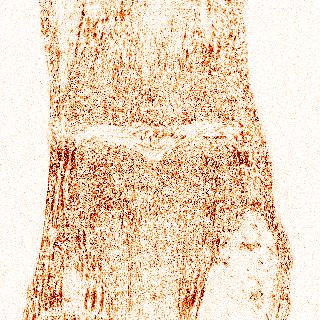}
  \caption*{\footnotesize WCRR-Mask}
\end{subfigure}%
\hfill
\begin{subfigure}[t]{.123\textwidth}
\includegraphics[width=\linewidth]{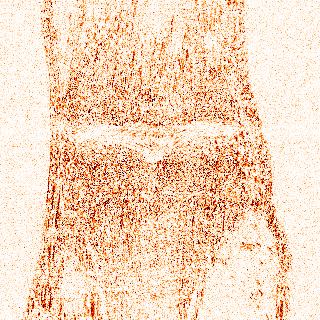}
  \caption*{\footnotesize Prox-DRUNet}
\end{subfigure}%
\hfill
\begin{subfigure}[t]{.123\textwidth}
\includegraphics[width=\linewidth]{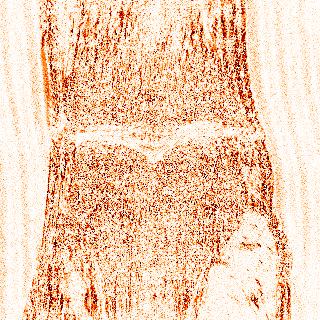}
  \caption*{\footnotesize patchNR}
\end{subfigure}%

\caption{4-fold single-coil MRI on PD data set.
The white box marks the zoomed area.
\textit{Top}: full image; \textit{middle}: zoomed-in part; \textit{bottom}: error.
} \label{fig:MRI_comparison_pd4_appendix}
\end{figure}

%-----------------------------------------------------------------------------------

\begin{figure}[t]
\centering
\begin{subfigure}[t]{.123\textwidth}  
\begin{tikzpicture}[spy using outlines=
{rectangle,white,magnification=4.5,size=2.115cm, connect spies}]
\node[anchor=south west,inner sep=0]  at (0,0) {\includegraphics[width=\linewidth]{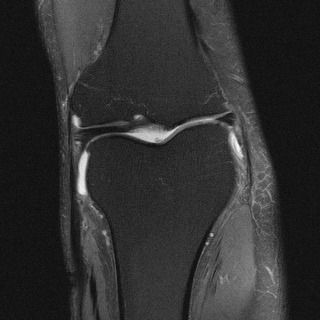}};
\spy on (.6,1.08) in node [right] at (-0.005,-1.1);
\end{tikzpicture}
%\caption*{\footnotesize GT}
\end{subfigure}%
\hfill
\begin{subfigure}[t]{.123\textwidth}
\begin{tikzpicture}[spy using outlines=
{rectangle,white,magnification=4.5,size=2.115cm, connect spies}]
\node[anchor=south west,inner sep=0]  at (0,0) {\includegraphics[width=\linewidth]{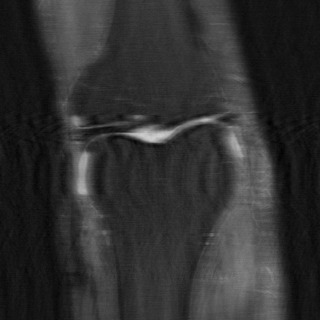}};
\spy on (.6,1.08) in node [right] at (-0.005,-1.1);
\end{tikzpicture}
%  \caption*{\footnotesize FBP}
\end{subfigure}%
\hfill
\begin{subfigure}[t]{.123\textwidth}
\begin{tikzpicture}[spy using outlines=
{rectangle,white,magnification=4.5,size=2.115cm, connect spies}]
\node[anchor=south west,inner sep=0]  at (0,0) {\includegraphics[width=\linewidth]{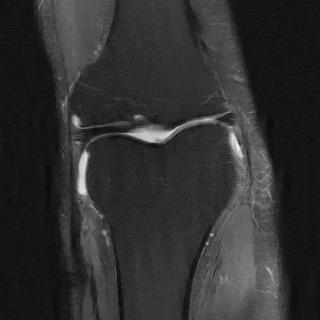}};
\spy on (.6,1.08) in node [right] at (-0.005,-1.1);
\end{tikzpicture}
%  \caption*{\footnotesize CRR}
\end{subfigure}%
\hfill
\begin{subfigure}[t]{.123\textwidth}
\begin{tikzpicture}[spy using outlines=
{rectangle,white,magnification=4.5,size=2.115cm, connect spies}]
\node[anchor=south west,inner sep=0]  at (0,0) {\includegraphics[width=\linewidth]{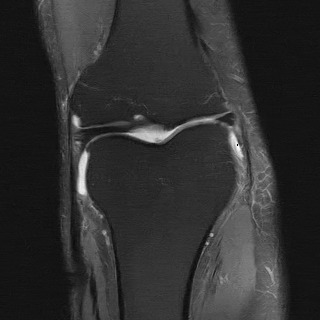}};
\spy on (.6,1.08) in node [right] at (-0.005,-1.1);
\end{tikzpicture}
%  \caption*{\footnotesize WCRR}
\end{subfigure}%
\hfill
\begin{subfigure}[t]{.123\textwidth}
\begin{tikzpicture}[spy using outlines=
{rectangle,white,magnification=4.5,size=2.115cm, connect spies}]
\node[anchor=south west,inner sep=0]  at (0,0) {\includegraphics[width=\linewidth]{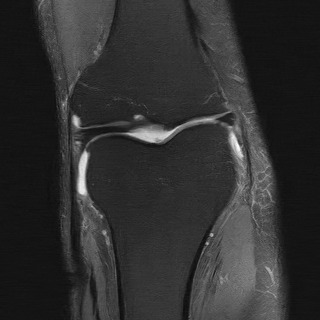}};
\spy on (.6,1.08) in node [right] at (-0.005,-1.1);
\end{tikzpicture}
%  \caption*{\footnotesize Mask-CRR}
\end{subfigure}%
\hfill
\begin{subfigure}[t]{.123\textwidth}
\begin{tikzpicture}[spy using outlines=
{rectangle,white,magnification=4.5,size=2.115cm, connect spies}]
\node[anchor=south west,inner sep=0]  at (0,0) {\includegraphics[width=\linewidth]{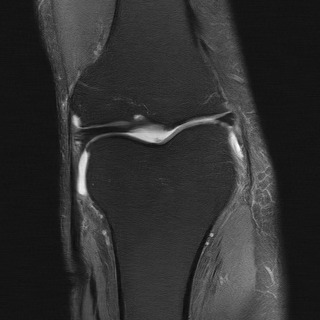}};
\spy on (.6,1.08) in node [right] at (-0.005,-1.1);
\end{tikzpicture}
%  \caption*{\footnotesize Mask-WCRR}
\end{subfigure}%
\hfill
\begin{subfigure}[t]{.123\textwidth}
\begin{tikzpicture}[spy using outlines=
{rectangle,white,magnification=4.5,size=2.115cm, connect spies}]
\node[anchor=south west,inner sep=0]  at (0,0) {\includegraphics[width=\linewidth]{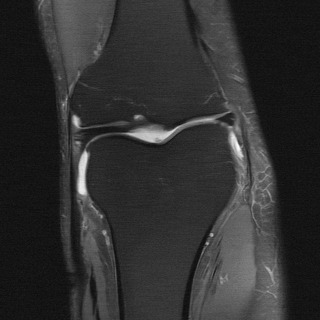}};
\spy on (.6,1.08) in node [right] at (-0.005,-1.1);
\end{tikzpicture}
%  \caption*{\footnotesize Prox-DRUNet}
\end{subfigure}%
\hfill
\begin{subfigure}[t]{.123\textwidth}
\begin{tikzpicture}[spy using outlines=
{rectangle,white,magnification=4.5,size=2.115cm, connect spies}]
\node[anchor=south west,inner sep=0]  at (0,0) {\includegraphics[width=\linewidth]{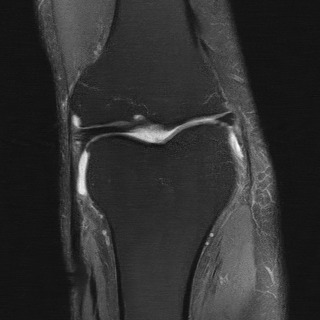}};
\spy on (.6,1.08) in node [right] at (-0.005,-1.1);
\end{tikzpicture}
%  \caption*{\footnotesize patchNR}
\end{subfigure}%

\begin{subfigure}[t]{.123\textwidth}
\includegraphics[width=\linewidth]{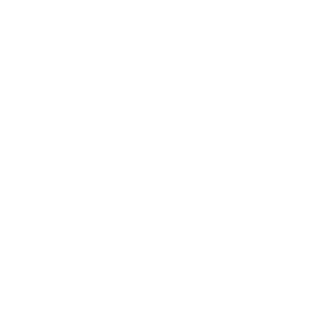}
  \caption*{\footnotesize GT}
\end{subfigure}%
\hfill
\begin{subfigure}[t]{.123\textwidth}
\includegraphics[width=\linewidth]{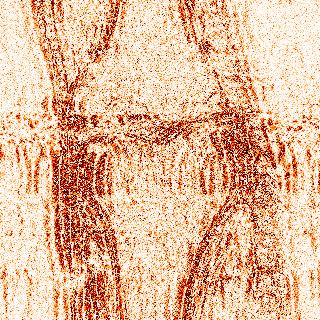}
  \caption*{\footnotesize Zero-filled}
\end{subfigure}%
\hfill
\begin{subfigure}[t]{.123\textwidth}
\includegraphics[width=\linewidth]{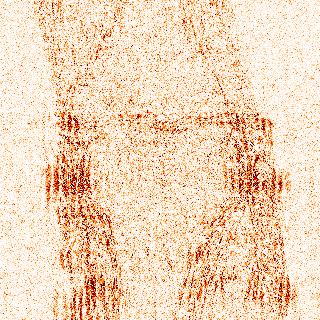}
  \caption*{\footnotesize CRR}
\end{subfigure}%
\hfill
\begin{subfigure}[t]{.123\textwidth}
\includegraphics[width=\linewidth]{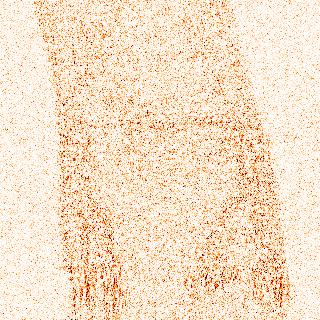}
  \caption*{\footnotesize WCRR}
\end{subfigure}%
\hfill
\begin{subfigure}[t]{.123\textwidth}
\includegraphics[width=\linewidth]{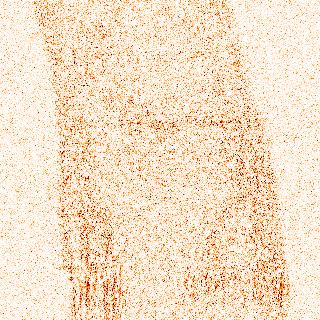}
  \caption*{\footnotesize CRR-Mask}
\end{subfigure}%
\hfill
\begin{subfigure}[t]{.123\textwidth}
\includegraphics[width=\linewidth]{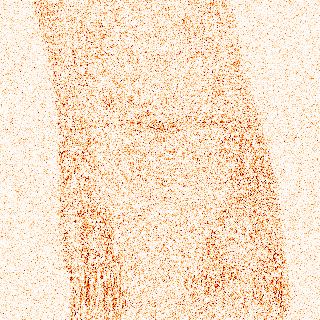}
  \caption*{\footnotesize WCRR-Mask}
\end{subfigure}%
\hfill
\begin{subfigure}[t]{.123\textwidth}
\includegraphics[width=\linewidth]{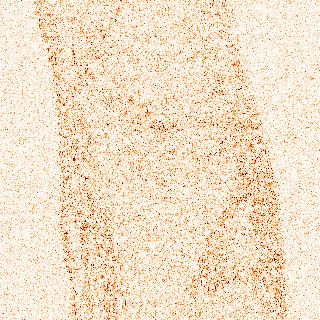}
  \caption*{\footnotesize Prox-DRUNet}
\end{subfigure}%
\hfill
\begin{subfigure}[t]{.123\textwidth}
\includegraphics[width=\linewidth]{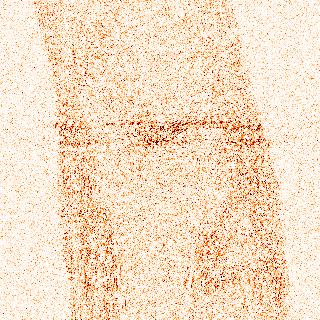}
  \caption*{\footnotesize patchNR}
\end{subfigure}%

%-----------------------------------------------------------------------------

\begin{subfigure}[t]{.123\textwidth}  
\begin{tikzpicture}[spy using outlines=
{rectangle,white,magnification=4.5,size=2.115cm, connect spies}]
\node[anchor=south west,inner sep=0]  at (0,0) {\includegraphics[width=\linewidth]{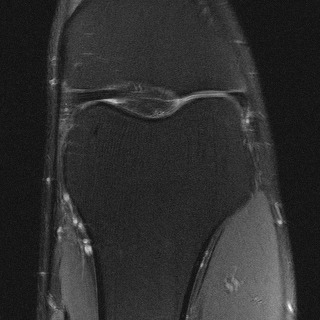}};
\spy on (.5,.6) in node [right] at (-0.005,-1.1);
\end{tikzpicture}
%\caption*{\footnotesize GT}
\end{subfigure}%
\hfill
\begin{subfigure}[t]{.123\textwidth}
\begin{tikzpicture}[spy using outlines=
{rectangle,white,magnification=4.5,size=2.115cm, connect spies}]
\node[anchor=south west,inner sep=0]  at (0,0) {\includegraphics[width=\linewidth]{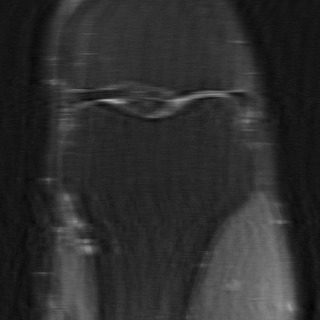}};
\spy on (.5,.6) in node [right] at (-0.005,-1.1);
\end{tikzpicture}
%  \caption*{\footnotesize FBP}
\end{subfigure}%
\hfill
\begin{subfigure}[t]{.123\textwidth}
\begin{tikzpicture}[spy using outlines=
{rectangle,white,magnification=4.5,size=2.115cm, connect spies}]
\node[anchor=south west,inner sep=0]  at (0,0) {\includegraphics[width=\linewidth]{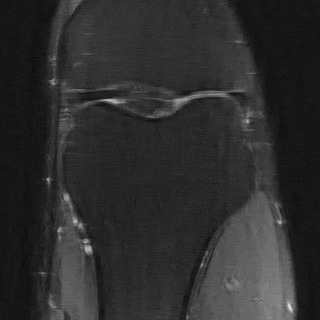}};
\spy on (.5,.6) in node [right] at (-0.005,-1.1);
\end{tikzpicture}
%  \caption*{\footnotesize CRR}
\end{subfigure}%
\hfill
\begin{subfigure}[t]{.123\textwidth}
\begin{tikzpicture}[spy using outlines=
{rectangle,white,magnification=4.5,size=2.115cm, connect spies}]
\node[anchor=south west,inner sep=0]  at (0,0) {\includegraphics[width=\linewidth]{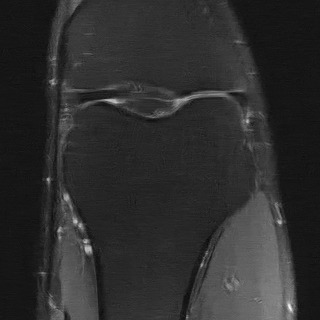}};
\spy on (.5,.6) in node [right] at (-0.005,-1.1);
\end{tikzpicture}
%  \caption*{\footnotesize WCRR}
\end{subfigure}%
\hfill
\begin{subfigure}[t]{.123\textwidth}
\begin{tikzpicture}[spy using outlines=
{rectangle,white,magnification=4.5,size=2.115cm, connect spies}]
\node[anchor=south west,inner sep=0]  at (0,0) {\includegraphics[width=\linewidth]{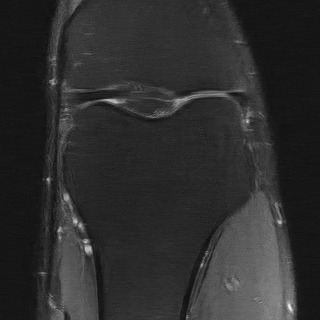}};
\spy on (.5,.6) in node [right] at (-0.005,-1.1);
\end{tikzpicture}
%  \caption*{\footnotesize Mask-CRR}
\end{subfigure}%
\hfill
\begin{subfigure}[t]{.123\textwidth}
\begin{tikzpicture}[spy using outlines=
{rectangle,white,magnification=4.5,size=2.115cm, connect spies}]
\node[anchor=south west,inner sep=0]  at (0,0) {\includegraphics[width=\linewidth]{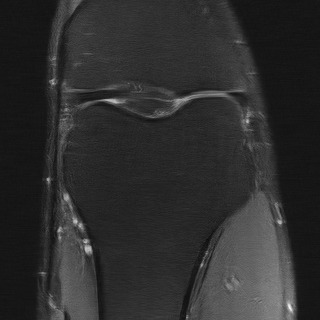}};
\spy on (.5,.6) in node [right] at (-0.005,-1.1);
\end{tikzpicture}
%  \caption*{\footnotesize Mask-WCRR}
\end{subfigure}%
\hfill
\begin{subfigure}[t]{.123\textwidth}
\begin{tikzpicture}[spy using outlines=
{rectangle,white,magnification=4.5,size=2.115cm, connect spies}]
\node[anchor=south west,inner sep=0]  at (0,0) {\includegraphics[width=\linewidth]{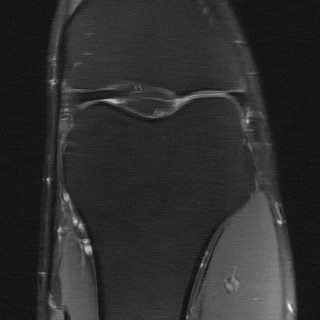}};
\spy on (.5,.6) in node [right] at (-0.005,-1.1);
\end{tikzpicture}
%  \caption*{\footnotesize Prox-DRUNet}
\end{subfigure}%
\hfill
\begin{subfigure}[t]{.123\textwidth}
\begin{tikzpicture}[spy using outlines=
{rectangle,white,magnification=4.5,size=2.115cm, connect spies}]
\node[anchor=south west,inner sep=0]  at (0,0) {\includegraphics[width=\linewidth]{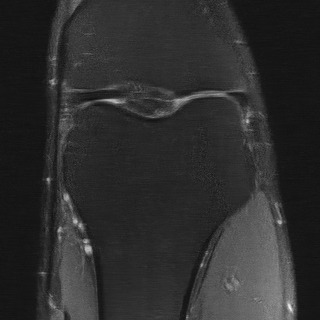}};
\spy on (.5,.6) in node [right] at (-0.005,-1.1);
\end{tikzpicture}
%  \caption*{\footnotesize patchNR}
\end{subfigure}%

\begin{subfigure}[t]{.123\textwidth}
\includegraphics[width=\linewidth]{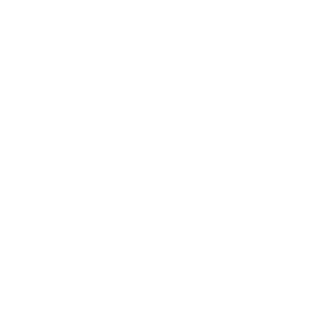}
  \caption*{\footnotesize GT}
\end{subfigure}%
\hfill
\begin{subfigure}[t]{.123\textwidth}
\includegraphics[width=\linewidth]{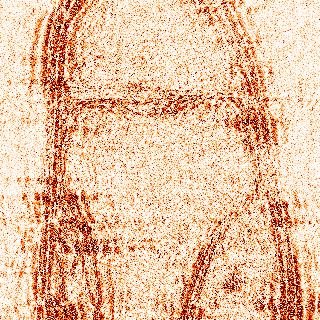}
  \caption*{\footnotesize Zero-filled}
\end{subfigure}%
\hfill
\begin{subfigure}[t]{.123\textwidth}
\includegraphics[width=\linewidth]{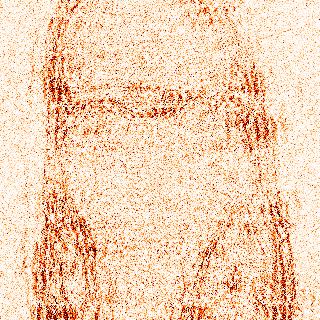}
  \caption*{\footnotesize CRR}
\end{subfigure}%
\hfill
\begin{subfigure}[t]{.123\textwidth}
\includegraphics[width=\linewidth]{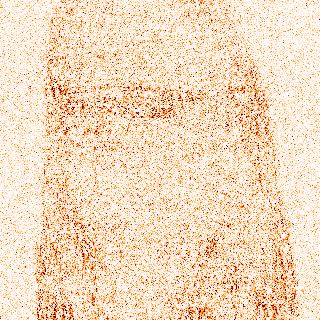}
  \caption*{\footnotesize WCRR}
\end{subfigure}%
\hfill
\begin{subfigure}[t]{.123\textwidth}
\includegraphics[width=\linewidth]{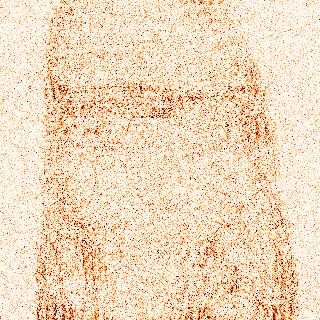}
  \caption*{\footnotesize CRR-Mask}
\end{subfigure}%
\hfill
\begin{subfigure}[t]{.123\textwidth}
\includegraphics[width=\linewidth]{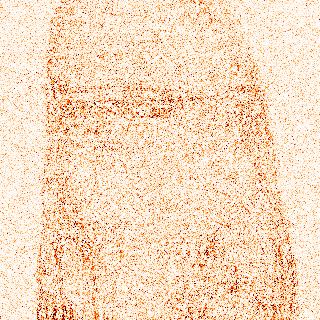}
  \caption*{\footnotesize WCRR-Mask}
\end{subfigure}%
\hfill
\begin{subfigure}[t]{.123\textwidth}
\includegraphics[width=\linewidth]{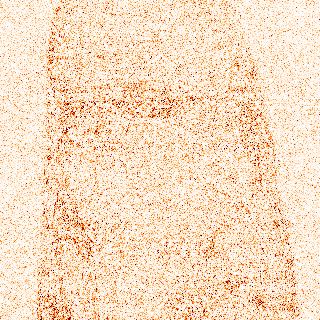}
  \caption*{\footnotesize Prox-DRUNet}
\end{subfigure}%
\hfill
\begin{subfigure}[t]{.123\textwidth}
\includegraphics[width=\linewidth]{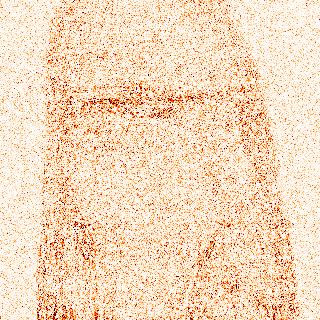}
  \caption*{\footnotesize patchNR}
\end{subfigure}%

%-----------------------------------------------------------------------------

\begin{subfigure}[t]{.123\textwidth}  
\begin{tikzpicture}[spy using outlines=
{rectangle,white,magnification=4.5,size=2.115cm, connect spies}]
\node[anchor=south west,inner sep=0]  at (0,0) {\includegraphics[width=\linewidth]{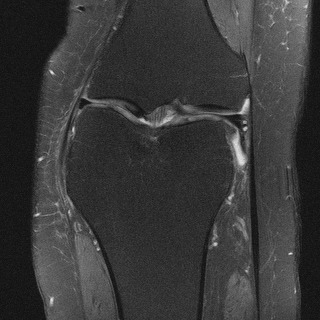}};
\spy on (1.1,1.32) in node [right] at (-0.005,-1.1);
\end{tikzpicture}
%\caption*{\footnotesize GT}
\end{subfigure}%
\hfill
\begin{subfigure}[t]{.123\textwidth}
\begin{tikzpicture}[spy using outlines=
{rectangle,white,magnification=4.5,size=2.115cm, connect spies}]
\node[anchor=south west,inner sep=0]  at (0,0) {\includegraphics[width=\linewidth]{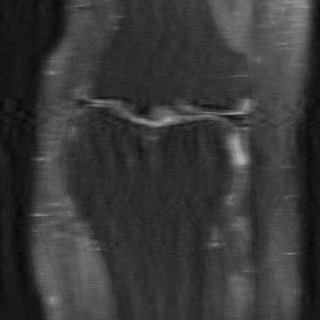}};
\spy on (1.1,1.32) in node [right] at (-0.005,-1.1);
\end{tikzpicture}
%  \caption*{\footnotesize FBP}
\end{subfigure}%
\hfill
\begin{subfigure}[t]{.123\textwidth}
\begin{tikzpicture}[spy using outlines=
{rectangle,white,magnification=4.5,size=2.115cm, connect spies}]
\node[anchor=south west,inner sep=0]  at (0,0) {\includegraphics[width=\linewidth]{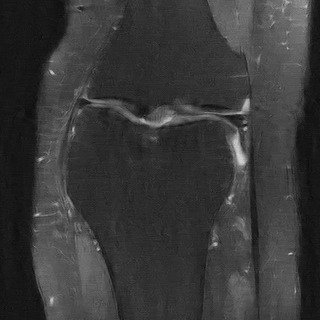}};
\spy on (1.1,1.32) in node [right] at (-0.005,-1.1);
\end{tikzpicture}
%  \caption*{\footnotesize CRR}
\end{subfigure}%
\hfill
\begin{subfigure}[t]{.123\textwidth}
\begin{tikzpicture}[spy using outlines=
{rectangle,white,magnification=4.5,size=2.115cm, connect spies}]
\node[anchor=south west,inner sep=0]  at (0,0) {\includegraphics[width=\linewidth]{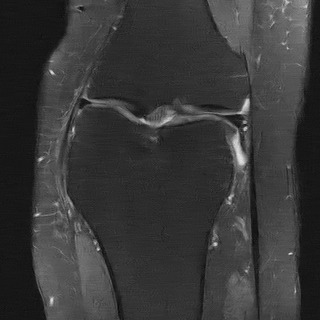}};
\spy on (1.1,1.32) in node [right] at (-0.005,-1.1);
\end{tikzpicture}
%  \caption*{\footnotesize WCRR}
\end{subfigure}%
\hfill
\begin{subfigure}[t]{.123\textwidth}
\begin{tikzpicture}[spy using outlines=
{rectangle,white,magnification=4.5,size=2.115cm, connect spies}]
\node[anchor=south west,inner sep=0]  at (0,0) {\includegraphics[width=\linewidth]{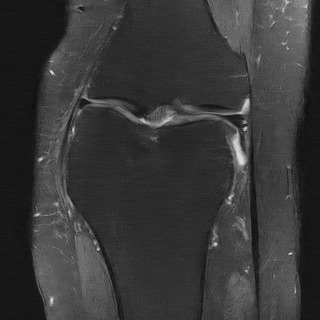}};
\spy on (1.1,1.32) in node [right] at (-0.005,-1.1);
\end{tikzpicture}
%  \caption*{\footnotesize Mask-CRR}
\end{subfigure}%
\hfill
\begin{subfigure}[t]{.123\textwidth}
\begin{tikzpicture}[spy using outlines=
{rectangle,white,magnification=4.5,size=2.115cm, connect spies}]
\node[anchor=south west,inner sep=0]  at (0,0) {\includegraphics[width=\linewidth]{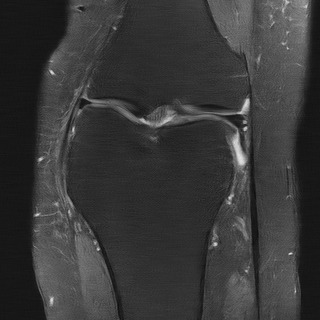}};
\spy on (1.1,1.32) in node [right] at (-0.005,-1.1);
\end{tikzpicture}
%  \caption*{\footnotesize Mask-WCRR}
\end{subfigure}%
\hfill
\begin{subfigure}[t]{.123\textwidth}
\begin{tikzpicture}[spy using outlines=
{rectangle,white,magnification=4.5,size=2.115cm, connect spies}]
\node[anchor=south west,inner sep=0]  at (0,0) {\includegraphics[width=\linewidth]{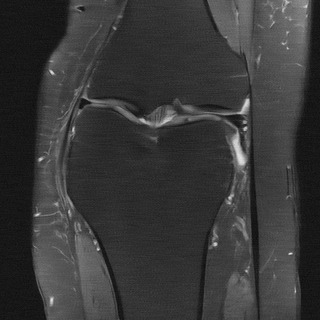}};
\spy on (1.1,1.32) in node [right] at (-0.005,-1.1);
\end{tikzpicture}
%  \caption*{\footnotesize Prox-DRUNet}
\end{subfigure}%
\hfill
\begin{subfigure}[t]{.123\textwidth}
\begin{tikzpicture}[spy using outlines=
{rectangle,white,magnification=4.5,size=2.115cm, connect spies}]
\node[anchor=south west,inner sep=0]  at (0,0) {\includegraphics[width=\linewidth]{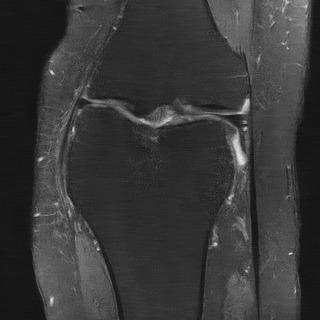}};
\spy on (1.1,1.32) in node [right] at (-0.005,-1.1);
\end{tikzpicture}
%  \caption*{\footnotesize patchNR}
\end{subfigure}%

\begin{subfigure}[t]{.123\textwidth}
\includegraphics[width=\linewidth]{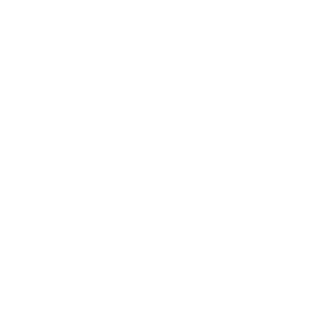}
  \caption*{\footnotesize GT}
\end{subfigure}%
\hfill
\begin{subfigure}[t]{.123\textwidth}
\includegraphics[width=\linewidth]{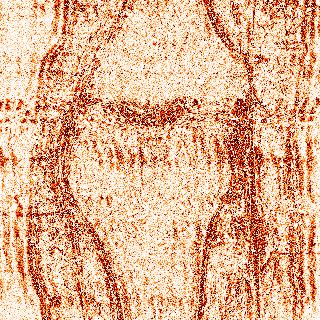}
  \caption*{\footnotesize Zero-filled}
\end{subfigure}%
\hfill
\begin{subfigure}[t]{.123\textwidth}
\includegraphics[width=\linewidth]{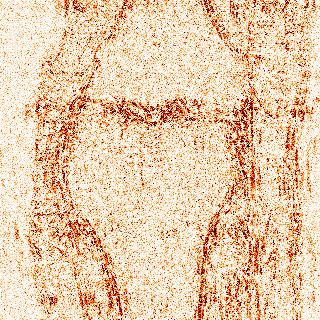}
  \caption*{\footnotesize CRR}
\end{subfigure}%
\hfill
\begin{subfigure}[t]{.123\textwidth}
\includegraphics[width=\linewidth]{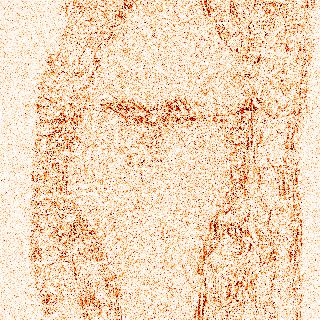}
  \caption*{\footnotesize WCRR}
\end{subfigure}%
\hfill
\begin{subfigure}[t]{.123\textwidth}
\includegraphics[width=\linewidth]{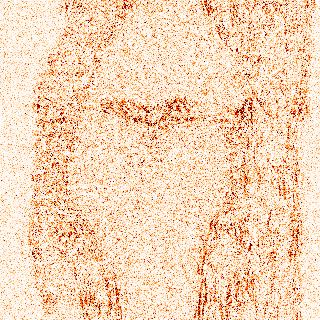}
  \caption*{\footnotesize CRR-Mask}
\end{subfigure}%
\hfill
\begin{subfigure}[t]{.123\textwidth}
\includegraphics[width=\linewidth]{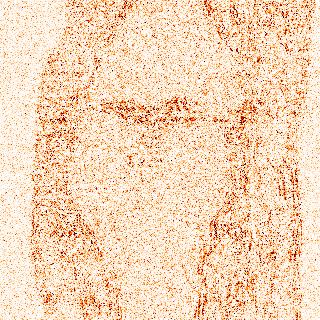}
  \caption*{\footnotesize WCRR-Mask}
\end{subfigure}%
\hfill
\begin{subfigure}[t]{.123\textwidth}
\includegraphics[width=\linewidth]{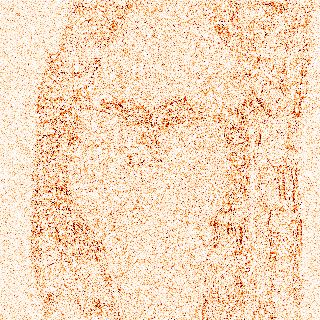}
  \caption*{\footnotesize Prox-DRUNet}
\end{subfigure}%
\hfill
\begin{subfigure}[t]{.123\textwidth}
\includegraphics[width=\linewidth]{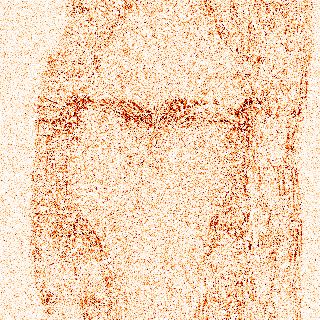}
  \caption*{\footnotesize patchNR}
\end{subfigure}%

\caption{4-fold single-coil MRI on PDFS data set.
The white box marks the zoomed area.
\textit{Top}: full image; \textit{middle}: zoomed-in part; \textit{bottom}: error.
} \label{fig:MRI_comparison_pdfs4_appendix}
\end{figure}

%-----------------------------------------------------------------------------------

\begin{figure}[t]
\centering
\begin{subfigure}[t]{.123\textwidth}  
\begin{tikzpicture}[spy using outlines=
{rectangle,white,magnification=4.5,size=2.115cm, connect spies}]
\node[anchor=south west,inner sep=0]  at (0,0) {\includegraphics[width=\linewidth]{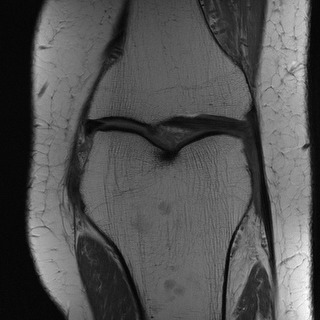}};
\spy on (1.65,.58) in node [right] at (-0.005,-1.1);
\end{tikzpicture}
%\caption*{\footnotesize GT}
\end{subfigure}%
\hfill
\begin{subfigure}[t]{.123\textwidth}
\begin{tikzpicture}[spy using outlines=
{rectangle,white,magnification=4.5,size=2.115cm, connect spies}]
\node[anchor=south west,inner sep=0]  at (0,0) {\includegraphics[width=\linewidth]{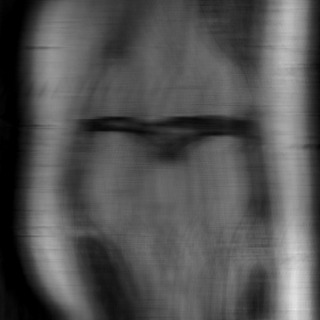}};
\spy on (1.65,.58) in node [right] at (-0.005,-1.1);
\end{tikzpicture}
%  \caption*{\footnotesize FBP}
\end{subfigure}%
\hfill
\begin{subfigure}[t]{.123\textwidth}
\begin{tikzpicture}[spy using outlines=
{rectangle,white,magnification=4.5,size=2.115cm, connect spies}]
\node[anchor=south west,inner sep=0]  at (0,0) {\includegraphics[width=\linewidth]{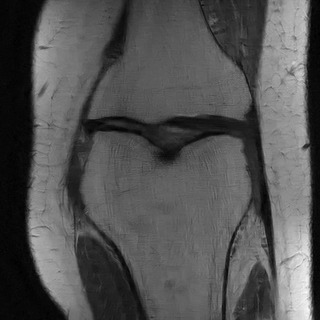}};
\spy on (1.65,.58) in node [right] at (-0.005,-1.1);
\end{tikzpicture}
%  \caption*{\footnotesize CRR}
\end{subfigure}%
\hfill
\begin{subfigure}[t]{.123\textwidth}
\begin{tikzpicture}[spy using outlines=
{rectangle,white,magnification=4.5,size=2.115cm, connect spies}]
\node[anchor=south west,inner sep=0]  at (0,0) {\includegraphics[width=\linewidth]{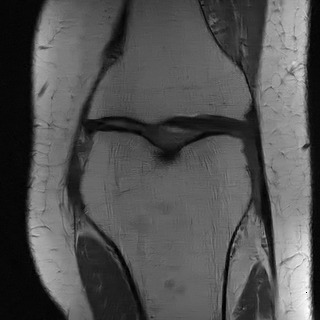}};
\spy on (1.65,.58) in node [right] at (-0.005,-1.1);
\end{tikzpicture}
%  \caption*{\footnotesize WCRR}
\end{subfigure}%
\hfill
\begin{subfigure}[t]{.123\textwidth}
\begin{tikzpicture}[spy using outlines=
{rectangle,white,magnification=4.5,size=2.115cm, connect spies}]
\node[anchor=south west,inner sep=0]  at (0,0) {\includegraphics[width=\linewidth]{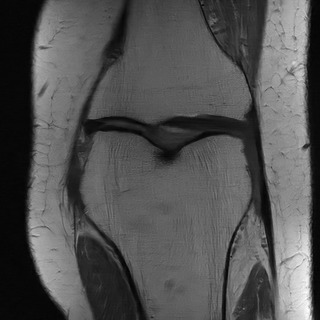}};
\spy on (1.65,.58) in node [right] at (-0.005,-1.1);
\end{tikzpicture}
%  \caption*{\footnotesize Mask-CRR}
\end{subfigure}%
\hfill
\begin{subfigure}[t]{.123\textwidth}
\begin{tikzpicture}[spy using outlines=
{rectangle,white,magnification=4.5,size=2.115cm, connect spies}]
\node[anchor=south west,inner sep=0]  at (0,0) {\includegraphics[width=\linewidth]{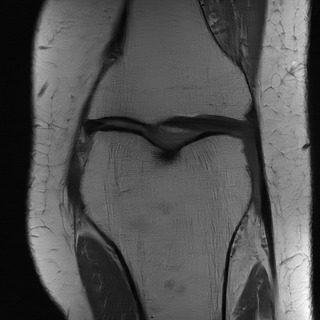}};
\spy on (1.65,.58) in node [right] at (-0.005,-1.1);
\end{tikzpicture}
%  \caption*{\footnotesize Mask-WCRR}
\end{subfigure}%
\hfill
\begin{subfigure}[t]{.123\textwidth}
\begin{tikzpicture}[spy using outlines=
{rectangle,white,magnification=4.5,size=2.115cm, connect spies}]
\node[anchor=south west,inner sep=0]  at (0,0) {\includegraphics[width=\linewidth]{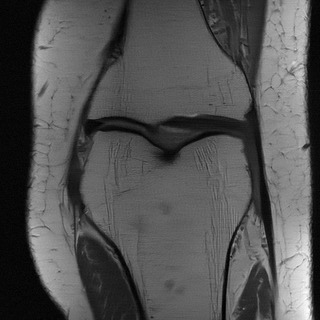}};
\spy on (1.65,.58) in node [right] at (-0.005,-1.1);
\end{tikzpicture}
%  \caption*{\footnotesize Prox-DRUNet}
\end{subfigure}%
\hfill
\begin{subfigure}[t]{.123\textwidth}
\begin{tikzpicture}[spy using outlines=
{rectangle,white,magnification=4.5,size=2.115cm, connect spies}]
\node[anchor=south west,inner sep=0]  at (0,0) {\includegraphics[width=\linewidth]{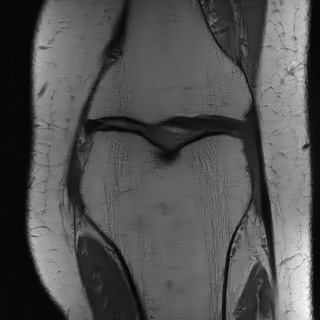}};
\spy on (1.65,.58) in node [right] at (-0.005,-1.1);
\end{tikzpicture}
%  \caption*{\footnotesize patchNR}
\end{subfigure}%

\begin{subfigure}[t]{.123\textwidth}
\includegraphics[width=\linewidth]{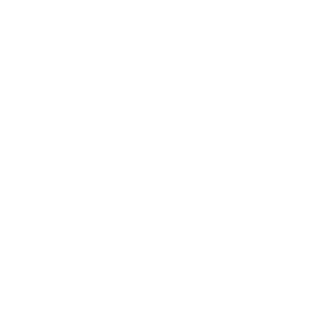}
  \caption*{\footnotesize GT}
\end{subfigure}%
\hfill
\begin{subfigure}[t]{.123\textwidth}
\includegraphics[width=\linewidth]{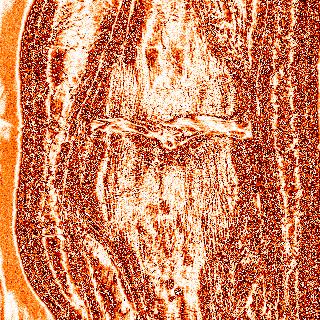}
  \caption*{\footnotesize Zero-filled}
\end{subfigure}%
\hfill
\begin{subfigure}[t]{.123\textwidth}
\includegraphics[width=\linewidth]{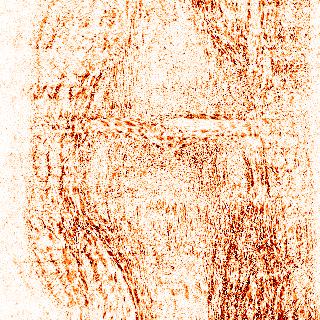}
  \caption*{\footnotesize CRR}
\end{subfigure}%
\hfill
\begin{subfigure}[t]{.123\textwidth}
\includegraphics[width=\linewidth]{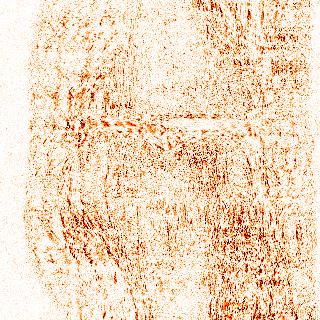}
  \caption*{\footnotesize WCRR}
\end{subfigure}%
\hfill
\begin{subfigure}[t]{.123\textwidth}
\includegraphics[width=\linewidth]{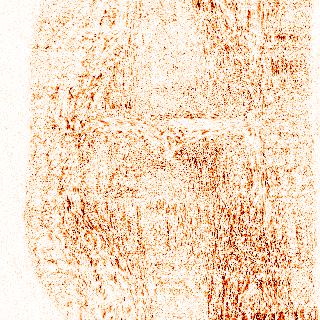}
  \caption*{\footnotesize CRR-Mask}
\end{subfigure}%
\hfill
\begin{subfigure}[t]{.123\textwidth}
\includegraphics[width=\linewidth]{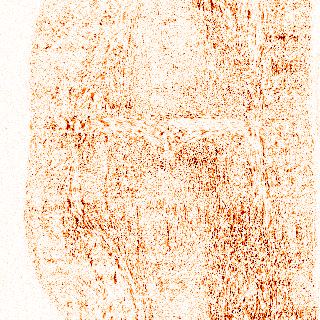}
  \caption*{\footnotesize WCRR-Mask}
\end{subfigure}%
\hfill
\begin{subfigure}[t]{.123\textwidth}
\includegraphics[width=\linewidth]{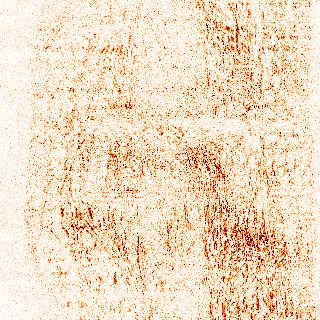}
  \caption*{\footnotesize Prox-DRUNet}
\end{subfigure}%
\hfill
\begin{subfigure}[t]{.123\textwidth}
\includegraphics[width=\linewidth]{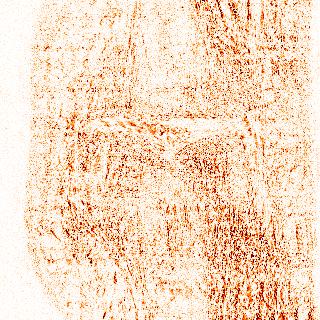}
  \caption*{\footnotesize patchNR}
\end{subfigure}%

%-----------------------------------------------------------------------------

\begin{subfigure}[t]{.123\textwidth}  
\begin{tikzpicture}[spy using outlines=
{rectangle,white,magnification=4.5,size=2.115cm, connect spies}]
\node[anchor=south west,inner sep=0]  at (0,0) {\includegraphics[width=\linewidth]{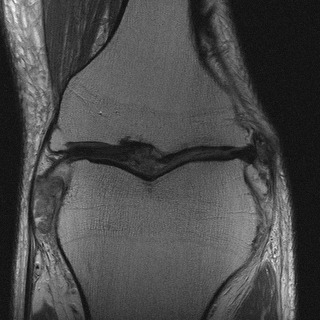}};
\spy on (1.75,.4) in node [right] at (-0.005,-1.1);
\end{tikzpicture}
%\caption*{\footnotesize GT}
\end{subfigure}%
\hfill
\begin{subfigure}[t]{.123\textwidth}
\begin{tikzpicture}[spy using outlines=
{rectangle,white,magnification=4.5,size=2.115cm, connect spies}]
\node[anchor=south west,inner sep=0]  at (0,0) {\includegraphics[width=\linewidth]{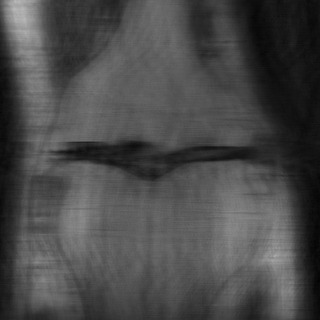}};
\spy on (1.75,.4) in node [right] at (-0.005,-1.1);
\end{tikzpicture}
%  \caption*{\footnotesize FBP}
\end{subfigure}%
\hfill
\begin{subfigure}[t]{.123\textwidth}
\begin{tikzpicture}[spy using outlines=
{rectangle,white,magnification=4.5,size=2.115cm, connect spies}]
\node[anchor=south west,inner sep=0]  at (0,0) {\includegraphics[width=\linewidth]{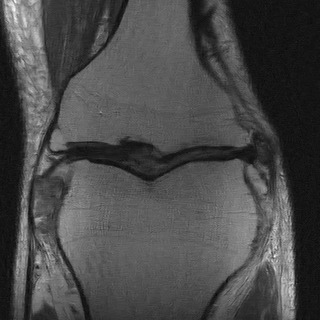}};
\spy on (1.75,.4) in node [right] at (-0.005,-1.1);
\end{tikzpicture}
%  \caption*{\footnotesize CRR}
\end{subfigure}%
\hfill
\begin{subfigure}[t]{.123\textwidth}
\begin{tikzpicture}[spy using outlines=
{rectangle,white,magnification=4.5,size=2.115cm, connect spies}]
\node[anchor=south west,inner sep=0]  at (0,0) {\includegraphics[width=\linewidth]{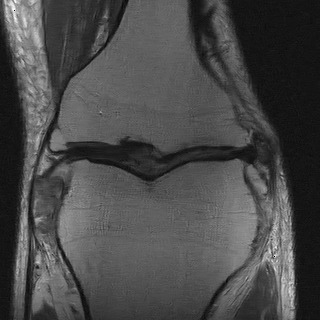}};
\spy on (1.75,.4) in node [right] at (-0.005,-1.1);
\end{tikzpicture}
%  \caption*{\footnotesize WCRR}
\end{subfigure}%
\hfill
\begin{subfigure}[t]{.123\textwidth}
\begin{tikzpicture}[spy using outlines=
{rectangle,white,magnification=4.5,size=2.115cm, connect spies}]
\node[anchor=south west,inner sep=0]  at (0,0) {\includegraphics[width=\linewidth]{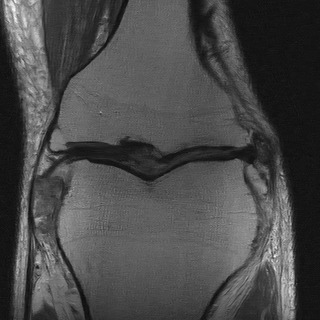}};
\spy on (1.75,.4) in node [right] at (-0.005,-1.1);
\end{tikzpicture}
%  \caption*{\footnotesize Mask-CRR}
\end{subfigure}%
\hfill
\begin{subfigure}[t]{.123\textwidth}
\begin{tikzpicture}[spy using outlines=
{rectangle,white,magnification=4.5,size=2.115cm, connect spies}]
\node[anchor=south west,inner sep=0]  at (0,0) {\includegraphics[width=\linewidth]{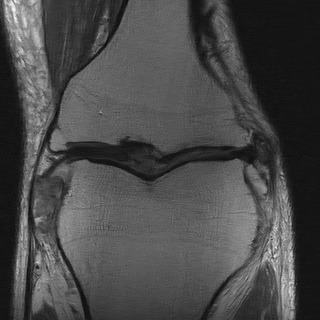}};
\spy on (1.75,.4) in node [right] at (-0.005,-1.1);
\end{tikzpicture}
%  \caption*{\footnotesize Mask-WCRR}
\end{subfigure}%
\hfill
\begin{subfigure}[t]{.123\textwidth}
\begin{tikzpicture}[spy using outlines=
{rectangle,white,magnification=4.5,size=2.115cm, connect spies}]
\node[anchor=south west,inner sep=0]  at (0,0) {\includegraphics[width=\linewidth]{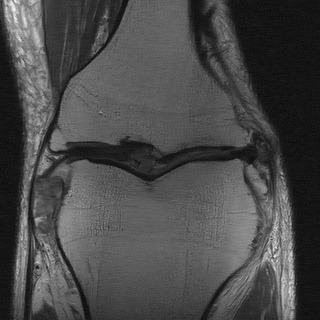}};
\spy on (1.75,.4) in node [right] at (-0.005,-1.1);
\end{tikzpicture}
%  \caption*{\footnotesize Prox-DRUNet}
\end{subfigure}%
\hfill
\begin{subfigure}[t]{.123\textwidth}
\begin{tikzpicture}[spy using outlines=
{rectangle,white,magnification=4.5,size=2.115cm, connect spies}]
\node[anchor=south west,inner sep=0]  at (0,0) {\includegraphics[width=\linewidth]{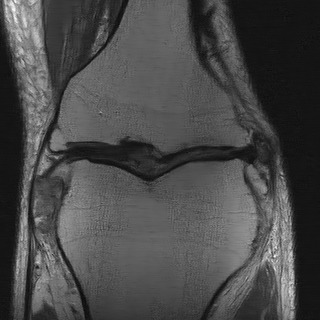}};
\spy on (1.75,.4) in node [right] at (-0.005,-1.1);
\end{tikzpicture}
%  \caption*{\footnotesize patchNR}
\end{subfigure}%

\begin{subfigure}[t]{.123\textwidth}
\includegraphics[width=\linewidth]{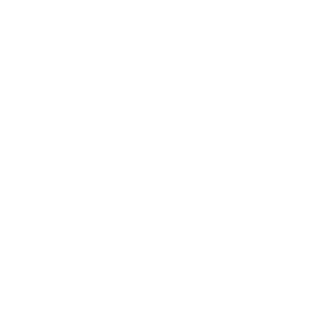}
  \caption*{\footnotesize GT}
\end{subfigure}%
\hfill
\begin{subfigure}[t]{.123\textwidth}
\includegraphics[width=\linewidth]{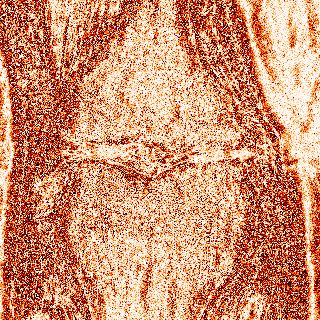}
  \caption*{\footnotesize Zero-filled}
\end{subfigure}%
\hfill
\begin{subfigure}[t]{.123\textwidth}
\includegraphics[width=\linewidth]{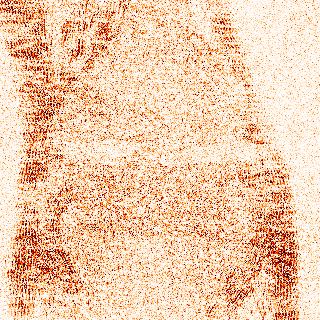}
  \caption*{\footnotesize CRR}
\end{subfigure}%
\hfill
\begin{subfigure}[t]{.123\textwidth}
\includegraphics[width=\linewidth]{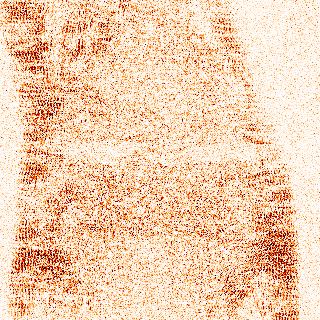}
  \caption*{\footnotesize WCRR}
\end{subfigure}%
\hfill
\begin{subfigure}[t]{.123\textwidth}
\includegraphics[width=\linewidth]{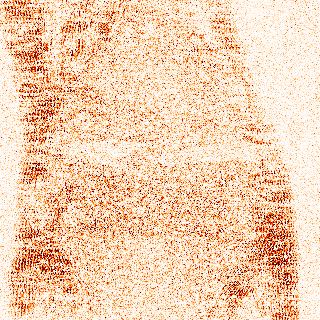}
  \caption*{\footnotesize CRR-Mask}
\end{subfigure}%
\hfill
\begin{subfigure}[t]{.123\textwidth}
\includegraphics[width=\linewidth]{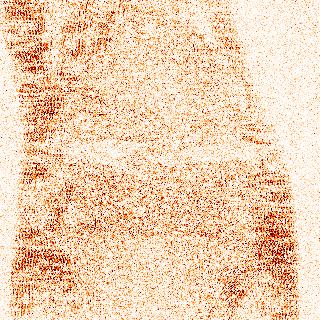}
  \caption*{\footnotesize WCRR-Mask}
\end{subfigure}%
\hfill
\begin{subfigure}[t]{.123\textwidth}
\includegraphics[width=\linewidth]{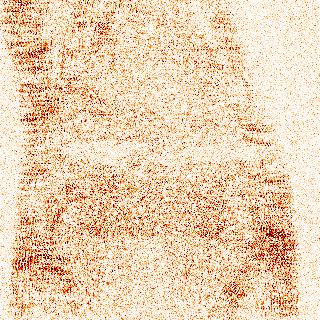}
  \caption*{\footnotesize Prox-DRUNet}
\end{subfigure}%
\hfill
\begin{subfigure}[t]{.123\textwidth}
\includegraphics[width=\linewidth]{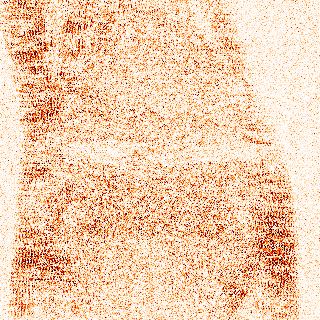}
  \caption*{\footnotesize patchNR}
\end{subfigure}%

%-----------------------------------------------------------------------------

\begin{subfigure}[t]{.123\textwidth}  
\begin{tikzpicture}[spy using outlines=
{rectangle,white,magnification=4.5,size=2.115cm, connect spies}]
\node[anchor=south west,inner sep=0]  at (0,0) {\includegraphics[width=\linewidth]{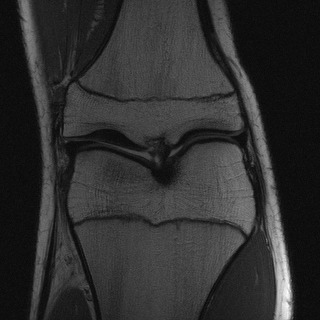}};
\spy on (1.6,1.15) in node [right] at (-0.005,-1.1);
\end{tikzpicture}
%\caption*{\footnotesize GT}
\end{subfigure}%
\hfill
\begin{subfigure}[t]{.123\textwidth}
\begin{tikzpicture}[spy using outlines=
{rectangle,white,magnification=4.5,size=2.115cm, connect spies}]
\node[anchor=south west,inner sep=0]  at (0,0) {\includegraphics[width=\linewidth]{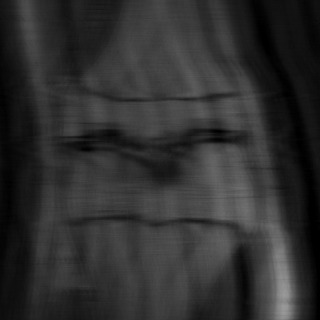}};
\spy on (1.6,1.15) in node [right] at (-0.005,-1.1);
\end{tikzpicture}
%  \caption*{\footnotesize FBP}
\end{subfigure}%
\hfill
\begin{subfigure}[t]{.123\textwidth}
\begin{tikzpicture}[spy using outlines=
{rectangle,white,magnification=4.5,size=2.115cm, connect spies}]
\node[anchor=south west,inner sep=0]  at (0,0) {\includegraphics[width=\linewidth]{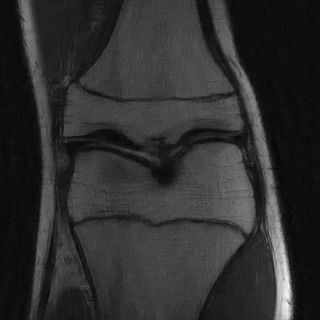}};
\spy on (1.6,1.15) in node [right] at (-0.005,-1.1);
\end{tikzpicture}
%  \caption*{\footnotesize CRR}
\end{subfigure}%
\hfill
\begin{subfigure}[t]{.123\textwidth}
\begin{tikzpicture}[spy using outlines=
{rectangle,white,magnification=4.5,size=2.115cm, connect spies}]
\node[anchor=south west,inner sep=0]  at (0,0) {\includegraphics[width=\linewidth]{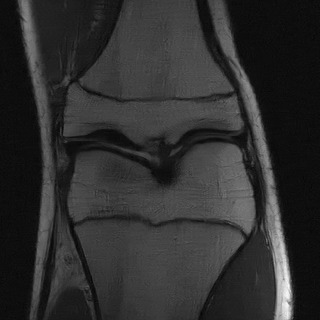}};
\spy on (1.6,1.15) in node [right] at (-0.005,-1.1);
\end{tikzpicture}
%  \caption*{\footnotesize WCRR}
\end{subfigure}%
\hfill
\begin{subfigure}[t]{.123\textwidth}
\begin{tikzpicture}[spy using outlines=
{rectangle,white,magnification=4.5,size=2.115cm, connect spies}]
\node[anchor=south west,inner sep=0]  at (0,0) {\includegraphics[width=\linewidth]{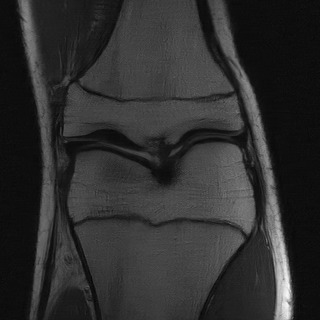}};
\spy on (1.6,1.15) in node [right] at (-0.005,-1.1);
\end{tikzpicture}
%  \caption*{\footnotesize Mask-CRR}
\end{subfigure}%
\hfill
\begin{subfigure}[t]{.123\textwidth}
\begin{tikzpicture}[spy using outlines=
{rectangle,white,magnification=4.5,size=2.115cm, connect spies}]
\node[anchor=south west,inner sep=0]  at (0,0) {\includegraphics[width=\linewidth]{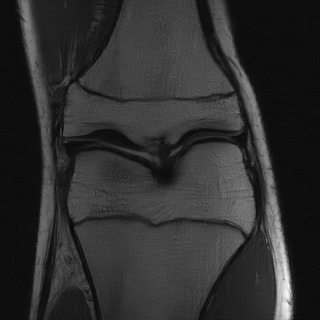}};
\spy on (1.6,1.15) in node [right] at (-0.005,-1.1);
\end{tikzpicture}
%  \caption*{\footnotesize Mask-WCRR}
\end{subfigure}%
\hfill
\begin{subfigure}[t]{.123\textwidth}
\begin{tikzpicture}[spy using outlines=
{rectangle,white,magnification=4.5,size=2.115cm, connect spies}]
\node[anchor=south west,inner sep=0]  at (0,0) {\includegraphics[width=\linewidth]{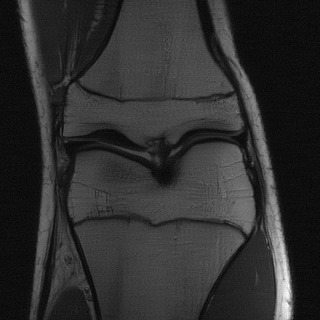}};
\spy on (1.6,1.15) in node [right] at (-0.005,-1.1);
\end{tikzpicture}
%  \caption*{\footnotesize Prox-DRUNet}
\end{subfigure}%
\hfill
\begin{subfigure}[t]{.123\textwidth}
\begin{tikzpicture}[spy using outlines=
{rectangle,white,magnification=4.5,size=2.115cm, connect spies}]
\node[anchor=south west,inner sep=0]  at (0,0) {\includegraphics[width=\linewidth]{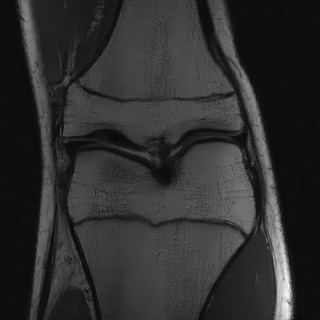}};
\spy on (1.6,1.15) in node [right] at (-0.005,-1.1);
\end{tikzpicture}
%  \caption*{\footnotesize patchNR}
\end{subfigure}%

\begin{subfigure}[t]{.123\textwidth}
\includegraphics[width=\linewidth]{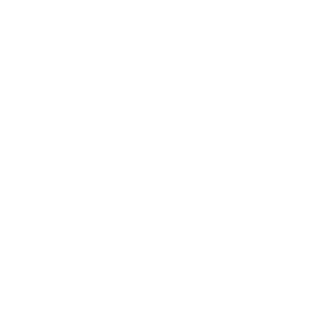}
  \caption*{\footnotesize GT}
\end{subfigure}%
\hfill
\begin{subfigure}[t]{.123\textwidth}
\includegraphics[width=\linewidth]{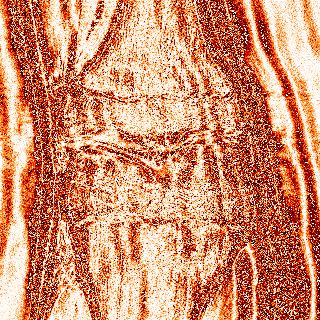}
  \caption*{\footnotesize Zero-filled}
\end{subfigure}%
\hfill
\begin{subfigure}[t]{.123\textwidth}
\includegraphics[width=\linewidth]{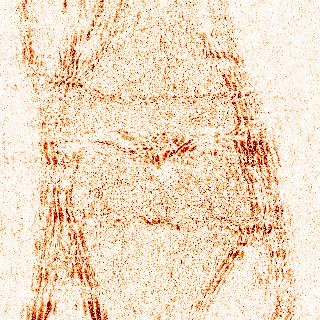}
  \caption*{\footnotesize CRR}
\end{subfigure}%
\hfill
\begin{subfigure}[t]{.123\textwidth}
\includegraphics[width=\linewidth]{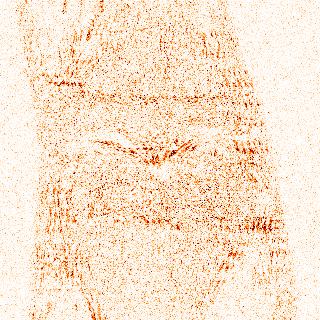}
  \caption*{\footnotesize WCRR}
\end{subfigure}%
\hfill
\begin{subfigure}[t]{.123\textwidth}
\includegraphics[width=\linewidth]{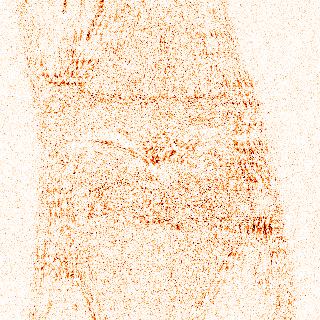}
  \caption*{\footnotesize CRR-Mask}
\end{subfigure}%
\hfill
\begin{subfigure}[t]{.123\textwidth}
\includegraphics[width=\linewidth]{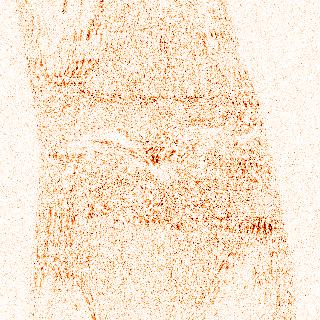}
  \caption*{\footnotesize WCRR-Mask}
\end{subfigure}%
\hfill
\begin{subfigure}[t]{.123\textwidth}
\includegraphics[width=\linewidth]{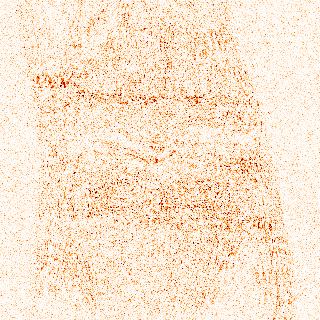}
  \caption*{\footnotesize Prox-DRUNet}
\end{subfigure}%
\hfill
\begin{subfigure}[t]{.123\textwidth}
\includegraphics[width=\linewidth]{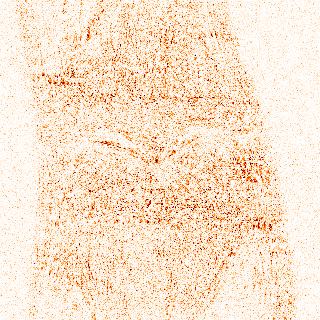}
  \caption*{\footnotesize patchNR}
\end{subfigure}%

\caption{8-fold multi-coil MRI on PD data set.
The white box marks the zoomed area.
\textit{Top}: full image; \textit{middle}: zoomed-in part; \textit{bottom}: error.
} \label{fig:MRI_comparison_pd8_appendix}
\end{figure}

%-----------------------------------------------------------------------------------

\begin{figure}[t]
\centering
\begin{subfigure}[t]{.123\textwidth}  
\begin{tikzpicture}[spy using outlines=
{rectangle,white,magnification=4.5,size=2.115cm, connect spies}]
\node[anchor=south west,inner sep=0]  at (0,0) {\includegraphics[width=\linewidth]{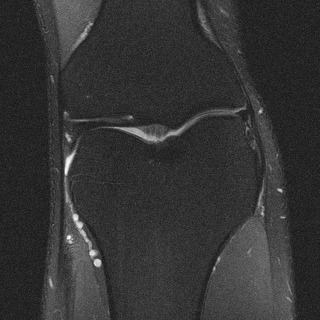}};
\spy on (1.65,1.28) in node [right] at (-0.005,-1.1);
\end{tikzpicture}
%\caption*{\footnotesize GT}
\end{subfigure}%
\hfill
\begin{subfigure}[t]{.123\textwidth}
\begin{tikzpicture}[spy using outlines=
{rectangle,white,magnification=4.5,size=2.115cm, connect spies}]
\node[anchor=south west,inner sep=0]  at (0,0) {\includegraphics[width=\linewidth]{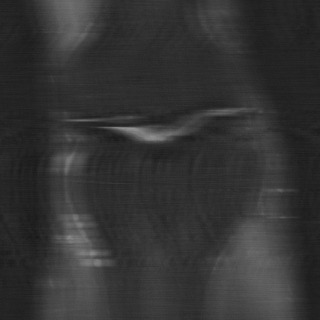}};
\spy on (1.65,1.28) in node [right] at (-0.005,-1.1);
\end{tikzpicture}
%  \caption*{\footnotesize FBP}
\end{subfigure}%
\hfill
\begin{subfigure}[t]{.123\textwidth}
\begin{tikzpicture}[spy using outlines=
{rectangle,white,magnification=4.5,size=2.115cm, connect spies}]
\node[anchor=south west,inner sep=0]  at (0,0) {\includegraphics[width=\linewidth]{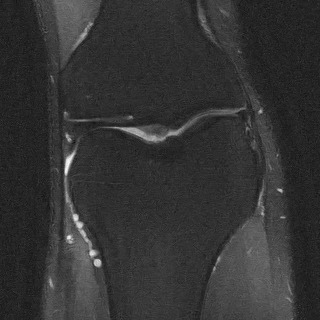}};
\spy on (1.65,1.28) in node [right] at (-0.005,-1.1);
\end{tikzpicture}
%  \caption*{\footnotesize CRR}
\end{subfigure}%
\hfill
\begin{subfigure}[t]{.123\textwidth}
\begin{tikzpicture}[spy using outlines=
{rectangle,white,magnification=4.5,size=2.115cm, connect spies}]
\node[anchor=south west,inner sep=0]  at (0,0) {\includegraphics[width=\linewidth]{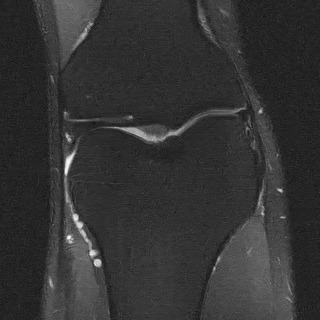}};
\spy on (1.65,1.28) in node [right] at (-0.005,-1.1);
\end{tikzpicture}
%  \caption*{\footnotesize WCRR}
\end{subfigure}%
\hfill
\begin{subfigure}[t]{.123\textwidth}
\begin{tikzpicture}[spy using outlines=
{rectangle,white,magnification=4.5,size=2.115cm, connect spies}]
\node[anchor=south west,inner sep=0]  at (0,0) {\includegraphics[width=\linewidth]{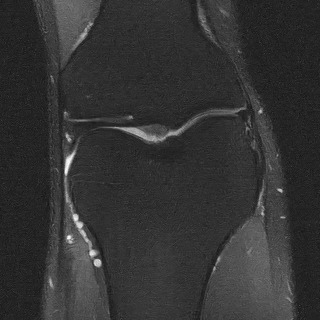}};
\spy on (1.65,1.28) in node [right] at (-0.005,-1.1);
\end{tikzpicture}
%  \caption*{\footnotesize Mask-CRR}
\end{subfigure}%
\hfill
\begin{subfigure}[t]{.123\textwidth}
\begin{tikzpicture}[spy using outlines=
{rectangle,white,magnification=4.5,size=2.115cm, connect spies}]
\node[anchor=south west,inner sep=0]  at (0,0) {\includegraphics[width=\linewidth]{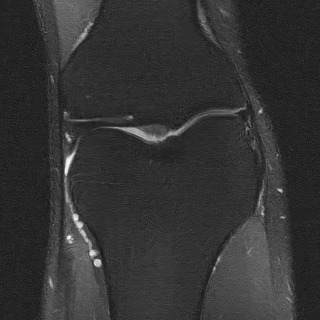}};
\spy on (1.65,1.28) in node [right] at (-0.005,-1.1);
\end{tikzpicture}
%  \caption*{\footnotesize Mask-WCRR}
\end{subfigure}%
\hfill
\begin{subfigure}[t]{.123\textwidth}
\begin{tikzpicture}[spy using outlines=
{rectangle,white,magnification=4.5,size=2.115cm, connect spies}]
\node[anchor=south west,inner sep=0]  at (0,0) {\includegraphics[width=\linewidth]{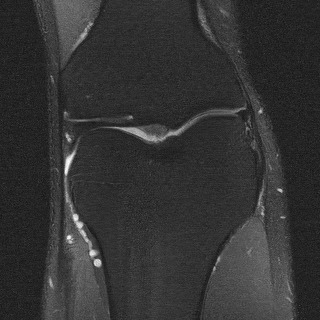}};
\spy on (1.65,1.28) in node [right] at (-0.005,-1.1);
\end{tikzpicture}
%  \caption*{\footnotesize Prox-DRUNet}
\end{subfigure}%
\hfill
\begin{subfigure}[t]{.123\textwidth}
\begin{tikzpicture}[spy using outlines=
{rectangle,white,magnification=4.5,size=2.115cm, connect spies}]
\node[anchor=south west,inner sep=0]  at (0,0) {\includegraphics[width=\linewidth]{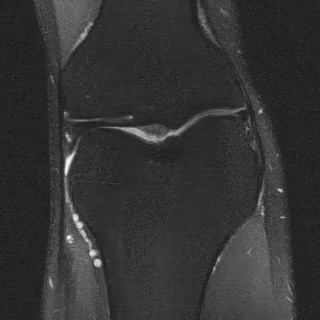}};
\spy on (1.65,1.28) in node [right] at (-0.005,-1.1);
\end{tikzpicture}
%  \caption*{\footnotesize patchNR}
\end{subfigure}%

\begin{subfigure}[t]{.123\textwidth}
\includegraphics[width=\linewidth]{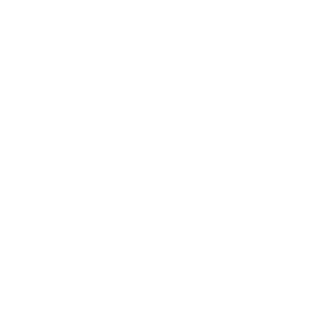}
  \caption*{\footnotesize GT}
\end{subfigure}%
\hfill
\begin{subfigure}[t]{.123\textwidth}
\includegraphics[width=\linewidth]{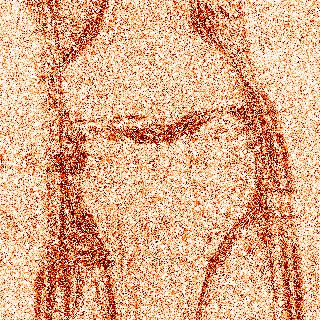}
  \caption*{\footnotesize Zero-filled}
\end{subfigure}%
\hfill
\begin{subfigure}[t]{.123\textwidth}
\includegraphics[width=\linewidth]{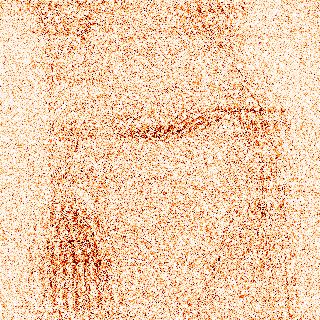}
  \caption*{\footnotesize CRR}
\end{subfigure}%
\hfill
\begin{subfigure}[t]{.123\textwidth}
\includegraphics[width=\linewidth]{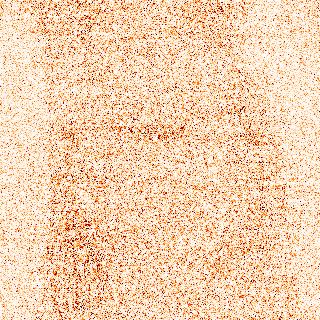}
  \caption*{\footnotesize WCRR}
\end{subfigure}%
\hfill
\begin{subfigure}[t]{.123\textwidth}
\includegraphics[width=\linewidth]{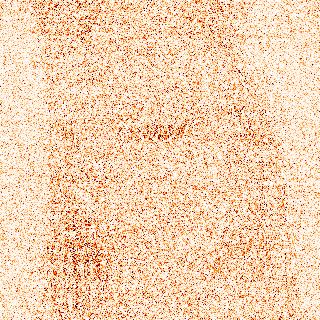}
  \caption*{\footnotesize CRR-Mask}
\end{subfigure}%
\hfill
\begin{subfigure}[t]{.123\textwidth}
\includegraphics[width=\linewidth]{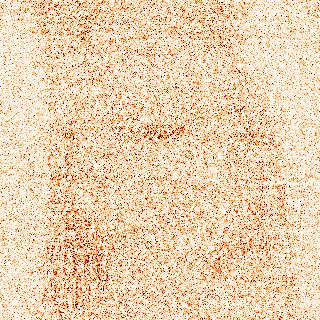}
  \caption*{\footnotesize WCRR-Mask}
\end{subfigure}%
\hfill
\begin{subfigure}[t]{.123\textwidth}
\includegraphics[width=\linewidth]{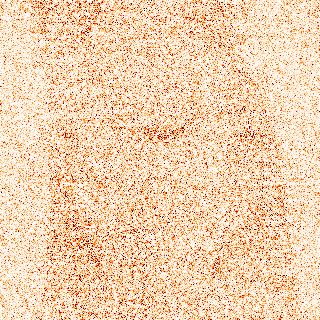}
  \caption*{\footnotesize Prox-DRUNet}
\end{subfigure}%
\hfill
\begin{subfigure}[t]{.123\textwidth}
\includegraphics[width=\linewidth]{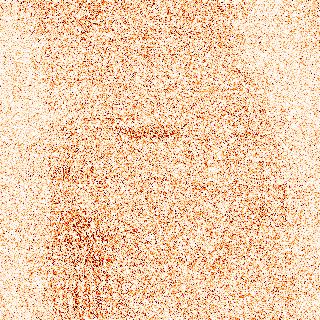}
  \caption*{\footnotesize patchNR}
\end{subfigure}%

%-----------------------------------------------------------------------------

\begin{subfigure}[t]{.123\textwidth}  
\begin{tikzpicture}[spy using outlines=
{rectangle,white,magnification=4.5,size=2.115cm, connect spies}]
\node[anchor=south west,inner sep=0]  at (0,0) {\includegraphics[width=\linewidth]{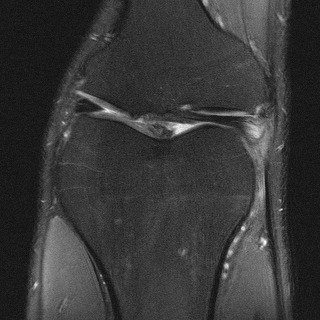}};
\spy on (1.15,1.3) in node [right] at (-0.005,-1.1);
\end{tikzpicture}
%\caption*{\footnotesize GT}
\end{subfigure}%
\hfill
\begin{subfigure}[t]{.123\textwidth}
\begin{tikzpicture}[spy using outlines=
{rectangle,white,magnification=4.5,size=2.115cm, connect spies}]
\node[anchor=south west,inner sep=0]  at (0,0) {\includegraphics[width=\linewidth]{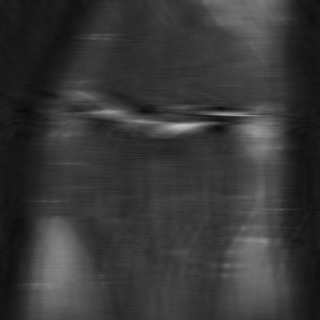}};
\spy on (1.15,1.3) in node [right] at (-0.005,-1.1);
\end{tikzpicture}
%  \caption*{\footnotesize FBP}
\end{subfigure}%
\hfill
\begin{subfigure}[t]{.123\textwidth}
\begin{tikzpicture}[spy using outlines=
{rectangle,white,magnification=4.5,size=2.115cm, connect spies}]
\node[anchor=south west,inner sep=0]  at (0,0) {\includegraphics[width=\linewidth]{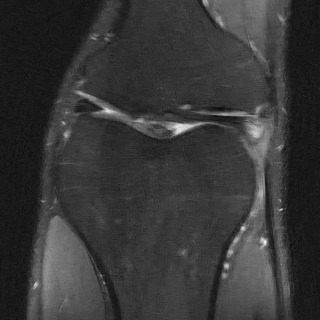}};
\spy on (1.15,1.3) in node [right] at (-0.005,-1.1);
\end{tikzpicture}
%  \caption*{\footnotesize CRR}
\end{subfigure}%
\hfill
\begin{subfigure}[t]{.123\textwidth}
\begin{tikzpicture}[spy using outlines=
{rectangle,white,magnification=4.5,size=2.115cm, connect spies}]
\node[anchor=south west,inner sep=0]  at (0,0) {\includegraphics[width=\linewidth]{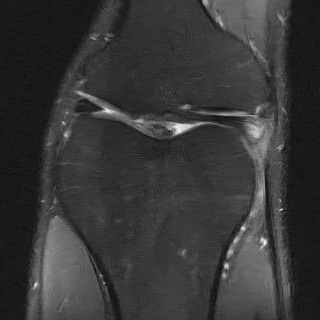}};
\spy on (1.15,1.3) in node [right] at (-0.005,-1.1);
\end{tikzpicture}
%  \caption*{\footnotesize WCRR}
\end{subfigure}%
\hfill
\begin{subfigure}[t]{.123\textwidth}
\begin{tikzpicture}[spy using outlines=
{rectangle,white,magnification=4.5,size=2.115cm, connect spies}]
\node[anchor=south west,inner sep=0]  at (0,0) {\includegraphics[width=\linewidth]{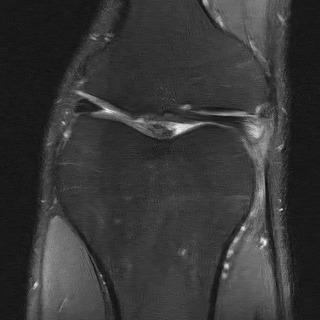}};
\spy on (1.15,1.3) in node [right] at (-0.005,-1.1);
\end{tikzpicture}
%  \caption*{\footnotesize Mask-CRR}
\end{subfigure}%
\hfill
\begin{subfigure}[t]{.123\textwidth}
\begin{tikzpicture}[spy using outlines=
{rectangle,white,magnification=4.5,size=2.115cm, connect spies}]
\node[anchor=south west,inner sep=0]  at (0,0) {\includegraphics[width=\linewidth]{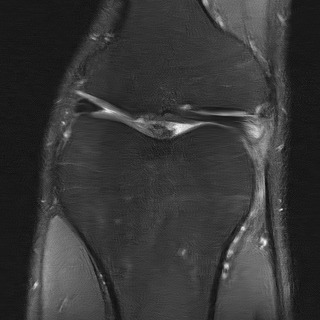}};
\spy on (1.15,1.3) in node [right] at (-0.005,-1.1);
\end{tikzpicture}
%  \caption*{\footnotesize Mask-WCRR}
\end{subfigure}%
\hfill
\begin{subfigure}[t]{.123\textwidth}
\begin{tikzpicture}[spy using outlines=
{rectangle,white,magnification=4.5,size=2.115cm, connect spies}]
\node[anchor=south west,inner sep=0]  at (0,0) {\includegraphics[width=\linewidth]{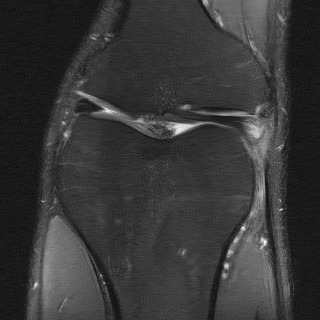}};
\spy on (1.15,1.3) in node [right] at (-0.005,-1.1);
\end{tikzpicture}
%  \caption*{\footnotesize Prox-DRUNet}
\end{subfigure}%
\hfill
\begin{subfigure}[t]{.123\textwidth}
\begin{tikzpicture}[spy using outlines=
{rectangle,white,magnification=4.5,size=2.115cm, connect spies}]
\node[anchor=south west,inner sep=0]  at (0,0) {\includegraphics[width=\linewidth]{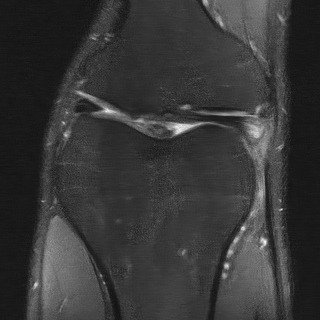}};
\spy on (1.15,1.3) in node [right] at (-0.005,-1.1);
\end{tikzpicture}
%  \caption*{\footnotesize patchNR}
\end{subfigure}%

\begin{subfigure}[t]{.123\textwidth}
\includegraphics[width=\linewidth]{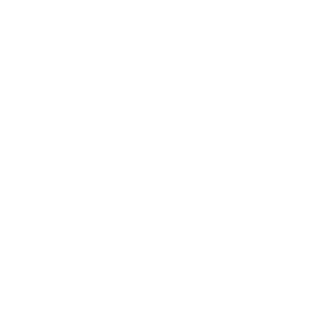}
  \caption*{\footnotesize GT}
\end{subfigure}%
\hfill
\begin{subfigure}[t]{.123\textwidth}
\includegraphics[width=\linewidth]{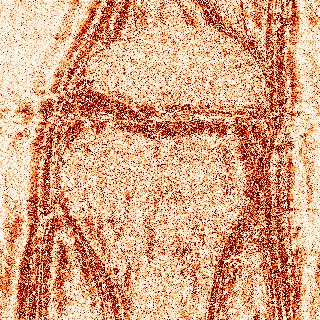}
  \caption*{\footnotesize Zero-filled}
\end{subfigure}%
\hfill
\begin{subfigure}[t]{.123\textwidth}
\includegraphics[width=\linewidth]{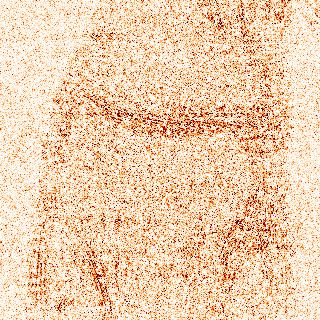}
  \caption*{\footnotesize CRR}
\end{subfigure}%
\hfill
\begin{subfigure}[t]{.123\textwidth}
\includegraphics[width=\linewidth]{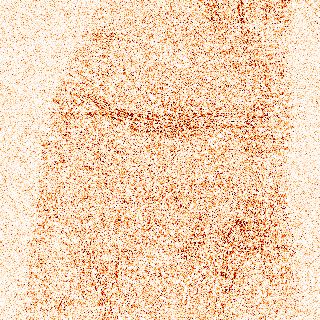}
  \caption*{\footnotesize WCRR}
\end{subfigure}%
\hfill
\begin{subfigure}[t]{.123\textwidth}
\includegraphics[width=\linewidth]{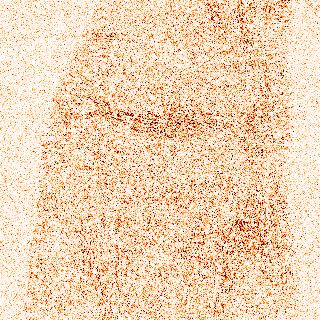}
  \caption*{\footnotesize CRR-Mask}
\end{subfigure}%
\hfill
\begin{subfigure}[t]{.123\textwidth}
\includegraphics[width=\linewidth]{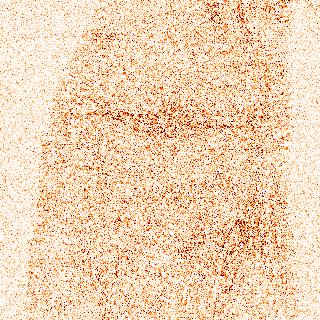}
  \caption*{\footnotesize WCRR-Mask}
\end{subfigure}%
\hfill
\begin{subfigure}[t]{.123\textwidth}
\includegraphics[width=\linewidth]{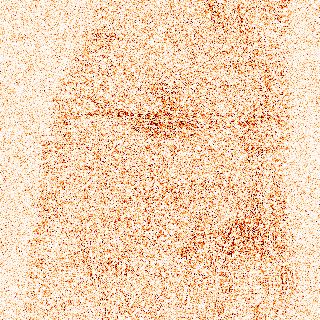}
  \caption*{\footnotesize Prox-DRUNet}
\end{subfigure}%
\hfill
\begin{subfigure}[t]{.123\textwidth}
\includegraphics[width=\linewidth]{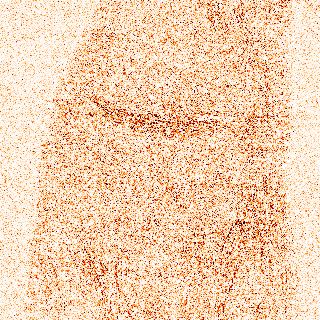}
  \caption*{\footnotesize patchNR}
\end{subfigure}%

%-----------------------------------------------------------------------------

\begin{subfigure}[t]{.123\textwidth}  
\begin{tikzpicture}[spy using outlines=
{rectangle,white,magnification=4.5,size=2.115cm, connect spies}]
\node[anchor=south west,inner sep=0]  at (0,0) {\includegraphics[width=\linewidth]{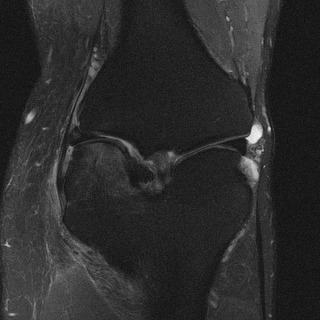}};
\spy on (1.6,1.15) in node [right] at (-0.005,-1.1);
\end{tikzpicture}
%\caption*{\footnotesize GT}
\end{subfigure}%
\hfill
\begin{subfigure}[t]{.123\textwidth}
\begin{tikzpicture}[spy using outlines=
{rectangle,white,magnification=4.5,size=2.115cm, connect spies}]
\node[anchor=south west,inner sep=0]  at (0,0) {\includegraphics[width=\linewidth]{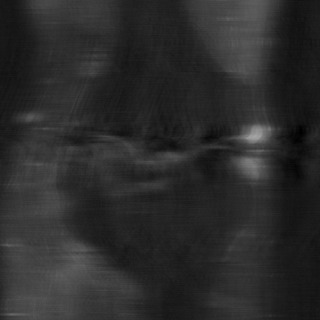}};
\spy on (1.6,1.15) in node [right] at (-0.005,-1.1);
\end{tikzpicture}
%  \caption*{\footnotesize FBP}
\end{subfigure}%
\hfill
\begin{subfigure}[t]{.123\textwidth}
\begin{tikzpicture}[spy using outlines=
{rectangle,white,magnification=4.5,size=2.115cm, connect spies}]
\node[anchor=south west,inner sep=0]  at (0,0) {\includegraphics[width=\linewidth]{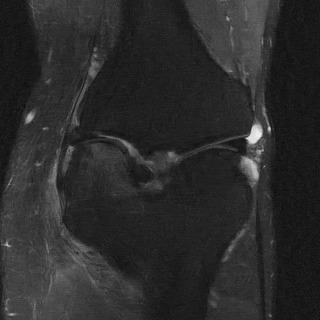}};
\spy on (1.6,1.15) in node [right] at (-0.005,-1.1);
\end{tikzpicture}
%  \caption*{\footnotesize CRR}
\end{subfigure}%
\hfill
\begin{subfigure}[t]{.123\textwidth}
\begin{tikzpicture}[spy using outlines=
{rectangle,white,magnification=4.5,size=2.115cm, connect spies}]
\node[anchor=south west,inner sep=0]  at (0,0) {\includegraphics[width=\linewidth]{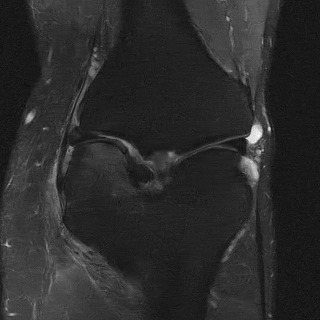}};
\spy on (1.6,1.15) in node [right] at (-0.005,-1.1);
\end{tikzpicture}
%  \caption*{\footnotesize WCRR}
\end{subfigure}%
\hfill
\begin{subfigure}[t]{.123\textwidth}
\begin{tikzpicture}[spy using outlines=
{rectangle,white,magnification=4.5,size=2.115cm, connect spies}]
\node[anchor=south west,inner sep=0]  at (0,0) {\includegraphics[width=\linewidth]{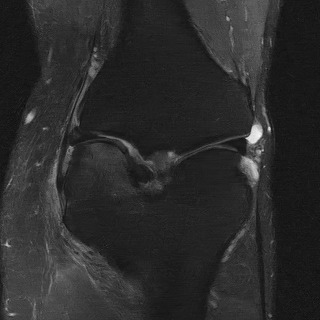}};
\spy on (1.6,1.15) in node [right] at (-0.005,-1.1);
\end{tikzpicture}
%  \caption*{\footnotesize Mask-CRR}
\end{subfigure}%
\hfill
\begin{subfigure}[t]{.123\textwidth}
\begin{tikzpicture}[spy using outlines=
{rectangle,white,magnification=4.5,size=2.115cm, connect spies}]
\node[anchor=south west,inner sep=0]  at (0,0) {\includegraphics[width=\linewidth]{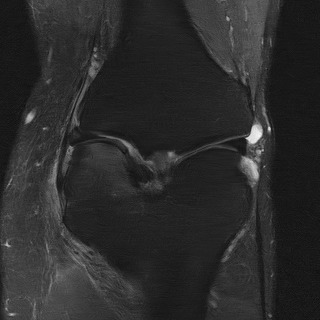}};
\spy on (1.6,1.15) in node [right] at (-0.005,-1.1);
\end{tikzpicture}
%  \caption*{\footnotesize Mask-WCRR}
\end{subfigure}%
\hfill
\begin{subfigure}[t]{.123\textwidth}
\begin{tikzpicture}[spy using outlines=
{rectangle,white,magnification=4.5,size=2.115cm, connect spies}]
\node[anchor=south west,inner sep=0]  at (0,0) {\includegraphics[width=\linewidth]{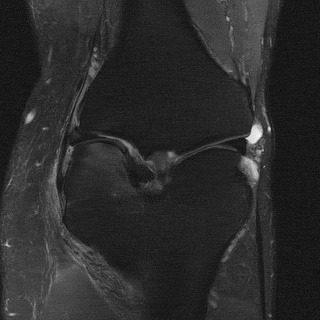}};
\spy on (1.6,1.15) in node [right] at (-0.005,-1.1);
\end{tikzpicture}
%  \caption*{\footnotesize Prox-DRUNet}
\end{subfigure}%
\hfill
\begin{subfigure}[t]{.123\textwidth}
\begin{tikzpicture}[spy using outlines=
{rectangle,white,magnification=4.5,size=2.115cm, connect spies}]
\node[anchor=south west,inner sep=0]  at (0,0) {\includegraphics[width=\linewidth]{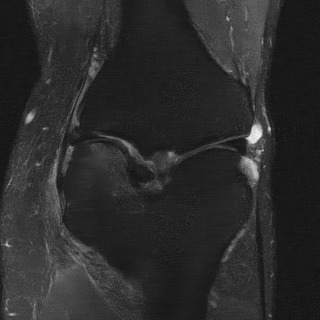}};
\spy on (1.6,1.15) in node [right] at (-0.005,-1.1);
\end{tikzpicture}
%  \caption*{\footnotesize patchNR}
\end{subfigure}%

\begin{subfigure}[t]{.123\textwidth}
\includegraphics[width=\linewidth]{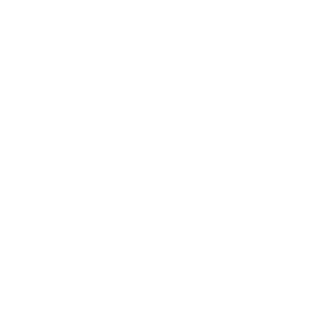}
  \caption*{\footnotesize GT}
\end{subfigure}%
\hfill
\begin{subfigure}[t]{.123\textwidth}
\includegraphics[width=\linewidth]{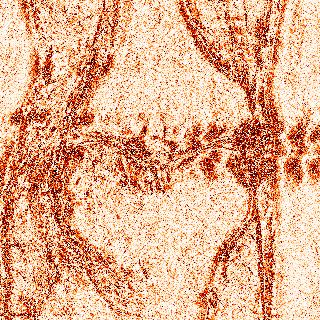}
  \caption*{\footnotesize Zero-filled}
\end{subfigure}%
\hfill
\begin{subfigure}[t]{.123\textwidth}
\includegraphics[width=\linewidth]{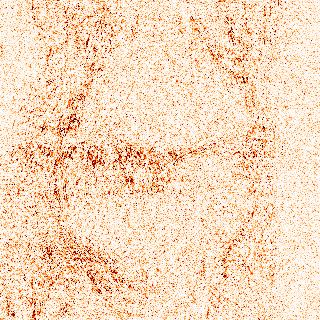}
  \caption*{\footnotesize CRR}
\end{subfigure}%
\hfill
\begin{subfigure}[t]{.123\textwidth}
\includegraphics[width=\linewidth]{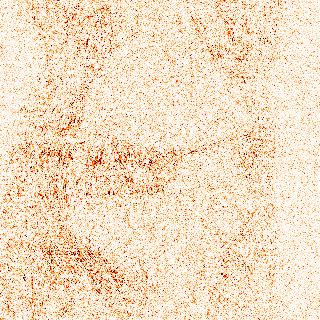}
  \caption*{\footnotesize WCRR}
\end{subfigure}%
\hfill
\begin{subfigure}[t]{.123\textwidth}
\includegraphics[width=\linewidth]{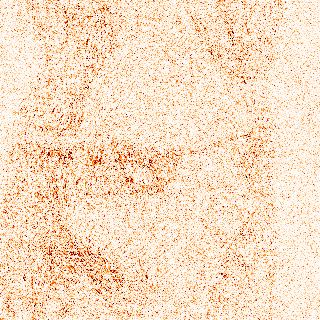}
  \caption*{\footnotesize CRR-Mask}
\end{subfigure}%
\hfill
\begin{subfigure}[t]{.123\textwidth}
\includegraphics[width=\linewidth]{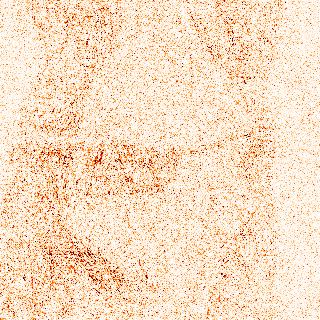}
  \caption*{\footnotesize WCRR-Mask}
\end{subfigure}%
\hfill
\begin{subfigure}[t]{.123\textwidth}
\includegraphics[width=\linewidth]{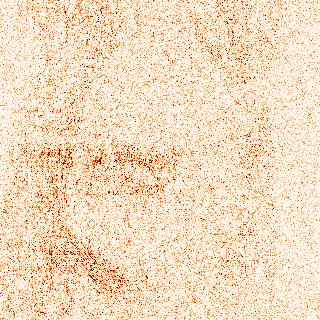}
  \caption*{\footnotesize Prox-DRUNet}
\end{subfigure}%
\hfill
\begin{subfigure}[t]{.123\textwidth}
\includegraphics[width=\linewidth]{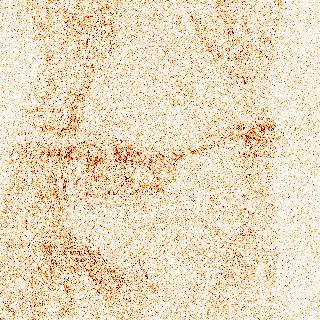}
  \caption*{\footnotesize patchNR}
\end{subfigure}%

\caption{8-fold multi-coil MRI on PDFS data set.
The white box marks the zoomed area.
\textit{Top}: full image; \textit{middle}: zoomed-in part; \textit{bottom}: error.
} \label{fig:MRI_comparison_pdfs8_appendix}
\end{figure}

%------------------------------------------------------------------------------------

\begin{figure}[t]
\centering
\begin{subfigure}[t]{.123\textwidth}  
\begin{tikzpicture}[spy using outlines=
{rectangle,white,magnification=10,size=2.115cm, connect spies}]
\node[anchor=south west,inner sep=0]  at (0,0) {\includegraphics[width=\linewidth]{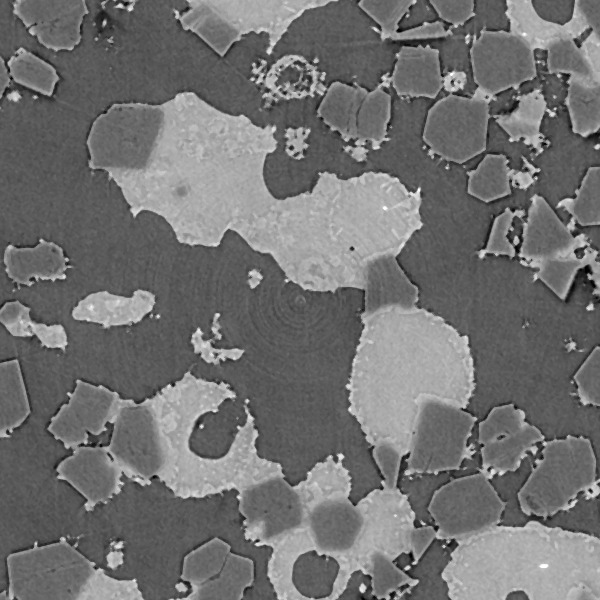}};
\spy on (.33,.55) in node [right] at (-0.005,-1.1);
\end{tikzpicture}
\caption*{\footnotesize HR}
\end{subfigure}%
\hfill
\begin{subfigure}[t]{.123\textwidth}
\begin{tikzpicture}[spy using outlines=
{rectangle,white,magnification=10,size=2.115cm, connect spies}]
\node[anchor=south west,inner sep=0]  at (0,0) {\includegraphics[width=\linewidth]{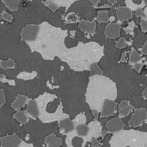}};
\spy on (.33,.55) in node [right] at (-0.005,-1.1);
\end{tikzpicture}
  \caption*{\footnotesize LR}
\end{subfigure}%
\hfill
\begin{subfigure}[t]{.123\textwidth}
\begin{tikzpicture}[spy using outlines=
{rectangle,white,magnification=10,size=2.115cm, connect spies}]
\node[anchor=south west,inner sep=0]  at (0,0) {\includegraphics[width=\linewidth]{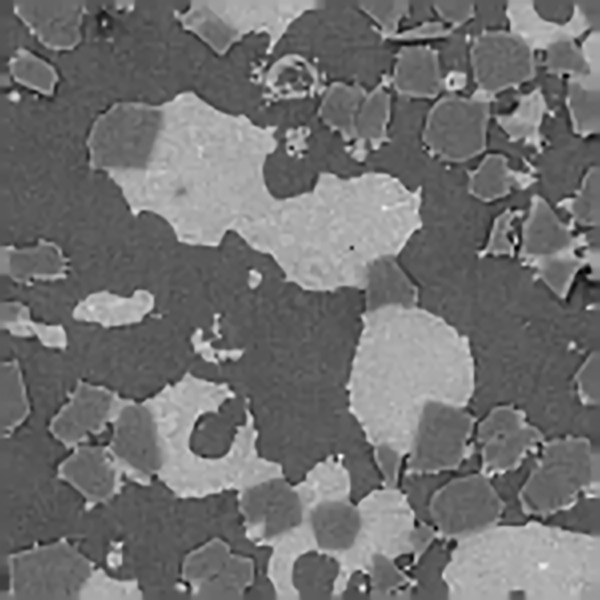}};
\spy on (.33,.55) in node [right] at (-0.005,-1.1);
\end{tikzpicture}
  \caption*{\footnotesize CRR}
\end{subfigure}%
\hfill
\begin{subfigure}[t]{.123\textwidth}
\begin{tikzpicture}[spy using outlines=
{rectangle,white,magnification=10,size=2.115cm, connect spies}]
\node[anchor=south west,inner sep=0]  at (0,0) {\includegraphics[width=\linewidth]{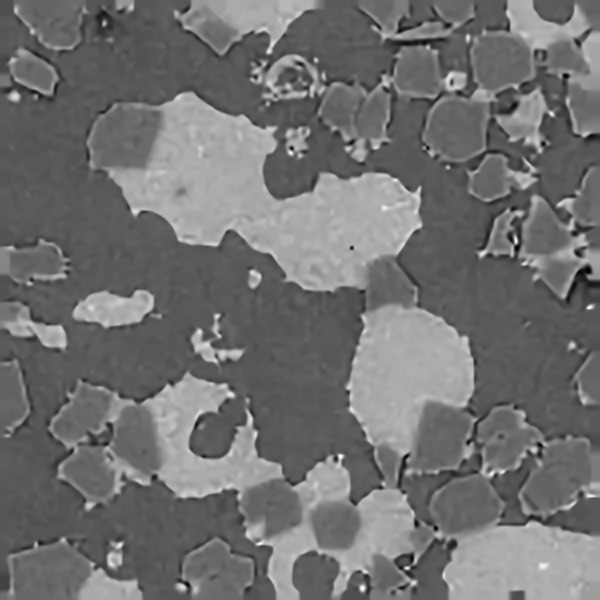}};
\spy on (.33,.55) in node [right] at (-0.005,-1.1);
\end{tikzpicture}
  \caption*{\footnotesize WCRR}
\end{subfigure}%
\hfill
\begin{subfigure}[t]{.123\textwidth}
\begin{tikzpicture}[spy using outlines=
{rectangle,white,magnification=10,size=2.115cm, connect spies}]
\node[anchor=south west,inner sep=0]  at (0,0) {\includegraphics[width=\linewidth]{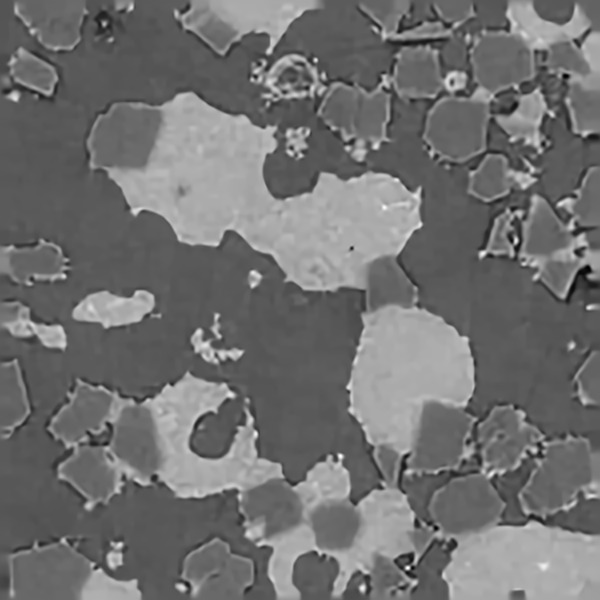}};
\spy on (.33,.55) in node [right] at (-0.005,-1.1);
\end{tikzpicture}
  \caption*{\footnotesize Mask-CRR}
\end{subfigure}%
\hfill
\begin{subfigure}[t]{.123\textwidth}
\begin{tikzpicture}[spy using outlines=
{rectangle,white,magnification=10,size=2.115cm, connect spies}]
\node[anchor=south west,inner sep=0]  at (0,0) {\includegraphics[width=\linewidth]{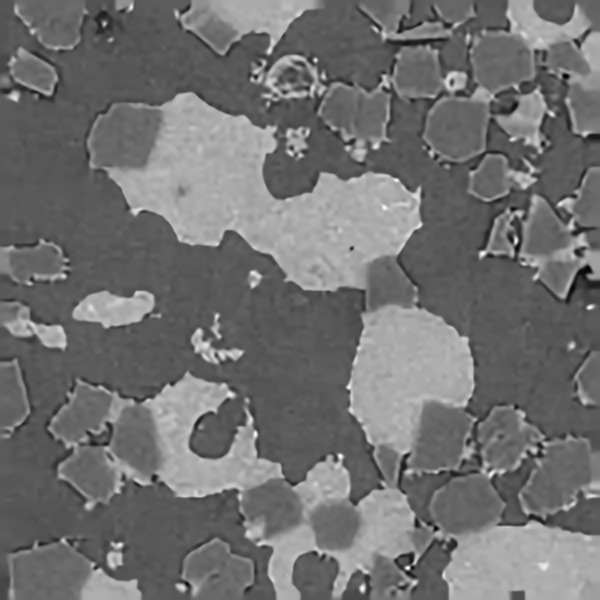}};
\spy on (.33,.55) in node [right] at (-0.005,-1.1);
\end{tikzpicture}
  \caption*{\footnotesize Mask-WCRR}
\end{subfigure}%
\hfill
\begin{subfigure}[t]{.123\textwidth}
\begin{tikzpicture}[spy using outlines=
{rectangle,white,magnification=10,size=2.115cm, connect spies}]
\node[anchor=south west,inner sep=0]  at (0,0) {\includegraphics[width=\linewidth]{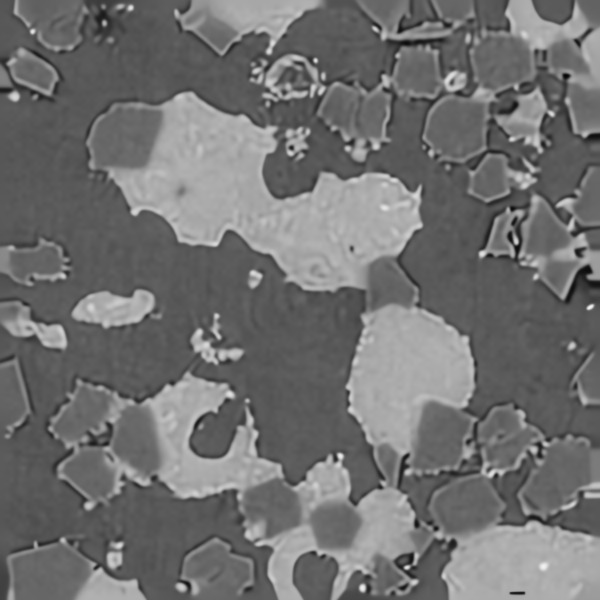}};
\spy on (.33,.55) in node [right] at (-0.005,-1.1);
\end{tikzpicture}
  \caption*{\footnotesize Prox-DRUNet}
\end{subfigure}%
\hfill
\begin{subfigure}[t]{.123\textwidth}
\begin{tikzpicture}[spy using outlines=
{rectangle,white,magnification=10,size=2.115cm, connect spies}]
\node[anchor=south west,inner sep=0]  at (0,0) {\includegraphics[width=\linewidth]{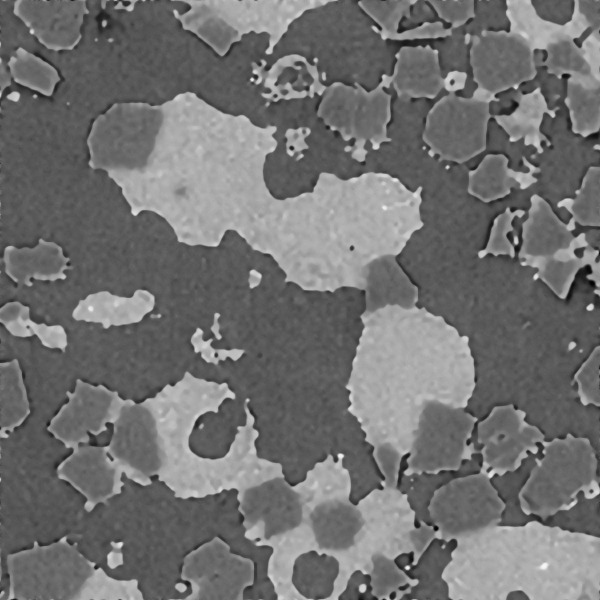}};
\spy on (.33,.55) in node [right] at (-0.005,-1.1);
\end{tikzpicture}
  \caption*{\footnotesize patchNR}
\end{subfigure}%

\begin{subfigure}[t]{.123\textwidth}  
\begin{tikzpicture}[spy using outlines=
{rectangle,white,magnification=10,size=2.115cm, connect spies}]
\node[anchor=south west,inner sep=0]  at (0,0) {\includegraphics[width=\linewidth]{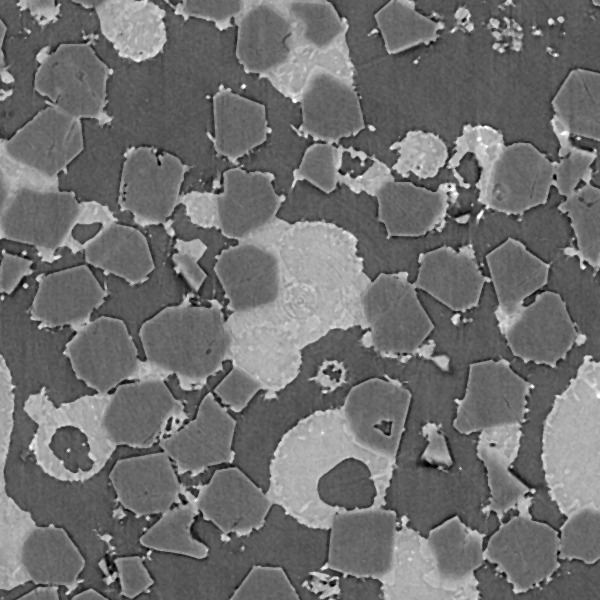}};
\spy on (.6,.59) in node [right] at (-0.005,-1.1);
\end{tikzpicture}
\caption*{\footnotesize HR}
\end{subfigure}%
\hfill
\begin{subfigure}[t]{.123\textwidth}
\begin{tikzpicture}[spy using outlines=
{rectangle,white,magnification=10,size=2.115cm, connect spies}]
\node[anchor=south west,inner sep=0]  at (0,0) {\includegraphics[width=\linewidth]{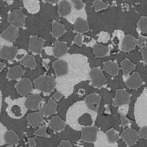}};
\spy on (.6,.59) in node [right] at (-0.005,-1.1);
\end{tikzpicture}
  \caption*{\footnotesize LR}
\end{subfigure}%
\hfill
\begin{subfigure}[t]{.123\textwidth}
\begin{tikzpicture}[spy using outlines=
{rectangle,white,magnification=10,size=2.115cm, connect spies}]
\node[anchor=south west,inner sep=0]  at (0,0) {\includegraphics[width=\linewidth]{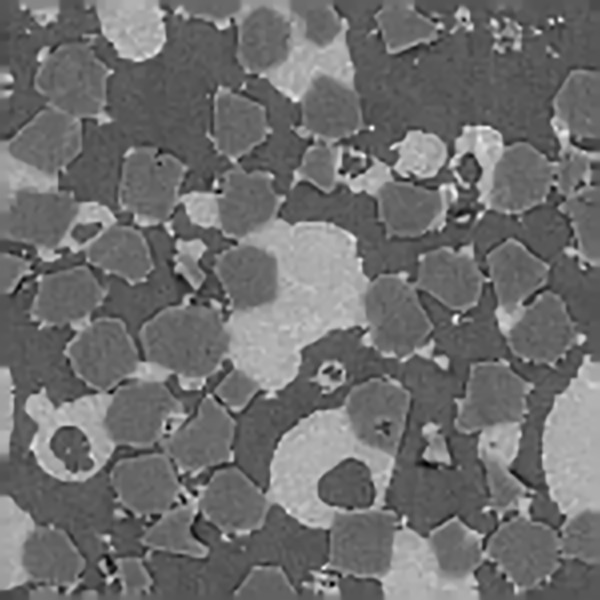}};
\spy on (.6,.59) in node [right] at (-0.005,-1.1);
\end{tikzpicture}
  \caption*{\footnotesize CRR}
\end{subfigure}%
\hfill
\begin{subfigure}[t]{.123\textwidth}
\begin{tikzpicture}[spy using outlines=
{rectangle,white,magnification=10,size=2.115cm, connect spies}]
\node[anchor=south west,inner sep=0]  at (0,0) {\includegraphics[width=\linewidth]{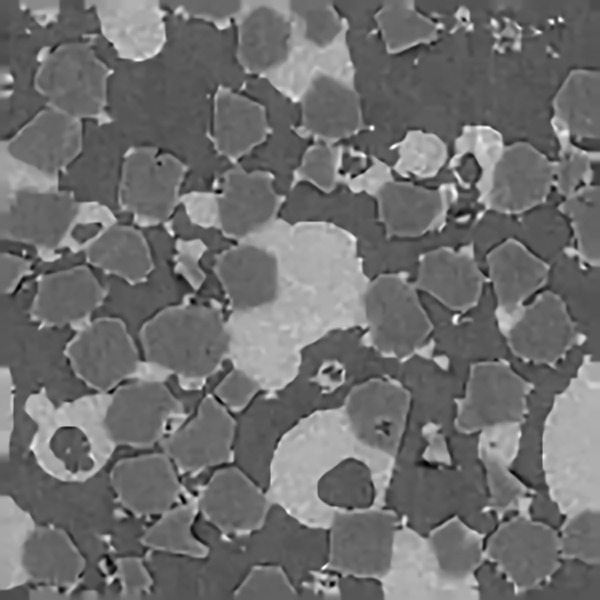}};
\spy on (.6,.59) in node [right] at (-0.005,-1.1);
\end{tikzpicture}
  \caption*{\footnotesize WCRR}
\end{subfigure}%
\hfill
\begin{subfigure}[t]{.123\textwidth}
\begin{tikzpicture}[spy using outlines=
{rectangle,white,magnification=10,size=2.115cm, connect spies}]
\node[anchor=south west,inner sep=0]  at (0,0) {\includegraphics[width=\linewidth]{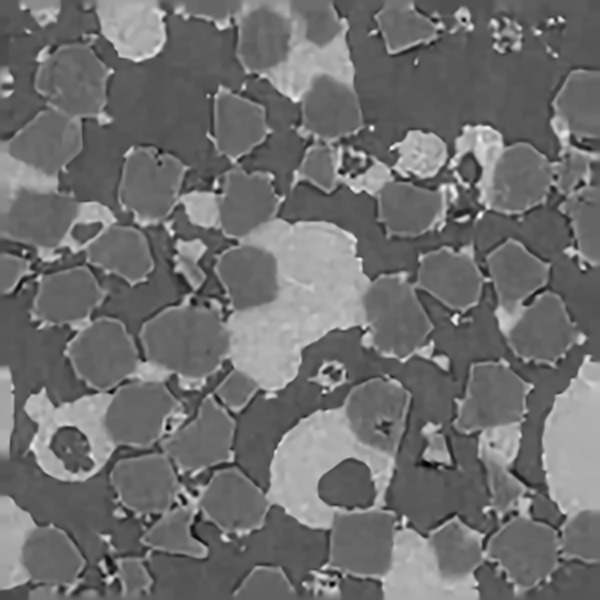}};
\spy on (.6,.59) in node [right] at (-0.005,-1.1);
\end{tikzpicture}
  \caption*{\footnotesize Mask-CRR}
\end{subfigure}%
\hfill
\begin{subfigure}[t]{.123\textwidth}
\begin{tikzpicture}[spy using outlines=
{rectangle,white,magnification=10,size=2.115cm, connect spies}]
\node[anchor=south west,inner sep=0]  at (0,0) {\includegraphics[width=\linewidth]{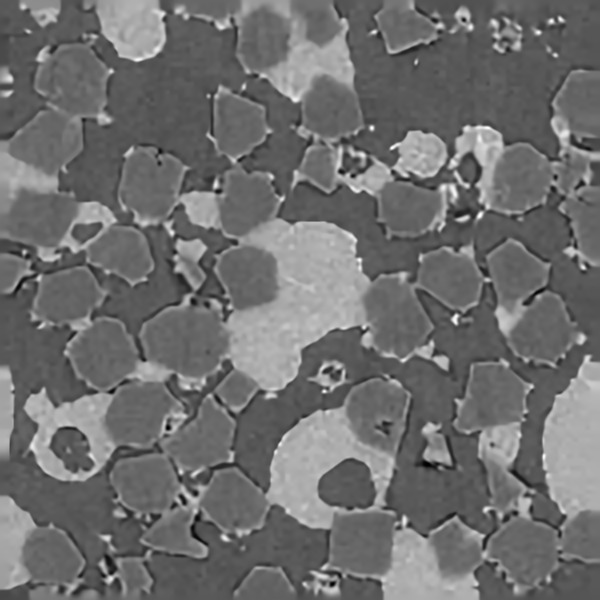}};
\spy on (.6,.59) in node [right] at (-0.005,-1.1);
\end{tikzpicture}
  \caption*{\footnotesize Mask-WCRR}
\end{subfigure}%
\hfill
\begin{subfigure}[t]{.123\textwidth}
\begin{tikzpicture}[spy using outlines=
{rectangle,white,magnification=10,size=2.115cm, connect spies}]
\node[anchor=south west,inner sep=0]  at (0,0) {\includegraphics[width=\linewidth]{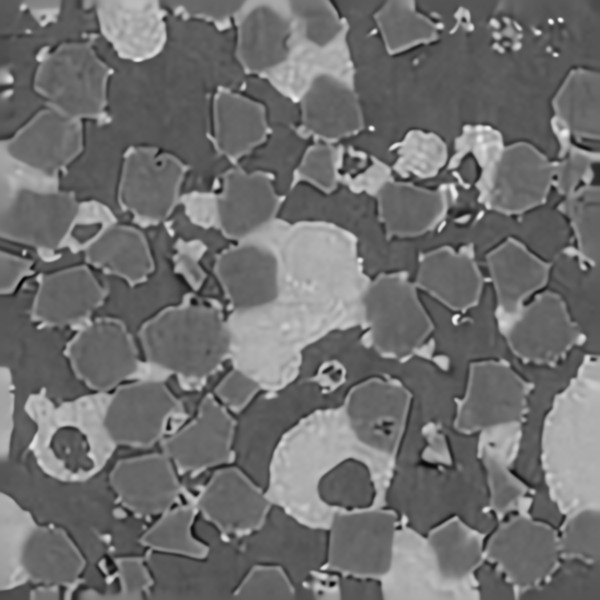}};
\spy on (.6,.59) in node [right] at (-0.005,-1.1);
\end{tikzpicture}
  \caption*{\footnotesize Prox-DRUNet}
\end{subfigure}%
\hfill
\begin{subfigure}[t]{.123\textwidth}
\begin{tikzpicture}[spy using outlines=
{rectangle,white,magnification=10,size=2.115cm, connect spies}]
\node[anchor=south west,inner sep=0]  at (0,0) {\includegraphics[width=\linewidth]{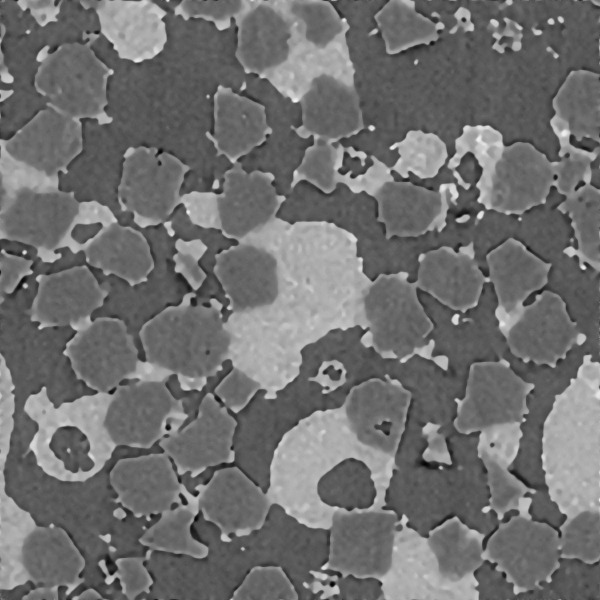}};
\spy on (.6,.59) in node [right] at (-0.005,-1.1);
\end{tikzpicture}
  \caption*{\footnotesize patchNR}
\end{subfigure}%

\begin{subfigure}[t]{.123\textwidth}  
\begin{tikzpicture}[spy using outlines=
{rectangle,white,magnification=10,size=2.115cm, connect spies}]
\node[anchor=south west,inner sep=0]  at (0,0) {\includegraphics[width=\linewidth]{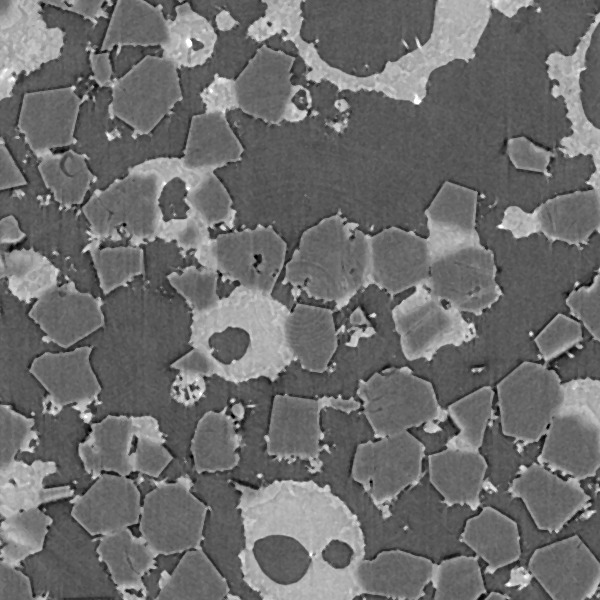}};
\spy on (.15,.53) in node [right] at (-0.005,-1.1);
\end{tikzpicture}
\caption*{\footnotesize HR}
\end{subfigure}%
\hfill
\begin{subfigure}[t]{.123\textwidth}
\begin{tikzpicture}[spy using outlines=
{rectangle,white,magnification=10,size=2.115cm, connect spies}]
\node[anchor=south west,inner sep=0]  at (0,0) {\includegraphics[width=\linewidth]{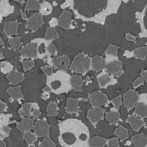}};
\spy on (.15,.53) in node [right] at (-0.005,-1.1);
\end{tikzpicture}
  \caption*{\footnotesize LR}
\end{subfigure}%
\hfill
\begin{subfigure}[t]{.123\textwidth}
\begin{tikzpicture}[spy using outlines=
{rectangle,white,magnification=10,size=2.115cm, connect spies}]
\node[anchor=south west,inner sep=0]  at (0,0) {\includegraphics[width=\linewidth]{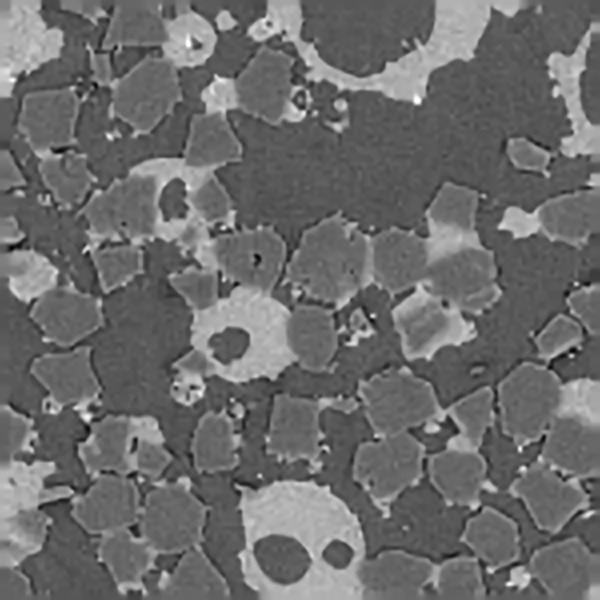}};
\spy on (.15,.53) in node [right] at (-0.005,-1.1);
\end{tikzpicture}
  \caption*{\footnotesize CRR}
\end{subfigure}%
\hfill
\begin{subfigure}[t]{.123\textwidth}
\begin{tikzpicture}[spy using outlines=
{rectangle,white,magnification=10,size=2.115cm, connect spies}]
\node[anchor=south west,inner sep=0]  at (0,0) {\includegraphics[width=\linewidth]{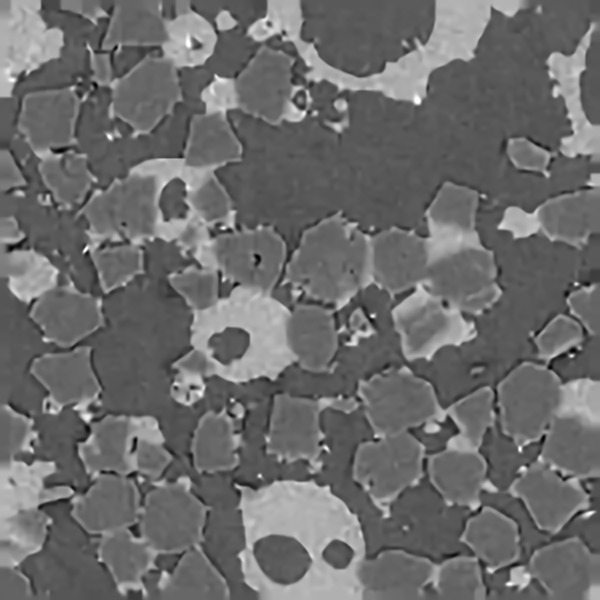}};
\spy on (.15,.53) in node [right] at (-0.005,-1.1);
\end{tikzpicture}
  \caption*{\footnotesize WCRR}
\end{subfigure}%
\hfill
\begin{subfigure}[t]{.123\textwidth}
\begin{tikzpicture}[spy using outlines=
{rectangle,white,magnification=10,size=2.115cm, connect spies}]
\node[anchor=south west,inner sep=0]  at (0,0) {\includegraphics[width=\linewidth]{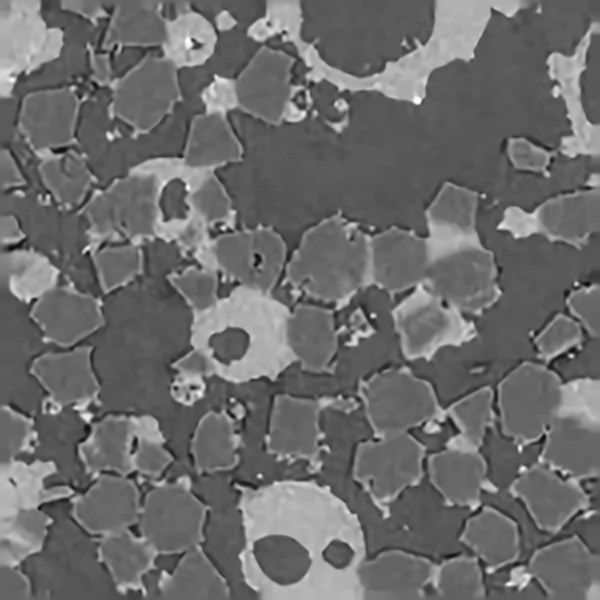}};
\spy on (.15,.53) in node [right] at (-0.005,-1.1);
\end{tikzpicture}
  \caption*{\footnotesize Mask-CRR}
\end{subfigure}%
\hfill
\begin{subfigure}[t]{.123\textwidth}
\begin{tikzpicture}[spy using outlines=
{rectangle,white,magnification=10,size=2.115cm, connect spies}]
\node[anchor=south west,inner sep=0]  at (0,0) {\includegraphics[width=\linewidth]{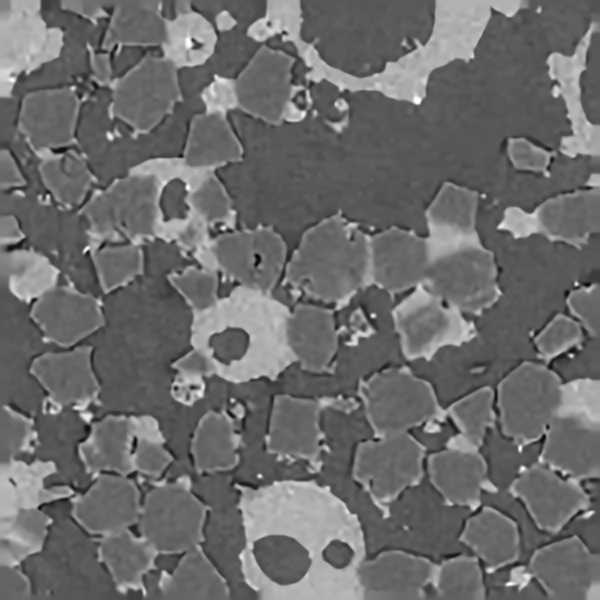}};
\spy on (.15,.53) in node [right] at (-0.005,-1.1);
\end{tikzpicture}
  \caption*{\footnotesize Mask-WCRR}
\end{subfigure}%
\hfill
\begin{subfigure}[t]{.123\textwidth}
\begin{tikzpicture}[spy using outlines=
{rectangle,white,magnification=10,size=2.115cm, connect spies}]
\node[anchor=south west,inner sep=0]  at (0,0) {\includegraphics[width=\linewidth]{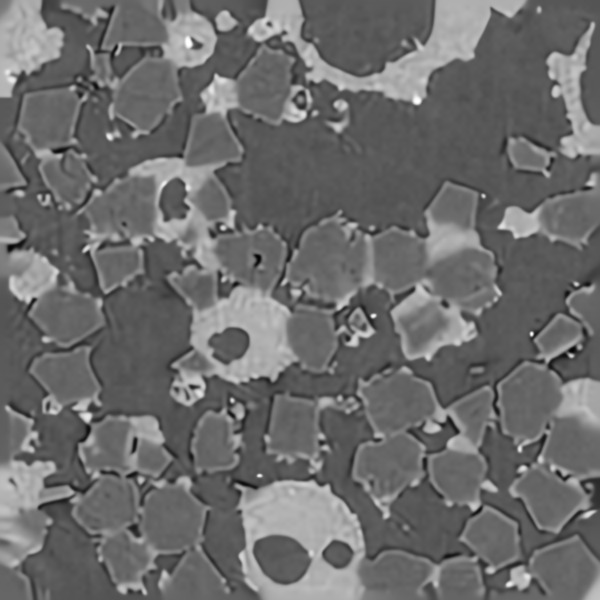}};
\spy on (.15,.53) in node [right] at (-0.005,-1.1);
\end{tikzpicture}
  \caption*{\footnotesize Prox-DRUNet}
\end{subfigure}%
\hfill
\begin{subfigure}[t]{.123\textwidth}
\begin{tikzpicture}[spy using outlines=
{rectangle,white,magnification=10,size=2.115cm, connect spies}]
\node[anchor=south west,inner sep=0]  at (0,0) {\includegraphics[width=\linewidth]{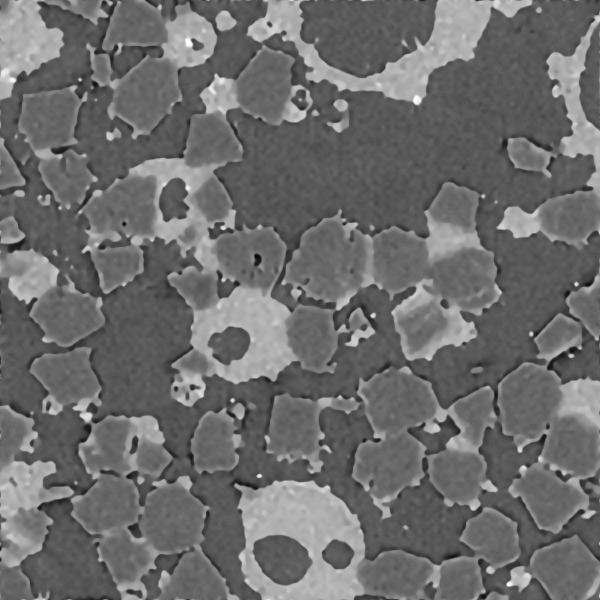}};
\spy on (.15,.53) in node [right] at (-0.005,-1.1);
\end{tikzpicture}
  \caption*{\footnotesize patchNR}
\end{subfigure}%

\caption{Superresolution of material microstructures.
The white box marks the zoomed area.
\textit{Top}: full image; \textit{bottom}: zoomed-in part.
} \label{fig:SiC_comparison_appendix}
\end{figure}
\end{document}